\documentclass[12pt]{amsart}
\usepackage{amsfonts, amsthm, amssymb, amsmath}
\usepackage{hyperref}
\usepackage[square, numbers, longnamesfirst]{natbib}
\usepackage[capitalise]{cleveref}
\hypersetup{colorlinks=false}

\usepackage{color}
\usepackage{tikz-cd}
\usepackage{physics}

\usepackage[group-separator={,}]{siunitx}

\newtheorem{theorem}{Theorem}[section]
\newtheorem{corollary}[theorem]{Corollary}
\newtheorem{proposition}[theorem]{Proposition}
\newtheorem{lemma}[theorem]{Lemma}

\theoremstyle{remark}
\newtheorem{remark}[theorem]{Remark}

\theoremstyle{definition}
\newtheorem{defi}[theorem]{Definition}
\def \tr {\operatorname{tr}}
\def \Diag {\operatorname{Diag}}
\def \Frob {\operatorname{Frob}}

\def \sign {\operatorname{sgn}}
\def \Sym {\operatorname{Sym}}
\def \sgn {\operatorname{sgn}}
\def \Ind {\operatorname{Ind}}

\def \orho {\overline{\rho}}
\newcommand{\Hom}{\operatorname{Hom}}
\renewcommand{\mod}{\textrm{ mod }}
\newcommand{\lm}{b}
\newcommand{\EFT}{\operatorname{EFT}}
\newcommand{\WF}{\operatorname{WF}}

\allowdisplaybreaks

\begin{document}

\author{Will Sawin}

\title[Sums over progressions in function fields]{Square-root cancellation for sums of factorization functions over squarefree progressions in $\mathbb F_q[t]$}

\begin{abstract} We prove estimates for the level of distribution of the M\"obius function, von Mangoldt function, and divisor functions in squarefree progressions in the ring of polynomials over a finite field. Each level of distribution converges to $1$ as $q$ goes to $\infty$, and the power savings converges to square-root cancellation as $q$ goes to $\infty$. These results in fact apply to a more general class of functions, the factorization functions, that includes these three. The divisor estimates have applications to the moments of $L$-functions, and the von Mangoldt estimate to one-level densities. 

\end{abstract}

\maketitle

\section{Introduction}

Let $\mathbb F_q$ be a finite field with $q$ elements and let $\mathbb F_q[t]$ be the ring of polynomials in one variable over $t$.  Let $\mathbb F_q[t]^+$ be the set of monic polynomials in $\mathbb F_q[t]$. For a natural number $n$, let $\mathcal M_n$ be the set of monic polynomials of degree $n$ in $\mathbb F_q[t]$, a set of size $q^n$. 

We can think of these objects as analogous to basic concepts in classical number theory -- $\mathbb F_q[t]$ is analogous to the integers, $\mathbb F_q[t]^+$ is analogous to the positive integers, and $\mathcal M_n$ is analogous to the set of integers in an interval $[X,2X]$ for $X \approx q^n$.

The main results in this paper will be analogous to \emph{level of distribution} results in number theory, for sums of natural arithmetic functions including the M\"{o}bius, von Mangoldt, and divisor functions, over arithmetic progressions to squarefree moduli.

We refer to irreducible monic polynomials as \emph{prime} polynomials. We define the function field analogues of the M\"{o}bius function, von Mangoldt function, and divisor functions as follows.
 \[ \mu(f) = \begin{cases} (-1)^r & f = \prod_{i=1}^r \pi_i\textrm{ with }\pi_i \textrm{ distinct primes}\\ 0 &\textrm{otherwise} \end{cases}\]
\[ \Lambda(f) = \begin{cases}  \deg \pi & f=\pi^r, \pi \textrm{ prime, } r \neq 0 \\ 0 & \textrm{otherwise} \end{cases}\]
\[ d_k(f) = \Bigl|  \Bigl\{ f_1,\dots f_k \in \mathbb F_q[t]^+ \mid f=  \prod_{i=1}^k f_i \Bigr\} \Bigr|  \]

\begin{theorem}[\cref{mobius-2}]\label{mobius-intro} Let $0<\theta<1 $ be a real number and let \[ q >  \frac{e^2 (1+\theta)^2 }{(1-\theta)^2}  \] be a prime power. There exists $\delta>0$ depending only on $q,\theta$ such that for any natural numbers $n,m$ with $m \leq \theta n$, squarefree $g \in \mathcal M_m$, and $a \in (\mathbb F_q[t]/g)^\times $,
\[ \sum_{\substack{  f \in \mathcal M_n\\ f \equiv a \mod g}} \mu(f) =  O \left( (q^{n-m})^{1-\delta} \right) \]
where the implicit constant depends only on $q, \theta,\delta$. \end{theorem}

In particular, we obtain a level of distribution $\theta> \frac{1}{2}$ as soon as $q > \lfloor 4 e^2 \rfloor =29$, and we obtain a level of distribution arbitrarily close to $1$ as long as $q$ is sufficiently large.

\citet[Theorem 1.7]{Sawin-Shusterman} proved a similar result, withou the squarefree moduli assumption, under the stricter condition $q> \frac{ (1+\theta)^2  p^2 e^2}{(1-\theta)^2} $ where $q$ is a power of a prime $p$, for $p>2$. In particular \citep{Sawin-Shusterman} proved level of distribution going to $1$ as $q$ goes to $\infty$ for fixed $p$, but not necessarily if $q$ goes to $\infty$.

We can take the power savings $\delta$ in \cref{mobius-intro} to be $\frac{1}{2} - \log_q \left( \frac{e (1+\theta) } { 1-\theta} \right)$ (\cref{mobius-2}), which goes to $\frac{1}{2}$ as $q$ goes to $\infty$ with fixed $\theta$. In other words, we obtain as close to full square-root cancellation as desired, as long as $q$ is sufficiently large. The same limiting behavior of $\delta$ holds for \cref{primes-intro} and \cref{divisor-intro}, below.

 \begin{theorem}[\cref{primes-2}]\label{primes-intro} Let $0<\theta<1 $ be a real number and let \[ q >  4^{ \frac{1}{1-\theta}  }  \left( \frac{1}{ 1-\theta} - \frac{1}{2} \right)^2e^2 \] be a prime power. There exists $\delta$ depending only on $q,\theta$ such that for any natural numbers $n,m$ with $m \leq \theta n $,  squarefree $g \in \mathcal M_m$, and $a \in (\mathbb F_q[t]/g)^\times $,
\[ \Bigl | \sum_{ \substack{ f \in \mathcal M_n \\  f \equiv a \mod g }}  \Lambda(f)  -  \frac{q^n}{ \phi(g) }  \Bigr|  =  O \left( (q^{n-m})^{1-\delta} \right) \]
where the implicit constant depends only on $q, \theta, \delta$.
\end{theorem}

In particular, we obtain a level of distribution $\theta> \frac{1}{2}$ as soon as $q> \lfloor 36 e^2 \rfloor  =266 $.

A result for level of distribution of the primes was proven in \citep[Theorem 1.9]{Sawin-Shusterman} (also to squarefree moduli, though it is likely possible to extend it to non-squarefree moduli with additional analytic effort). The result here is stronger in multiple directions --  we remove the dependence on $p$ in the level and in the (hidden) power savings, we obtain a level of distribution $>1/2$ for much smaller values of $q$, and our level of distribution approaches $1$ as $q$ goes to $\infty$, while \citep[Theorem 1.8]{Sawin-Shusterman} obtain a level of distribution at most $\frac{1}{2} + \frac{1}{126}$.

For the analogous problem over the integers, a great number of similar results have been proven with averages on the modulus $g$. There are different results depending on whether the residue class is fixed, one averages against well-factorizable weights, or if other nice conditions hold. Since none of these restrictions exist in our setting (though we do have the restriction to squarefree moduli), we compare against the result obtaining the best level of distribution, which was $\theta =\frac{3}{5}$  obtained by \citet{MaynardLevelII}, for triply well-factorizable weights, which we improve on for $q>\lfloor  2^7 e^2\rfloor = 945$.

\begin{theorem}[\cref{divisor-2}]\label{divisor-intro}
Let $k>1$ be a natural number, $0<\theta<1 $ be a real number and let \[ q>  \frac{e^2 (k-1)^2 }{ (1-\theta)^2 }   k^{ \frac{2\theta}{1-\theta}} \]
 be a prime power. There exists $\delta>0$ depending only on $q,k,\theta$ such that for any natural numbers $n,m$ with $m \leq \theta n$, squarefree $g \in \mathcal M_m$, and $a \in (\mathbb F_q[t]/g)^\times $,
\[ \sum_{\substack{  f \in \mathcal M_n\\ f \equiv a \mod g}} d_k(f) = \frac{1}{\phi(g) }\sum_{\substack{  f \in \mathcal M_n\\  \gcd(f,g)}} d_k(f)    + O \left( (q^{n-m})^{1-\delta} \right) \]
where the implicit constant depends only on $q,k,\theta$.
\end{theorem} 

We are not aware of prior work on the level of distribution of the divisor function over function fields to compare to. Instead let us compare to some analogous levels proved over the integers. For $k=2$, we can do better than $\theta=\frac{2}{3}$ due to Hooley, Linnik, and Selberg, as soon as $q > \lfloor e^2 2^2 \cdot 2^4 \rfloor =472$. For $k=3$, we can do better than the value $\theta=\frac{12}{23}$ due to  \citet*{FKM2014}, as long as $q > \lfloor \frac{ e^2 \cdot 2^2 \cdot 12^2}{ \cdot 11^2} 3^{ \frac{24}{11}} \rfloor = 386$.

For any $k$, to obtain a level of distribution beyond $1/2$, we need $q > e^2 (k-1)^2 k^2$. Thus we do not obtain an analogue of results on the level of distribution of the divisor function valid for sufficiently large $k$, such as the recent work of \citet{FouvryRadziwill}.

%Note that the restriction $\frac{1}{k}< \theta$ is immaterial since one can obtain an exact estimate for $\theta \leq \frac{1}{k}$ by choosing whichever of $f_1,\dots, f_k$ has the highest degree

Another point of comparison to \cref{primes-intro} and \cref{divisor-intro} could be Brun--Titchmarsh type theorems, which give upper bounds, rather than asymptotics, for sums of arithmetic functions in progressions. The original work \cite{Brun-Titchmarsh} handled sums of the von Mangoldt function, and it was generalized by \citet{Shiu1980} to a class of multiplicative functions including the divisor function. Function field analogues, for the divisor function, were proven by \citet[\S6]{Tamam2014} and \citet[\S6]{Andrade2020}. Compared to \cref{primes-intro} and \cref{divisor-intro}, these give a weaker statement, but in a wider range. It would be interesting to prove a Brun--Titchmarsh-type bound in the level of generality of \cref{squarefree-bound-precise} below. 

The fact that our levels of distributions approach $1$ as $q $ goes to $\infty$ could be analogized with  the level of distribution obtained by \citet[Theorem, (6.3)]{E1996} for the indicator function of numbers whose sum of digits in a given base is congruent to a fixed value, which goes to $1$ as the base goes to $\infty$ -- we could compare $q$ to the base of a digit expansion.

In addition to level of distribution results, we have two results about $L$-functions.

We say a character $\chi\colon (\mathbb F_q[t]/g)^\times \to \mathbb C^\times$ is \emph{primitive} if it does not factor through $ (\mathbb F_q[t]/h)^\times$ for any proper divisor $h$ of $g$. We say $\chi$ is \emph{even} if it is trivial when restricted to $\mathbb F_q^\times$ and \emph{odd} otherwise.

 For $f$ a polynomial, write $|f|=q^{\deg f}$. For $\chi\colon (\mathbb F_q[t]/g)^\times \to \mathbb C^\times$ a primitive character, define the $L$-function
\[ L(s, \chi) = \sum_{ \substack{ f \in \mathbb F_q [t]^+ \\ \gcd(f,g)=1}} \chi(f) \abs{f}^{-s}. \]

First, we have a power savings results for certain moments of $L$-functions.

\begin{theorem}[\cref{moments-2}]\label{moments-intro} Let $k>1$ be a natural number and $q> k^2 2^{ 2k-2} $ a prime power. There exists $\delta>0$ such that for any squarefree monic polynomial $g \in \mathbb F_q[t]$ and $s\in \mathbb C$ with $\operatorname{Re} s \geq \frac{1}{2}$, we have   \[   \Biggl| \frac{1}{\phi^*(g)   } \sum_{\substack{ \chi \colon (\mathbb F_q[t]/g)^\times \to \mathbb C^\times \\ \textrm{primitive} }} L\left(s,\chi\right)^k   - 1   \Biggr |   = O ( \abs{g}^{-\delta} ) \] where $\phi^*(g)$ denotes the number of primitive characters mod $g$, with the implied constant depending only on $q, k, \delta$. In fact, we may take any $\delta < \frac{1}{2} - \log_q ( k 2^{k-1} )$. \end{theorem}

Note that this moment, unlike the most commonly studied moments of $L$-functions, does not have an absolute value.

For $k=1$, a similar estimate (over both $\mathbb F_q[t]$ and $\mathbb Z$) follows only from the functional equation, and for $k=2$ from the functional equation plus bounds for Kloosterman sums. The first nontrivial case, for $k=3$, was handled over $\mathbb Z$ by \citet{Zacharias} (for prime moduli, but allowing a character twist that we do not consider here). The $k=4$ case seems to be open over $\mathbb Z$. We can handle it over $\mathbb F_q[T]$ for $q> 1024$.

The method of proof of \cref{moments-intro} can likely be adapted to estimate moments twisted by an additional $\chi(a)$ factor or shifted moments involving $\prod_{i=1}^k L(s_i,\chi)$ for distinct $s_i$, but, for simplicity, we don't pursue these generalizatoins here.

Second, we have a result on the one-level density of the zeroes of $L$-functions. When $\chi$ is primitive and odd, all the zeroes of $L(s,\chi)$ lie on the critical line (i.e. have real part $1/2$).  The following lemma describes the distribution of the zeroes on the critical line, on average over $\chi$.

  \begin{theorem}[\cref{ol-2}]\label{ol-intro} Let $\lambda>1.07$ be a real number and let $q$ be a prime power with \[ q > \frac{ (\lambda+1)^4 e^4}{4} .\] 

 Let $\psi$ be a Schwartz function on $\mathbb R$ whose Fourier transform $\hat{\psi}$ is supported on $[-\lambda,\lambda]$. Then for a squarefree polynomial $g$ of degree $m$, we have
 \[ \sum_{\substack{  \chi \colon (\mathbb F_q[t] / g)^\times \to \mathbb C^\times  \\ \textrm{primitive} \\ \textrm{odd}}}  \sum_{\gamma_\chi}  \psi \Bigl( \frac{ m \log(q) \gamma_\chi}{2\pi} \Bigr ) = \left( \hat{\psi}(0) + o(1) \right)  \phi^{*, odd}(g)  \]
 where $\frac{1}{2} + i \gamma_\chi$ are the (infinitely many) roots of $L(s,\chi)$, $\phi^{*,odd}(g)$ is the number of odd primitive characters modulo $g$, and $o(1)$ goes to $0$ as $m$ goes to $\infty$ for fixed $\lambda, q ,\psi$.
 \end{theorem}
 
 This result fits into the Katz-Sarnak philosophy that densities of low-lying zeroes of $L$-functions in families, in the large conductor limit, should match predictions associated to scaling limits of random matrices  \cite{KS1,KS2, ILS}.  This philosophy was motivated by results in the function field setting, calculating the densities in the $q \to \infty$ limit. For this family, the $q\to\infty$ limit was calculated in \cite{PDCQKR}. Theorem \ref{ol-intro} represents progress towards the Katz-Sarnak conjectures outside the $q\to\infty$ limit.

Over $\mathbb Z$, a similar result was proven on average over the modulus $g$, with $\lambda = 2 + \frac{50}{ 1093}$, by \citet{DPR}. We can improve on this over $\mathbb F_q[t]$ for $q> \lfloor (3+ \frac{50}{ 1093})^4 e^4 / 4 \rfloor =1174$.  Without the average over the modulus, it should be possible to prove a result by elementary means for $\lambda <2$, which we can improve on for a similar range of $q$.

\subsection{Precise result}

We now explain the key definition of a factorization function on $\mathcal M_n$, associated to a representation $\rho$ of $S_n$.

\begin{defi}\label{fac-fun} Fix $\rho$ a (finite-dimensional) representation of $S_n$, not necessarily irreducible, defined over $\mathbb C$. 

For $f$ a polynomial of degree $n$ over $\overline{\mathbb F}_q$, let $V_f$ be the free $\mathbb C$-vector space generated by tuples $(a_1,\dots,a_n) \in \overline{\mathbb F}_q^n$ with $\prod_{i=1}^n (t-a_i)=f$. Then $V_f$ admits a natural $S_n$ action by permuting the $a_i$, as well as, if $f \in \mathbb F_q[t]$, an action of $\Frob_q$ by sending $(a_1,\dots,a_n) $ to $(a_1^q,\dots, a_n^q)$. These actions commute.

Let $\Frob_q$ act trivially on the space of $\rho$. This commutes with the $S_n$ action.

For $f\in \mathbb F_q[t]$, let \[F_\rho (f) = \tr (\Frob_q, (V_f \otimes \rho)^{S_n} )\] where the action of $\Frob_q$ on $(V_f \otimes \rho)^{S_n}$ is well-defined because the $\Frob_q$ and $S_n$-actions on $V_f \otimes \rho$ commute.  \end{defi}

Note that, because representations of $S_n$ are always self-dual, we have $(W \otimes \rho)^{S_n} = \operatorname{Hom}_{S_n}(W, \rho)$ for any representation $W$.  Interpreting $(V_f \otimes \rho)^{S_n} $, and other similar expressions appearing later, as a space of invariants or a Hom-space is a matter of preference.

Our main theorem, \cref{squarefree-bound-precise}, is an estimate for the sum of $F_\rho(f)$ over $f$ in an arithmetic progression, from which we can deduce the previous theorems by specializing to particular $\rho$. To state the estimate, we need to introduce some additional notation.

For a representation $\rho$ of $S_n$, and $\sgn$ the sign representation of $S_n$, let $V_{\rho}$ be the representation of $GL_{m-1}$ obtained by Schur-Weyl duality from $\rho \otimes \sgn$, i.e. \[ V_{\rho} =\Bigl( (\mathbb C^{m-1} )^{\otimes n } \otimes \rho \otimes \sgn \Bigr)^{S_n} .\] 

For $d$ from $0$ to $n$, let \[V^d_{\rho} = \Bigl( (\mathbb C^{m} )^{\otimes (n-d) } \otimes \rho \otimes \sgn_{S_{n-d}} \Bigr)^{S_{n-d} \times S_d } ,\] viewed as a representation of $GL_m$. 
Note that we restrict $\rho$ to $S_{n-d} \times S_d $ using the natural inclusion $S_{n-d} \times S_d \subseteq S_n$ so that we may take $S_{n-d} \times S_d$-invariants.

Denote by $\Diag(\lambda_1,\dots, \lambda_{m}) \in GL_{m}$ the diagonal matrix with diagonal entries $\lambda_1,\dots, \lambda_{m}$, and similarly for  $\Diag(\lambda_1,\dots, \lambda_{m-1}) \in GL_{m-1}$

Let
\[ C_1(\rho) =\ \sum_{d=0}^n (-1)^{d} \frac{\partial^m\tr \left( \Diag(\lambda_1,\dots, \lambda_m),  V^{d }_\rho \right ) }{\partial \lambda_1 \dots \partial \lambda_m }\Big|_{\lambda_1,\dots,\lambda_m=1} \ - \sum_{d=m}^{n} (-1)^{d-m}  \binom{d}{m} \dim V^{d }_\rho\]
and
\[C_2(\rho) = \frac{\partial^{m-1} \tr( \Diag(\lambda_1,\dots, \lambda_{m-1}, V_{\rho} ) ) }{ \partial \lambda_1 \dots \partial \lambda_{m-1}} \Big|_{\lambda_1,\dots,\lambda_{m-1}=1} .\]

The notation $|_{\lambda_1,\dots,\lambda_m=1} $ refers to, after taking the derivative, evaluating the expressions with all the $\lambda_i$ variables equal to $1$.

Note that $V_\rho, V^d_\rho, C_1(\rho),C_2(\rho)$ depend on $m$ and $n$ but we do not include those variables in the notation as, unlike $\rho$, they will not vary in the proof.

\begin{theorem}\label{squarefree-bound-precise} Let  $\mathbb F_q$ be a finite field. Let $g$ be a squarefree polynomial of degree $m$ over $\mathbb F_q$ and $a$ an invertible residue class mod $g$.  

Let $n\geq m$ be a natural number and let $\rho$ be a representation of $S_n$.  
Then
\[ \Bigl| \sum_{ \substack{ f \in \mathcal M_n\\  f \equiv a \mod g }}  F_\rho(f)  -\frac{1}{ \phi(g) }  \sum_{ \substack{ f \in \mathcal M_n \\ \gcd(f,g)=1 }}  F_\rho(f) \Bigr| \leq  2 ( C_1(\rho)  + C_2(\rho) \sqrt{q} ) q^{\frac{n-m}{2} }.\]
\end{theorem}

We will deduce from Theorem \ref{squarefree-bound-precise} all the results stated above, by explicitly estimating $C_1(\rho)$ and $C_2(\rho)$, in Section \ref{s-applications}. We expect there will be further applications using different representations $\rho$, so we hope that the arguments there demonstrate the methods needed to calculate the quantities $C_1,C_2$ in general. 

A similar result was proven for short interval sums in \cite[Proposition 4.3]{SawinShort}. The dependence on $q$ of the error term is similar in both results, however, \cref{squarefree-bound-precise} has much better dependence on $n$ and $m$, in particular for $\rho$ the sign representation and other representations corresponding to partitions of $n$ with many parts of size $1$. This improved dependence allows us to obtain asymptotic results for fixed $q$ for the divisor function, M\"obius function, and von Mangoldt function, while \cite{SawinShort} could only obtain them for the divisor function. The reason for this improved estimate is that, in this setting, we are able to apply the characteristic cycle, and associated strong Betti number bounds \cite{massey-formula}, rather than the bounds of \cite{KatzBetti}, which are weaker in this context.

\cref{squarefree-bound-precise} implies that the average value of $F_\rho$ on $f \equiv a \mod g$ and on $f$ coprime to $g$ converge to the same value as $q \to \infty$ for fixed $n,m$ as long as $n-m \geq 2$. For $n-m \geq 3$, this was obtained in \cite{BBSR}. We note that \cref{squarefree-bound-precise} gives savings a large power of $q$ when $n-m$ is sufficiently large, and that \cite{BBSR} is valid for arithmetic progressions to arbitrary moduli, not just squarefree moduli. 

Another prior work on sums like those in \cref{squarefree-bound-precise} as $q\to\infty$ is the work of Keating and Rudnick, later with Rodgers and Roditty-Gershon, which studied the variance of sums in progressions -- i.e. the average square over $a$ of the left side of \cref{squarefree-bound-precise}. These gave asymptotics for the variance of sums of the M\"obius function \cite{KRmobius}, the von Mangoldt function \cite{KRprimes}, the divisor function \cite{krrr}, and general $F_\rho$ \cite{Rodgers2018}. These asymptotics imply an average error of $q^{ \frac{n-m}{2}}$. This means that, in the large $q$ aspect, the worst-case bound of \cref{squarefree-bound-precise} loses only a factor of $O(q^{ \frac{1}{2}})$ compared to the known average value.

\subsection{Factorization types}\label{ss-ft}

For $f$ a monic polynomial over $\mathbb F_q$, expressed as $\prod_{i=1}^r \pi_i^{e_i}$ for distinct prime factors $\pi_i$ and exponents $e_i$, let the \emph{extended factorization type} of $f$, denoted $\omega_f$, be the multiset of pairs $(\deg \pi_i, e_i)$. This notion was introduced by Gorodetsky~\citep{Gorodetsky2020}.

We are interested in the distribution of the extended factorization type of a random polynomial $f$, or a random polynomial $f$ satisfying some congruence conditions. The point of this is that a wide variety of arithmetic functions can be expressed in terms of the extended factorization type, so their distribution is determined by the distribution of the extended factorization type. These include the M\"obius, divisor, and von Mangoldt functions mentioned earlier, variants of them like the Liouville function and the indicator function of the primes, and further functions like the number of prime factors and the maximum degree of a prime factor. They also include many more subtle functions. Some of these have been the subject of recent research such as the maximum over $n$ of the number of divisors of degree $n$, which was studied in \cite{FGK} as the $\mathbb F_q[t]$-analogue of the Hooley function over $\mathbb Z$ \cite{HooleyNumbers}.

For $f$ of degree $n$, $\omega_f$ lies in the set $\EFT_n$ of multisets $M$ of pairs  of positive integers $(a,b)$ such that $\sum_{ (a,b) \in M} ab=n$. (Here a multiset of pairs is a set of pairs where each pair is given a positive integer multiplicity, and summing over the multiset is summing over the set with each term multiplied by the multiplicity.)  We will compare distributions on extended factorization types by comparing measures on $\EFT_n$. To that end, for $f$ a polynomial, let $\delta_{\omega_f}$ be the Dirac delta measure on $\omega_f$. Then for $S$ a set of polynomials, \[\frac{1}{ \abs{S}} \sum_{ f\in S} \delta_{\omega_f}\] is a probability measure describing the distribution of a random polynomial in $S$. 

\begin{theorem}[\cref{eft-final}]\label{eft-intro} Let $0< \theta<1$ be a real number and $q$ a prime power. Let $n$ be a natural number, let $g$ be a squarefree polynomial of degree $m\leq \theta n $ over $\mathbb F_q$, and let $a$ be an invertible residue class mod $g$.  Then the total variation distance between the probability measures 
\[ \frac{1}{q^{n-m}}  \sum_{ \substack{ f \in \mathcal M_n \\ f \equiv a \mod g}} \delta_{\omega_f} \]
and
\[ \frac{1}{\phi(g) q^{n-m}}  \sum_{ \substack{ f \in \mathcal M_n \\ \gcd(f,g)=1 }} \delta_{\omega_f} \]
is at most
\[  \frac{1}{ \sqrt{2\pi} (4+e) e^{ (4+e) } } \frac {  L^{9/2}   e^{2L} }{  L^L} +o(1)  \]
where \[ L = \frac{1}{2}   \left( \frac{(2- 2 \theta) \sqrt{q} } { (1+2 \theta)  e} \right)^{ (1-\theta) }  \] and $o(1)$ goes to $0$ as $n$ goes to $\infty$ but may depend on $q,\theta$.
\end{theorem}

For $\theta < \frac{1}{2}$, the same result was proven by Gorodetsky \citep{Gorodetsky2020}, with error term going to $0$ as $n$ goes to $\infty$ with fixed $q$, using the Riemann Hypothesis. Thus \cref{eft-intro} is a beyond the Riemann Hypothesis analogue of \citep[Theorem 1.2]{Gorodetsky2020}. However, the cost for our result simultaneously going beyond Riemann and (unlike the results above) allowing an arbitrary factorization function is that our error term doesn't go to zero as $n$ goes to infinity, but only goes to a small fraction. Still, we can make this fraction arbitrarily small for a given $\theta$ by taking $q$ sufficiently large, and the decrease is \emph{exponential} in $q$.

For example, taking $\theta$ just above $1/2$, for $q> \num{5000000}$, the error term is $<0.1$, for $q>\num{20000000}$, the error term is $<0.001$,  and for $q> \num{100000000}$, the error term is $<q^{-1}$.

\subsection{Strategy of proof}

The proof of \cref{squarefree-bound-precise} relies on methods from sheaf cohomology, including the characteristic cycle, classical vanishing cycles theory, and properties of perverse sheaves. We begin by constructing a sheaf, $(sym_* \mathbb Q_\ell \otimes \rho)^{S_n}$, on the space of degree $n$ polynomials whose trace function (i.e. the trace of Frobenius on the stalk at a particular point) is $F_\rho$.  We restrict $(sym_* \mathbb Q_\ell \otimes \rho)^{S_n}$  to the space of degree $n$ polynomials prime to $g$, which maps to the space of invertible residue classes mod $g$ (a torus $T$). By the Lefschetz fixed point formula, the sum in a progressions of $F_\rho$ is the trace of Frobenius on the compactly supported cohomology of a fiber with coefficients in $(sym_* \mathbb Q_\ell \otimes \rho)^{S_n}$. The main goal of the paper will be to understand these compactly supported cohomology groups, and how they vary from fiber to fiber.

One fundamental tool to describe how cohomology varies between different fibers is vanishing cycles theory (see \cite{sawinvanishing}). Given any curve in the torus $T$ of invertible residue classes, vanishing cycles describes the difference between the cohomology of the fiber over a generic point in that curve and the cohomology of the fiber over any particular point as a sum of local contributions over singular points of the special fiber. Conveniently, there are only finitely many relevant singular points. This would imply that only the degree $n-m$ and $n+1-m$ cohomology groups can vary between the fibers, and the rest are fixed. However, this implication is only valid if the fibers are compact, as vanishing cycles requires a proper morphism. Thus, to apply this strategy, we must compactify.

In fact, we are able to find a very well-behaved compactification of this morphism (or, more precisely, the morphism $\pi$ from the space of polynomials together with a factorization into linear factors to the space of residue classes mod $g$). This compactification is not well-behaved because it is smooth (it is not) but rather because it has local models which are isomorphic to the product of a singular scheme with $T$, making its local behavior constant in $T$ in various senses. In particular, this implies that the vanishing cycles contribution at the boundary we added to compactify vanishes, and so we can indeed deduce that only the degree $n-m$ and $n+1-m$ cohomology groups vary between the fibers. Furthermore, this compactification, together with a perversity argument, allows us to deduce that the cohomology in higher degrees is independent of $a\in T$, and therefore the trace of Frobenius on this cohomology essentially matches the main term. So it remains to bound the trace of Frobenius on the cohomology groups in degrees $n-m$ and $n+1-m$. This compactification is constructed in \cref{sec-compactification} and the consequences of its existence are described in \cref{sec-structure}.

The main tool to control the trace of Frobenius on these cohomology groups is Deligne's Riemann Hypothesis, which implies that the eigenvalues of Frobenius on these cohomology groups in degree $i$ have size at most $q^{i/2}$. Since the trace of an operator is at most the largest eigenvalue times the dimension, to bound the trace it will suffice to bound the dimension of these cohomology groups. The constants $C_1(\rho)$ and $C_2(\rho)$ are our bounds for the cohomology in degree $n-m$ and $n+1-m$ respectively (more precisely, for the non-constant part of that cohomology). 

The cohomology in degree $n-m$ is bounded by an Euler characteristic argument in \cref{sec-Euler}. At the generic point of $T$, there can only be nonconstant cohomology in degree $n-m$, and so to calculate the rank in degree $n-m$, it suffices to calculate the Euler characteristic of the cohomology and then subtract off the Euler characteristic of the constant part. We apply the Lefschetz fixed point formula to several automorphisms without fixed points to deduce relations between Euler characteristics for different representations $\rho$, which give us enough information to compute all the Euler characteristics.

The cohomology in degree $n+1-m$ is bounded by a characteristic cycle argument in \cref{sec-Massey}. In \cite{massey-formula}, Betti number bounds were proven for the stalks of perverse sheaves in terms of certain local intersection-theoretic multiplicities of the characteristic cycle. The relevant cohomology groups are indeed the stalks of a certain perverse sheaf on $T$. We calculate exactly the characteristic cycle of this perverse sheaf and then estimate the relevant local intersection multiplicities, bounding them by global intersection numbers that can be explicitly computed.

To deduce results like Theorems \ref{mobius-intro}, \ref{primes-intro}, and \ref{divisor-intro} from \cref{squarefree-bound-precise}, we first identify a representation $\rho$ such that $F_\rho$ is our desired function (except, for the von Mangoldt function, we must use an alternating sum of $F_\rho$s). We then must evaluate $C_1(\rho)$ and $C_2(\rho)$ for this representation, which requires getting a formula for the traces of diagonal matrices acting on the representations $V_\rho$ and $V^d_\rho$.  It turns out that these traces fit nicely into generating functions in an auxiliary variable $u$ where the coefficient of $u^n$ arises from a representation of $S_n$.

The relevant formulas are somewhat familiar from classical number theory -- if we detect the condition $f \equiv a \mod g$ appearing in \cref{mobius-intro,primes-intro,divisor-intro} by Dirichlet characters, then the contribution of each character $\chi$ can be expressed in terms of its Dirichlet $L$-function (respectively by its inverse, logarithmic derivative, or $k$th power) and thus can be expressed in terms of the roots of the Dirichlet $L$-function. The contribution of a character $\chi$ will match the trace $\Diag( q^{-\alpha_1},\dots, q^{-\alpha_{m-1}})$ on $V_\rho$, for $\alpha_1,\dots, \alpha_{m-1}$ the roots of $L(s,\chi)$, and so we obtain the trace of a general diagonal matrix by replacing $q^{-\alpha_i}$ with $\lambda_i$. Similarly, the alternating sum in $d$ of the trace of $\Diag( q^{-\alpha_1},\dots, q^{-\alpha_{m}})$ on $V^d_\rho$, where $\alpha_1,\dots,\alpha_m$ are the trivial roots of the $L$-function of the trivial character, will match the contribution of the trivial character.

Given these generating function expressions, it is straightforward to differentiate with respect to the $\lambda_i$ variables, set the $\lambda_i$ to $1$, and extract the coefficient of a particular power of $u$, which will usually have a simple expression in binomial coefficients. We then use an elementary inequality for binomial coefficients to bound them by an exponential in $n-m$, and we obtain power savings in the sum when the base of this exponential is strictly less than $\sqrt{q}$. 

The strategy to deduce \cref{moments-intro} and \cref{ol-intro} is similar, except that we start with a sum over characters and must first manipulate it into a sum in progressions -- standard in analytic number theory -- and that we need slightly different estimates for the binomial coefficients (exponential in $m$ rather than $n-m$). The most difficult part is to sieve out the imprimitive characters from the primitive ones, but this causes calculational complexity more than fundamental difficulty.

These applications are all worked out in \cref{s-applications}.

Finally, in \cref{s-anatomy}, we deduce \cref{eft-intro} from \cref{squarefree-bound-precise}. We handle each extended factorization type separately (in this aspect, the argument is similar to \cite{Gorodetsky2020}). We detect a given extended factorization type using a sieve which is an alternating sum of terms where each term is a factorization function as in \cref{fac-fun}. For some factorization types, this leads to a power savings estimate, but for others, where a certain entropy is large, it does not. To handle these terms, we bound the total number of polynomials whose factorization has large entropy. We do this by essentially probabilistic methods -- we use the Poisson-Dirichlet process, which describes the limiting behavior of the prime factorizations of large numbers and large degree polynomials, and prove a tail estimate for this process by comparing with a certain simple Markov chain.

The restriction to progressions with squarefree modulus is due to the nature of our compactification. It might be possible to find a compactification with similarly good properties for the analogous morphism given a progression to non-squarefree modulus, in which case the arguments of this paper should generalize.

\subsection{Acknowledgments}

This was written while the author served as a Clay Research Fellow, based on earlier work where the author was supported by Dr. Max R\"{o}ssler, the Walter Haefner Foundation and the ETH Zurich Foundation, and completed while the author was supported by NSF grant DMS-2101491.

I would like to thank Mark Shusterman for helpful conversations and the anonymous referee for numerous helpful comments.

\tableofcontents

\section{Notation and preliminaries }

We fix throughout this paper almost the same data discussed in Theorem \ref{squarefree-bound-precise}: a finite field  $\kappa= \mathbb F_q$, a squarefree polynomial $g$ of degree $m$ over $\mathbb F_q$, a natural number $n \geq m$, and a representation $\rho$ of $S_n$ over $\mathbb Q_\ell$. (Since all representations of $S_n$ are defined over $\mathbb Q$, we can freely transfer them between different fields of characteristic $0$.)

However, we will not fix a residue class $a$ mod $g$. Instead, we define a torus on which this residue class may vary.

 Let $T$ be the algebraic torus whose $R$-points are $(R[t]/g)^\times$ for every ring $R$ over $\mathbb F_q$. (Thus $T$ is an algebraic torus of dimension $m$, which splits over the splitting field of $g$.) Let $C = \operatorname{Spec} \mathbb F_q[t,g^{-1}]$. Let $\pi$ be the natural map $C^{ n} \to T$ that sends $(\alpha_{1} ,\dots,\alpha_n )$ to the residue class of $ \prod_{i=1}^{n} (t-\alpha_{i})$ mod $g$. 
 
 Let $S \subset \mathbb P^1(\overline{\kappa}) $ be the union of the set of roots of $g$ with $\infty$, so that $C_{\overline{\kappa} } = \mathbb P^1_{\overline{\kappa}}  \setminus S$.

 We use $C^{(n)}$ to denote the $n$th symmetric power of $C$.

\subsection{Relating functions to sheaves}

The key object of study in this paper is $ R \pi_! \mathbb Q_\ell$, the derived pushforward with compact support. The product $C^n$ has a natural action of $S_n$ permuting the factors, which preserves the map $\pi$, inducing an action of $S_n$ on $R \pi_! \mathbb Q_\ell$ by functoriality. We can study different $S_n$-isotypic components of this action by taking the tensor product with a representation $\rho$ of $S_n$ and taking $S_n$-invariants.

The next lemma explains how $ R \pi_! \mathbb Q_\ell$ is related to the main sum appearing in Theorem \ref{squarefree-bound-precise}:

\begin{lemma}\label{trace-function} For $a \in T(\mathbb F_q)$, we have 
\begin{equation}\label{eq-trace-function} \sum_{ \substack{ f \in \mathcal M_n \\  f \equiv a \mod g }}  F_\rho(f)  =  \sum_{j =0}^{2(n-m)}  (-1)^j \tr \left(\Frob_q, \left( R^{j} \pi_! \mathbb Q_\ell \otimes \rho\right)^{S_n}_a \right) \end{equation} where $\left( R^{j} \pi_! \mathbb Q_\ell \otimes \rho\right)^{S_n}_a$ is the stalk of $\left( R^{j} \pi_! \mathbb Q_\ell \otimes \rho\right)^{S_n}$ at $a$. 

Furthermore, we have
\begin{equation}\label{eq-trace-complex} \sum_{ \substack{ f \in \mathcal M_n \\ \gcd(f,g)=1}}  F_\rho(f)  =  \sum_{j =0}^{2n}  (-1)^j \tr \Bigl(\Frob_q, H^j\Bigl( \bigl(H^*_c(C_{\overline{\kappa}}, \mathbb Q_\ell)\bigr)^{\otimes n} \otimes \rho \Bigr)^{S_n}\Bigr) \end{equation}
where $\bigl(H^*_c(C_{\overline{\kappa}}, \mathbb Q_\ell)\bigr)^{\otimes n} $ denotes the $n$-fold graded tensor product of the graded vector space $H^*_c(C_{\overline{\kappa}}, \mathbb Q_\ell)$, with an action of $S_n$ by permuting the factors and applying the Koszul sign rule, and $H^j $ takes the $j$'th graded piece of a graded vector space.

\end{lemma}

\begin{proof} Letting $sym\colon C^n \to C^{(n)}$ be the natural map, and letting $\tau\colon C^{(n)} \to T$ be the map that sends an unordered tuple $\alpha_1,\dots, \alpha_n$ to $\prod_{i=1}^n (T-\alpha_i) \mod g$, we have \[ \left( R^{j} \pi_! \mathbb Q_\ell \otimes \rho\right)^{S_n}= \left( \left(R^{j} \tau_!  sym_* \mathbb Q_\ell \right) \otimes \rho\right)^{S_n} = R^j \tau_! \left( sym_* \mathbb Q_\ell \otimes \rho\right)^{S_n} .\]

By the Grothendieck-Lefschetz fixed point formula, because $\tau^{-1} (a)$ has dimension $\leq n-m$,
\[ \sum_{j =0}^{2(n-m)}  (-1)^j \tr \left(\Frob_q, \left( R^{j} \pi_! \mathbb Q_\ell \otimes \rho\right)^{S_n}_a \right) =  \sum_{ f \in (\tau^{-1}(a) ) (\mathbb F_q) } \tr (\Frob_q,  (sym_* \mathbb Q_\ell \otimes \rho)^{S_n}_ f) .\]

Now viewing $C^{(n)}$ as the space of monic polynomials of degree $n$ prime to $g$, and $\tau^{-1}(a)$ as the space of monic polynomials of degree $n$ congruent to $a$ modulo $g$, we have
\[ (\tau^{-1}(a) ) (\mathbb F_q) = \{ f \in \mathcal M_n | f \equiv a \mod g\}. \] Next note that \[ (sym_* \mathbb Q_\ell \otimes \rho)^{S_n}_f =((sym_* \mathbb Q_\ell)_f \otimes \rho)^{S_n}  \] and $(sym_* \mathbb Q_\ell)_f$ is the free vector space on the fiber of $sym$ over $f$ and thus is isomorphic as a representation of $S_n$ and $\operatorname{Frob}_q$ to $V_f$. Hence, by definition, \begin{equation}\label{sym-vs-F} \tr (\Frob_q,  (sym_* \mathbb Q_\ell \otimes \rho)^{S_n}_ f)  = F_\rho(f) .\end{equation} This gives \eqref{eq-trace-function}.

For \eqref{eq-trace-complex}, the proof is similar. We have 
\[ \sum_{ \substack{ f \in \mathcal M_n \\ \gcd(f,g)=1 }}  F_\rho(f) = \sum_{ \substack{ f \in \mathcal M_n\\ \gcd(f,g)=1} }  \tr (\Frob_q,  (sym_* \mathbb Q_\ell \otimes \rho)^{S_n}_ f) = \sum_{j =0}^{2n}  (-1)^j \tr\Bigl(\Frob_q, H^j_c\bigl ( C_{\overline{\kappa}}^{(n)},   (sym_* \mathbb Q_\ell \otimes \rho)^{S_n}\bigr)\Bigr)\] \[ =\sum_{j =0}^{2n} (-1)^j \tr\Bigl(\Frob_q,\bigl ( H^j_c (C_{\overline{\kappa}}^{(n)},   sym_* \mathbb Q_\ell) \otimes \rho\bigr)^{S_n}   \Bigr) = \sum_{j =0}^{2n}  (-1)^j \tr\left(\Frob_q,\bigl ( H^j_c (C_{\overline{\kappa}}^n,    \mathbb Q_\ell) \otimes \rho\bigr)^{S_n}   \right)\] \[ =  \sum_{j =0}^{2n}  (-1)^j \tr \Bigl(\Frob_q, H^j\Bigl( \bigl(H^*_c(C_{\overline{\kappa}}, \mathbb Q_\ell)\bigr)^{\otimes n} \otimes \rho \Bigr)^{S_n}\Bigr) .\]
Here the first equality is by \eqref{sym-vs-F}, the second follows from the Grothendieck-Lefschetz fixed point formula (using  $\dim( C^{(n)})= n$), the third is because taking $ V\mapsto (V\otimes \rho)^{S_n}$ is an exact functor and thus commutes with cohomology, the fourth is compatibility of compactly-supported cohomology with finite pushforwards \cite[XVII, Theorem 5.1.8(a)]{sga4-3}, and the last is the K\"unneth formula \cite[XVII, Theorem 5.4.3]{sga4-3}. 
\end{proof}

Using this lemma, Theorem \ref{squarefree-bound-precise} will follow from properties of $R \pi_! \mathbb Q_\ell$. The properties we need, other than Deligne's bound for the Frobenius eigenvalues, will all be purely geometric in nature, so to prove them we will work over the algebraically closed field $\overline{\kappa}$. 

\subsection{Representations of the symmetric group and partitions}

We review here some standard facts about irreducible representations of the symmetric group and their relationships to partitions \cite[\S4]{FultonHarris}, which we will use later in the paper. We will need to apply these to representations $\orho$ potentially different from the fixed representation $\rho$.

\begin{lemma}\label{schur-vanishing-lemma} Let $\orho$ be an irreducible representation of $S_n$ corresponding to a partition with at least one part of size $\geq m.$ Then $V_{\orho}=0$. \end{lemma}

\begin{proof} By definition, we must show that $\Bigl( (\mathbb C^{m-1} )^{\otimes n } \otimes \orho \otimes \sgn \Bigr)^{S_n} =0$, or equivalently, that $\orho \otimes \sgn$ does not appear in $(\mathbb C^{m-1} )^{\otimes n } $.

Note that $\orho \otimes \sgn$ is the irreducible representation corresponding to the conjugate partition \cite[Exercise 4.4(c)]{FultonHarris}, a partition into at least $m$ parts. Furthermore $(\mathbb C^{m-1} )^{\otimes n }$ is the permutation representation of $S_n$ acting on the set of $m-1$-colorings of $n$ elements, which is a sum of permutation representations acting on partitions of $n$ elements into at most $m-1$ parts.

Because the only summands of permutation representations acting on a partition are irreducible representation corresponding to partitions greater than or equal to it in the dominance ordering \cite[Exercise 4.46]{FultonHarris}, which must have a lesser or equal number of parts, indeed $\orho \otimes \sgn$ cannot be a summand of $(\mathbb C^{m-1} )^{\otimes n }$.
\end{proof}

\begin{lemma}\label{other-vanishing-lemma} Let $\orho$ be an irreducible representation of $S_n$ corresponding to a partition with all parts of size $<m$. Then $V_{\orho}^d=0$ for $d \geq m$. \end{lemma}

\begin{proof} By definition, we must show that  \[ \Bigl( (\mathbb C^{m} )^{\otimes (n-d) } \otimes \orho \otimes \sgn_{S_{n-d}} \Bigr)^{S_{n-d} \times S_d } =0 \] for $d\geq m$. Since $S_d$ acts on $(\mathbb C^{m} )^{\otimes (n-d) } \otimes \orho \otimes \sgn_{S_{n-d}}$ only through its action on $\orho$, it suffices to show $\orho^{ S_d} = 0$, or in other words that $\orho$ is not a summand of the permutation representation acting on partitions into one part of size $d$ and $n-d$ parts of size $1$. The representation $\orho$ corresponds to a partition with all parts of size $<m$ and so is not higher in the dominance ordering than a partition with one part of size $d\geq m$, and so $\orho$ cannot be a summand of this permutation representation by \cite[Exercise 4.46]{FultonHarris}. \end{proof}

We also include here an alternative formula for $C_2(\rho)$, obtained by breaking $V_\rho$ into its $\Diag(\lambda_1,\dots, \lambda_{m-1})$-eigenspaces, which will be convenient to use in \cref{cc-nice-formula} later.

\begin{lemma}\label{C2-e-formula} For $\rho$ a representation of $S_n$, we have \[C_2(\rho) =  \sum_{ \substack{ (e_1,\dots, e_{m-1}) \in \mathbb N^{m-1}  \\ \sum_{i=1}^{m-1} e_i=n}}  \Bigl ( \prod_{i=1}^{m-1} e_i \Bigr)  \dim ( \rho \otimes \sgn)^{ \prod_{i=1}^{m-1} S_{e_i}}  . \] \end{lemma}

\begin{proof}  %By definition 
  Let us first calculate the coefficient of $\prod_{i=1}^{m-1} \lambda_i^{e_i}$ in $ \tr( \Diag(\lambda_1,\dots, \lambda_{m-1}), V_{\rho} )  $ for $e_1,\dots, e_{m-1} \in \mathbb N$. This is the dimension of the $ \prod_{i=1}^{m-1} \lambda_i^{e_i}$-eigenspace of the maximal torus of $GL_{m-1}$ inside $V_{\rho}$. By definition, $V_{\rho} = \left( (\mathbb C^{m-1})^{\otimes n} \otimes \rho \otimes \sgn\right)^{S_n}.$
  
  The  $ \prod_{i=1}^{m-1} \lambda_i^{e_i}$-eigenspace in $(\mathbb C^{m-1})^{\otimes n} $ is generated by tensor products of $n$-tuples of basis vectors where the $i$'th basis vector occurs $e_i$ times. Thus it is nonempty only if $\sum_{i=1}^{m-1} e_i =n$, its dimension is $\frac{n!}{ \prod_{i=1}^{m-1} (e_i)!}$, and as a representation of $S_n$ it is isomorphic to $\Ind_{\prod_{i=1}^{m-1} S_{e_i}}^{S_n} \mathbb C$.
  
  Thus, the $ \prod_{i=1}^{m-1} \lambda_i^{e_i}$-eigenspace of $V_{\rho}$ is isomorphic to \[ \left(  \left( \Ind_{\prod_{i=1}^{m-1} S_{e_i}}^{S_n} \mathbb C\right) \otimes \rho \otimes \sgn \right)^{S_n} = ( \rho \otimes \sgn)^{ \prod_{i=1}^{m-1} S_{e_i}} .\]
  
  Hence \[ \tr( \Diag(\lambda_1,\dots, \lambda_{m-1}, V_{\rho} ) ) = \sum_{ \substack{ (e_1,\dots, e_{m-1}) \in \mathbb N^{m-1}  \\ \sum_{i=1}^{m-1} e_i=n}} \left(\prod_{i=1}^{m-1} \lambda_i ^{e_i} \right)\dim ( \rho \otimes \sgn)^{ \prod_{i=1}^{m-1} S_{e_i}} \] and so by the definition of $C_2(\rho)$, \[ C_2(\rho) =  \frac{\partial^{m-1} \tr( \Diag(\lambda_1,\dots, \lambda_{m-1}, V_{\rho} ) ) }{ \partial \lambda_1 \dots \partial \lambda_{m-1}} \Big|_{\lambda_1,\dots,\lambda_{m-1}=1} \]
    \[= \sum_{ \substack{ (e_1,\dots, e_{m-1}) \in \mathbb N^{m-1}  \\ \sum_{i=1}^{m-1} e_i=n}}   \frac{\partial^{m-1}  \prod_{i=1}^{m-1} \lambda_i ^{e_i} }{ \partial \lambda_1 \dots \partial \lambda_{m-1}}  \Big|_{\lambda_1,\dots,\lambda_{m-1}=1} \dim ( \rho \otimes \sgn)^{ \prod_{i=1}^{m-1} S_{e_i}}  \] \[= \sum_{ \substack{ (e_1,\dots, e_{m-1}) \in \mathbb N^{m-1}  \\ \sum_{i=1}^{m-1} e_i=n}}  \Bigl ( \prod_{i=1}^{m-1} e_i \Bigr)  \dim ( \rho \otimes \sgn)^{ \prod_{i=1}^{m-1} S_{e_i}}  . \qedhere\] \end{proof}

 The following fact will let us prove a key identity between characters of $S_n$.     We will use this for $\mathcal S$ the set of partitions into $<m$ parts, but state it in slightly greater generality.

    \begin{lemma}\label{permutation-restriction-fact} Let $\mathcal S$ be a set of partitions of $n$ that is upward-closed in the dominance ordering. Let $\chi$ be a character of $S_n$ obtained as a linear combination of the characters of irreducible representations corresponding to conjugate partitions of elements of $\mathcal S$.
    
    If $\chi$ vanishes on all permutations with cycles of size $n_1,\dots, n_r$, with $(n_1,\dots, n_r)\in \mathcal S$, then $\chi=0$.\end{lemma}

    \begin{proof} For $(n_1, \dots , n_r)$ a partition of $n$, the representation $\Ind_{S_{n_1} \times \dots \times S_{n_r}}^{S_n} \sgn$ is the sum of the representation of $S_n$ corresponding to the conjugate of $n_1+ \dots + n_r$ with representations corresponding to partitions lower in the dominance order, so if $(n_1,\dots, n_r) \in \mathcal S$ then $\Ind_{S_{n_1} \times \dots \times S_{n_r}}^{S_n} \sgn$  is a sum of irreducible representations corresponding to conjugate partitions of elements of $\mathcal S$. Using this, we can inductively write $\chi$ as a sum of characters of representations $\Ind_{S_{n_1} \times \dots \times S_{n_r}}^{S_n} \sgn$ for $(n_1,\dots, n_r)\in \mathcal S$ by writing it as a sum of characters of irreducible representations and in each step cancelling the remaining representation highest in the dominance order. Assuming for contradiction that $\chi\neq 0$, this sum must be nontrivial.

Now given a nontrivial sum of characters of representations $\Ind_{S_{n_1} \times \dots \times S_{n_r}}^{S_n} \sgn$, choose some partition $(n_1 , \dots , n_r)$ appearing that is maximal in the dominance order, and choose $\sigma$ to be a permutation with $r$ cycles, of size $n_1,\dots, n_r$. Applying the Frobenius formula for the character of an induced representation, we see that  $\tr ( \sigma , \Ind_{S_{n_1} \times \dots \times S_{n_r}}^{S_n} \sgn) \neq 0$, but $\tr ( \sigma , \Ind_{S_{n^{'}_1} \times \dots \times S_{n^{'}_{r'}}}^{S_n} \sgn)=0$ for $n_1',\dots, n_{r'}'$ any partition that is not larger in the dominance order, since $\sigma$ is conjugate to an element of $S_{n_1} \times \dots \times S_{n_r}$, the product of the $n_i$-cycle in $S_{n_i}$, but not conjugate to any element of $S_{n^{'}_1} \times \dots \times S_{n^{'}_r}$. It follows that $\chi$ is nontrivial on $\sigma$, contradicting our assumption.
\end{proof}

The following lemma will help in calculating $V^d_\rho$ for an induced representation

\begin{lemma}\label{induction-restriction} Let $n, d,n_1,n_2$ be natural numbers with $n_1+n_2=n$ and $d\leq n$. For $\rho_1, \rho_2$ representations of $S_{n_1},S_{n_2}$ with $n_1+n_2 =n$, we have
\[ \operatorname{Res}_{S_n}^{ S_d \times S_{n-d}} \Ind_{S_{n_1} \times S_{n_2}}^{S_n}   (\rho_1   \otimes \rho_2 ) = \bigoplus_{e =\max(0, d- n_2 ) }^{ \min(n_1 ,d) }   \operatorname{Ind}_{ S_e \times S_{d-e} \times S_{n_1-e} \times  S_{n_2-d+e}  }^{S_d\times S_{n-d}}  (\rho_1   \otimes \rho_2)\] where we view $\rho_1  \otimes \rho_2$ as a representation of $S_e \times S_{d-e} \times S_{n_1-e} \times  S_{n_2-d+e} $ by embedding $S_e \times S_{n_1-e}$ into $S_{n_1}$ and $S_{d-e} \times S_{n_2-d+e}$ into $S_{n_2}$, and we view $S_e \times S_{d-e} \times S_{n_1-e} \times  S_{n_2-d+e} $ as a subgroup of $S_d \times S_{n-d}$ by embedding  $S_e \times S_{d-e} $ into $S_d$ and $ S_{n_1-e} \times  S_{n_2-d+e} $  into $S_{n-d}$.
\end{lemma}

\begin{proof} By \cite[Proposition 22]{serre-lrfg} we have
\[ \operatorname{Res}_{S_n}^{ S_d \times S_{n-d}} \Ind_{S_{n_1} \times S_{n_2}}^{S_n}   (\rho_1   \otimes \rho_2 ) = \bigoplus_{\sigma \in (S_d \times S_{n-d} ) \backslash S_n / (S_{n_1} \times S_{n_2}) } \operatorname{Ind}_{ (S_{n_1} \times S_{n_2}) \cap  \sigma^{-1} (S_d \times S_{n-d} )  \sigma }^{ S_d \times S_{n-d} }  (\rho_1   \otimes \rho_2) . \]
Double cosets  $(S_d \times S_{n-d} ) \sigma (S_{n_1} \times S_{n_2}) $ are classified by the number $e$ of elements of $\{1,\dots n_1\}$ whose image under $\sigma \colon \{1,\dots, n\}\to \{1,\dots, n\}$ is contained in $\{1,\dots,d\}$. 

For $\sigma$ in the double coset corresponding to $e$, the intersection $(S_{n_1} \times S_{n_2}) \cap  \sigma^{-1} (S_d \times S_{n-d} ) \sigma$ consists of all those permutations that stabilize the sets
\[\{1,\dots, n_1\}, \{n_1+1, \dots, n\},  \sigma^{-1}  ( \{1,\dots, d\} ), \sigma^{-1} ( \{d+1,\dots,n \}).\]

Without loss of generality, $\sigma$ fixes $\{1,\dots, e\}$, sends $\{e+1, \dots n_1\}$ to $\{ d+1, n_1+d-e\}$, sends  $\{n_1+1, \dots n_1+d-e\}$ to $\{e_1,\dots ,d\}$, and fixes $\{n_1+d-e+1,\dots,n\}$.

Then $\sigma^{-1} ( ( \{1,\dots, d\} ) =\{1,\dots, e, n_1+1, \dots n_1+d-e\}$ and $\sigma^{-1} ( \{d+1,\dots,n \})= \{ e_1,\dots, n_1, n_1+d-e+1,\dots n\} $. Any permutation that stabilizes these four sets stabilizes their intersections
\[ \{1,\dots, e\} , \{e_1,\dots, n_1\}, \{n_1+1,\dots,n_1+d-e\}, \{n_1+d-e+1, \dots, n\}\]
and any permutation that stabilizes these four intervals stabilizes their unions, the original sets, so 
\[(S_{n_1} \times S_{n_2}) \cap  \sigma^{-1} (S_d \times S_{n-d} ) = S_{e} \times S_{n_1-e} \times S_{d-e} \times S_{n_2+d-e} \]
and the embeddings we must induct and restrict on are as described in the statement.
\end{proof}

\subsection{Perverse sheaves} 

By convention, a variety is a separated, geometrically integral scheme of finite type over a field. 

The following lemma about perverse sheaves will be useful later. 

\begin{lemma}\label{perverse-easy-inequality} Let $X$ be a normal variety of dimension $n$, with generic point $\eta$, and let $K$ be a perverse $\ell$-adic sheaf on $X$. For any $x\in X$ we have \begin{equation}\label{dimension-stalk-generic} \dim \mathcal H^{-n}(K)_x \leq \dim \mathcal H^{-n}(K)_{\eta} .\end{equation}
\end{lemma}

\begin{proof} Let $U$ be the maximal open set on which $\mathcal H^{-n}(K)$ is lisse and let $j \colon U \to X$ be the open immersion. Let $\mathcal K$ be the kernel of the adjunction map $\mathcal H^{-n}(K) \to j_* j^* \mathcal H^{-n}(K)$.

Since $K$ is perverse, $\mathcal H^i(K)$ vanishes for $i<-n$. Truncating the complex defining $K$ in degree $-n$ produces a complex with cohomology in degree only $-n$ and equal to $\mathcal H^{-n}(K)$ in that degree, i.e. a complex quasi-isomorphic to $\mathcal H^{-n}(K) [n]$, and truncation gives a map $\mathcal H^{-n}(K) [n]\to K$ which is an isomorphism in degree $-n$. Composed with the natural map $\mathcal K[n] \to \mathcal H^{-n}(K) [n]$ arising from the injection $\mathcal K \to \mathcal H^{-n}(K)$, we obtain a map $\mathcal K [n] \to K$ which is nontrivial as soon as $\mathcal K$ is because it is an injection on cohomology in degree $-n$.

However, since $\mathcal K$ is supported in dimension $<n$, the sheaf $\mathcal K[n]$ is supported in perverse degree $<0$. Because perversity is a $t$-structure, any complex that maps nontrivially to a perverse sheaf must have perverse cohomology in some nonnegative degree, so  the map $\mathcal K [n] \to K$ must be trivial and thus $\mathcal K$ is trivial. Then $\mathcal H^{-n}(K)\to j_* j^* \mathcal H^{-n}(K)$ is injective so
\[  \dim \mathcal H^{-n}(K)_x  \leq  \dim(j_* j^* \mathcal H^{-n}(K))_x  \leq \dim (j_* j^* \mathcal H^{-n}(K))_\eta =\dim \mathcal H^{-n}(K)_\eta)\]
with the middle inequality  because the stalk of $j_*$ of a lisse sheaf at a closed point of a normal variety may be computed as the inertia invariants of the stalk of the lisse sheaf at at the generic point.  \end{proof}

\subsection{The singular support and the characteristic cycle}

We review some notation and terminology of \citet{Beilinson} and \citet{saito1}. (Our formulations of the definitions are mainly adapted from \citep{saito1}). All schemes are over a perfect field $k$, which in the remainder of the paper we can specialize to be the algebraic closure of a finite field. We follow the modified terminology of \cite{saitomixed} in using \emph{acyclic} instead of \emph{transversal} in \cref{C-transversal-1} and \cref{pair-acyclic}, to avoid confusion between different uses of the phrase ``$C$-transversal". We begin each definition with a parenthetical highlighting what is to be defined.

\begin{defi}\label{C-transversal-1} ($C$-acyclic morphism) \citep[Definition 3.5(1)]{saito1} Let $X$ be a smooth scheme over $k$ and let $C \subseteq T^* X$ be a closed conical subset of the cotangent bundle. Let $f\colon  X\to Y$ be a morphism of smooth schemes over $k$.

We say that $f\colon  X \to Y$ is \emph{$C$-acyclic} if the inverse image $df^{-1}(C)$ by the canonical morphism $df\colon X \times_Y T^* Y \to T^* X$ is a subset of the zero-section $X \subseteq X \times_Y T^* Y$.
\end{defi}

\begin{defi}($f_\circ C$) \citep[(1.2)]{Beilinson}\label{circ-forward} In the same setting as Definition \ref{C-transversal-1}, if $f$ is proper, let $f_\circ C$ be the projection from $X \times_Y T^* Y$ to $T^*Y $ of the inverse image $df^{-1}(C)$.

\end{defi}

\begin{defi}($f_! A$) \citep[(2.3)]{saito-direct}\label{shriek-forward}  In the same setting as Definitions \ref{C-transversal-1} and \ref{circ-forward}, let $A$ be an algebraic cycle of codimension $\dim X$ on $T^* X$ supported on $C$. Assume also that $f_\circ C$ has dimension $\dim Y$.

Let $f_!  A$ be the pushforward from $X \times_Y T^* Y$ to $T^*Y $ of the intersection-theoretic inverse image $df^* C$. \end{defi}

\begin{defi}($C$-transversal morphism and $h^\circ C$) \citep[Definition 3.1]{saito1} Let $X$ be a smooth scheme over $k$ and let $C \subseteq T^* X$ be a closed conical subset of the cotangent bundle. Let $h\colon  W \to X$ be a morphism of smooth schemes over $k$.

Let $h^* C$ be the pullback of $C$ from $T^* X$ to $W \times_X T^* X$ and let $K$ be the inverse image of the $0$-section $W \subseteq T^* W$ by the canonical morphism $dh\colon  W \times_X T^* X \to T^* W$. 

We say that $h\colon  W\to X$ is \emph{$C$-transversal} if the intersection $h^* C \cap K$ is a subset of the zero-section $W \subseteq W \times_X T^* X$.

If $h\colon  W \to X$ is $C$-transversal, we define a closed conical subset $h^\circ C \subseteq T^* W$ as the image of $h^* C$ under $dh$ (it is closed by \citep[Lemma 3.1]{saito1}). \end{defi}

\begin{defi}($C$-acyclic pair) \citep[Definition 3.5(2)]{saito1}\label{pair-acyclic} We say that a pair of morphisms $h\colon  W \to X$ and $f\colon W\to Y$ of smooth schemes over $k$ is \emph{$C$-acyclic}, for $C \subseteq T^* X$ a closed conical subset of the cotangent bundle, if $h$ is $C$-transversal and $f$ is $h^\circ C$-acyclic. \end{defi}

%\begin{defi}\citep[Th. Finitude, Definition 2.12]{sga4h} We say that a morphism $f\colon  W \to Y$ is locally acyclic relative to $K \in D^b_x(W, \mathbb F_\ell)$ if for each geometric point $x \in W$ geometric point  $t \in Y$ specializing to $f(x)$, $W_x$ the henselization of $W$ at $x$ and $W_{x,t}$ the fiber of $X_x$ over $t$, the natural map $H^* ( W_x, K) \to H^* ( W_{x,t}, K)$ is an isomorphism.  \end{defi}
%
%For us, the main advantage of local acyclicity is that, when $Y$ is a smooth curve, the local acyclicity of $f$ implies that $R\Phi_f K$ vanishes, since taking $t$ the generic point, $H^* ( W_{x,t}, K)$ is the stalk of $R \Psi_ K$ at $x$ and $H^* ( W_x, K) $ is the stalk of $K$ at $x$, so the map is an isomorphism if and only if the mapping cone $R(\Phi_f K)_x$ vanishes.

\begin{defi}(the singular support) \citep[1.3]{Beilinson}\label{def-singular-support} For $K \in D^b_c(X, \mathbb F_\ell)$, let the \emph{singular support $SS(K)$ of $K$} be the smallest closed conical subset $C \in T^* X$ such that for every $C$-acyclic pair $h\colon  W \to X$ and $f\colon  W\to Y$, the morphism $f\colon  W\to Y$ is locally acyclic relative to $h^* K$. 

\end{defi}
The existence and uniqueness of $SS(K)$ is \citep[Theorem 1.3]{Beilinson}, which also proves that if $X$ has dimension $n$ then $SS(K)$ has dimension $n$ as well.

\begin{defi}(properly $C$-transversal morphism) \citep[Definition 7.1(1)]{saito1} Let $X$ be a smooth scheme of dimension $n$ over $k$ and let $C \subseteq T^* X$ be a closed conical subset of the cotangent bundle with each irreducible component of dimension $n$. Let $W$ be a smooth scheme of dimension $m$ over $k$ and let $h\colon  W \to X$ be a morphism over $k$.

We say that $h$ is \emph{properly $C$-transversal} if it is $C$-transversal and each irreducible component of $h^* C$ has dimension $m$. \end{defi} 

\begin{defi}($h^! A$) \citep[Definition 7.1(2)]{saito1} Let $X$ be a smooth scheme of dimension $n$ over $k$ and let $A$ be an algebraic cycle of codimension $n$ on $ T^* X$ whose support $C$ is a closed conical subset of the cotangent bundle (necessarily of dimension $n$).

 Let $W$ be a smooth scheme of dimension $m$ over $k$ and let $h\colon  W \to X$ be a properly $C$-transversal morphism over $k$.
 
 We say that $h^{!} A$ is $(-1)^{n-m}$ times the pushforward along $dh\colon   W \times_X T^* X \to T^* W$ of the pullback along $h\colon   W \times_X T^* X \to T^* X$ of $A$, with the pullback and pushforward in the sense of intersection theory.  \end{defi}

Here the pushforward in the sense of intersection theory is well-defined because, by \citep[Lemma 3.1]{saito1}, $dh$ is finite when restricted to $h^* C$ (with the induced reduced subscheme structure), i.e. finite when restricted to the support of $h^* A$.

\begin{defi}(isolated $C$-characteristic points) \citep[Definition 5.3(1)]{saito1} Let $X$ be a smooth scheme of dimension $n$ over $k$  and let $C\subseteq T^* X$ be a closed conical subset of the cotangent bundle. Let $Y$ be a smooth curve over $k$ and $f\colon  X\to Y$ a morphism over $k$. 

We say a closed point $x \in X$ is at most an \emph{isolated $C$-characteristic point} of $f$ if $f$ is $C$-acyclic when restricted to some open neighborhood of $x$ in $X$, minus $x$.  We say that $x\in X$ is an \emph{isolated $C$-characteristic point} of $f$ if this holds, but $f$ is not $C$-acyclic when restricted to any open neighborhood of $x$. \end{defi}

\begin{defi}($\operatorname{dimtot}$) For $V$ a representation of the Galois group of a local field over $\mathbb F_\ell$ (or a continuous $\ell$-adic representation), we define $\operatorname{dimtot} V$ to be the dimension of $V$ plus the Swan conductor of $V$. For a complex $W$ of such representations, we define $\operatorname{dimtot}W $ to be the alternating sum $\sum_i (-1)^i \operatorname{dimtot} \mathcal H^i(W)$ of the total dimensions of its cohomology objects. \end{defi}

\begin{defi}(the characteristic cycle) \citep[Definition 5.10]{saito1} Let $X$ be a smooth scheme of dimension $n$ over $k$ and $K$ an object of $D^b_c(X, \mathbb F_\ell)$. Let the \emph{characteristic cycle of $K$}, $CC(K)$, be the unique $\mathbb Z$-linear combination of irreducible components of $SS(K)$ such that for every \'{e}tale morphism $j\colon  W \to X$, every morphism $f\colon  W\to Y$ to a smooth curve and every at most isolated $h^\circ SS(\mathcal F)$-characteristic point $u \in W$ of $f$, we have
\begin{equation}\label{Milnor-formula}  - \operatorname{dimtot} \left( R \Phi_f(j^* K) \right)_u =  (j^* CC(K), f^*\omega )_{T^*W,u} \end{equation} where $\omega$ is a meromorphic one-form on $Y$ with no zero or pole at $f(u)$.

\end{defi}

Here the notation $(,)_{T*W ,u}$ denotes the intersection number in $T^* W$ over the point $u$ (i.e. at $(u, f^* \omega(u)) \in T^* X$).

The existence and uniqueness is \citep[Theorem 5.9]{saito1}, except for the fact that the coefficients lie in $\mathbb Z$ and not $\mathbb Z[1/p]$, which is \citep[Theorem 5.18]{saito1} and is due to Beilinson, based on a suggestion by Deligne.

Saito defined the characteristic cycle for complexes in $D^b_c(X,\mathbb F_\ell)$. However, it can be extended straightforwardly to complexes in $D^b_C(X,\mathbb Q_\ell)$ by choosing a representative in $D^b_c(X, \mathbb Z_\ell)$ and projecting to $D^b_c(X,\mathbb F_\ell)$. The $\operatorname{dimtot}$ of the vanishing cycles along any map is independent of the choice of representative because it is an alternating sum over cohomology groups, and hence, by the uniqueness of the characteristic cycle, the characteristic cycle is independent of the choice of perverse representative.

The singular support is not necessarily independent of the choice of representatives. However, for perverse $\mathbb Q_\ell$-sheaves, we can choose a perverse $\mathbb Z_\ell$-representative  \cite[Section 2.2.18]{bbd}, and mod out by any $\ell$-power torsion to make it $\ell$-torsion-free, so its projection to $D^b_c(X,\mathbb F_\ell)$ will be a perverse sheaf over $\mathbb F_\ell$. For perverse sheaves over $\mathbb F_\ell$, the singular support is equal to the support of the characteristic cycle \cite[Proposition 4.14.2]{saito1}, and thus is independent of the choice of representative. We will study the singular support only for perverse sheaves.

In addition to reviewing the theory of the characteristic cycle, we introduce one new result. Saito conjectured \cite[Conjecture 1]{SaitoProper} that for $f\colon X \to Y$ a proper morphism of smooth schemes over a perfect field $k$, with $X$ of dimension $n$ and $Y$ of dimension $m$, and $\mathcal F$ a constructible complex on $X$, if every irreducible component of $f_\circ  SS(\mathcal F)$ has dimension $\leq m$, then \[ CC ( R f_* \mathcal F) = f_! CC( \mathcal F).\] In \cite[Theorem 2.2.5]{saito-direct}, Saito proved this with the additional assumptions that $X$ and $Y$ are projective. We would like to use this result, but in our setting $Y$ is not projective, though the morphism $f$ is. We could compactify $Y$, but calculating the singular support of the extended $f$ would be difficult. Instead we use the following lemma, which has almost the same proof as \cite[Lemma 2]{SaitoProper}.

\begin{lemma}\label{cc-pushforward} Let $f\colon X \to Y$ be a proper morphism of smooth schemes over a perfect field $k$. Assume that every irreducible component of $X$ has dimension $n$ and every irreducible component of $Y$ has dimension $m$. Let $\mathcal F$ be a constructible complex on $X$. Let $SS(\mathcal F) \subseteq T^* X$ be the singular support. Let $S$ be a closed subset of $T^* Y$. If $df^\circ(SS(\mathcal F))$ has dimension $\leq m$ and the projection from $df^{-1} ( SS(\mathcal F) ) \subseteq (T^* Y)\times_Y X$ to $T^* Y$ is finite away from $S$, then  \[ CC ( R f_* \mathcal F) - f_! CC( \mathcal F) \] is a cycle supported on $S$. 

\end{lemma}

We will apply this lemma with $S$ the zero section. We can think of the assumption of this lemma about finite support, for $S$ the zero section, as a relatively generic condition on the morphism $f$ and the sheaf $\mathcal F$. Calculating the multiplicity of the zero section in the pushforward may be done by base-changing to the generic point of $Y$ and using the earlier result \cite[Theorem 7.13]{saito1}. So under this condition, \cref{cc-pushforward} can be used to verify \cite[Conjecture 1]{SaitoProper}.

\begin{proof} We may assume that $k$ is algebraically closed. Let $j\colon V \to Y$ be an \'{e}tale morphism and $g \colon  V \to T$ be a morphism to a smooth curve $T$ with isolated characteristic point $v \in V$ with respect to $f_\circ SS(\mathcal F)$. Let $\omega$ be a nonvanishing one-form on $T$. By definition, we have 
\begin{equation}\label{Milnor-proof}  - \operatorname{dimtot} \left( R \Phi_g(j^* R f_* \mathcal F) \right)_uv=  (j^* CC(Rf_* \mathcal F),  g^*\omega )_{T^*V,v} .\end{equation}
Let us first check that as long as (the graph of ) $g^* \omega$ does not intersect $S$ over $v$, we have
\begin{equation}\label{Milnor-check}  - \operatorname{dimtot} \left( R \Phi_g(j^* R f_* \mathcal F) \right)_uv=  (j^* f_!  CC(\mathcal F), g^*\omega )_{T^*V,v} .\end{equation}
By replacing $Y$ with $V$, our assumptions on $\mathcal F$ are preserved, so we may assume $V=Y$. 

By passing to a further open subset of $Y$, we may assume that $f_\circ SS(\mathcal F)$ intersects $g^* \omega$ only over $v$. Thus  $f_\circ SS(\mathcal F) \cap g^* \omega$ does not intersect $S$. Hence the projection $df^* SS(\mathcal F) \to T^* Y$ is finite when restricted to $g^* \omega$, and so the projection $(df^* SS(\mathcal F) \cap f^* g^* \omega) \to T^* Y$ is finite.

Because the image of the projection $(df^* SS(\mathcal F) \cap f^* g^* \omega) \to T^* Y$ is contained in the unique intersection point of $f_\circ SS(\mathcal F)$ and $g^* \omega$, we can conclude that $df^* SS(\mathcal F) \cap  f^* g^* \omega$ is finite.

The composition $g \circ f \colon X\to T$ is $SS(\mathcal F)$-acyclic wherever $f^* g^* \omega$ does not intersect $SS(\mathcal F)$.  The complementary set where $f^* g^* \omega$ does intersect $SS(\mathcal F)$ is the projection from $(T^* Y)\times_Y X$ to $X$ of $df^* SS(\mathcal F) \cap f^* g^* \omega $ and thus is a finite set. So $g \circ f \colon X \to T$ is $SS(\mathcal F)$-acyclic away from finitely many points, which are all isolated characteristic points. Let $I$ be this finite set of points of $X$. 

By the definition of singular support, the vanishing cycles sheaf $R \Phi (\mathcal F)$ vanishes away from these finitely many points. Because vanishing cycles commutes with proper pushforward (see \cite[XIII, (1.3.6.3)]{sga7-2} for the analogue for nearby cycles and deduce the claim from the definition of vanishing cycles as a mapping cone) we have an isomorphism
\[R \Phi_v ( R f_* \mathcal F, g) \cong \bigoplus_{ u \in I}  R\Phi_u ( \mathcal F, g \circ f ). \] 

Then by \eqref{Milnor-formula}, we have 
\[ -\operatorname{dimtot} R \Phi_v ( R f_* \mathcal F, g) =  - \sum_{ u \in I}  \operatorname{dimtot}  R\phi_u ( \mathcal F, g \circ f ) \] \[ = \sum_{u \in I} ( CC(\mathcal F), f^* g^*\omega)_{ T^* X,u}   = ( f_* CC(\mathcal F), g^*\omega) _{T^* Y, v} \] verifying \eqref{Milnor-check}.

We now wish to use \eqref{Milnor-proof} and \eqref{Milnor-check} to verify that $ CC ( R f_* \mathcal F) - f_! CC( \mathcal F)$ is supported on $S$. This is equivalent to stating that, for each irreducible component $Z_i$ of $f_\circ SS(\mathcal F)$ not contained in $S$, the multiplicity of $Z_i$ in $ CC ( R f_* \mathcal F) $ equals the multiplicity of $Z_i$ in $ f_! CC( \mathcal F)$.

To check this, it suffices to take a point $(v,\omega_0)$ in $Z_i$ not contained in $S$ or any other irreducible component of $f_\circ SS(\mathcal F)$, and show there exists an open subset $V$ of $Y $, a map $g\colon  V \to T$, and a nonvanishing differential $\omega$ on $T$ such that $f_\circ SS(\mathcal F)$ intersects $g^* \omega$ at $(v,\omega_0)$ but not any other point over a neighborhood of $v$. That will ensure that $v$ is an isolated characteristic point and that $(dg)^* \omega$ does not intersect $S$ over $v$, so that
\[ \operatorname{mult}_{ CC (R f_* \mathcal  F)} (Z_i)  \cdot ( Z_i, g^*\omega)_{ T^*V,v} = (j^* CC(Rf_* \mathcal F),  g^*\omega )_{T^*V,v} \] \[= (j^* f_!  CC(\mathcal F), g^*\omega )_{T^*V,v} = \operatorname{mult}_{ CC (R f_* \mathcal  F)} (Z_i)  \cdot  ( Z_i, g^*\omega)_{ T^*V,v}  \]
and ensure that $( Z_i, g^*\omega)_{ T^*V,v}  \neq 0$ so we can conclude
and thus \[ \operatorname{mult}_{ CC (R f_* \mathcal  F)} (Z_i)  = \operatorname{mult}_{ CC (R f_* \mathcal  F)} (Z_i) , \] as desired.

We now finish the argument by constructing the required open subset $V$ of $Y $, map $g\colon  V \to T$, and nonvanishing differential $\omega$ on $T$ such that $f_\circ SS(\mathcal F)$ intersects $g^* \omega$ at $(v,\omega_0)$ but not any other point over a neighborhood of $v$.

We can take $V$ to be an affine neighborhood of $v$, which we may embed into $\mathbb A^N$ for some $N$, take $T = \mathbb A^1$, $\omega=dx$, and take $g$ to be a generic cubic polynomial function on $\mathbb A^N$ whose derivative at $v$ is $\omega_0$.  Then certainly the only point over $v$ that $g^* \omega$ intersects is $(v,\omega_0)$. For each point $(v', \omega')$ of $SS(\mathcal F)$ with $v' \neq v$, the condition that $g^*\omega (v') = \omega'$, equivalently, $d g(v') =\omega'$, is a codimension $m$ condition on the cubic polynomial $g$ which is linearly independent of the condition on the derivative of $g$ at  $v$ is $\omega_0$.  Thus for a generic $g$, this codimension $m$ condition is satisfied on a subset of the $m$-dimensional scheme \[f_\circ SS(\mathcal F) \setminus (f_\circ SS(\mathcal F) \cap (T^* Y)_v)\]  with dimension $m-m=0$, and hence is satisfied for only finitely many points. Choosing an open neighborhood of $v$ avoiding these finite points, we confirm that $f_\circ SS(\mathcal F)$ intersects the graph of $g^* \omega$ at $(v,\omega_0)$ but not any other point over a neighborhood of $v$. Thus the multiplicities of $Z_i$ in the two cycles are equal, as desired.\end{proof}

\section{The compactification}\label{sec-compactification}

Let $X \subseteq \left( \mathbb P^1 \right)^{ n} \times T$ be the graph of $\pi$. Since $X$ is canonically isomorphic to $C^n$, we abuse notation slightly by writing $\pi$ also for the map $X \to  T$.

Let $\overline{X}$ be the closure of $X$ in $\left( \mathbb P^1 \right)^{ n} \times T$.  Let $s\colon  \overline{X} \to ( \mathbb P^1)^n \times T$ be the closed immersion. Let $u\colon X \to \overline{X} $ be the inclusion and $\overline{\pi} \colon X \to T$ the projection, so that $\pi = \overline{\pi} \circ u$.

 By construction, $\overline{\pi}$ is proper, and thus it defines a compactification of the map $\pi$. Compactifications are helpful in a number of \'{e}tale cohomology calculations, and we will use this compactification in multiple ways -- to study vanishing cycles, to estimate the characteristic cycle of $R \pi_! \mathbb Q_\ell$, and, together with Artin's affine theorem, to control the higher-degree perverse cohomology of $R \pi_! \mathbb Q_\ell$. Doing all these things requires information on the local structure of the compactification. The first step of this, in \cref{boundary-divisor-calculation}, is a simple description of the boundary $\overline{X} \setminus X$ of the compactification. Afterwards, we will focus on defining an \'{e}tale-local model for $\overline{X}$. This means for each point $((\alpha_1,\dots,\alpha_n), a)\in \overline{X} \setminus X$ we will define a scheme $\overline{Y} \times T$ and an open neighborhood of $((\alpha_1,\dots,\alpha_n), a)$ with an \'{e}tale map to $\overline{Y} \times T$. This will allow us to easily transfer calculations between $\overline{X}$ and $\overline{Y} \times T$, where we take advantage of the simpler geometry of $\overline{Y}$ -- in particular, the fact that it splits as a product.

Let $Z$ be the locus in $(\mathbb P^1)^n$ consisting of tuples $(\alpha_1,\dots,\alpha_n)$ such that, for each $x \in S$, there exists $i $ from $1$ to $n$ with $\alpha_i = x$.

\begin{lemma}\label{boundary-divisor-calculation} $\overline{X} \setminus X $ and $Z \times T$ are equal as closed subschemes of $\left( \mathbb P^1 \right)^{ n} \times T$. \end{lemma}

\begin{proof} Since $Z \times T$ consists of tuples where at least one $\alpha_i \in S$ and therefore at least one $\alpha_i\notin C$, the sets $X$ and $Z\times T$ do not intersect. So it suffices to check that $\overline X = X \cup (Z \times T)$.

We check this equality locally over an open cover of $\left( \mathbb P^1 \right)^{ n} \times T$.  For $A \subset \{1,\dots,n\}$, let $U_A$ be the open subset of $(\mathbb P^1)^n \times T$ consisting of tuples $((\alpha_1,\dots,\alpha_n), a)$ with $\alpha_i \neq \infty$ if $i \notin A$ and $\alpha_i \neq 0$ if $i\in A$. Thus $\alpha_i$ is a regular function on $U_A$ for $i\notin A$ and $\alpha_i^{-1}$ is a regular function for $i \in A$. Then the $U_A$ cover $\left( \mathbb P^1 \right)^{ n} \times T$, so it suffices to show for all $A$ that 
\begin{equation}\label{local-bdc} \overline{X} \cap U_A = ( X \cup (Z \times T))  \cap U_A.\end{equation}

The graph of $\pi$ is the locally closed subset defined by the open condition $\alpha_1,\dots, \alpha_n \in C$ and the system of equations  
\begin{equation}\label{pi-graph-equation-1} \prod_{i=1}^n ( x - \alpha_i) = a(x)\end{equation}
for $x\in S\setminus\{\infty\}$. However, these equations may not be regular functions on $U_A$. To fix this, we multiply by $\alpha_i^{-1}$ for $i \in A$, obtaining
\begin{equation}\label{pi-graph-equation-2} \Bigl( \prod_{i \in A} (\alpha_i^{-1} x - 1) \Bigr) \Bigl(\prod_{i \notin A} (x-\alpha_i ) \Bigr) = a(x) \prod_{i\in A} \alpha_i^{-1} \end{equation}
for $x \in S\setminus \{\infty\}$. These equations define a closed subset $\widetilde{X}_A$ of $U_A$ which includes $X\cap U_A$, and therefore includes $\overline{X}\cap U_A$.

We will now show $\widetilde{X}_A$ is contained in $( X \cup (Z \times T)) \cap U_A$. In other words, we assume $((\alpha_1,\dots,\alpha_n), a) \in U_A$ solves \eqref{pi-graph-equation-2} and prove that $((\alpha_1,\dots,\alpha_n), a)\in X$ or $((\alpha_1,\dots,\alpha_n), a)\in  Z \times T$.

Suppose first that none of the $\alpha_i$ are equal to $\infty$.  The right side of \eqref{pi-graph-equation-2} is then nonzero, so the left side is nonzero, so we have $\alpha_i \neq x$ for all $i$ and all $x\in S \setminus\{\infty\}$, but since we also have $\alpha_i \neq \infty$, we have $\alpha_i\notin S$, i.e. $\alpha_i\in C$. Multiplying \eqref{pi-graph-equation-2} by $\prod_{i\in A} \alpha_i$, we obtain \eqref{pi-graph-equation-1}, which together with $\alpha_i\in X$ implies $((\alpha_1,\dots,\alpha_n), a)\in X$.

 Suppose instead that at least one of the $\alpha_i$ equals $\infty$. Then the right side of \eqref{pi-graph-equation-2} is zero, so the left side is zero, so for each $x\in S\setminus\{\infty\}$, there must exist $i$ with $\alpha_i=x$. The same is true for $x=\infty$, so we have $((\alpha_1,\dots,\alpha_n),a) \in Z \times T$ by definition.
 
 Conversely, we have already seen that any point in $X$ solves \eqref{pi-graph-equation-2}, and can see that every point in $Z\times T$ solves \eqref{pi-graph-equation-2} since both sides of the equation vanish on $Z \times T$. It follows that 
 \[( X \cup (Z \times T)) \cap U_A = \widetilde{X}_A\supseteq  \overline{X} \cap U_A    \] 
 
To show \eqref{local-bdc}, it suffices to show  $\widetilde{X}_A  \subseteq \overline{X} \cap U_A $, or equivalently that $X\cap U_A$ is dense in $\widetilde{X}_A$. Because \eqref{pi-graph-equation-2} is a system of $m$ equations in $n+m$ variables, each irreducible component of its solution set $\widetilde{X}_A$ must have dimension $\geq n$. Thus, to check that $X\cap U_A$ is dense in its solution set $\widetilde{X}_A$, it suffices to check that the complement of $X\cap U_A$ in $\widetilde{X}_A$ has dimension $<n$. The complement of $X\cap U_A$ in the solution set is $(Z \times T) \cap U_A$, which has dimension $<n$ because $\dim Z \leq n - (m+1)$, as, on each irreducible component of $Z$, only $n - (m+1)$ of the $\alpha_i$s may vary. 
 \end{proof}
 
 Let $z\colon Z \times T \to \overline{X}$ be the closed complement of $u$.

We will build local models for the pair $X \subseteq \overline{X} $, i.e. spaces, locally in the \'{e}tale topology, isomorphic to $\overline{X}$, with open subspaces locally isomorphic to $X$.  The purpose of these local models is to prove, in \cref{locally-acyclic-easy}, that $u_! {\mathbb Q}_\ell$ is universally locally acyclic relative to $\overline{\pi}$ in a neighborhood of the boundary of $\overline{X}$; in \cref{singular-support-embedded}, that the singular support of $s_* u_! \mathbb Q_\ell$, restricted to $Z \times T$, intersects the cotangent space to $T$ only at $0$; and, in \cref{lisse-on-boundary}, that $R \overline{\pi}_*  z_* z^*   R u_*  \mathbb Q_\ell$ is a complex of lisse sheaves with unipotent global monodromy.  After proving these, we can forget the specific geometry of our local model and just work with \cref{locally-acyclic-easy,singular-support-embedded,lisse-on-boundary}.

Our local models will depend on the data of a map $\mu\colon\{1,\dots, n \} \to S$, which we require to be surjective, a section $\gamma \colon S \setminus \{\infty\} \to \{1,\dots, n\}$ of $\mu$, and a choice of $x_0 \in S \setminus \{\infty \} $. We fix this data, and suppress the dependence on it, until the point where it becomes useful to vary it, which we will highlight.

Let $U \subset (\mathbb P^1)^n$ be the open set consisting of tuples $(\alpha_1,\dots, \alpha_n) \in (\mathbb P^1)^{n}$ such that $\alpha_i \notin S \setminus \{ \mu(i) \}$ for all $i$ and $x \mapsto \alpha_{\gamma(x)}$ is injective as a map $ S \setminus \{\infty\} \to\mathbb P^1$.  Let $o \colon U \times T \to ( \mathbb P^1)^n \times T$ be the open immersion. The definition of $U$ is motivated by the fact that the map $e$, to be defined next, is \'{e}tale on this set, as we will see in \cref{etale-locus-boundary}.

Let $e\colon U\times T  \to \mathbb A^n\times T$ send a tuple $((\alpha_1,\dots, \alpha_n), a)$ to the tuple $((\lm_1,\dots, \lm_n),a)$, where we have
\[ \lm_i =\begin{cases}   \begin{aligned} \mu(i) - \alpha_i  & \hspace{10pt}\textrm{if } \mu(i) \neq \infty\textrm{ and }i \neq \gamma(\mu(i)),\\
 \frac{1}{ x_0 - \alpha_i } &  \hspace{10pt}\textrm{if }\mu(i)= \infty\\
\frac{ (x- \alpha_i )  \prod_{j \notin \mu^{-1}(x)} (x- \alpha_j ) }{ a(x) \prod_{j \in \mu^{-1}(\infty) }(x_0 -  \alpha_j ) }  & \hspace{10pt} \textrm{if }i=\gamma(x)\textrm{ for }x  \in S\setminus\{\infty\} \end{aligned}\end{cases} \] 

Let us check that this map is well-defined on the whole of $U$. Note first that $i = \gamma(x)$ for some $x \in S \setminus \infty$ if and only if $\mu(i)\neq\infty$ and $i= \gamma(\mu(i))$, by the definition of $\gamma$ as a section of $\mu$ away from $\infty$, so for each $i$ exactly one of the cases occurs.

 For $(\alpha_1,\dots, \alpha_n ) \in U$, if $\mu(i) = \infty$ then $\alpha_i \neq x_0$ by definition of $U$, since $x_0\in S \setminus \{\infty\}$, so the denominators $x_0 -\alpha_j$ do not vanish. We can only have a pole in the numerator if $\alpha_j= \infty$. This can only happen if $j \in \mu^{-1}(\infty)$, in which case the pole in the numerator is cancelled by the corresponding pole of $x_0 - \alpha_j$ in the denominator.
 
We have defined this complicated map $e$ because, expressed in the $b_i$ variables, the complicated equations \eqref{pi-graph-equation-2} of $\overline{X}$ (taking $A = \mu^{-1}(\infty)$) become the simple equations \eqref{Y-equation} below. 

Let $\overline{Y}  \subseteq \mathbb A^n$ be the closed subset where
\begin{equation}\label{Y-equation} \prod_{i \in \mu^{-1} (x)}  \lm_i = \prod_{i \in \mu^{-1} (\infty) } \lm_i \end{equation} for all $x \in S \setminus \infty$.

Let $Y \subseteq \overline{Y} $ be the open subset where, in addition, \[ \prod_{i \in \mu^{-1}(x)} \lm_i\neq 0\] for some (equivalently, every) $x \in S$.

We now prove a series of lemmas showing that $\overline{Y} \times T$ is an \'{e}tale-local model for $X$. We first check that $o^{-1}(X) = e^{-1}(Y \times T)$, then that $e$ is \'{e}tale, and finally that $o^{-1}(\overline{X}) =e^{-1}(\overline{Y} \times T)$. (The map $o$ is, by construction, an open immersion.) 

\begin{lemma}\label{open-set-comparison}  We have \[ o^{-1}(X) = e^{-1} (Y\times T) \] as locally closed subsets of $U \times T$. \end{lemma}

\begin{proof} Because $X$ is the graph of $\pi$, $o^{-1}(X)$ consists of tuples $((\alpha_1,\dots,\alpha_n),a)$ where $\alpha_i \neq x$ for any $i$ from $1$ to $n$ and any $x \in S$, and 
\begin{equation}\label{graph-equation}  \prod_{i=1}^n ( x - \alpha_i) = a(x) \end{equation} for all $x \in S\setminus\{\infty\}$.

To check $o^{-1}(X) \subseteq  e^{-1} (Y\times T) $, observe that for $x \in S \setminus \infty$, for $(\lm_1,\dots, \lm_n) =e (\alpha_1,\dots,\alpha_n)$, we have
\begin{equation}\label{finite-a-product} \prod_{i \in \mu^{-1} (x) } \lm_i =\left( \prod_{ i \in \mu^{-1}(x)  \setminus \{\gamma(x)\}} (x- \alpha_i )   \right) \frac{ (x- \alpha_{\gamma(x)} )  \prod_{j \notin \mu^{-1}(x)  } (x- \alpha_j ) }{ a(x) \prod_{j \in \mu^{-1}(\infty)} ( x_0 - \alpha_j  ) } = \frac{ \prod_{ j=1}^n (x- \alpha_j) }{a(x) \prod_{j \in \mu^{-1}(\infty)} ( x_0 - \alpha_j  )} \end{equation}
and
\begin{equation}\label{infinite-a-product} \prod_{i \in \mu^{-1}(\infty)} \lm_i = \frac{1}{  \prod_{j \in \mu^{-1}(\infty)} ( x_0 - \alpha_j  )}\end{equation}
hence if \eqref{graph-equation} holds and $\alpha_i \neq x$ for all $i,x$ then \eqref{finite-a-product} and \eqref{infinite-a-product} are equal and nonzero for all $x$.

Conversely, if \eqref{finite-a-product} and \eqref{infinite-a-product} are equal and nonzero for all $x$, then \eqref{graph-equation} holds, and thus because $a(x) \neq 0$, $\alpha_i\neq x$ for all $i,x$. Hence $e^{-1} (Y\times T) \subseteq o^{-1}(X)$. 
\end{proof}

\begin{lemma}\label{etale-locus-boundary} The map $e$ is \'etale. \end{lemma}

\begin{proof} Because $e$ is a map between smooth varieties, it suffices to check that the $n+m \times n+m$ Jacobian matrix of $e$ is invertible.

For the rows corresponding to the $a$ variables and the $\lm_i$ with either $\mu(i)=\infty$ or  $i \neq \gamma(\mu(i))$, we will check that the diagonal entry, and only the diagonal entry, of the row is nonzero.  This follows because $e$ preserves the $a$ variables, and because $\lm_i$ is defined as either $\frac{1}{x_0 -\alpha_i}$ or $x_{\mu(i) }-\alpha_i$ so it depends only on $\alpha_i$ and its derivative with respect to $\alpha_i$ is nonzero.

Removing these rows and the corresponding columns, it suffices to check that the $m \times m$ matrix with entries $\frac{ \partial \lm_{\gamma(x)}} {\partial \alpha_{\gamma(y)}}$ is invertible for $x,y \in S \setminus \infty$. We have

\[  \frac{\partial}{ \partial \alpha _{\gamma(y)} }  \frac{ (x- \alpha_{\gamma(x)} )  \prod_{j \notin \mu^{-1}(x)} (x- \alpha_j ) }{ a(x) \prod_{j \in \mu^{-1}(\infty) }(x_0 -  \alpha_j ) } = \frac{-1}{ x- \alpha_{\gamma(y)} }   \frac{ (x- \alpha_{\gamma(x)} )  \prod_{j \notin \mu^{-1}(x)} (x- \alpha_j ) }{ a(x) \prod_{j \in \mu^{-1}(\infty) }(x_0 -  \alpha_j ) } \] since the factor $(x-\alpha_{\gamma(y)})$ appears once in the numerator, and no other factor depending on $\alpha_{\gamma(y)}$ appears in the numerator or denominator, regardless of if $y=x$ and so $\gamma(y) = \gamma(x)$ or $y \neq x$ and so $\mu(\gamma(y)) \neq x$. 

Thus the determinant of this $m \times m$ matrix is \[ \det M \cdot  (-1)^m  \prod_{ x\in S \setminus\{\infty\}} \left(  \frac{ (x- \alpha_{\gamma(x)} )  \prod_{j \notin \mu^{-1}(x)} (x- \alpha_j ) }{ a(x) \prod_{j \in \mu^{-1}(\infty) }(x_0 -  \alpha_j ) } \right) \] where $M$ is the matrix whose entries are given by $M_{xy} = \frac{ 1}{ x- \alpha_{\gamma(y)}}$. Because $M$ is a Cauchy matrix, its determinant is \[ \det M = \frac{ \prod_{ \{x,y \} \subseteq S \setminus \{\infty\}, x\neq y} (x-y) (\alpha_{\gamma(x)} - \alpha_{\gamma(y)})  } {\prod_{x \in S \setminus \{\infty\}} \prod_{y \in S \setminus \{\infty\}} (x- \alpha_{\gamma(y)} ) } .\] All the factors in the numerator of $\det M$ are nonzero by the definition of $U$. The only possible zero factor in $  \frac{ (x- \alpha_{\gamma(x)} )  \prod_{j \notin \mu^{-1}(x)} (x- \alpha_j ) }{ a(x) \prod_{j \in \mu^{-1}(\infty) }(x_0 -  \alpha_j ) } $ is $(x-\alpha_{\gamma(x)})$ as we have $\alpha_j \neq x$ for $j \notin \mu^{-1}(x)$ by the definition of $U$, and these factors cancel the matching factors in the denominator of $\det M$, so in fact the whole product is nonzero, and thus the Jacobian is nonzero.\end{proof}

\begin{lemma}\label{closure-comparison}  We have \[ o^{-1} ( \overline{X} ) = e^{-1} (\overline{Y} \times T ) \] as closed subsets of $U \times T$.  \end{lemma}

\begin{proof} Let us check that 
\[ o^{-1} ( \overline{X} )  = \overline{ o^{-1} ( X) } = \overline{ e^{-1} (Y \times T) } = e^{-1} (\overline{Y} \times T ). \] 

The first equality holds because closure commutes with restriction to an open subset. The second equality holds by \cref{open-set-comparison}. Because closure commutes with \'{e}tale pullbacks, to check the third equality it suffices to show that $\overline{Y}$ is the closure of $Y$. Because $\overline{Y}$ is closed and contains $Y$ by definition, it suffices to show $Y$ is dense in $\overline{Y}$. To do this, observe that $\overline{Y}$ is defined by $m$ equations in $n$ variables, so each irreducible component of $\overline{Y}$ has dimension $\geq n-m$. On the other hand, by counting parameters, $\overline{Y} \setminus Y$ has dimension $n-m-1$.
 \end{proof}

Let \[o_1\colon (U \times T) \cap \overline{X} \to \overline{X} \] and \[ o_2 \colon (U \times T) \cap X \to X \]be the open immersions. Let \[ e_1\colon (U \times T) \cap \overline{X} = (U \times T) \times_{\mathbb A^n} \overline{Y} \to \overline{Y} \times T \] and  \[ e_2\colon (U \times T) \cap X = (U \times T) \times_{\mathbb A^n} Y \to Y \times T \]   be the \'{e}tale maps whose existence is guaranteed by Lemmas \ref{open-set-comparison} and \ref{closure-comparison} earlier. Let \[ u\colon  X \to \overline{X} ,\] \[ u_U \colon (U \times T) \cap X \to (U \times T) \cap \overline{X} ,\] and \[ u_Y\colon Y \to \overline{Y} \] be the open immersions.
 
We have a commutative diagram with Cartesian squares

\begin{equation}\label{small-local-model}\begin{tikzcd}
%& Z & (U \cap Z) \arrow[l,"o_3"] \arrow[r,"e_3"] & (\overline{Y} \setminus Y)  \arrow[d,"z_Y"] \\ 
%Z \times T  \arrow[d, "z"]  \arrow[ur,"pr_1"] & (U \cap Z) \times T  \arrow[l,"o_3 \times id"] \arrow[ur,"pr_1"]  \arrow[r,"e_3 \times id"] \arrow[d,"z_U"] & (\overline{Y} \setminus Y ) \times T \arrow[d,"z_Y \times id"]\arrow[ur,"pr_1"]  
\overline{X} & (U \times T) \cap \overline{X} \arrow[l ,"o_1"] \arrow[r,"e_1"]& \overline{Y} \times T\arrow[r, "pr_1"]   &  \overline{Y}  \\
X \arrow[u, "u"]&  (U \times T) \cap X \arrow[l,"o_2"]  \arrow[r,"e_2"] \arrow[u,"u_U"] & Y \times T \arrow[u,"u_Y \times id"]\arrow[r,"pr_1"]  & Y \arrow[u,"u_Y"]
\end{tikzcd} \end{equation}

In the following lemmas, for simplicity of notation, we use the equals sign to denote a canonical isomorphism of sheaves or complexes of sheaves. The exact sense in which these isomorphisms are canonical does not need to be made precise because, in applications, we will only use the fact that these are isomorphisms.

\begin{lemma}  We have  \begin{equation}\label{eq-shriek-flat} o_1^*    u_!  \mathbb Q_\ell =  e_1^* pr_1^* u_{Y!} \mathbb Q_\ell . \end{equation}  \end{lemma}

\begin{proof}  Observe that extension by zero commutes with \'{e}tale and smooth pullbacks, so
\begin{align*} &  o_1^*  u_!  \mathbb Q_\ell  = u_{U!} o_2^* \mathbb Q_\ell = u_{U!}\mathbb Q_\ell =   u_{U!} e_2^* \mathbb Q_\ell  =  e_1^*  (u_Y \times id)_! \mathbb Q_\ell \\ & = e_1^* (u_Y \times id)_! pr_1^* \mathbb Q_\ell = e_1^* pr_1^* u_{Y!} \mathbb Q_\ell . \qedhere  \end{align*} \end{proof} 
%e_1^* ( (u_{Y!} \mathbb Q_\ell) \boxtimes \mathbb Q_\ell).\end{split} \end{equation*}

%((R u_{Y*} \mathbb Q_\ell) \boxtimes \mathbb Q_\ell) = (e_3 \times id)^* (z_Y \times id)^*( (R u_{Y*} \mathbb Q_\ell) \boxtimes \mathbb Q_\ell)  = (e_3^* z_Y^* R u_{Y_*} \mathbb Q_\ell) \boxtimes \mathbb Q_\ell .\end{split} \end{equation*}

We now recall that the open sets $U$ and maps we have constructed depend on data -- more precisely a surjection $\mu\colon \{1,\dots,n\} \to S$, a section $\gamma$ of $\mu$ over  $ S \setminus \{ \infty \}$, and a fixed $x_0 \in S \setminus \{\infty\}$.

\begin{lemma} \label{Zariski-open-cover} Fix $x_0 \in S \setminus \{\infty\}$. For any $ (\alpha_1,\dots, \alpha_n)\in Z$, we can choose $\mu\colon  \{1,\dots, n\} \to S$ surjective and choose a section $\gamma$ of $\mu$ over $S\setminus\{\infty\}$ so that
\[ (\alpha_1,\dots, \alpha_n) \in U\]
and $\alpha_{\gamma(x)}=x$ for all $x\in S\setminus\{\infty\}$.

\end{lemma}

\begin{proof} %Because the maps $o_1$ are open immersions, it suffices to prove that every point of $Z \times T$ is contained in the image of some $o_1$. Because the image of $o_1$ consists of points contained in $U \times T$, it suffices to show that, given any point $(\alpha_1,\dots, \alpha_n) \in Z$, we can choose $\mu, \gamma(x)$ so that $(\alpha_1,\dots, \alpha_n) \in U$. 

 For each $i$ from $1$ to $n$, let $\mu(i) = \alpha_i$ if $\alpha_i \in S$ and choose $\mu(i)$ arbitrarily otherwise. Then $\alpha_i \notin S\setminus \{\mu(i)\}$. Furthermore, by the definition of $Z$, for each $x\in S$ there exists some $i$ with $\alpha_i = x$. Hence, for this $i$, we have $\mu(i)=x$. Thus $\mu$ is surjective. For each $x$ choose $\gamma(x)$ to be some $i$ with $\alpha_i=x$.

 We have $\alpha_{\gamma(x)} = x$ and so $x\mapsto \alpha_{\gamma(x)}$ is injective. By the definition of $U$, we have $(\alpha_1,\dots, \alpha_n) \in U$. 
 
 The condition $\alpha_{\gamma(x)}=x$ holds by construction. \end{proof}

We are now ready to state and prove \cref{locally-acyclic-easy} and \cref{singular-support-embedded}.

\begin{lemma} \label{locally-acyclic-easy} The sheaf $ u_! \mathbb Q_\ell$ is universally locally acyclic relative to $\overline{\pi}$ in a neighborhood of $Z \times T$. \end{lemma}

\begin{proof} This is a local question, so it suffices to prove it in a neighborhood of each point in $Z \times T$. By \cref{Zariski-open-cover}, one of the maps $o_1$ furnishes such a neighborhood for any given point, since they are open immersions whose image, by definition, consists of the points $( (\alpha_1,\dots, \alpha_n), a)$ with $(\alpha_1,\dots, \alpha_n)\in U $. So it suffices to show that $o_1^*    u_!  \mathbb Q_\ell $ is universally locally acyclic relative to the projection to $T$.

By \cite[Th. Finitude, Corollary 2.16]{sga4h}, every sheaf on a variety over a field $\kappa$ is universally locally acyclic relative to the projection to $\operatorname{Spec} \kappa$. In particular, $u_{Y!} \mathbb Q_\ell$ is universally locally acyclic relative to the projection to a point. Thus the base change $pr_1^* u_{Y!} \mathbb Q_\ell$ is universally locally acyclic relative to the base-changed projection $pr_2$.  Since \'{e}tale pullbacks preserve universal local acyclicity, $e_1^* pr_1^* u_{Y!} \mathbb Q_\ell$ is universally locally acyclic. The universal local acyclicity of $o_1^*    u_!  \mathbb Q_\ell $  then follows from \eqref{eq-shriek-flat}.\end{proof}

Recall the definition of singular support given in \cref{def-singular-support}.

 \begin{lemma}\label{singular-support-embedded}  The fiber of the singular support of $s_* u_! \mathbb Q_\ell$ over a point $x$ in $Z \times T$, viewed as a closed subset of the cotangent space of $(\mathbb P^1)^n \times T$ at $x$, intersects the cotangent space at $T$ only in zero. \end{lemma}
 
 \begin{proof} By \cref{Zariski-open-cover}, we can choose data so that $x$ lies in the image of $o_1$.  Let $s' \colon \overline{Y} \to \mathbb A^n$ and $\tilde{s} \colon  (U \times T) \cap \overline{X} \to \overline{X}$ be the closed immersions. Then by proper base change (in the easy special case of closed immersions) and \eqref{eq-shriek-flat} we have 
 %Then by \eqref{eq-shriek-flat}, $o_1^*    u_!  \mathbb Q_\ell =  e_1^* pr_1^* u_{Y!} \mathbb Q_\ell$, so for $s'\colon \overline{Y} \to \mathbb A^n$ the closed immersion, we have
\begin{equation*}o^*  s_* u_!\mathbb Q_\ell = \tilde{s}_* o_1^*    u_!  \mathbb Q_\ell  = \tilde{s}_*  e_1^* pr_1^* u_{Y!} \mathbb Q_\ell = e^* pr_1^* s'_* u_{Y!} \mathbb Q_\ell . \end{equation*}
 
By \citep[Theorem 1.4(i)]{Beilinson}, because $o, e$ and $pr_1$ are smooth, we have \begin{equation}\label{embedded-eq} o^\circ SS(s_*u_! \mathbb Q_\ell) = e^{\circ} pr_1^{\circ} SS( s'_* u_{Y_!} \mathbb Q_\ell) .\end{equation}
 
 Fix a vector $\omega$ in  $ SS(s_*u_! \mathbb Q_\ell) $ that lies in the cotangent space of $T$. Because $o$ is compatible with the projection to $T$, $do(\omega)$ also lies in the cotangent space of $T$. 
 
Because $\omega \in SS (s_*u_! \mathbb Q_\ell)$, using \eqref{embedded-eq}, $do(\omega)$ is the image under $de$ of a vector  $v' \in pr_1^{\circ} SS( s'_* u_{Y!} \mathbb Q_\ell) $. Because $e$ is \'{e}tale and compatible with the projection to $T$,  $v'$ must also lie in the cotangent space of $T$. Furthermore, $v'$ must lie in the image of $d pr_1$ by construction. This implies $v'=0$ because the image of $dpr_1 $, for $pr_1\colon \mathbb A^n \times T \to \mathbb A^n$ the projection, is the cotangent space of $\mathbb A^n$. Then because $e$ and $o$ are \'{e}tale, $do(\omega)=0$  and $\omega=0$ as well.  \end{proof}

Proving \cref{lisse-on-boundary} will be more difficult. 

Fixing once more the local data $\mu, \gamma,x_0$, observe that we have $(U \times T)  \cap (Z \times T)  = (U \cap Z) \times T$, so the restriction of $o_1 \colon (U \times T ) \cap \overline{X} \to \overline{X}$ to the boundary $Z \times T$ of $\overline{X}$ is the product of the open immersion $U \cap Z \to Z$ with the identity $T \to T$.  We would like to use this to show that the restriction $ z^*  R u_* {\mathbb Q}_\ell$ to the boundary $Z \times T$ is locally a pullback from $Z$. Unfortunately the restriction to $(U \cap Z) \times T$ of the map $e$ is not necessarily the product of an \'{e}tale morphism with the identity, because our formulas for the $\lm_i$ depend on $a \in T$.   To fix this, we must introduce some new spaces, where this product formula does hold.

Let $\tilde{Z} \subseteq Z$ be the closed subset of $Z$ where $\alpha_{\gamma(x)} =x$ for all $x \in S \setminus \{\infty\}$.  Let $\tilde{z}\colon  \tilde{Z} \times T \to \overline{X}$ be the closed immersion.  Let \[o_3\colon U \cap \tilde{Z} \to \tilde{Z}\]  be the open immersion and  \[ \tilde{z}_U\colon ( U \cap \tilde{Z} ) \times T \to (U \times T) \cap \overline{X} \] the restricted closed immersion.

Let $\tilde{Y}$ be the closed set in $\overline{Y}$ where $\lm_{\gamma(x)} =0 $ for $x \in S \setminus\{0\}$. Let \[z_Y\colon \tilde{Y}  \to \overline{Y} \]  be the closed immersion.

Let $e_3 \colon \tilde{Z} \to \tilde{Y} $ send a tuple $(\alpha_1,\dots, \alpha_n)$ to the tuple $(\lm_1,\dots, \lm_n)$, where we have
\[ \lm_i =\begin{cases}   \begin{aligned} \mu(i) - \alpha_i  & \hspace{10pt}\textrm{if } \mu(i) \neq \infty\textrm{ and }i \neq i_{\mu(i)},\\
 \frac{1}{ x_0 - \alpha_i } &  \hspace{10pt}\textrm{if }\mu(i)= \infty\\
0   & \hspace{10pt} \textrm{if }i=\gamma(x)\textrm{ for }x  \in S\setminus\{\infty\} \end{aligned}\end{cases} \] 

We can check that the restriction of $e_1$ to $(U \cap \tilde{Z}) \times T$ is given by $e_3 \times id$ by observing that the formula defining $e$ restricts to the formula defining $e_3$ if $\alpha_{\gamma(x)}-x=0$ for all $x\in S\setminus\{\infty\}$.

We can expand our earlier diagram \eqref{small-local-model} to 
 
\begin{equation}\label{big-local-model} \begin{tikzcd}
& \tilde{Z} & (U \cap \tilde{Z}) \arrow[l,"o_3"'] \arrow[r,"e_3"] & \tilde{Y}   \arrow[d,"z_Y"] \\ 
\tilde{Z} \times T  \arrow[d, "\tilde{z}"]  \arrow[ur,"pr_1"] & (U \cap \tilde{Z} ) \times T  \arrow[l,"o_3 \times id"] \arrow[ur,"pr_1"]  \arrow[r,"e_3 \times id"] \arrow[d,"\tilde{z}_U"] & \tilde{Y}  \times T \arrow[d,"z_Y \times id"]\arrow[ur,"pr_1"]   &  \overline{Y}     \\
\overline{X} & (U \times T) \cap \overline{X} \arrow[l ,"o_1"] \arrow[r,"e_1"]& \overline{Y} \times T\arrow[ur, "pr_1"'] & Y \arrow[u,"u_Y"']  \\
X \arrow[u, "u"]&  (U \times T) \cap X \arrow[l,"o_2"]  \arrow[r,"e_2"] \arrow[u,"u_U"] & Y \times T \arrow[u,"u_Y \times id"]\arrow[ur,"pr_1"']  
\end{tikzcd} \end{equation}

This is a commutative diagram with Cartesian squares.

\begin{lemma} We have \begin{equation}\label{eq-pushforward-flat}  (o_3 \times id)^* \tilde{z}^*   R u_*  \mathbb Q_\ell = pr_1^* e_3^* z_Y^* R u_{Y*} \mathbb Q_\ell .\end{equation}  \end{lemma}

\begin{proof} Because derived pushforward  commutes with \'{e}tale and smooth pullbacks,
\begin{align*} & (o_3 \times id)^* \tilde{z}^*   R u_*  \mathbb Q_\ell  = \tilde{z}_U^* o_1^*  R u_*  \mathbb Q_\ell  = \tilde{z}_U^*  R u_{U*} o_2^* \mathbb Q_\ell =\tilde{z}_U^*  R u_{U*}\mathbb Q_\ell = \tilde{z}_U^*  R u_{U*} e_2^* pr_1^* \mathbb Q_\ell   \\ & = \tilde{z}_U^* e_1^* R (u_Y \times id)_* pr_1^* \mathbb Q_\ell =\tilde{z}_U^* e_1^*  pr_1^* R u_{Y*} \mathbb Q_\ell = pr_1^* e_3^* z_Y^* R u_{Y*} \mathbb Q_\ell. \qedhere \end{align*}\end{proof}

\begin{lemma}\label{lisse-on-boundary}  The cohomology sheaves of the complex \[ R \overline{\pi}_*  z_* z^*   R u_*  \mathbb Q_\ell \] are lisse sheaves on $T$ with unipotent global monodromy. \end{lemma}

Here ``unipotent global monodromy" means that the representation of the geometric fundamental group of $T$ associated with the lisse sheaf is an iterated extension of trivial representations (so that each element acts by a unipotent matrix).

\begin{proof} We have $\overline{\pi} \circ z = pr_2$, where $pr_2 \colon Z \times T \to T$ is the projection.

For each choice of $\mu$ and $\gamma(x)$ for $x\in S \setminus \infty$, the set $\tilde{Z} \cap U$ is locally closed in $Z$. By \cref{Zariski-open-cover}, these locally closed sets cover $Z$. Given a finite cover of a Noetherian scheme $Z$ by locally closed sets, we can refine it into a stratification by choosing a generic point, choosing one of the locally closed sets that contains it, observing that it contains an open neighborhood of that generic point, removing that open neighborhood, and repeating the process. (This may involve using a given locally closed set repeatedly, which is fine.)  Let $(Z_i)_{1\leq i \leq N}$ be a locally closed stratification refining $\tilde{Z} \cap U$, and let $r_i\colon Z_i \to Z_i \cap U$ be the locally closed immersion.

By \cref{eq-pushforward-flat}, the pullback of $R u_* \mathbb Q_\ell$ to $(U \cap  \tilde{Z} ) \times T$ is $pr_1^* e_3^* \tilde{z}_Y^* R u_{Y*} \mathbb Q_\ell$. It follows that the pullback of $Ru_* \mathbb Q_\ell$ to $Z_i \times T$ is $pr_1^* r_i^* e_3^* \tilde{z}_Y^* R u_{Y*} \mathbb Q_\ell$ on $Z_i$. 

%By \cref{Zariski-open-cover} and \eqref{eq-pushforward-flat} we have an open cover $U_i$ of $Z$ such that, restricted to each open $U_i \times T$, the complex $z^*   R u_*  \mathbb Q_\ell$ is pulled back from $U_i$. We can refine this open cover to a stratification into locally closed subsets $Z_i \times T$ such that the pullback of $R u_*  \mathbb Q_\ell$ to each stratum $Z_i \times T$ is a pullback from $Z_i$.

We have a spectral sequence computing $R pr_{2*} z^*   R u_*  \mathbb Q_\ell$ whose first page is
\[ \bigoplus_{i=1}^N  Rpr_{2!} \left( pr_1^*  r_i^* e_3^* \tilde{z}_Y^* R u_{Y*} \mathbb Q_\ell \right) \] where $pr_1\colon Z_i \times T \to Z_i$ and $pr_2 \colon Z_i \times T \to T$ are the projections. 

 The construction of the spectral sequence is essentially identical to \cite[(3)]{Petersen2017}. To construct this spectral sequence, let $w_i \colon \bigcup_{j \leq i} Z_j \times T \to (\tilde{Z} \cap U)\times T$ be the closed immersion (closed by the definition of stratification).
We can filter the complex   $ z^*   R u_*  \mathbb Q_\ell$ into the series of complexes $w_{i*} w_i^* z^*   R u_*  \mathbb Q_\ell$. The $i$'th associated graded object of this filtration is then the restriction of $z^* R u^* \mathbb Q_\ell$ to $Z_i \times T$. This filtration induces a filtration on  $ Rpr_{2*} z^*   R u_*  \mathbb Q_\ell $ whose $i$'th associated graded is $Rpr_{2!} \left( pr_1^*  r_i^* e_3^* \tilde{z}_Y^* R u_{Y*} \mathbb Q_\ell \right) $ where $pr_1\colon Z_i \times T \to Z_i$ and $pr_2 \colon Z_i \times T \to T$ are the projections. The spectral sequence of a filtered complex associated to this filtration is our desired spectral sequence.

By proper base change, $Rpr_{2!} \left( pr_1^* r_i^* e_3^* \tilde{z}_Y^* R u_{Y*} \mathbb Q_\ell \right)$ is the pullback from a point of \[H^*_c ( Z_i, r_i^* e_3^* \tilde{z}_Y^* R u_{Y*} \mathbb Q_\ell ),\] and hence its cohomology sheaves are constant.  The kernels and cokernels of any map of constant sheaves are constant sheaves. Since each page of the spectral sequence is obtained from the previous by taking kernels and cokernels of differentials, it follows by induction that each sheaf appearing on each page of the sequence is constant. The final page of the spectral sequence gives the associated graded sheaves of a filtration on the cohomology sheaves of the limit $R pr_{2*} z^*   R u_*  \mathbb Q_\ell$, so the cohomology sheaves of the limit are iterated extensions of constant sheaves and thus are lisse with unipotent global monodromy (because extensions of lisse sheaves give extensions of representations). \end{proof}

\section{Structure of the complex}\label{sec-structure}

We will understand the complex $R \pi_! \mathbb Q_\ell $ by studying its perverse cohomology in every degree. The first step follows quickly from \cref{lisse-on-boundary}.

\begin{lemma} \label{lisse-higher-cohomology} The perverse cohomology sheaf \[ {}^p\mathcal H^j ( R \pi_! \mathbb Q_\ell ) \] vanishes for $j< n$ and is lisse with unipotent global monodromy for $j>n$.
 \end{lemma} 

\begin{proof} To check that it vanishes for $j < n$ we use the fact that $\mathbb Q_\ell [n]$ is perverse on $C^n$, $\pi$ is an affine morphism, and \citep[Corollary 4.1.2]{bbd}.

We have the distinguished triangle \[  u_! \mathbb Q_\ell \to R u_* \mathbb Q_\ell \to z_* z^* R u_* \mathbb Q_\ell .\] 

Applying $R \overline{\pi}_*$ and using $\overline{\pi} \circ u = \pi$, this gives a distinguished triangle
\[  R \pi_! \mathbb Q_\ell \to R \pi_* \mathbb Q_\ell \to R \overline{\pi}_* z_* z^* R u_* \mathbb Q_\ell \] and a long exact sequence
\[ \dots \to {}^p \mathcal H^{j-1} (R \overline{\pi}_* z_* z^* R u_* \mathbb Q_\ell) \to   {}^p \mathcal H^j (  R \pi_! \mathbb Q_\ell )  \to {}^p \mathcal H^j ( R \pi_* \mathbb Q_\ell) \to \dots \]

Using the fact that $\mathbb Q_\ell[n]$ is perverse on $C^n$, $\pi$ is an affine morphism, and \citep[Theorem 4.1.1]{bbd}, ${}^p \mathcal H^j ( R \pi_* \mathbb Q_\ell)$ vanishes when $j>n$. Thus $ {}^p \mathcal H^j (  R \pi_! \mathbb Q_\ell )$ is a quotient of $ {}^p \mathcal H^{j-1} (R \overline{\pi}_* z_* z^* R u_* \mathbb Q_\ell)$ in this range.

By Lemma \ref{lisse-on-boundary}, all the cohomology sheaves of $R \overline{\pi}_* z_* z^* R u_* \mathbb Q_\ell$ are lisse with unipotent global monodromy. When the cohomology sheaves of a complex on a smooth variety of dimension $m$ are lisse, the $i$th perverse cohomology sheaf is simply the $i-m$'th usual cohomology sheaf, since the usual truncation in degree $\leq i-m$ satisfies the hypotheses of the perverse truncation in degree $\leq i$. It follows that all the perverse cohomology sheaves of $R \overline{\pi}_* z_* z^* R u_* \mathbb Q_\ell$ are lisse with unipotent global monodromy. Because the quotient of a lisse perverse sheaf remains lisse, and its monodromy representation is a quotient representation, ${}^p\mathcal H^j ( R \pi_! \mathbb Q_\ell ) $ is lisse with unipotent global monodromy when $j>n$. \end{proof}

Observe that $R \pi_! \mathbb Q_\ell$ admits an action of $S_n$, arising from the action of $S_n$ on $C^n$, which preserves $\pi$.  The complex $\left( R \pi_! \mathbb Q_\ell \otimes \rho\right)^{S_n}$ will be of great interest to us because, by \cref{trace-function}, its trace function at a point $a$ is precisely the sum of $F_\rho$ over $f$ congruent to $a$ mod $g$.

It follows from \cref{lisse-higher-cohomology} that \[{}^p \mathcal H^j \left( \left( R \pi_! \mathbb Q_\ell \otimes \rho\right)^{S_n}\right) \] vanishes for $j<n$ and is lisse with unipotent global monodromy for $j<n$, because these properties are stable under tensor product with $\rho$ and taking $S_n$-invariants, which are also exact functors. For $\rho$ an irreducible representation, we can give an even more precise description of $\left( R \pi_! \mathbb Q_\ell \otimes \rho\right)^{S_n}$. Depending on the partition corresponding to the irreducible reprsentation $\rho$, we will either show that $\left( R \pi_! \mathbb Q_\ell \otimes \rho\right)^{S_n}$ is a sum of shifts of constant sheaves (in \cref{irrep-constant}) or that $\left( R \pi_! \mathbb Q_\ell \otimes \rho\right)^{S_n}[n]$ is a perverse sheaf (in \cref{irrep-perverse}). The first case, the constant sheaves, will give the main term in our estimate for the sum of $F_\rho$. The second case, the perverse sheaves, will give the error term, and the next two sections will be devoted to controlling the stalks of these perverse sheaves. 

 %We use the usual notation where irreducible representations of $S_n$ correspond to partitions of $n$, with $n$ corresponding to the trivial representation and $1+ \dots +1$ corresponding to the sign representation. 

\begin{lemma}\label{irrep-constant} Suppose $\rho$ is an irreducible representation of $S_n$ corresponding to a partition with at least one part of size at least $m$. Then $\left( R \pi_! \mathbb Q_\ell \otimes \rho\right)^{S_n}$ is the direct sum of its cohomology sheaves, and these cohomology sheaves are geometrically constant. \end{lemma}

\begin{proof} Let $(n_1, \dots , n_r)$ be the partition of $n$ corresponding to $\rho$, and assume $n_1 \geq m$. By the construction of irreducible representations of $S_n$, $\rho$ is a summand of $\Ind_{S_{n_1} \times \dots \times S_{n_{r}} }^{S_n} \mathbb Q_\ell$.  It suffices to show that the complex $ \left( R \pi_! \mathbb Q_\ell \otimes \Ind_{S_{n_1} \times \dots S_{n_{r}} }^{S_n} \mathbb Q_\ell \right)^{S_n} $ is the direct sum of its cohomology sheaves, and these cohomology sheaves are geometrically constant, because both of these properties pass to summands. By the self-duality of representations of $S_n$ and Frobenius reciprocity, we have
\[ \left( R \pi_! \mathbb Q_\ell \otimes \Ind_{S_{n_1} \times \dots S_{n_{r}} }^{S_n} \mathbb Q_\ell \right)^{S_n} = \Hom_{S_n} ( \Ind_{S_{n_1} \times \dots S_{n_{r}} }^{S_n} \mathbb Q_\ell, R \pi_! \mathbb Q_\ell) \] \[= \Hom_{S_{n_1} \times \dots \times S_{n_r}}(\mathbb Q_\ell, R\pi_! \mathbb Q_\ell)= \left( R \pi_! \mathbb Q_\ell \right)^{S_{n_1} \times \dots \times S_{n_r}}.\] 

By definition we have \[ C^{(n_1)} \times \dots \times C^{(n_r)} = C^{n_1}/ S_{n_1} \times \dots \times C^{n_r}/ S_{n_r} = C^n/ (S_{n_1}\times \dots \times S_{n_r}).\] Let  $sym^{\rho}\colon  C^n \to C^{(n_1)} \times \dots \times C^{(n_r)}$ be the quotient map $C^n \to C^n/ (S_{n_1}\times \dots \times S_{n_r})$ and let $\pi^\rho\colon C^{(n_1)} \times \dots \times C^{(n_r)} \to T$ be the map obtained from the universal property of a quotient by the $S_{n_1}\times \dots S_{n_r}$-invariant map $\pi \colon C^n \to T$.  By definition, we have $\pi^\rho \circ sym^{\rho} =\pi$.  % Concretely, $sym^{\rho}$ send a tuple $\alpha_1,\dots, \alpha_n$ to a tuple of unordered tuples $((\alpha_1,\dots,\alpha_{n_1} ), (\alpha_{n_1+1},\dots, \alpha_{n_1+n_2}),\dots, (\alpha_{n-n_r+1},\dots, \alpha_n))$ and $\pi^\rho$ sends a tuple of unordered tuples $((\alpha_1,\dots,\alpha_{n_1} ), (\alpha_{n_1+1},\dots, \alpha_{n_1+n_2}),\dots, (\alpha_{n-n_r+1},\dots, \alpha_n))$ to $\prod_{i=1}^n (t-\alpha_i) $. 

Because $sym^\rho$ is finite,
\[ \left( R \pi_! \mathbb Q_\ell \right)^{S_{n_1} \times \dots \times  S_{n_r}} = \left( R \pi^\rho_! sym^\rho_* \mathbb Q_\ell \right )^{S_{n_1} \times \dots \times S_{n_r}} = R \pi^\rho_! \left( sym^\rho_* \mathbb Q_\ell\right )^{S_{n_1} \times \dots  \times S_{n_r}} .\]

Because $S_{n_1} \times \dots \times  S_{n_r}$ acts transitively on the fibers of $sym^\rho$, the natural map $\mathbb Q_\ell \to \left( sym^\rho_* \mathbb Q_\ell \right)^{S_{n_1} \times \dots  \times S_{n_r}} $ is an isomorphism on stalks at each point, hence is an isomorphism, so \[ R \pi^\rho_! \left( sym^\rho_* \mathbb Q_\ell \right)^{S_{n_1} \times \dots  \times S_{n_r}}  = R \pi^\rho_! \mathbb Q_\ell \] and it suffices to show $ R \pi^\rho_! \mathbb Q_\ell $ is the sum of its cohomology sheaves, and these cohomology sheaves are geometrically constant.

To do this, note that we can view $C^{(n_i)}$ as the moduli space of monic polynomials in $t$ of degree $n_i$ which are prime to $g$, and then $\pi^{\rho}$ is the map multiplying polynomials mod $g$. We have an isomorphism \[C^{(n_1)} \times \dots \times C^{(n_r)} \cong T \times \mathbb A^{n_1-m} \times C^{(n_2)} \times \dots \times C^{(n_r)} \] that sends a tuple of polynomials $(f_1,\dots, f_r)$ with $f_1 = T^{n_1} + c_1 T^{n_1-1} + \dots + c_{n_1} $ to 
\[\Bigl(\prod_{i=1}^{r} f_i \mod g, (c_1,\dots, c_{n_1-m}) , f_2,\dots, f_r  \Bigr) .\] To check that the map is an isomorphism, it suffices to check that, for any \[ a \in T, (c_1,\dots, c_{n_1 - m }) \in \mathbb A^n, f_2 \in C^{(n_2)},\dots, f_r \in C^{(n_r)}\]  there is a unique polynomial $f_1$ such that \[f_1 \mod g = \frac{a}{\prod_{i=2}^r f_i \mod g} \] and the leading terms of $f_1$ are \[T^{n_1} + c_1 T^{n_1-1} + \dots + c_{n_1-m} T^m + \dots\] and, furthermore, check that the coefficients of $f_1$ are regular functions on $T \times \mathbb A^{n_1-m} \times C^{(n_2)} \times \dots \times C^{(n_r)}$. 

A unique polynomial $f_1$ exists because, after subtracting off the leading terms, every residue class mod $g$ is represented by a unique polynomial of degree $<m$. We can write the coefficients of $f_1$ as regular functions %in $c_1,\dots c_{n_1-m}$ as well as the coefficients of $a$ and $ f_2,\dots, f_m$,
 on $T \times \mathbb A^{n_1-m} \times C^{(n_2)} \times \dots \times C^{(n_r)}$ by using polynomial long division.

With this isomorphism in hand, it suffices to show that $R pr_{1!} \mathbb Q_\ell$ is the direct sum of its cohomology sheaves, and these sheaves are geometrically constant, for $pr_1\colon   T \times \mathbb A^{n_1-m} \times C^{(n_2)} \times \dots \times C^{(n_r)} \to T$ the projection. This follows from proper base change, as $R pr_{1!} \mathbb Q_\ell$ is the pullback of  $H^*_c( \mathbb A^{n_1-m} \times C^{(n_2)} \times \dots \times C^{(n_r)}_{\overline{\kappa}},\mathbb Q_\ell) $ from a point, and that complex splits as a direct sum of shifts of vector spaces (as does any complex of sheaves on a point). \end{proof}

\begin{lemma}\label{rho-cohomology} For $\rho$ any representation of $S_n$, we have an isomorphism  \[ H^*_c \left( T_{\overline{\kappa} } , \left( R \pi_! \mathbb Q_\ell \otimes \rho\right)^{S_n}\right) = \left( H^*_c ( C_{\overline{\kappa}} , \mathbb Q_\ell)^{\otimes n} \otimes \rho \right)^{S_n}\]
and for all integers $d$ we have \[ \dim H^{n+d}_c \left( T_{\overline{\kappa} } , \left( R \pi_! \mathbb Q_\ell \otimes \rho\right)^{S_n}\right) = \dim   V_{\rho}^d .\]

 \end{lemma}

\begin{proof} First observe that \[ H^*_c \left( T_{\overline{\kappa} } , \left( R \pi_! \mathbb Q_\ell \otimes \rho\right)^{S_n}\right)= \left( H^*_c(T_{\overline{\kappa}}, R\pi_! \mathbb Q_\ell) \otimes \rho \right)^{ S_n} = \left( H^*_c (C^n_{\overline{\kappa}}, \mathbb Q_\ell ) \otimes \rho \right)^{S_n}  = \left( H^*_c ( C_{\overline{\kappa}} , \mathbb Q_\ell)^{\otimes n} \otimes \rho \right)^{S_n} \] where the three inequalities use, respectively, that taking $S_n$-invariants is an exact functor and thus commutes with cohomology, the Leray spectral sequence with compact supports \cite[XVII, Theorem 5.1.8(a)]{sga4-3}, and the K\"unneth formula \cite[XVII, Theorem 5.4.3]{sga4-3}.

Next observe that $H^j_c ( C_{\overline{\kappa}}, \mathbb Q_\ell)$ has dimension $0$ unless $i=1$ or $i=2$, dimension $m$ if $i=1$, and dimension $1$ if $i=2$.  Thus
\[ \hspace{-.35in} H^*_c ( C_{\overline{\kappa}} , \mathbb Q_\ell)^{\otimes n}   \cong (  H^0( C_{\overline{\kappa}} , \mathbb Q_\ell) \oplus  H^1( C_{\overline{\kappa}} , \mathbb Q_\ell) \oplus  H^2( C_{\overline{\kappa}} , \mathbb Q_\ell) )^{\otimes n}  \cong ( 0 \oplus (\mathbb Q_\ell)^m \oplus \mathbb Q_\ell)^{\otimes n} \textrm{ as }\mathbb Q_\ell\textrm{-vector spaces}.\]

The $n+d$th graded piece is generated by tensors of $n-d$ vectors in $\mathbb Q_\ell^m$ with $d$ vectors in $\mathbb Q_\ell$, and thus we have
\[ \mathcal H^{n+d} \Bigl( H^*_c ( C_{\overline{\kappa}} , \mathbb Q_\ell)^{\otimes n}  \Bigr) \cong \operatorname{Ind}_{S_{n-d} \times S_d }^{S_n}  ( \mathbb Q_\ell^{m} )^{\otimes n-d} \textrm{ as }\mathbb Q_\ell\textrm{-vector spaces}.\]
Here the induced representation is a sum over $S_n / (S_{n-d} \times S_d)$, i.e. a sum over $d$-element subsets of $\{1,\dots,n\}$. 
However, this is not necessarily an isomorphism as $S_n$-representations.  To obtain an isomorphism, before inducing, we must twist by the sign character of the permutation group acting on the odd degree tensor factors, i.e., we must twist by the sign representation of $S_{n-d}$. Thus 
%The cohomology in degree $n+d$ of the tensor product $H^*_c ( C_{\overline{\kappa}} , \mathbb Q_\ell)^{\otimes n} $  is the sum over size $d$ subsets of $\{1,\dots, n\}$ of the tensor product of $\mathbb Q_\ell$ for each $i$ in the subset with $\mathbb Q_\ell^{m}$ for each $i$ not in the subset.  The action of $S_n$ is by permuting the subsets and permuting the tensor product, except twisted by the sign representation for $i$ in that subset because those terms correspond to cohomology in odd degrees. Thus 
\[ \mathcal H^{n+d} \Bigl( H^*_c ( C_{\overline{\kappa}} , \mathbb Q_\ell)^{\otimes n}  \Bigr) = \operatorname{Ind}_{S_{n-d} \times S_d }^{S_n}  \Bigl( ( \mathbb Q_\ell^{m} )^{\otimes n-d} \otimes \sgn_{S_{n-d}} \Bigr) \textrm{ as }S_n\textrm{-representations} \] so (because tensor products of complexes of $\mathbb Q_\ell$-vector spaces commute with taking cohomology)
\[ \mathcal H^{n+d} \Bigl( H^*_c ( C_{\overline{\kappa}} , \mathbb Q_\ell)^{\otimes n}\otimes \rho \Bigr)= \operatorname{Ind}_{S_{n-d} \times S_d }^{S_n}  \Bigl( ( \mathbb Q_\ell^{m} )^{\otimes n-d} \otimes \rho \otimes \sgn_{S_{n-d}} \Bigr) \] and thus, by Frobenius reciprocity,
\[\hspace{-.5in} \mathcal H^{n+d} \Bigl( H^*_c ( C_{\overline{\kappa}} , \mathbb Q_\ell)^{\otimes n}\otimes \rho \Bigr)^{S_n} = \Bigl( \operatorname{Ind}_{S_{n-d} \times S_d }^{S_n}  \Bigl( ( \mathbb Q_\ell^{m} )^{\otimes n-d} \otimes \rho \otimes \sgn_{S_{n-d}} \Bigr) \Bigr) ^{S_n} =  \Bigl( ( \mathbb Q_\ell^{m} )^{\otimes n-1} \otimes \rho \otimes \sgn_{S_{n-d}} \Bigr)^{S_{n-d} \times S_d }\] which, after replacing $\mathbb Q_\ell$ with $\mathbb C$, is  $V_{\rho}^d$. 
\end{proof} 
\begin{lemma}\label{irrep-cohomology-vanishing} Suppose $\rho$ is an irreducible representation of $S_n$ corresponding to a partition with all parts of size $<m$. Then \[  H^*_c \left( T_{\overline{\kappa} } , \left( R \pi_! \mathbb Q_\ell \otimes \rho\right)^{S_n}\right)  \]  is concentrated in degrees between $n$ and $n+m-1$. \end{lemma}

\begin{proof} This follows from \cref{rho-cohomology} and \cref{other-vanishing-lemma}. \end{proof}

\begin{lemma}\label{irrep-perverse} Suppose $\rho$ is an irreducible representation of $S_n$ corresponding to a partition with all parts of size $<m$. Then $\left( R \pi_! \mathbb Q_\ell \otimes \rho\right)^{S_n}[n]$ is a perverse sheaf. \end{lemma}

\begin{proof}
Assume for contradiction that ${}^p \mathcal H^j \left( \left( R \pi_! \mathbb Q_\ell \otimes \rho\right)^{S_n} \right) $ is nonzero for some $j>n$. Let $j$ be the largest such. Because ${}^p \mathcal H^j \left( \left( R \pi_! \mathbb Q_\ell \otimes \rho\right)^{S_n} \right) $  is lisse with unipotent monodromy, $H^m_c \Bigl( T_{\overline{\kappa}}, {}^p \mathcal H^j \left( \left( R \pi_! \mathbb Q_\ell \otimes \rho\right)^{S_n} \right) \Bigr) $, which is equal to the monodromy coinvariants of ${}^p \mathcal H^j \left( \left( R \pi_! \mathbb Q_\ell \otimes \rho\right)^{S_n} \right) $, is nonzero. %Since all other nonzero perverse cohomology sheaves have degree $<j$, and by \cite[(4.2.4)]{bbd} their cohomology is in degree $\leq m$, none of the differentials
We now use the spectral sequence associated to the perverse filtration of  $\left(R \pi_! \mathbb Q_\ell \otimes \rho\right)^{S_n}$, which computes $H^{p+q}_c \left( T_{\overline{\kappa} } , \left( R \pi_! \mathbb Q_\ell \otimes \rho\right)^{S_n}\right)$  from $E^{p,q}_1 = H^{2p+q}_c \Bigl( T_{\overline{\kappa}}, {}^p \mathcal H^{-p} \left( \left( R \pi_! \mathbb Q_\ell \otimes \rho\right)^{S_n} \right) \Bigr)$, and whose differential on the $r$th page sends $ E^{p,q}_r$ to $E^{p +r, q+1 - r}$ \cite[Definition 3.6.1]{dCM}.

The only differentials which could cancel the $p=-j, q=2j+m $ term $H^m_c \Bigl( T_{\overline{\kappa}}, {}^p \mathcal H^j \left( \left( R \pi_! \mathbb Q_\ell \otimes \rho\right)^{S_n} \right) \Bigr)$ are those going to $E^{p+r,q+1-r}_r$, which must vanish as \[ E^{p+r, q+1-r}_1 = H^{m+1+r}_c \Bigl( T_{\overline{\kappa}}, {}^p \mathcal H^{j-r} \left( \left( R \pi_! \mathbb Q_\ell \otimes \rho\right)^{S_n} \right) \Bigr)= 0 \] since perverse sheaves have cohomology in degree $\leq m$ by \cite[(4.2.4)]{bbd}, and those coming from $E^{p-r, q+r-1}_r$, which must vanish as \[ E^{p-r, q+1-r}_1 = H^{m-r-1}_c \Bigl( T_{\overline{\kappa}}, {}^p \mathcal H^{j+r} \left( \left( R \pi_! \mathbb Q_\ell \otimes \rho\right)^{S_n} \right) \Bigr)=0\] since $j$ is the maximal degree with nonzero perverse cohomology.

Since no differentials cancel the $p=-j, q=2j+m $ term, we have $H^{m+j}_c \left( T_{\overline{\kappa} } , \left( R \pi_! \mathbb Q_\ell \otimes \rho\right)^{S_n}\right) \neq 0$. Because $j>n$, this contradicts \cref{irrep-cohomology-vanishing}, which implies $H^{m+j}_c \left( T_{\overline{\kappa} } , \left( R \pi_! \mathbb Q_\ell \otimes \rho\right)^{S_n}\right) =0$ for $j \geq n$. \end{proof}

The following cohomology vanishing result is the main reason we are able to obtain such strong bounds. Using Deligne's Riemann Hypothesis, the maximum degree in which $\left( R \pi_! \mathbb Q_\ell \otimes \rho\right)^{S_n}$ has nonvanshing cohomology determines the power of $q$ that will appear in our bound for the trace function of $\left( R \pi_! \mathbb Q_\ell \otimes \rho\right)^{S_n}$. Since the maximum degree $n+1-m$ provided by \cref{support-degrees} is close to the middle degree $n-m$ (i.e. the dimension of a fiber of $\pi$), the bound we obtain will be close to square-root cancellation.

\begin{lemma}\label{support-degrees} Suppose that  $n \geq m$ and $\rho$ is an irreducible representation of $S_n$ corresponding to a partition with all parts of size $<m$.

Then  $\left( R \pi_! \mathbb Q_\ell \otimes \rho\right)^{S_n}$ is concentrated in degrees $n-m$ and $n+1-m$. \end{lemma}

\begin{proof} Because $\left( R \pi_! \mathbb Q_\ell \otimes \rho\right)^{S_n}[n]$ is perverse on a variety of dimension $m$, it is concentrated in degree $\geq -m$, so $\left( R \pi_! \mathbb Q_\ell \otimes \rho\right)^{S_n}$ is concentrated in degree $\geq n-m $. 

Fix a point $x$ in $T$. We will show that the stalk of $\left( R \pi_! \mathbb Q_\ell \otimes \rho\right)^{S_n}$ at $x$ is concentrated in degree $\leq n+1-m$. Because $x$ is arbitrary, this will prove that $\left( R \pi_! \mathbb Q_\ell \otimes \rho\right)^{S_n}$ is concentrated in degree $\leq n+1-m$ which, combined with concentration in degree $\geq n-m$, gives the statement.

 Let $\operatorname{Spec} R$ be the spectrum of a discrete valuation ring with special point $s$ and generic point $\eta$ and let $i\colon \operatorname{Spec} R\to T$ be a map such that $i(s)=x$ and $i(\eta)$ is a generic point of $T$. We have a vanishing cycles distinguished triangle (\cite[XIII, (2.1.8.9)]{sga7-2}, applied to $ i^* u_! \mathbb Q_\ell$)  \begin{equation}\label{vanishing-cycles-sequence} (R \overline{\pi}_* u_! \mathbb Q_\ell)_s \to (R \overline{\pi}_* u_! \mathbb Q_\ell)_{\eta} \to H^* ( \overline{\pi}^{-1} (x), R \Phi_R i^* u_! \mathbb Q_\ell) .\end{equation}

At any point where  $u_! \mathbb Q_\ell$ is universally locally acyclic relative to $\overline{\pi}$, by definition $i^* u_! \mathbb Q_\ell$ is locally acyclic and so $R \Phi_R i^* u_! \mathbb Q_\ell=0$ at that point by \cite[XIII, Proposition 2.1.4]{sga7-2} and the definition of local acyclicity. By \cref{locally-acyclic-easy}, this holds in an open neighborhood of $Z \times T$ in $\overline{X}$. Its complement is a closed subset of $\overline{X}$, hence proper over $T$, and also a closed subset of $X$, hence affine over $T$, and thus is finite over $T$. Thus the complement contains finitely many points in the fiber over $s$. Thus  $R \Phi_R i^* u_! \mathbb Q_\ell$ is supported at finitely many points.

Because $n \geq m$, the fibers of $\overline{\pi}$ have dimension $n-m$, so $ i^* u_! \mathbb Q_\ell$ is a sheaf on a variety of dimension $\leq n+1-m$, and thus $i^* u_! \mathbb Q_\ell[ m-1-n]$ is semiperverse. It follows from the (semi)perversity of vanishing cycles \cite[Corollary 4.6]{Illusie94} that $R \Phi_R i^* u_! \mathbb Q_\ell [m-1-n]$ is semiperverse, and in particular $R \Phi_R i^* u_! \mathbb Q_\ell$ is concentrated in degree $\leq n+1-m$. Because $ R \Phi_R i^* u_! \mathbb Q_\ell$ is concentrated at finitely many points and in degree $\leq n+1-m$, $ H^* ( \overline{\pi}^{-1} (x), R \Phi_R i^* u_! \mathbb Q_\ell)$ is concentrated in degree $n +1-m$.  Applying \eqref{vanishing-cycles-sequence} and noting that $\overline{\pi} \circ u = \pi$, the natural map \[ (R \pi_! \mathbb Q_\ell )_x \to (R \pi_! \mathbb Q_\ell)_{i(\eta)}\] is an isomorphism in degrees  $>n-m+1$.  Tensoring with $\rho$ and taking $S_n$-invariants, the map is still an isomorphism in degrees $>n-m+1$.

By \cref{irrep-perverse}, $\left( R \pi_! \mathbb Q_\ell \otimes \rho\right)^{S_n}[n]$ is perverse, so at the generic point $i(\eta)$, the shifted perverse sheaf $\left( R \pi_! \mathbb Q_\ell \otimes \rho\right)^{S_n}$ is concentrated in degree $n-m$. Thus $\left( R \pi_! \mathbb Q_\ell \otimes \rho\right)^{S_n}_x$ is concentrated in degree $\leq n+1-m$.
\end{proof}

 The next two sections will be devoted to showing that the dimension of the stalk of $\left( R \pi_! \mathbb Q_\ell \otimes \rho\right)^{S_n}$ in degree $n-m$ is bounded by $C_1(\rho)$  and the dimension of its stalk in degree $n+1-m$ is bounded by $C_2(\rho)$. This will then quickly imply the main theorem.

%We have a commutative diagram
%
%\[ \begin{tikzcd} (R \pi_! \mathbb Q_\ell)_x  \arrow[r,"a"] \arrow[d,"b"] & (R \pi_! \mathbb Q_\ell)_{i(\eta) }\arrow[d,"d"] \\
%N_x \arrow[r,"c" ]  & N_{i(\eta) } \end{tikzcd} \]
%
%We have just seen that $a$ is an isomorphism in degree $>n-m+1$ and a surjection in degree $n-m+1$. Because $K$ is perverse, it is lisse in degree $-m$ over the generic point $i(\eta)$, so $d$ is an isomorphism in degree $>n-m$ and because $N$ is lisse, $c$ is an isomorphism in all degrees. So $b$ is an isomorphism in degree $n-m+1$ and a surjection in degree $n-m+1$.
%
% Hence $K_x[n]$, which is the mapping cone of $b$, vanishes in degrees $> n-m+1$. So $K_x$ vanishes in degrees $>1-m$. But because $K_x$ is perverse on a variety of dimension $-m$, it vanishes in degrees $<-m$ as well. \end{proof} 
 
  \section{Euler characteristic and generic rank}\label{sec-Euler}

For $\orho$ an irreducible representation of $S_n$, let $\chi_{\orho}$ be the Euler characteristic of the stalk of $\left( R \pi_! \mathbb Q_\ell \otimes \orho\right)^{S_n}$ at the generic point. In other words, for $\eta$ a geometric generic point of $T$, let \[ \chi_{\orho }= \sum_j (-1)^j \dim  ( R^j \pi_! \mathbb Q_\ell \otimes \orho)^{S_n}_{\eta} .\]

The main goal of this section is to prove, in Lemma \ref{rho-Euler-formula}, the formula
\begin{equation}\label{chi-orho-formula} \chi_{\orho} =\sum_{d=0}^n (-1)^{d+n+m} \frac{\partial^m}{\partial \lambda_1 \dots \partial \lambda_m } \tr \Bigl ( \Diag(\lambda_1,\dots, \lambda_m),V_{\orho}^d \Bigr)\Big|_{\lambda_1,\dots,\lambda_m=1}. \end{equation} 

Let us first see how \eqref{chi-orho-formula} will help us. When $\orho$ is an irreducible representation corresponding to a partition into parts of size $<m$, we know from \cref{irrep-perverse} that $\left( R \pi_! \mathbb Q_\ell \otimes \rho\right)^{S_n}$ is concentrated in degree $n-m$ at the generic point, so \eqref{chi-orho-formula} will enable us to compute the rank of  $\left( R \pi_! \mathbb Q_\ell \otimes \rho\right)^{S_n}$ in this degree. Using \cref{perverse-easy-inequality}, we will bound the rank of $\left( R^{n+m} \pi_! \mathbb Q_\ell \otimes \rho\right)^{S_n}$ at non-generic points as well. This will lead to the $C_1(\rho)$ component of the error term in \cref{squarefree-bound-precise}.

To prove \eqref{chi-orho-formula}, we will use \cref{chi-easy-case,rho-derivative-formula}, where we will show that \eqref{chi-orho-formula} holds for $\orho$ corresponding to a partition with one part of size $\geq m$. In \cref{Lefschetz-vanishing,Lefschetz-linear-equation}, we will apply the Lefschetz fixed point formula to a map with no fixed points to derive equations \eqref{sigma-rho-equation} satisfied by the quantities $\chi_{\orho}$. We will show in \cref{rho-equation-uniqueness} that these equations, together with \eqref{chi-orho-formula} for $\orho$ corresponding to a partition with one part of size $\geq m$, have a unique solution. This enables us to prove \eqref{chi-orho-formula} by checking that the right side of \eqref{chi-orho-formula} also satisfies these equations, which we will do in Lemma \ref{rho-Euler-formula}.

\begin{lemma}\label{Lefschetz-vanishing} Let $\sigma \in S_n$ have at most $m-1$ cycles and let $\eta \in T$ be the geometric generic point. Then \begin{equation}\label{lefschetz-permutation} \sum_j (-1)^j \tr(\sigma, (R^j \pi_!  \mathbb Q_\ell)_{\eta} ) =0 .\end{equation} \end{lemma}

\begin{proof}  The group $S_n$ acts on the space $\overline{X}$, its open subset $X$, and its closed complement $Z \times T$, by permuting $(\alpha_1,\dots,\alpha_n)$.  Thus $\sigma$ acts naturally on the sheaf $u_!\mathbb Q_\ell$. Using $\pi = \overline{\pi} \circ u$ and proper base change, we have \begin{equation} \label{lefschetz-perm2} \sum_j (-1)^j \tr(\sigma, (R^j \pi_!  \mathbb Q_\ell)_{\eta} ) = \sum_j (-1)^j \tr(\sigma, (R^j \overline{\pi}^* u_!  \mathbb Q_\ell)_{\eta} )= \sum_j (-1)^j \tr( \sigma, H^j ( \overline{\pi}^{-1}(\eta), u_! \mathbb Q_\ell ) ) . \end{equation}

The Lefschetz fixed point formula \cite[III, Corollary 4.8]{sga5} expresses $\sum_j (-1)^j \tr( \sigma, H^j ( \overline{\pi}^{-1}(\eta), u_! \mathbb Q_\ell ) ) $ as an integral over the fixed points of $\sigma$. We will check that there are no fixed points of $\sigma$ in $\overline{\pi}^{-1} (\eta)$. This implies \[ \sum_j (-1)^j \tr( \sigma, H^j ( \overline{\pi}^{-1}(\eta), u_! \mathbb Q_\ell ) ) =0 ,\] which, combined with \eqref{lefschetz-perm2}, implies \eqref{lefschetz-permutation}.

First let us check there are no fixed points of $\sigma$ in $Z$, and thus none in $Z \times T$. Any fixed point of $\sigma$ in $(\mathbb P^1)^n$ is an $n$-tuple of points of $\mathbb P^1$ which takes the same value on different indices $i$ in the same cycle of $\sigma$, and thus takes at most $m-1$ distinct values. On the other hand, every tuple in $Z$ by definition takes all the values of $S$, and thus takes at least $m+1$ different values.

Next let us check there are no fixed points of $\sigma$ in the complement of $Z \times T$, which is the generic fiber of $\pi$ in $X = C^n$. Again, any fixed point of $\sigma$ is an $n$-tuple which takes at most $m-1$ distinct values. Hence the space of fixed points in $C^n$ has dimension at most $m-1$. Thus the image of the space of fixed points in $T$ has dimension at most $m-1$ and hence cannot include the generic point. \end{proof}

\begin{lemma}\label{Lefschetz-linear-equation} Let $\sigma \in S_n$ have at most $m-1$ cycles. Then
\begin{equation}\label{sigma-rho-equation} \sum_{ \orho\textrm{ irrep of }S_n} \chi_{\orho} \tr(\sigma, \orho)  =0 .\end{equation}\end{lemma}

\begin{proof} This follows from \cref{Lefschetz-vanishing} and the observation that, as a representation of $S_n,$
\[ (R^j \pi_!  \mathbb Q_\ell)_{\eta}  = \bigoplus_{\orho \textrm{ irrep of }S_n} \orho^{  \dim  ( R^j \pi_! \mathbb Q_\ell \otimes \orho)^{S_n}_{\eta} } \]
so
\[ \tr(\sigma, (R^j \pi_!  \mathbb Q_\ell)_{\eta} ) = \sum_{ \orho \textrm{ irrep of }S_n} \dim  ( R^j \pi_! \mathbb Q_\ell \otimes \orho)^{S_n}_{\eta}   \tr(\sigma, \orho)\]
and thus
\[ \sum_j (-1)^j \tr(\sigma, (R^j \pi_!  \mathbb Q_\ell)_{\eta} ) = \sum_{ \orho \textrm{ irrep of }S_n} \chi_{\orho}   \tr(\sigma, \orho). \qedhere \]
\end{proof}

\begin{lemma}\label{chi-easy-case} Let $\overline{\rho}$ be an irreducible representation of $S_n$ corresponding to a partition with at least one part of size $\geq m$. Then \begin{equation}\label{direct-rho-equation} \chi_{\orho}=  \sum_{d=m}^{n} (-1)^{d+n}  \binom{d}{m} \dim V_{\orho}^d.\end{equation} \end{lemma}

\begin{proof} %By \cref{irrep-constant} we know that $(R \pi_! \mathbb Q_\ell \otimes \orho)^{S_n}$ is the direct sum of its cohomology sheaves, and these sheaves are constant. Together with \cref{rho-cohomology} to relate $V^d_{\orho}$ to the cohomology of $(R \pi_! \mathbb Q_\ell \otimes \orho)^{S_n}$ and the standard computation of the cohomology of the constant sheaf on $T_{\overline{k}} \cong \mathbb G_m^m$, this implies 
We have \begin{align*} &\dim V^d_{\orho} =  \dim(( H^{n+d}_c ( T_{\overline{\kappa}} , R \pi_! \mathbb Q_\ell \otimes \orho))^{S_n})= \sum_j \dim ((H^{n+d-j}_c ( T_{\overline{\kappa}}, R^j  \pi_! \mathbb Q_\ell \otimes \orho))^{S_n} )\\ & = \sum_j \binom{m}{n+d-j-m}   \dim  ( ( R^j \pi_! \mathbb Q_\ell \otimes \orho)^{S_n}_{\eta})\end{align*} where the first equality is \cref{rho-cohomology}, the second uses, from \cref{irrep-constant}, the fact that $(R \pi_! \mathbb Q_\ell \otimes \orho)^{S_n}$ is the direct sum of its cohomology sheaves, and the final equality uses, from \cref{irrep-constant}, the fact that those cohomology sheaves are constant, together with the standard computation of the cohomology of the constant sheaf on $T_{\overline{k}} \cong \mathbb G_m^m$. Thus
\[ \sum_{d=m}^{n} (-1)^{d+n}  \binom{d}{m} \dim V_{\orho}^d =  \sum_{d=m}^{n} (-1)^{d+n}  \binom{d}{m}  \sum_j \binom{m}{n+d-j-m}   \dim  (( R^j \pi_! \mathbb Q_\ell \otimes \orho)^{S_n}_{\eta})\]
which implies \eqref{direct-rho-equation} when combined with the combinatorial identity
\begin{equation}\label{easy-case-combinatorics}  \sum_{d=m}^n (-1)^{d+n}  \binom{d}{m}  \binom{m}{n+d-j-m}  =   (-1)^j \end{equation}
for $(n-m) \leq j \leq 2(n-m)$ and the fact that $R^j \pi_! \mathbb Q_\ell$ vanishes unless $(n-m) \leq j \leq 2(n-m)$.  So it remains to check \eqref{easy-case-combinatorics}.

To do this, we observe that the coefficient of $u^{d-m}$ in $\frac{1}{(1-u)^{m+1}}$ is $\binom{d}{m}$ and the coefficient of $u^{2m -n-d+j}$  in $(1-u)^m$ is $(-1)^{d+n+j} \binom{m }{ n+d-j-m}$. Multiplying, the coefficient of $u^{ m-n+j}$ in $\frac{(1-u)^m}{(1-u)^{m+1}}$ is  \[ (-1)^j  \sum_{d=-\infty}^\infty (-1)^{d+n}  \binom{d}{m}  \binom{m}{n+d-j-m}  .\]   Because $j\geq n-m$, the coefficient of $u^{ m-n+j}$ in $\frac{(1-u)^m}{(1-u)^{m+1}} = \frac{1}{1-u}$ is $1$.  Furthermore, we have $\binom{d}{m}=0$ if $m<d$ and we have $ \binom{m}{n+d-j-m} =0$ if $d>n$ because then $n+d-j-m > 2n-j-m > 2n - 2(n-m)-m=m$, so 
\[ 1=  (-1)^j  \sum_{d=-\infty}^\infty (-1)^{d+n}  \binom{d}{m}  \binom{m}{n+d-j-m} =(-1)^j  \sum_{d=m}^n (-1)^{d+n}  \binom{d}{m}  \binom{m}{n+d-j-m} ,\]
giving \eqref{easy-case-combinatorics}. \end{proof}

\begin{lemma}\label{rho-equation-uniqueness} The system of equations given by \eqref{sigma-rho-equation} for each $\sigma \in S_n$ with at most $m-1$ cycles and \eqref{direct-rho-equation} for each $\orho$ corresponding to a partition with at least one part of size $\geq m$ has a unique solution $(\chi_{\orho})_{\orho \textrm{ irrep of }S_n}$. \end{lemma}

\begin{proof} Suppose there are two distinct solutions $\chi_{\orho}$ and $\chi_{\orho}'$. Then $\sum_{\orho} (\chi_{\orho}-\chi_{\orho}') \tr(\sigma, \orho)$ is a nontrivial linear combination of the characters of $S_n$ corresponding to partitions into parts of size $<m$, which vanishes on all permutations with $<m$ cycles, contradicting \cref{permutation-restriction-fact}, taking $\mathcal S$ to be the set of partitions into $<m$ parts so the conjugate partitions of elements of $\mathcal S$ are exactly the partitions into parts of size $<m$. \end{proof}

\begin{lemma}\label{underivative-divisible} Let $\overline{\rho}$ be an irreducible representation of $S_n$ corresponding to a partition with at least one part of size $\geq m$. The polynomial
\begin{equation} \label{underivative-expression}  \sum_{d=0}^n (-1)^{d} \tr \Bigl ( \Diag(\lambda_1,\dots, \lambda_m),  V_{\orho}^d \Bigr) ,\end{equation} is divisible by $\prod_{i=1}^m ( \lambda_i - 1)$ in $\mathbb Q[\lambda_1,\dots, \lambda_m]$. \end{lemma}

\begin{proof} It suffices to show that \eqref{underivative-expression} vanishes when $\lambda_i=1$ for some $i$, and by symmetry, it suffices to show this when $\lambda_m=1$. Then we calculate, by Frobenius reciprocity, 
\[ V_{\orho}^d =  \Bigl( (\mathbb C^{m-1}\oplus \mathbb C )^{\otimes (n-d) } \otimes \orho \otimes \sgn_{S_{n-d} } \Bigr)^{S_{n-d}\times S_d}  =  \oplus_{k=0}^{n-d}  \Bigl( (\mathbb C^{m-1})^{n-d-k} \otimes \orho \otimes \sgn_{S_{n-d-k} }  \otimes \sgn_k \Bigr)^{S_{n-d-k} \times S_k \times S_d}\] \[=  \oplus_{k=0}^{n-d}  \Bigl( (\mathbb C^{m-1})^{n-d-k} \otimes \orho \otimes  \sgn_{S_{n-d-k} }  \otimes \Ind_{S_k \times S_d}^{S_{k+d}}\sgn_k \Bigr)^{S_{n-d-k} \times S_{k +d} }\]
so
\begin{equation}\label{kd-sum} \begin{split} &\sum_{d=0}^n (-1)^{d} \tr \Bigl ( \Diag(\lambda_1,\dots, \lambda_{m-1},1), V_{\orho}^d \Bigr) \\
= &\sum_{d=0}^n \sum_{k=0}^{n-d} (-1)^{d} \tr \Bigl ( \Diag(\lambda_1,\dots, \lambda_{m-1}),\Bigl( (\mathbb C^{m-1})^{n-d-k} \otimes \orho \otimes  \sgn_{S_{n-d-k} }  \otimes \Ind_{S_k \times S_d}^{S_{k+d}}\sgn_k \Bigr)^{S_{n-d-k} \times S_{k +d} } \Bigr ) .\end{split}\end{equation}

Next we claim that, for each $c>0$, the terms corresponding to pairs $k,d$ with $k+d=c$ cancel in \eqref{kd-sum}. The function that takes a representation $\tilde{\rho}$ of $S_{c}$ to \[ \tr \Bigl ( \Diag(\lambda_1,\dots, \lambda_{m-1}),\Bigl( (\mathbb C^{m-1})^{n-d-k} \otimes \orho \otimes  \sgn_{S_{n-d-k} }  \otimes \tilde{\rho} \Bigr)^{S_{n-d-k} \times S_{k +d} } \Bigr )\] is additive on representations of $S_{c}$ and hence extends to a linear function on the representation ring of $S_{c}$. To check that the terms corresponding to $k+d=c$ in \eqref{kd-sum} cancel for every $c>0$, by this linearity it suffices to check that
\[ \sum_{\substack{ k, d \in \mathbb N \\ k+d=c}} (-1)^d  [ \Ind_{S_k \times S_d}^{S_{c}}\sgn_k] = 0 \] in the representation ring of $S_{c}$ for $c>0$.  

To check this, note that for $\sigma\in S_n$ we have
\[ \tr \Bigl( \sigma, \sum_{\substack{ k, d \in \mathbb N \\ k+d=c}} (-1)^d  [ \Ind_{S_k \times S_d}^{S_{c}}\sgn_k] \Bigr) =  \sum_{\substack{ k, d \in \mathbb N \\ k+d=c}} (-1)^d  \tr(\sigma,  \Ind_{S_k \times S_d}^{S_{c}}\sgn_k) = \sum_{\substack{ k, d \in \mathbb N \\ k+d=c}} (-1)^d \sum_{ \substack{ S \subseteq \{1,\dots, c\}\\ \sigma(S) =S\\ \abs{S} = k }} \sgn ( \sigma \mid_S ) \]
where $\sgn (\sigma \mid_S) $ is $(-1)^k$ times $(-1)$ raised to the number of cycles of $\sigma$ contained in $S$, counting fixed points as cycles, giving
\[ \tr \Bigl( \sigma, \sum_{\substack{ k, d \in \mathbb N \\ k+d=c}} (-1)^d  [ \Ind_{S_k \times S_d}^{S_{c}}\sgn_k] \Bigr) = (-1)^{c} \sum_{ \substack{ S \subseteq \{1,\dots, c\}\\ \sigma(S) =S\\ \abs{S} = k }} (-1)^{ \abs{\{ \textrm{cycles of }\sigma\textrm{ contained in }S\}}} \] \[ = (-1)^c ( 1+ (-1))^{ \abs{ \{\textrm{cycles of }\sigma\}} }= (-1)^c 0^{ \abs{\{\textrm{cycles of }\sigma\}}}= 0 \] because the number of cycles of $\sigma$ is nonzero as long as $c>0$. Since the character vanishes, the virtual representation vanishes. (Alternatively, by the Littlewood-Richardson rule, the only irreducible representations occurring in these induced representations are those corresponding to the hook Young diagrams, and each occurs twice, with opposite sign.) 

Hence we can replace \eqref{kd-sum} with only the $k=d=0$ terms, getting 

\begin{align*} & \sum_{d=0}^n (-1)^{d} \tr \Bigl ( \Diag(\lambda_1,\dots, \lambda_{m-1},1), V_{\orho}^d \Bigr) \\
= & \tr \Bigl ( \Diag(\lambda_1,\dots, \lambda_{m-1}), \Bigl( (\mathbb C^{m-1} )^{\otimes n } \otimes \orho \otimes \sgn \Bigr)^{S_{n}} \Bigr) = \tr \Bigl ( \Diag(\lambda_1,\dots, \lambda_{m-1}), V_{\orho} \Bigr) = 0 \end{align*} because $\orho$ corresponds to a partition with at least one part of size $\geq m$ and thus $V_{\orho}=0$ by Lemma \ref{schur-vanishing-lemma}.
\end{proof}

\begin{lemma}\label{rho-derivative-formula} Let $\overline{\rho}$ be an irreducible representation of $S_n$ corresponding to a partition with at least one part of size $\geq m$. Then \begin{equation}\label{eq-rho-derivative} \begin{split} & \sum_{d=m}^{n} (-1)^{d+n}  \binom{d}{m} \dim V_{\orho}^d \\
= &\sum_{d=0}^n (-1)^{d+n+m} \frac{\partial^m}{\partial \lambda_1 \dots \partial \lambda_m } \tr \Bigl ( \Diag(\lambda_1,\dots, \lambda_m), V_{\orho}^d \Bigr)  \Big|_{\lambda_1,\dots,\lambda_m=1} .\end{split}\end{equation} \end{lemma}

\begin{proof} By \cref{underivative-divisible}, there exists a polynomial $F$ in variables $\lambda_1,\dots,\lambda_m$ such that \[ \sum_{d=0}^n (-1)^{d+n+m} \tr \Bigl ( \Diag(\lambda_1,\dots, \lambda_m), V_{\orho}^d \Bigr) =  F(\lambda_1,\dots, \lambda_m) \prod_{i=1}^m (\lambda_i-1) .\]  Then 

\begin{equation}\label{eq-derivative-evaluation} \begin{split} &  \sum_{d=0}^n (-1)^{d+n+m} \frac{\partial^m}{\partial \lambda_1 \dots \partial \lambda_m } \tr \Bigl ( \Diag(\lambda_1,\dots, \lambda_m), V_{\orho}^d  \Bigr)  \Big|_{\lambda_1,\dots,\lambda_m=1} \\ 
& = \frac{\partial^m}{\partial \lambda_1 \dots \partial \lambda_m } \Bigl( F(\lambda_1,\dots, \lambda_m) \prod_{i=1}^m (\lambda_i-1) \Bigr) \Big|_{\lambda_1,\dots,\lambda_m=1} \\ 
& = F(1,\dots, 1). \end{split}\end{equation}
On the other hand, for a formal variable $u$ \[ ( u-1)^m  F(u,\dots, u) = \sum_{d=0}^n (-1)^{d+n+m}  \tr \Bigl ( \Diag(u,\dots, u),V_{\orho}^d \Bigr) \] 
\[ = \sum_{d=0}^n (-1)^{d+n+m}   u^{n-d} \dim V_{\orho}^d \]
%so if we write $F(u,\dots, u) = \sum_{i=0}^{\infty} a_i u^i$ then \[  \dim V_{\orho}^d = (-1)^{d+n+m} \sum_i (-1)^{n-d-i}  a_i  \binom{m}{m- n+ d+ i}  \] and thus \[ \sum_{d=m}^{n} (-1)^{d+n}  \binom{d}{m} \dim V_{\orho}^d = \sum_{d=m}^{n} (-1)^{n-m-d-i}  \binom{d}{m}   \sum_i (-1)^{n-d-i}  \binom{m}{m-n+ d + i} a_i \]
 so plugging in $u = \frac{1}{1+v}$ and multiplying both sides by $(-1)^m (1+v)^n$, we obtain
\[ (1+v)^n \frac{ v^m}{(1+v)^m} F\left(\frac{1}{1+v},\dots, \frac{1}{1+v}\right ) =  \sum_{d=0}^n (-1)^{d+n}   (1+v)^d \dim V_{\orho}^d \] 

The constant term of $(1+v)^n \frac{ 1}{(1+v)^m} F\left(\frac{1}{1+v},\dots, \frac{1}{1+v}\right )$ as a formal power series in $v$ is $F(1,\dots, 1)$, so $F(1,\dots, 1)$ is the coefficient of $v^m$ in $\sum_{d=0}^n (-1)^{d+n}   (1+v)^d \dim V_{\orho}^d$, which is \[ \sum_{d=m}^{n} (-1)^{d+n}  \binom{d}{m} \dim V_{\orho}^d. \] Combined with \eqref{eq-derivative-evaluation}, this proves \eqref{eq-rho-derivative}. \end{proof}

\begin{lemma}\label{rho-Euler-formula} For any irreducible representation $\orho$, we have
\begin{equation}\label{eq-ref} \chi_{\orho} =\sum_{d=0}^n (-1)^{d+n+m} \frac{\partial^m}{\partial \lambda_1 \dots \partial \lambda_m } \tr \Bigl ( \Diag(\lambda_1,\dots, \lambda_m),V_{\orho}^d \Bigr) \Big|_{\lambda_1,\dots,\lambda_m=1}. \end{equation} \end{lemma}

\begin{proof} By \cref{rho-equation-uniqueness}, it suffices to check that the right-hand-side of \eqref{eq-ref} satisfies \eqref{direct-rho-equation} and \eqref{sigma-rho-equation}. That \eqref{direct-rho-equation} is satisfied is the content of \cref{rho-derivative-formula}. To check \eqref{sigma-rho-equation}, given $\sigma$ a permutation with fewer than $m$ cycles, we must show the vanishing of
\[ \sum_{\orho}\sum_{d=0}^n (-1)^{d+n+m} \frac{\partial^m}{\partial \lambda_1 \dots \partial \lambda_m } \tr \Bigl ( \Diag(\lambda_1,\dots, \lambda_m), \Bigl( (\mathbb C^{m} )^{\otimes (n-d) } \otimes \orho \otimes \sgn_{S_{n-d} } \Bigr)^{S_{n-d}\times S_d} \Bigr)  \tr(\sigma, \orho) \Big|_{\lambda_1,\dots,\lambda_m=1} \]
%\[\hspace{-.85in} = \sum_{\orho}\sum_{d=0}^n (-1)^{d+n+m} \frac{\partial^m}{\partial \lambda_1 \dots \partial \lambda_m } \tr \Bigl ( \Diag(\lambda_1,\dots, \lambda_m),  \Bigl(  \Ind_{S_{n-d} \times S_d}^{S_n} \Bigl ((\mathbb C^{m} )^{\otimes (n-d) }  \otimes \sgn_{S_{n-d} } \Bigr)  \otimes \orho\Bigr)^{S_{n}} \Bigr)  \tr(\sigma, \orho) \Big|_{\lambda_1,\dots,\lambda_m=1}\]
Using the identity for any representation $W$ of $S_n$, $\sum_{\orho} \tr(\sigma, \orho) \dim (W \otimes \orho)^{S_n} = \tr(\sigma,W)$, and Frobenius reciprocity, it suffices to prove the vanishing of 
\begin{equation}\label{eq-sum-trace} \sum_{d=0}^n (-1)^{d+n+m} \frac{\partial^m}{\partial \lambda_1 \dots \partial \lambda_m } \tr \Bigl ( \Diag(\lambda_1,\dots, \lambda_m) \cdot \sigma ,   \Ind_{S_{n-d} \times S_d}^{S_n} \Bigl ((\mathbb C^{m} )^{\otimes (n-d) }  \otimes \sgn_{S_{n-d} } \Bigr)   \Bigr) \end{equation}
To do this, note that the Frobenius formula for the character of an induced representation expresses
\begin{equation}\label{eq-lambda-trace} \tr \Bigl ( \Diag(\lambda_1,\dots, \lambda_m) \cdot \sigma ,   \Ind_{S_{n-d} \times S_d}^{S_n} \Bigl ((\mathbb C^{m} )^{\otimes (n-d) }  \otimes \sgn_{S_{n-d} } \Bigr)   \Bigr)\end{equation}
as a sum of terms of the form
\begin{equation}\label{tr-sigma'} \tr \Bigl ( \Diag(\lambda_1,\dots, \lambda_m) \cdot \sigma' , (\mathbb C^{m} )^{\otimes (n-d) }  \otimes \sgn_{S_{n-d} }    \Bigr)\end{equation} where $\sigma'$ is a permutation in $S_{n-d} \times S_d$ conjugate to $\sigma$.  We can choose a basis for $(\mathbb C^{m} )^{\otimes (n-d)} $ consisting of tensor products of $(n-d)$-tuples of standard basis vectors of $\mathbb C^m$. We can express the action of $\Diag(\lambda_1,\dots, \lambda_m) \cdot \sigma' $ on this basis as a matrix, and then \eqref{tr-sigma'} is the sum of the diagonal entries of this matrix, twisted by $\sign_{S_{n-d}}(\sigma')$. The only nonzero diagonal entries occur for tuples of basis vectors stable under $\sigma'$.  Because $\sigma$ has fewer than $m$ cycles, $\sigma' $ has fewer than $m$ cycles when restricted to $S_{n-d}$. Thus fewer than $m$ different basis vectors can occur in a nonzero diagonal entry. Hence the eigenvalue of $\Diag(\lambda_1,\dots, \lambda_m) $ on any nonzero diagonal entry is a monomial that uses fewer than $m$ of the variables $\lambda_1,\dots, \lambda_m$. Thus the trace \eqref{eq-lambda-trace} is a sum of monomials in fewer than $m$ of the variables. Therefore the partial derivative of \eqref{eq-lambda-trace} in the variables $\lambda_1,\dots,\lambda_m$ vanishes, so the sum of partial derivatives \eqref{eq-sum-trace} vanishes, as desired. \end{proof} 

\begin{lemma}\label{rho-Euler-difference} For $\orho$ an irreducible representation of $S_{n}$,
\begin{equation}\begin{split}\label{c1} &\sum_{d=0}^n (-1)^{d+n+m} \frac{\partial^m \tr \bigl ( \Diag(\lambda_1,\dots, \lambda_m), V^d_{\orho}  \bigr)}{\partial \lambda_1 \dots \partial \lambda_m } \Big|_{\lambda_1,\dots,\lambda_m=1}  - \sum_{d=m}^{n} (-1)^{d+n}  \binom{d}{m} \dim V^d_{\orho} \end{split}\end{equation}  is $\chi_{\orho}$ if $\orho$ corresponds to a partition with all parts of size $<m$ and $0$ otherwise. \end{lemma}

\begin{proof} If $\orho$ corresponds to a partition with at least one part of size $\geq m$, \eqref{c1} vanishes by \cref{rho-derivative-formula}.

If $\orho$ corresponds to a partition with all parts of size $<m$, we observe that $V^d_{\orho}=0$ for $d\geq m$ by \cref{other-vanishing-lemma} and thus 
\[ \sum_{d=m}^{n} (-1)^{d+n}  \binom{d}{m} \dim V^d_{\orho}= 0 \] 
and so  \eqref{c1} is $\chi_{\orho}$ by \cref{rho-Euler-formula}. \end{proof}

\begin{lemma}\label{generic-rank-formula} For $\orho$ is an irreducible representation of $S_n$ corresponding to a partition with all parts of size $<m$, the dimension of the stalk of $\left( R^{n-m} \pi_! \mathbb Q_\ell \otimes \orho\right)^{S_n}$ at the generic point is $C_1(\orho)$.\end{lemma}

\begin{proof} This follows from \cref{rho-Euler-difference} and the definition of $C_1(\orho)$ once we observe, from \cref{irrep-perverse}, that $\left( R \pi_! \mathbb Q_\ell \otimes \orho\right)^{S_n}_{\eta}$  is concentrated in degree $n-m$ and so its rank in degree $n-m$ is $(-1)^{n-m}$ times its Euler characteristic. \end{proof}

\begin{proposition}\label{Euler-main-bound} Suppose that $\orho$ is an irreducible representation corresponding to a partition with all parts of size $<m$. Then the dimension of the stalk of  $\left( R^{n-m} \pi_! \mathbb Q_\ell \otimes \orho\right)^{S_n}$ at any point of $T$ is $\leq C_1(\orho)$. \end{proposition}

\begin{proof} By \cref{irrep-perverse}, $ \left( R\pi_! \mathbb Q_\ell \otimes \orho\right)^{S_n}[n] $ is perverse on a normal variety of dimension $m$. Thus by \cref{perverse-easy-inequality}, its rank in degree $-m$ at any point is at most its rank at the generic point. The result then follows from Lemma \ref{generic-rank-formula}. \end{proof}

The following bound will be useful when we need a less precise estimate:

\begin{lemma}\label{weak-C1-bound} We have
\[ C_1(\rho) \leq \sum_{d=0}^n  \frac{\partial^m}{\partial \lambda_1 \dots \partial \lambda_m } \tr \Bigl ( \Diag(\lambda_1,\dots, \lambda_m), V_{\orho}^d \Bigr) \Big|_{\lambda_1,\dots,\lambda_m=1}  .\] \end{lemma}

\begin{proof}  First note that $\tr \Bigl ( \Diag(\lambda_1,\dots, \lambda_m), V_{\orho}^d \Bigr) $ is a polynomial in $\lambda_1,\dots, \lambda_m$ with nonnegative coefficients so its derivative in $\lambda_1,\dots, \lambda_m$, evaluated at $1$, is nonnegative. 

Both sides are additive in $\rho$, so it suffices to handle the case when $\rho$ is irreducible. If $\rho$ corresponds to a partition with all parts of size $<m$, then by \cref{other-vanishing-lemma},
\[ C_1(\rho) = \sum_{d=0}^n (-1)^d   \frac{\partial^m \tr \bigl ( \Diag(\lambda_1,\dots, \lambda_m), V_{\orho}^d \bigr)}{\partial \lambda_1 \dots \partial \lambda_m }  \Big|_{\lambda_1,\dots,\lambda_m=1}   \leq  \sum_{d=0}^n  \frac{\partial^m\tr \bigl ( \Diag(\lambda_1,\dots, \lambda_m), V_{\orho}^d \bigr) }{\partial \lambda_1 \dots \partial \lambda_m } \Big|_{\lambda_1,\dots,\lambda_m=1}  \]
For any other $\rho$, by \cref{rho-derivative-formula},
\[ C_1(\rho) = 0 \leq  \sum_{d=0}^n  \frac{\partial^m}{\partial \lambda_1 \dots \partial \lambda_m } \tr \Bigl ( \Diag(\lambda_1,\dots, \lambda_m), V_{\orho}^d \Bigr)\Big|_{\lambda_1,\dots,\lambda_m=1}  .\qedhere \]  \end{proof}

 \section{Characteristic cycle and Massey bound}\label{sec-Massey}

 To apply the characteristic cycle theory to $u_! \mathbb Q_\ell$, we must first embed $\overline{X}$ into a smooth variety, which we do using the closed immersion \[s\colon \overline{X} \to (\mathbb P^1)^n \times T \] mentioned in \cref{sec-structure}. Furthermore, let \[ (sym \times id)\colon (\mathbb P^1)^n \times T \to  (\mathbb P^1)^{(n)} \times T \] be the quotient by $S_n$, and \[ pr_2\colon   (\mathbb P^1)^{(n)} \times T \to T\] the projection. Then by definition \[ \overline{\pi} = pr_2 \circ (sym \times id) \circ s.\]
 
 The utility of defining the map $sym$ is that it is $S_n$-invariant, so $ (sym\times id)_* s_* u_! \mathbb Q_\ell $ admits a natural $S_n$ action, which we use to define  $  ((sym_n \times id)_* s_* u_! \mathbb Q_\ell \otimes \rho )^{S_n}$. We will calculate the characteristic cycle of $ ( (sym_n \times id)_* s_* u_! \mathbb Q_\ell \otimes \rho )^{S_n}$ and then project under $pr_2$ to find the characteristic cycle of $(R \pi_! \mathbb Q_\ell \otimes \rho)^{S_n}$.  It will be difficult to calculate the multiplicity of the zero section in the characteristic cycle of $(R \pi_! \mathbb Q_\ell \otimes \rho)^{S_n}$ by this method, but this multiplicity, which matches the Euler characteristic of the generic fiber, was already computed in \cref{generic-rank-formula}.

 \begin{lemma}\label{singular-support-symmetric}  The fiber of the singular support of $( (sym \times id)_* s_* u_! \mathbb Q_\ell \otimes \rho)^{S_n} $ over a point $x$ in $(Z/S_n) \times T$, viewed as a closed subset of the cotangent space of $(\mathbb P^1)^n \times T$ at $x$, intersects the cotangent space at $T$ only in zero. \end{lemma}
 
 \begin{proof} Because the singular support can only decrease on passing to a direct summand, it suffices to prove the same property for the singular support of $ (sym \times id)_* s_* u_! \mathbb Q_\ell $. 
 
 By \citep[Lemma 2.2(ii)]{Beilinson}, \[ SS ( (sym \times id)_* s_* u_! \mathbb Q_\ell ) \subseteq (sym \times id)_\circ SS ( s_* u_! \mathbb Q_\ell).\]   Here $(sym \times id)_{\circ}$ is the image under $(sym \times id)$ of the  inverse image under $d (sym \times id)$. So the fiber of $SS ( (sym \times id)_* s_* u_! \mathbb Q_\ell ) $ over a point $x$ is contained in the union over $y \in (sym \times id)^{-1}(x)$ of the inverse image under $d(sym \times id)$ of the fiber of $SS ( s_* u_! \mathbb Q_\ell)$ at $y$.
 
 For $x \in (Z/S_n) \times T$, we must have $y \in Z \times T$. If a vector $\omega$ in the fiber of $SS ( (sym \times id)_* s_* u_! \mathbb Q_\ell )$ over $x$ lies in the cotangent space of $T$ and is nonzero, then its image under $d(sym \times id)$ at each point $y$ lies in the cotangent space of $T$ and is nonzero, since $sym \times id$ is compatible with the projection to $T$, and so it cannot lie in the fiber of $SS ( s_* u_! \mathbb Q_\ell)$ over $y$ by \cref{singular-support-embedded}. \end{proof}

 We can also factor $(sym \times id) \circ s \circ u = \tilde{s} \circ sym$ for $sym\colon C^n \to C^{(n)}$ the quotient map and $\tilde{s}\colon  C^{(n)} \to (\mathbb P^1)^{(n)} \times T$ the quotient of the graph of $\pi$ by $S_n$. We can express the maps relevant here in the following commutative diagram:
 
 \[ \begin{tikzcd}
     X\arrow[d,"u"]&  C^n \arrow[l,equal] \arrow[r,"sym"] &   C^{ (n)} \arrow[d , "\tilde{s}"]  \\
\overline{X}\arrow[r,"s"]\arrow[drr,"\overline{\pi}"] &  (\mathbb P^1)^n \times T \arrow[r, "sym \times id"]  &    (\mathbb P^1)^{(n)} \times T \arrow[d, "pr_2"]  \\
& & T\end{tikzcd} \] 
  Note that the maps $sym, sym\times id,$ and $\overline{\pi}$ are each invariant under the natural action of $S_n$ on, respectively, $C^n, (\mathbb P^1)^n\times T$, and $\overline{X}$.
 
  Our next step is to calculate the characteristic cycle of $(sym_* \mathbb Q_\ell \otimes \rho)^{S_n}$.
 
 To do this, note that we can view points of $C^{(n)}$ as monic polynomials $\prod_{i=1}^n (T-\alpha_i)$ of degree $n$, prime to $g$.  Let $P_n$ be the vector space of polynomials of degree $<n$ and $P_n^\vee$ its dual space. Viewing points of $C^{(n)}$ as monic polynomials of degree $n$, we can view the tangent space at any point as $P_n$, so the tangent bundle is $P_n \otimes \mathcal O_{ C^{(n)}}$ and the cotangent bundle is $P_n^\vee \otimes \mathcal O_{C^{(n)}}$. 
   
 \begin{defi} We says a \emph{weight function} is a function $w$ from the positive natural numbers to natural numbers such that \[\sum_{k=1}^{\infty} k w(k) =n.\]
 
 We write $\WF$ for the set of weight functions.
 
 Fix a weight function $w$. Let $D_w = \prod_k C^{ (w(k))}$, so that the points of $D_w$ parameterize tuples  $(f_k)_{k\in \mathbb N^+}$ where $f_k$ is a monic polynomial of degree $w(k)$, prime to $g$.
 
 Define a map $ev_w$ from $D_w$  to $C^{ (n)}$ that sends a tuple $(f_k)_{k\in \mathbb N^+}$ to $\prod_{k=1}^{\infty}  f_k^k$.
 
Define a closed subset $A_w$ of $C^{(n)}$ as the image of $ev_w$.  

Define a vector bundle $W_w$ on $D_w$ as the kernel of the dual map \[  \Bigl(\cdot  \prod_k f_k^{k-1} \Bigr) ^T \colon  P_n^\vee  \otimes \mathcal O_{D_w}  \to P_{ \sum_k  w(k)  }^\vee  \otimes \mathcal O_{D_w}  \]  to the map of vector bundles \[  \Bigl(\cdot  \prod_k f_k^{k-1} \Bigr) \colon P_{ \sum_k w(k) }  \otimes \mathcal O_{D_w}  \to P_n \otimes   \mathcal O_{D_w} .\]  In other words, the fiber of $W_w$ over a tuple $(f_k)_{k \in \mathbb N^+}$ is the vector space of linear forms on polynomials of degree $<n$ that vanish on all polynomial multiples of $\prod_{k=1}^{\infty} f_k^{k-1}$. 
 
Since $W_w$ is a subbundle of $ P_n^\vee  \otimes \mathcal O_{D_w}  = ev_w^* T^* C^{(n)}$,  we have a map $W_w \to   T^* C^{(n)}$. Define $B_w$ as the image of this map.  \end{defi} 

By construction, $B_w$ is an irreducible closed subset of $T^* C^{(n)}$. The vector bundle $W_w$ has fiber dimension $n - \sum_{k=1}^{\infty} w(k)$ and base dimension $\sum_{k=1}^{\infty} w(k)$, so the total space has dimension $n$. Because $ev_w$ is injective on the generic point, $B_w$ also has dimension $n$ and the projection $W_w \to B_w$ has degree $1$. % is an irreducible closed subset of $T^* C^{(n)}$ with multiplicity one. 

For $w$ a weight function, let $S_w = \prod_{k=1}^{\infty} S_k^{w(k)}$. Then $S_w$ is a subgroup of $S_n$, the stabilizer of a partition into $w(1)$ parts of size $1$, $w(2)$ parts of size $2$, and so on. Let  \[ M_{\rho}(w) = \dim \Bigl(  ( \rho \otimes \sgn) ^{ S_w}  \Bigr) .\] 

This section will be devoted to, first, giving in \cref{characteristic-cycle-torus} a formula for the characteristic cycle of $ (R \pi_! \mathbb Q_\ell[n] \otimes \rho)^{S_n}$ in terms of the $M_{\rho(w)}$ and some cycles $B'_w$, to be defined later, related to $B_w$, and second, using that formula to bound the dimensions of the stalks of $ (R^{n+1-m} \pi_! \mathbb Q_\ell \otimes \rho)^{S_n}$ by $C_2(\rho)$. This will involve proving a succession of formulas for the characteristic cycles of $(sym_* \mathbb Q_\ell [n] \otimes \rho)^{S_n}$, $\tilde{s}_! (sym_* \mathbb Q_\ell [n] \otimes \rho)^{S_n}$, and finally $( R \pi_! \mathbb Q_\ell [n] \otimes \rho)^{S_n}$, each formula proved using the previous one, as well as re-expressing $C_2(\rho)$ as a sum involving the coefficients $M_\rho(w)$ in \cref{cc-nice-formula}.

\begin{lemma}\label{cc-easy} The characteristic cycle of $(sym_* \mathbb Q_\ell [n] \otimes \rho)^{S_n} $ is \[ \sum_{ w\in \WF } M_{\rho}(w)  [ B_w] .\] \end{lemma}

\begin{proof} This is exactly \citep[Theorem 2.10]{SawinSupNorms}, specialized to the case $K = \mathbb Q_\ell$. \end{proof}

We now study the characteristic cycle on $(\mathbb P^1)^{(n)} \times T$. We can still view the tangent space of $(\mathbb P^1)^{(n)}$ at a point $(\alpha_1,\dots,\alpha_n)$ as the vector space of polynomials of degree $<n$, at least where none of the $\alpha_i$ are $\infty$, and we can view the tangent space of $T$ at a point $a$ as the vector space of polynomials mod $g$.

\begin{defi} Fix a weight function $w$. Let $\tilde{A}_w$ be the closure in $(\mathbb P^1)^{(n)} \times T$  of the image of $\tilde{s} \circ ev_w$.  We will define a closed conical cycle $\tilde{B}_w \in T^* ( (\mathbb P^1)^{(n)} \times T)$ lying over $\tilde{A}_w$

Let $\tilde{W}_w$ be the inverse image of $W_w$ under \[ ev_w^* d \tilde{s} \colon ev_w^* \tilde{s}^* T^* ( (\mathbb P^1)^{(n)}  \times T) \to ev_w^* T^* C^{(n)} ,\] so that $\tilde{W}_w$ fits into the commutative diagram where every square is Cartesian
\[ \begin{tikzcd}
\tilde{W}_w  \arrow[r]  \arrow[d] &  ev_w^*  \tilde{s}^* T^* \left( (\mathbb P^1)^{(n)} \times T\right)  \arrow[r] \arrow[d, "ev_w^* d \tilde{s}"] &   \tilde{s}^* T^* \left( (\mathbb P^1)^{(n)} \times T\right)  \arrow[r] \arrow[d,"d\tilde{s}",swap]   \arrow[dd, bend left] &  T^* \left( (\mathbb P^1)^{(n)} \times T\right) \arrow[dd] \\
W_w \arrow[r] & ev_w^* T^*  C^{(n) }  \arrow[r] \arrow[d]& T^* C^{(n)} \arrow[d] \\
& D_w \arrow[r,"ev_w"]  & C^{(n)} \arrow[r, "\tilde{s}"] & (\mathbb P^1)^{(n)} \times T \end{tikzcd} \] 

Let $\tilde{B}_w$ be the closure of the image in $T^* \left( (\mathbb P^1)^{(n)} \times T\right)$ of $\tilde{W}_w.$ \end{defi}

Concretely, we can as before represent $ev_w^* T^* C^{(n)} $ as $P_n^\vee \otimes \mathcal O_{D_w}$, we can represent $ ev_w^* \tilde{s}^* T^* ( (\mathbb P^1)^{(n)} \times T) $ as $(P_n \oplus P_m )^\vee \otimes \mathcal O_{D_w}$, and we can represent $ev_w^* d \tilde{s}$ as the transpose \[ (id + \textrm{mod }g)^T \colon (P_n \oplus P_m )^\vee \otimes \mathcal O_{D_w}  \to P_n^\vee \otimes \mathcal O_{D_w} \] of the map
 \[ (id + \textrm{mod }g) \colon P_n \otimes \mathcal O_{D_w} \to (P_n \oplus P_m ) \otimes \mathcal O_{D_w} \] that sends a polynomial $h$ to $(h, h \mod g)$. This follows from expressing $\tilde{s}$ as the map sending $f$ to $(f, f \mod g)$ and differentiating. 
 
 Hence $\tilde{W}_w$ is the inverse image under $(id + \textrm{mod }g)^T$ of the kernel of $( \cdot \prod_{k=1}^{\infty} f_k^{k-1})^T$ and thus is the kernel of the composition $( \cdot \prod_{k=1}^{\infty} f_k^{k-1})^T \circ (id + \textrm{mod }g)^T$.

%
%Define $\tilde{B}_w$ as the closure in $T^* \left( (\mathbb P^1)^{(n)} \times T\right)$ of the image under $\tilde{s}$ of the inverse image under $d \tilde{s}$ of $B_w$. 
%
%More concretely, let $\tilde{A}_w$ be the closure in $(\mathbb P^1)^{(n)} \times T$ of the set consisting of pairs $( \prod_{k=1}^{\infty} f_k^k,  \prod_{k=1}^{\infty} f_k^k \mod g)$ where $f_k$ is a polynomial of degree $w(k)$ prime to $g$. 
%
%This is the image of the map $D_w$ sending $(f_k)_{k \in \mathbb N^+}$ to $( \prod_{k=1}^{\infty} f_k^k,  \prod_{k=1}^{\infty} f_k^k \mod g)$.
%
%Let $\tilde{B}_w$ be the closure in $T^* \left( (\mathbb P^1)^{(n)} \times T\right)$ of the image under this map of the vector bundle on $D_w$ whose fiber over $(f_k)_{k \in \mathbb N^+}$  consists of pairs of a linear form on polynomials of degree $<n$ and a linear form on polynomials mod $g $ whose sum vanishes on all multiples of $\prod_{k=1}^{\infty} f_k^{k-1}$.  \end{defi}

\begin{lemma}\label{prime-cycle-formula}  The characteristic cycle of $\tilde{s}_! (sym_* \mathbb Q_\ell [n] \otimes \rho)^{S_n}$ is  \[ \sum_{w\in \WF } M_{\rho}(w)  [ \tilde{B}_w] \] plus a union of irreducible components supported on  $(Z/ S_n ) \times T$. 

\end{lemma}

\begin{proof}  
It suffices to prove that the restriction of the characteristic cycle of $\tilde{s}_! (sym_* \mathbb Q_\ell [n] \otimes \rho)^{S_n}$ to the open complement of $(Z/ S_n) \times T$ is $\sum_{w\in \WF } M_{\rho}(w)  [ \tilde{B}_w] $.

Restricted to that open set, $\tilde{s}$ is a closed immersion.  From the Cartesian squares in the above commutative diagram, we can see that $[\tilde{B}_w] $ is the closure of the pushforward of $d\tilde{s}^* [B_w]$ to $T^* ( (\mathbb P^1)^{(n)} \times T)$. Thus, restricted to that open set, $[\tilde{B}_w]$ is the pushforward along $\tilde{s}$ of $d \tilde{s}^* [B_w]$, which by definition is $ \tilde{s}_! [B_w]$. The result then follows from \cite[Lemma 5.13(2)]{saito1}, which gives the compatibility of characteristic cycles with closed immersions, and Lemma \ref{cc-easy}.\end{proof}

\begin{lemma}\label{prime-support-formula} The singular support of $\tilde{s}_! (sym_* \mathbb Q_\ell [n] \otimes \rho)^{S_n}$  is the union of $\tilde{B}_w$ over all weight functions $w$ with $M_{\rho}(w)>0$ plus a union of irreducible components supported on  $(Z/ S_n ) \times T$.  \end{lemma}

\begin{proof} First note that $\tilde{s}_! (sym_* \mathbb Q_\ell \otimes \rho)^{S_n}[n] $ is a perverse sheaf because perversity is preserved by pushforward and compactly supported pushforward with respect to quasi-finite affine morphisms \cite[Corollary 4.1.3]{bbd}, as well as taking summands. By \cite[Proposition 5.14]{saito1}, it follows that the singular support of $\tilde{s}_! ((sym_* \mathbb Q_\ell \otimes \rho)^{S_n}[n] $  is the support of its characteristic cycle. The claim now follows from \cref{prime-cycle-formula}.  \end{proof}

 \begin{defi} We say a weight function $w$ is $m$-bounded if  $\sum_{k=1}^{\infty} w(k) <m$. 
 
 Fix $w$ an $m$-bounded weight function. let $A_{w}'$ be the image in $T$ of $pr_2 \circ \tilde{s} \circ ev_w$.  We will construct a closed conical subset of $T^* T$ supported over $A_w'$.
 
 To do this, let $W_w'$ be the inverse image of $\tilde{W}_w$ under \[ ev_w^* \tilde{s}^* d pr_2\colon ev_w^* \tilde{s}^* pr_2^* T^* T  \to  ev_w^* \tilde{s}^* T^* (( \mathbb P^1)^{(n)} \times T) .\]  By construction, $W_{w'}$ is a sub-vector bundle of $ ev_w^* \tilde{s}^* pr_2^* T^* T $ and thus maps to $T^*T $. Let $B_w'$ be the image of $W_w'$ in $T^* T$.  We write $[B_w']$ for the cycle-theoretic pushforward of $[W_w']$ to $T^* T$ , so that $[B_w']$ is a cycle with nontrivial multiplicity if the map $W_{w'} \to T^* T$ is not generically one-to-one.
 \end{defi}
 
 We do not need to take closures, since we will check that $A_{w}'$ is closed in \cref{w-small-proper}.
 
 Concretely, representing $ ev_w^* \tilde{s}^* T^* ( (\mathbb P^1)^{(n)} \times T) $ as $(P_n \oplus P_m )^\vee \otimes \mathcal O_{D_w}$, we can represent $ev_w^* \tilde{s}^* pr_2^* T^* T$ as $(P_m)^\vee \otimes \mathcal O_{D_w} $ and  $ev_w^* \tilde{s}^* d pr_2$ as the transpose of projection onto the second factor. Thus $W_w'$ is the kernel of \[ \Bigl( \cdot \prod_{k=1}^{\infty} f_k^{w(k)-1}\Bigr)^T \circ ( \textrm{mod } g)^T \colon   P_m^\vee  \otimes   \mathcal O_{D_w} \to P^\vee_{ \sum_k w(k)} \otimes  \mathcal O_{D_w}.\]  Because $f_1,\dots, f_k$ are prime to $g$ and $\sum_{k} w(k) < n $ , the dual map \[ (\textrm{mod }g)\circ \Bigl(\cdot \prod_{k=1}^{\infty} f_k^{w(k)-1} \Bigr) \colon P_{ \sum_k w(k)} \otimes  \mathcal O_{D_w} \to  P_m  \otimes   \mathcal O_{D_w}\] is injective. Thus $W_w'$ is a vector bundle of fiber dimension $m - \sum_k w(k)$ over a base of dimension $\sum_k w(k)$ and hence has total space dimension $m$.

 \begin{lemma}\label{w-small-proper}  For $w$ an $m$-bounded weight function,  the composition $\tilde{s} \circ ev_{w}$ is proper. In particular,  $A_{w}'$ is a closed subset of $T$.   \end{lemma}
 
 \begin{proof} By definition, $A_{w'}$ is the image of $pr_2 \circ \tilde{s} \circ ev_w$ for $pr_2 \colon (\mathbb P^1)^n \times T \to T$ the projection. Because $pr_2$ is proper, the claim about $A_{w'}$ follows from the properness of $\tilde{s} \circ ev_w$.

 Let $\prod_k sym_{w(k)} \colon   \prod_k C^{ w(k) } \to  \prod_k C^{ (w(k))} = D_w$ be the map sending a tuple $  ( (b_{k,j})_{j=1}^{w(k)} )_{k=1}^{\infty}$ to $ (\prod_{j=1}^{w(k)} (T - b_{j,k} ) )_{k=1}^{\infty}$. Then $ \prod_k sym_{w(k)}$ is surjective, so it suffices to show that $\tilde{s} \circ ev_w \circ \prod_k sym_{w(k)}$ is proper by \cite[01W0]{stacks} (since $\tilde{s} \circ ev_w$ is clearly separated of finite type).

 We have a commutative diagram
 
 \[ \begin{tikzcd}
    \prod_k C^{w(k)} \arrow[ rr, "\prod_k sym_{w(k)}"] , \arrow[d, "rep"] & & \prod_k C^{ (w(k))}\arrow[d, "ev_w"] \\ 
     X\arrow[d,"u"]&  C^n \arrow[l,equal] \arrow[r,"sym"] &   C^{ (n)} \arrow[d , "\tilde{s}"]  \\
\overline{X}\arrow[r,"s"]  &  (\mathbb P^1)^n \times T \arrow[r, "sym \times id"]  &    (\mathbb P^1)^{(n)} \times T    \end{tikzcd}\]
where $rep$ sends $ ( (b_{k,j})_{j=1}^{w(k)} )_{k=1}^{\infty}$ to the tuple 
\[ ( b_{1,1},\dots, b_{1,w(1)}, b_{2,1}, b_{2,1}, b_{2,2}, b_{2,2}, \dots, b_{2, w(2)}, b_{2,w(2)}, b_{3,1} ,b_{3,1}, b_{3,1} ,b_{3,2}, b_{3,2},b_{3,2} \dots ) \]
 obtained by repeating each $b_{k,j}$ $k$ times.  So $ \tilde{s} \circ ev_w \circ \prod_k sym_{w(k)}=  (sym \times id) \circ s \circ u \circ rep$. Since $ (sym \times id) $ and $ s $ are proper, it suffices to show that $ u \circ rep$ is proper.
 
Properness of a morphism can be checked on an open cover \cite[01W2]{stacks}. The restriction of $u \circ rep$ to the open set $X$ is simply $rep$, which is finite, hence proper. The restriction of $u \circ rep$ to the complement of the closure of the image of $u \circ rep$ is a morphism from the empty scheme and thus is proper. So it suffices to show these two open sets cover $\overline{X}$, i.e. that the closure of the image of $u \circ rep$ is contained in $X$.
 
 We defined $\overline{X} $ as a closed subset of $(\mathbb P^1)^n \times T$.  The image of $u \circ rep$ is contained in the subset $Y_w$ of tuples of the form 
 \[ ( ( b_{1,1},\dots, b_{1,w(1)}, b_{2,1}, b_{2,1}, b_{2,2}, b_{2,2}, \dots, b_{2, w(2)}, b_{2,w(2)}, b_{3,1} ,b_{3,1}, b_{3,1} ,b_{3,2}, b_{3,2},b_{3,2} \dots ), a) \]
 where now $b_{k,j} \in \mathbb P^1$ and $a \in T$. This subset $Y_w$ is closed because it is defined by equating some of the variables $\alpha_1,\dots, \alpha_n$.  By \cref{boundary-divisor-calculation}, the boundary $\overline{X} \setminus X$ is $Z \times T$, where $Z$ consists of tuples $\alpha_1,\dots, \alpha_n$ where for each $x \in S$ there exists $i$ from $1$ to $n$ with $\alpha_i =x$. This can only intersect $Y_w$ if for each $x \in S$ we have $k$ and $1\leq j \leq w(k)$ such that $b_{k,j} =x$. This can only happen if the total number $\sum_{k=1}^{\infty} w(k)$ of possible $b_{k,j}$ is at least the cardinality $|S|= m+1$ of $S$, which is false by our assumption $\sum_{k=1}^{\infty} w(k)< m$. So the intersection $Y_w \cap \overline{X} \setminus X$ vanishes, and thus the closure of the image of $u \circ rep$ is contained in $X$, as desired. \end{proof}

\begin{lemma}\label{small-w-push} Let $w$ be an $m$-bounded weight function.

Then $pr_{2 \circ } \tilde{B}_w$ has dimension $m$ and $pr_{2!} [\tilde{B}_w] = [B_w']$.  \end{lemma}

\begin{proof}  By definition, $pr_{2!}[ \tilde{B}_w]$ is the image under $pr_2\colon (\mathbb P^1)^{(n)}  \times T^*T \to T^*T $ of the inverse image under $dpr_2 \colon  (\mathbb P^1)^{(n)}  \times T^*T \to T^* ( (\mathbb P^1)^{(n)}  \times T) $ of $[\tilde{B}_w]$.  

The natural map $\tilde{W}_w \to T^*  ( (\mathbb P^1)^{(n)}  \times T) $ is the composition of the closed immersion $\tilde{W}_w \to ev_w^* \tilde{s}^* T^*  ( (\mathbb P^1)^{(n)}  \times T) $ with the  base change $ev_w^* \tilde{s}^* T^*  ( (\mathbb P^1)^{(n)}  \times T) \to T^*  ( (\mathbb P^1)^{(n)}  \times T) $ of $\tilde{s} \circ ev_w$. 

By \cref{w-small-proper}, $\tilde{s} \circ ev_w  $ is proper, so this base change is proper, and thus $\tilde{W}_w \to T^*  ( (\mathbb P^1)^{(n)}  \times T) $ is proper. Hence $\tilde{B}_w$, defined as the closure of the image of $\tilde{W}_w$, is simply the image, and $[\tilde{B}_w]$ is the pushforward of $[\tilde{W}_w]$ along this proper map. We have a commutative diagram with Cartesian squares
\[ \hspace{-.5in} \begin{tikzcd}
W_w' \arrow[r] \arrow[d] &  ev_w^*  \tilde{s}^*  \left( (\mathbb P^1)^{(n)} \times T^* T\right)   \arrow[r]  \arrow[d," ev_w^*\tilde{s}^* dpr_2 "]  &   \tilde{s}^*  \left( (\mathbb P^1)^{(n)} \times T^* T\right)   \arrow[r]    \arrow[d, "\tilde{s}^* d pr_2"]   &   (\mathbb P^1)^{(n)} \times T^* T  \arrow[d, "dpr_2"]  \arrow[r] & T^* T \\ 
\tilde{W}_w  \arrow[r] &  ev_w^*  \tilde{s}^* T^* \left( (\mathbb P^1)^{(n)} \times T\right)  \arrow[r]&   \tilde{s}^* T^* \left( (\mathbb P^1)^{(n)} \times T\right)  \arrow[r]    &  T^* \left( (\mathbb P^1)^{(n)} \times T\right) \end{tikzcd} \]
Using this diagram, $(dpr_2)^* [ \tilde{B}_w] $ is the pushforward of $ (ev_w^*\tilde{s}^* dpr_2)^* [\tilde{W}_w]$. To check that 
\[  (ev_w^*\tilde{s}^* dpr_2)^* [\tilde{W}_w]=[W_w']\]
we use the fact that the intersection of the vector bundles $\tilde{W}_w$ and $  ev_w^*  \tilde{s}^*   \left( (\mathbb P^1)^{(n)} \times T^* T\right) $ inside $ ev_w^*  \tilde{s}^* T^* \left( (\mathbb P^1)^{(n)} \times T\right)$ is transverse (because as we saw above it has the expected dimension), so the intersection-theoretic pullback is simply the inverse image with multiplicity one. 

Thus $pr_{2!} [\tilde{B}_w]$ is the pushforward of $[W_w']$ along the map $W_w' \to T^*T$, which by definition is $[B_w']$.  

To understand $pr_{2\circ} \tilde{B}_w$, we perform the same argument with set-theoretic pullback and pushforward instead of intersection-theoretic pullback and pushforward. We obtain by the same commutative diagram that $pr_{2 \circ} \tilde{B}_w= B_w' $ and thus has dimension $m$.
\end{proof}

\begin{lemma}\label{big-w-push} Let $w$ be weight function that is not $m$-bounded.

Then $pr_{2 \circ } \tilde{B}_w$ has dimension $\leq m$ and $pr_{2!} [\tilde{B}_w] $ is a scalar multiple of the zero section $[T]$. \end{lemma}

  \begin{proof} The zero section of $T^*T $ has dimension $m$. Thus every algebraic cycle of dimension $m$ supported on the zero section is a scalar multiple of the zero section. Because $pr_{2!} [\tilde{B}_w]$ is supported on $pr_{2 \circ } \tilde{B}_w$ by definition, it suffices to show that $pr_{2\circ } \tilde{B}_w$ is contained in the zero section.
  
  %By definition, $pr_{2!}[ \tilde{B}_w]$ is the image under $pr_2\colon (\mathbb P^1)^{(n)}  \times T^*T \to T^*T $ of the intersection-theoretic inverse image under $dpr_2 \colon  (\mathbb P^1)^{(n)}  \times T^*T \to T^* ( (\mathbb P^1)^{(n)}  \times T) $. This is an algebraic cycle class of dimension $m$, contained in $pr_{2 \circ } \tilde{B}_w$ which is the image under $pr_2$ of the set-theoretic image under $dpr_2$ of $\tilde{B}_w$.
  
%  It suffices to show that $pr_{2\circ } \tilde{B}_w$ is the zero section, as the zero section has dimension $m$, and this implies that  $pr_{2!} [\tilde{B}_w] $ is a scalar multiple of the zero section. 
    
That $pr_{2\circ } \tilde{B}_w$ is contained in the zero section is equivalent to saying that the inverse image under $dpr_2$ of $\tilde{B}_w$ is contained in the zero section of $(\mathbb P^1)^{(n)}  \times T^*T$. Because $dpr_2$ is injective, it is equivalent to say that $\tilde{B}_w$ intersects the image of $dpr_2$ only in the zero section of $ T^* ( (\mathbb P^1)^{(n)}  \times T)$.
  
  By definition, $\tilde{B}_w$ is the closure of the image of $\tilde{W}_w$ inside $ T^* ( (\mathbb P^1)^{(n)}  \times T)$. Let us first check that the image of $\tilde{W}_w$ intersects the image of $dpr_2$ only in the zero section.  Each point in the image of $\tilde{W}_w$ consists of a polynomial $\prod_{k=1}^{\infty} f_k^k$ and an element $(h_1,h_2)$ of $P_n^\vee \times P_m^\vee$ in the kernel of the composition $( \cdot \prod_{k=1}^{\infty} f_k^{k-1})^T \circ (id + \textrm{mod }g)^T$.  This point lies in the image of $d pr_2$ if and only $(h_1,h_2)$ it is contained in $0 \times P_m^\vee$, in which case $h_1=0$ and $h_2$ lies in the kernel of the composition $( \cdot \prod_{k=1}^{\infty} f_k^{k-1})^T \circ ( \textrm{mod }g)^T$.  To show the point lies in the zero section, we must show $h_2=0$, for which it suffices to show that $( \cdot \prod_{k=1}^{\infty} f_k^{k-1})^T \circ ( \textrm{mod }g)^T$ is injective, or, equivalently, that the transpose map $(\textrm{mod }g) \circ ( \cdot \prod_{k=1}^{\infty} f_k^{k-1})$ is surjective. This map takes a polynomial of degree $< \sum_{k=1}^{\infty} w(k)$, multiplies by $\prod_{k=1}^{\infty} f_k^{k-1}$, and reduces mod $g$ to obtain a polynomial of degree $<m$. Because $\sum_{k=1}^{\infty} w(k) \geq m$, every residue class mod $g$ arises from a polynomial of degree  $< \sum_{k=1}^{\infty} w(k)$, and because all the $f_k$ are coprime to $g$, multiplying by $\prod_{k=1}^{\infty} f_k^{k-1}$ is a bijection on residue classes, so indeed this map is surjective. 
 
 The map $\tilde{s}$ is proper away from the complement $(Z/S_n) \times T$ of $X/ S_n$ in $\overline{X}/S_n$.  Because $ev_w$ is proper, the pullback $ev_w^*\tilde{s}^* T^* ( (\mathbb P^1)^ {(n)} \times T) \to  T^* ( (\mathbb P^1)^ {(n)} \times T)$ is proper away from the cotangent fibers over $(Z/S_n) \times T$.

 Because the map $\tilde{W}_w \to T^* ( (\mathbb P^1)^{(n)}  \times T)$ is the composition of the closed immersion $\tilde{W}_w \to ev_w^*\tilde{s}^* T^* ( (\mathbb P^1)^ {(n)} \times T)$ with the map $ev_w^*\tilde{s}^* T^* ( (\mathbb P^1)^ {(n)} \times T) \to  T^* ( (\mathbb P^1)^ {(n)} \times T)$,  the map $\tilde{W}_w  \to T^* ( (\mathbb P^1)^{(n)}  \times T)$ is also proper away from the cotangent fibers over $(Z/S_n) \times T$. Thus if the closure of the image of $\tilde{W}_w$ intersects the image of $dpr_2$ outside the zero section, but the image of $\tilde{W}_w$ does not, this intersection must take place in the non-proper locus $(Z/S_n) \times T$. But this is impossible because, by \cref{prime-cycle-formula}, $\tilde{B}_w$ is an irreducible component of the characteristic cycle of $\tilde{s}_! (sym_* \mathbb Q_\ell \otimes \rho)^{S_n}$, taking $\rho = \mathbb Q_\ell[S_n]$ so 
\[ M_\rho(w) =\dim ( \mathbb Q_\ell [S_n] \otimes \sgn )^{S_w} = \dim ( \mathbb Q_\ell[S_n])^{S_w} = |S_n/S_w| \neq 0,\]  and by \cref{singular-support-symmetric}, every component of the characteristic cycle of 
 \[\tilde{s}_! (sym_* \mathbb Q_\ell \otimes \rho)^{S_n} = ( (sym \times id)_* s_* u_! \mathbb Q_\ell \otimes \rho)^{S_n},\] restricted to $(Z/S_n ) \times T$, intersects the image of $dpr_2$ only at zero. \end{proof}

\begin{lemma}\label{characteristic-cycle-torus} The characteristic cycle of $( R \pi_! \mathbb Q_\ell [n] \otimes \rho)^{S_n}$ is equal to \[ \sum_{ \substack{w\in \WF \\ \sum_{k=1}^{\infty} w(k) < m }} M_{\rho}(w)  [ B'_w] \] plus some integer multiple of the zero section. \end{lemma} 

\begin{proof} We will use the identity
\[ ( R \pi_! \mathbb Q_\ell \otimes \rho)^{S_n} = R pr_{2!}  \tilde{s}_! (sym_* \mathbb Q_\ell \otimes \rho)^{S_n} .\]
We apply \cref{cc-pushforward} to the map $pr_2$ and the constructible complex $  \tilde{s}_! (sym_* \mathbb Q_\ell \otimes \rho)^{S_n} $. We take the closed set $S$ to be the zero section. 

We first check that  $ pr_{2 \circ} SS ( \tilde{s}_! (sym_* \mathbb Q_\ell \otimes \rho)^{S_n} )$ has dimension $\leq m$. It suffices to check that $pr_{2\circ}$ of each irreducible component of $SS ( \tilde{s}_! (sym_* \mathbb Q_\ell \otimes \rho)^{S_n} )$ has dimension $\leq m$.  By \cref{prime-support-formula}, these irreducible components consist of the $\tilde{B}_w $ and some components supported on $(Z/ S_n) \times T$. That  $ \dim pr_{2 \circ} \tilde{B}_w  \leq m$ for all $w$ is checked in \cref{small-w-push,big-w-push}. For the irreducible components supported on $(Z/S_n) \times T$, we know from \cref{singular-support-symmetric} that they intersect the image of $d pr_2$ only in the zero section and so their image under $pr_{2\circ}$ is contained in the zero section, which has dimension $m$.

We next check that the map from $ dpr_2^*  SS ( \tilde{s}_! (sym_* \mathbb Q_\ell \otimes \rho)^{S_n} )$ to $T^* T$ is finite away from the zero-section of $T^*T$. Since this map from a closed subset of $(\mathbb P^1)^{(n)} \times T^*T$ to $T^*T$ is certainly proper, it suffices to check it has finite fibers over every pair of a point $a \in T$ and nonzero vector $v$ in the cotangent space to $T$ at $a$.  Again using \cref{singular-support-symmetric}, the fiber over $(a,v)$ cannot intersect $(Z/S_n) \times T$. Thus the fiber is contained in $C^n$. Hence because the fiber is closed, and $C^n$ is affine, the fiber is affine, and therefore, because the fiber is proper, the fiber is finite.

Thus we can apply \cref{cc-pushforward}, obtaining
\[ CC ( ( R \pi_! \mathbb Q_\ell \otimes \rho)^{S_n}) = CC ( R pr_{2!}  \tilde{s}_! (sym_* \mathbb Q_\ell \otimes \rho)^{S_n}) = pr_{2!} CC ( \tilde{s}_! (sym_* \mathbb Q_\ell \otimes \rho)^{S_n}) \] modulo cycles supported on the zero-section.
We apply \cref{prime-cycle-formula} to obtain \[  CC ( \tilde{s}_! (sym_* \mathbb Q_\ell \otimes \rho)^{S_n}) = \sum_{\substack{ w\in \WF}} M_{\rho}(w)  [ \tilde{B}_w] \] plus irreducible components supported on $(Z/S_n) \times T$ so \[ pr_{2!} CC ( \tilde{s}_! (sym_* \mathbb Q_\ell \otimes \rho)^{S_n})  = \sum_{\substack{ w\in \WF }} M_{\rho}(w)  pr_{2!}  [ \tilde{B}_w] \]  plus $pr_{2!}$ of irreducible components supported on $(Z/S_n) \times T$. By \cref{singular-support-symmetric}, $pr_{2!}$ of the remaining irreducible components is supported on the zero section.

The claim then follows from applying \cref{small-w-push} or \cref{big-w-push} to each term $pr_{2!}  [ \tilde{B}_w]$.
\end{proof}

\begin{lemma}\label{component-multiplicity} Let $w$ be a weight function. Assume $\sum_{k=1}^\infty w(k)  =m-1$. The multiplicity of $A_{w}'$ at a point of $T$ is at most \[ \frac{ (m -1)!}{ \prod_{k=1}^\infty w(k)! } \prod_{k=1}^\infty k^{w(k)},\] \end{lemma}

\begin{proof} Consider the $m+1$-dimensional vector space $H^0( S, \mathcal O(n))$ and its $m$-dimensional projectivization $\mathbb P ( H^0( S, \mathcal O(n)))$. We can embed $T$ into $\mathbb P ( H^0( S, \mathcal O(n)))$ by sending $a\in T$ to the restriction to $S$ of a monic polynomial $f$ of degree $n$ congruent to $a$ mod $g$. (The restriction to $S \setminus \infty$ is independent of the choice of polynomial because of the congruence condition, and the restriction to $\infty$ is independent of the choice of polynomial because it is monic.)

Because $\dim A_{w}' = \sum_{k=1}^\infty w(k) = m -1$ and $\dim T = m$, $ A_{w}' $ is a locally closed hypersurface in this projective space, and its multiplicity at any point is bounded by its degree. It suffices to show that this degree is $\frac{ (m -1)!}{ \prod_{k=1}^\infty w(k)! } \prod_{k=1}^\infty k^{w(k)}$.

By definition, $A_w'$ is the image in $T$ of the map $pr_2 \circ \tilde{s} \circ ev_w \colon  D_w  = \prod_{k=1}^{\infty} C^{ (w(k))} \to T$ that sends a tuple $(f_k)_{k=1}^{\infty}$ of monic polynomials prime to $g$ to $\prod_{k=1}^{\infty}  f_k^k \mod g$.  Viewing $T$ as a subset of $\mathbb P ( H^0( S, \mathcal O(n)))$, we see that $A_w'$ is the image of the map from $D_w$ to $\mathbb P ( H^0( S, \mathcal O(n)))$ that sends $(f_k)_{k=1}^{\infty}$ to the restriction of  $\prod_{k=1}^{\infty}  f_k^k \in H^0 (\mathbb P^1, \mathcal O(n))$. 

Let $\overline{A_w}$ be the set of sections in $H^0 (\mathbb P^1, \mathcal O(n))$ of the form $\prod_{k=1}^{\infty} f_k^k$ for $f_k \in H^0 (\mathbb P^1, \mathcal O(w(k)))$. Because $D_w$ is a dense subset of $\prod_{k=1}^{\infty} \mathbb P ( H^0 (\mathbb P^1, \mathcal O(w(k))))$, $A_{w}'$ is a dense subset of the image of $\overline{A_w} $  under the linear projection  \begin{equation}\label{linear-projection}  \mathbb P ( H^0(\mathbb P^1, \mathcal O(n))) \to \mathbb P ( H^0( S, \mathcal O(n))).\end{equation}

The indeterminacy locus of \eqref{linear-projection} consists of sections that vanish on all the points of $S$. Since $\prod_{k=1}^{\infty} f_k^{k}$ vanishes only where one of the $f_k$ does, it vanishes at at most $\sum_{k=1}^{\infty} w(k) =m-1$ points of $S$, so $\overline{A_w}$ does not intersect the indeterminacy locus, so $\deg A_{w}' = \deg \overline{A_w} $.

Because $\overline{ A_{w }} $ is the image of a map from \[\prod_{k=1}^\infty \mathbb P ( H^0( \mathbb P^1, \mathcal O(w(k)))) =  \prod_{k=1}^\infty\mathbb P^{w(k)}\] to $\mathbb P^n$, its degree is the $m-1$-fold self-intersection of the pullback of the hyperplane class $\mathcal O(1)$ of $\mathbb P^n$ to  $\prod_{k=1}^\infty\mathbb P^{w(k)}$

Because the map $(f_k)_{k =1}^{\infty}  \to \prod_{k=1}^{\infty} f_k^{k} $ has degree $k$ on the $k$th factor, the pullback of the hyperplane class is the sum over $k$ of $k$ times the hyperplane class of $\mathbb P^{w(k)}$.  Taking the $m-1$st power, we can ignore all terms where the hyperplane class of $\mathbb P^{w(k)}$ appears more than $w(k)$ times, leaving only the terms where it appears exactly $w(k)$ times for each $k$, each of which contributes $\prod_{k=1}^\infty k^{w(k)}$, and there are $\frac{ (m -1)!}{ \prod_{k=1}^\infty w(k)! }$ such terms. 
\end{proof}

To relate $C_2(\rho)$ to $M_\rho$, it is helpful to have an alternative description of $C_2(\rho)$. We obtain one by grouping the terms in \cref{C2-e-formula} into $S_{m-1}$-orbits.

   \begin{lemma}\label{cc-nice-formula} We have \[ C_2(\rho) = \sum_{ \substack{w\in \WF \\ \sum_{k=1}^{\infty} w(k) =m-1 }}  \frac{ (m -1)!}{ \prod_{k=1}^\infty w(k)! } \prod_{k=1}^\infty k^{w(k)} \dim ( \rho \otimes \sgn) ^{ S_w}.\] \end{lemma}

   \begin{proof} We will group the terms in the sum of Lemma \ref{C2-e-formula} into groups parameterized by functions $w$. Given a tuple $e_1,\dots, e_{m-1}$ such that $\sum_{i=1}^{m-1} e_i=n$, we can define the function \[w(k) = | \{ i \in \{1,\dots, n \} \mid e_i = k \} |.\]  With this definition, we have
  \[   \prod_{i=1}^{m-1} e_i   \dim ( \rho \otimes \sgn)^{ \prod_{i=1}^{m-1} S_{e_i}}   =\begin{cases}  \prod_{k=1}^{\infty} k^{w(k)}  \dim (\rho \otimes \sgn)^{S_w} & e_i \neq 0 \textrm{ for all }i \\ 0 & e_i=0 \textrm{ for some }i. \end{cases} . \]
  We always have $\sum_{k=1}^{\infty} k w(k) =n$ and we have $\sum_{k=1}^{\infty} w(k) = m-1$ unless some $e_i = 0$. Finally, the number of $e_1,\dots, e_{m-1}$ which produce a given $w$ satisfying these conditions is $\frac{(m-1)!}{\prod_{k=1}^{\infty} w(k)!}$. 
  
  By grouping together all the tuples $e_1,\dots,e_{m-1}$ that produce a given $w$, and ignoring those tuples where some $e_i=0$, we obtain
  \[ C_2(\rho) =  \sum_{ \substack{ (e_1,\dots, e_{m-1}) \in \mathbb N^{m-1}  \\ \sum_{i=1}^{m-1} e_i=n}}  \Bigl( \prod_{i=1}^{m-1} e_i \Bigr)   \dim ( \rho \otimes \sgn)^{ \prod_{i=1}^{m-1} S_{e_i}} \] \[ = \sum_{ \substack{ w\in \WF  \\ \sum_{k=1}^{\infty} w(k)=m-1}} \frac{ (m-1)! \prod_{k=1}^{\infty} k^{w(k)} }{ \prod_{k=1}^{\infty} w(k)!}  \dim (\rho \otimes \sgn)^{S_w } . \qedhere \] \end{proof}

 \begin{proposition}\label{cc-main-bound} Suppose that $\rho$ is an irreducible representation corresponding to a partition with all parts of size $<m$.
 
 Then the dimension of the stalk of  $\left( R^{n+1-m} \pi_! \mathbb Q_\ell \otimes \rho\right)^{S_n}$ at any point of $T$ is $\leq C_2(\rho)$. \end{proposition}
 
 \begin{proof} By \cref{irrep-perverse}, $\left(R \pi_! \mathbb Q_\ell [n]  \otimes \rho \right)^{S_n}$ is perverse. It follows from \citep[Theorem 1.4]{massey-formula} that the rank of its stalk in degree $1-m$ at a point $a\in T$, which is the stalk of $\left(R \pi_! \mathbb Q_\ell \otimes \rho\right)^{S_n}$ in degree $n+1-m$ at $a$, is at most the $m-1$st polar multiplicity at $a$.  By definition (\cite[Definition 1.3]{massey-formula}), this is \[ \operatorname{mult}_a \left(  \pi_*  \mathbb P ( CC ( \left(R \pi_! \mathbb Q_\ell \otimes \rho\right)^{S_n}[n])) \right) \] where $\operatorname{mult}_a$ is the ordinary multiplicity of a cycle at $a$, $\pi\colon  \mathbb P ( T^* T) \to T$ is the projection from the projectivized cotangent bundle, and $\mathbb P (  C)$ for a conical cycle $C$ denotes the projectivized cycle.  
 
   Note that the characteristic cycle has dimension $m$, so its projectivization has dimension $m-1$, and thus this pushforward is taken as an $m-1$-dimensional cycle.
 
Projectivization commutes with addition of cycles, and the projectivization of the zero section is the empty cycle, so by \cref{characteristic-cycle-torus} we have
\[ \mathbb P ( CC ( \left(R \pi_! \mathbb Q_\ell \otimes \rho\right)^{S_n}[n]))  = \sum_{ \substack{w\in \WF \\ \sum_{k=1}^{\infty} w(k) < m }} M_{\rho}(w) \mathbb P (  [ B'_w] ) .\] 
 
 Thus we have 
 \[ \pi_*  \mathbb P ( CC ( \left(R \pi_! \mathbb Q_\ell \otimes \rho\right)^{S_n}[n]))  =  \sum_{ \substack{ w\in \WF \\ \sum_{k=1}^{\infty} w(k) < m }} M_{\rho}(w)  \pi _* \mathbb P (  [ B'_w] ).\]
 
 For $w$ with $\sum_{k=1}^{\infty} w(k) < m-1$,  $\dim A_w' = \sum_{k=1}^{\infty} w(k) < m-1$, and so $\pi _* \mathbb P (  [ B'_w] )$ is a cycle of dimension $m-1$ contained in a scheme $A_w'$ of dimension $<m-1$ and thus vanishes. 
 
 Next we check for $w$ with $\sum_{k=1}^\infty w(k) = m-1$ that \[ \pi_* \mathbb P( [B'_w]) = [A'_w].\]
 
 For these $w$, $W'_w$ is a line bundle on $D_w$. Thus the projectivization of the fiber of $W'_w$ over any point of $D_w$ is a single point.
 
 Hence  $\mathbb P (B'_w)$ is the image of a map from $D_w$ to $\mathbb P(T^*T)$ that sends each point $z$ to a single point inside the fiber of $\mathbb P(T^* T)$ over $pr_2( \tilde{s} (ev_w(z)))$. 
 Thus $\pi _* \mathbb P ([B'_w])$ is the image of the map $pr_2 \circ \tilde{s} \circ ev_w$ from $D_w$ to $T$. Therefore, $\pi_* \mathbb P( [B'_w]) = [A'_w]$ by definition. Hence
 
\[ \operatorname{mult}_a \left(  \pi_*  \mathbb P ( CC ( \left(R \pi_! \mathbb Q_\ell \otimes \rho\right)^{S_n}[n])) \right)  =    \sum_{ \substack{ w\in \WF \\ \sum_{k=1}^{\infty} w(k) =m-1 }}  M_\rho(w) \operatorname{mult}_a  [A'_w] \]
\[ \leq  \sum_{ \substack{w\in \WF \\ \sum_{k=1}^{\infty} w(k) =m-1 }}  M_\rho(w) \frac{ (m -1)!}{ \prod_{k=1}^\infty w(k)! } \prod_{k=1}^\infty k^{w(k)} = C_2(\rho) \] by \cref{component-multiplicity} and \cref{cc-nice-formula}.  \end{proof}

\section{The main theorem}

Before proving \cref{squarefree-bound-precise}, we have the following intermediate step, which will also be useful on its own in cases where a main term is not necessary.

\begin{proposition}\label{alt-main} Let  $\mathbb F_q$ be a finite field. Let $g$ be a squarefree polynomial of degree $m$ over $\mathbb F_q$ and $a$ an invertible residue class mod $g$.  

Let $n\geq m$ be a natural number and let $\rho$ be an irreducible representation of $S_n$ corresponding to a partition with all parts of size $<m$.  
Then
\[ \Bigl| \sum_{ \substack{ f \in \mathcal M_n \\  f \equiv a \mod g }}  F_\rho(f)  \Bigr| \leq   q^{\frac{n-m}{2} } ( C_1(\rho)  + C_2(\rho) \sqrt{q} ). \]
\end{proposition}

\begin{proof} %From \eqref{trace-function} we obtain 
%\[ \Bigl| \sum_{ \substack{ f \in \mathcal M_n \\  f \equiv a \mod g }}  F_\rho(f)  \Bigr| \leq \sum_{j} \Bgil \tr \Bigl ( \Frob_q, \left( R^{j} \pi_! \mathbb Q_\ell \otimes \rho\right)^{S_n}_a \Bigr) \right| .\]
By Deligne's Riemann Hypothesis \cite[Theorem 1]{weil-ii}, all eigenvalues of $\Frob_q$ acting on $(R^{j} \pi_! \mathbb Q_\ell)_a$ are $\leq q^{j/2}$, and thus the same is true for  eigenvalues of $\Frob_q$ on  $\left( R^{j} \pi_! \mathbb Q_\ell \otimes \rho\right)^{S_n}_a$. Hence
\[ \left| \tr \left(\Frob_q, \left( R^{j} \pi_! \mathbb Q_\ell \otimes \rho\right)^{S_n}_a \right) \right| \leq q^{j/2} \dim\left( R^{j} \pi_! \mathbb Q_\ell \otimes \rho\right)^{S_n}_a . \]
We have
\[  \dim\left( R^{j} \pi_! \mathbb Q_\ell \otimes \rho\right)^{S_n}_a \leq \begin{cases}  C_1(\rho) & j = n-m \\ C_2(\rho) & j = n+1-m \\ 0 & \textrm{otherwise} \end{cases} \] by \cref{Euler-main-bound}, \cref{cc-main-bound}, and \cref{support-degrees} respectively.
%It follows from \cref{support-degrees} that $\left( R^{j} \pi_! \mathbb Q_\ell \otimes \rho\right)^{S_n}_a$ vanishes for $j \neq n-m, n+1-m$.  Furthermore, for $j= n-m$  its dimension is $\leq C_1(\rho)$ by \cref{Euler-main-bound}, and for $j = n+1-m$ its dimension is at most $C_2(\rho)$ by \cref{cc-main-bound}. 
Thus, by \cref{trace-function},
\[ \Bigl| \sum_{ \substack{ f \in \mathcal M_n \\  f \equiv a \mod g }}  F_\rho(f)  \Bigr| \leq  \sum_{j =0}^{2(n-m)}  \Bigl |\tr \left(\Frob_q, \left( R^{j} \pi_! \mathbb Q_\ell \otimes \rho\right)^{S_n}_a \right)\Bigr| \leq  C_1(\rho) q^{\frac{n-m}{2}} + C_2(\rho) q^{ \frac{n+1-m}{2} }. \qedhere \] \end{proof}

We recall the statement of \cref{squarefree-bound-precise}, and prove it. 

\begin{theorem} Let  $\mathbb F_q$ be a finite field. Let $g$ be a squarefree polynomial of degree $m$ over $\mathbb F_q$ and $a$ an invertible residue class mod $g$.  

Let $n\geq m$ be a natural number and let $\rho$ be a representation of $S_n$.  
Then
\[ \Bigl| \sum_{ \substack{ f \in \mathcal M_n \\  f \equiv a \mod g }}  F_\rho(f)  -\frac{1}{ \phi(g) }  \sum_{ \substack{ f \in\mathcal M_n \\ \gcd(f,g)=1 }}  F_\rho(f) \Bigr| \leq   2( C_1(\rho)  + C_2(\rho) \sqrt{q} ) q^{\frac{n-m}{2} }.  \]

\end{theorem}

\begin{proof} By definition, the functions $C_1(\rho)$ and $C_2(\rho)$ are additive in $\rho$, as is $F_{\rho}(f)$. Because $\rho$ is a sum of irreducible representations, we can reduce by linearity to the case when $\rho$ is irreducible. We now split into two cases, depending on whether the partition of $n$ corresponding to $\rho$ has a part of size $\geq m$.

If it does, then by \cref{irrep-constant}, $ \left( R^{j} \pi_! \mathbb Q_\ell \otimes \rho\right)^{S_n}$ is a geometrically constant sheaf, so its trace function $\tr \left(\Frob_q, \left( R^{j} \pi_! \mathbb Q_\ell \otimes \rho\right)_a^{S_n} \right)$ is constant. Hence by \cref{trace-function},  $\sum_{ \substack{ f \in \mathcal M_n \\  f \equiv a \mod g }}  F_\rho(f)  $ is independent of $a$. It follows that
\[\sum_{ \substack{ f \in \mathcal M_n \\  f \equiv a \mod g }}  F_\rho(f)  -\frac{1}{ \phi(g) }  \sum_{ \substack{ f \in \mathcal M_n \\ \gcd(f,g)=1 }}  F_\rho(f)  = 0 .\]

We have $C_1(\rho)=0$ by \cref{rho-Euler-difference} and $C_2(\rho)=0$ by \cref{schur-vanishing-lemma}. Thus our desired inequality follows, in this case, from $0 \leq 0$.

%Because $\rho \otimes \sgn$ corresponds to a partition with at least $m$ parts, its Schur-Weyl dual representation $V_\rho$ vanishes, so $C_1(\rho)=C_2(\rho)=0$. Because $O (\rho,q)$ is a sum of dimensions, it is is certainly nonnegative, so the inequality follows from \[ |0| \leq O(\rho,q) = q^{\frac{n-m}{2}} (0 + 0\sqrt{q} ) +O (\rho,q) \] in this case.

We now assume the partition of $n$ corresponding to $\rho$ does not have a part of size $\geq m$.  We apply \cref{alt-main} to bound $ \sum_{ \substack{ f \in \mathcal M_n \\  f \equiv a \mod g }}  F_\rho(f) $.  We calculate

%By \cref{irrep-cohomology-vanishing}, $H^j\left( \left(H^*(C_{\overline{\kappa}}, \mathbb Q_\ell)\right)^{\otimes n} \otimes \rho \right)^{S_n} $ vanishes except for $j$ between $n+m-1$ and $m$. Furthermore, because all eigenvalues of $\Frob_q$ on $H^1(C_{\overline{\kappa}}, \mathbb Q_\ell)$ have absolute value $1$, all eigenvalues of $\Frob_q$ on $H^j\left( \left(H^*(C_{\overline{\kappa}}, \mathbb Q_\ell)\right)^{\otimes n} \otimes \rho \right)^{S_n} $ have absolute value $q^{ j-n}$, so the trace of Frobenius on it has absolute value at most its dimension times $q^{i-n}$.

 %Setting $d=j-n$, since  $H^1(C_{\overline{\kappa}}, \mathbb Q_\ell)$ has dimension $m$, \[ H^j\left( \left(H^*(C_{\overline{\kappa}}, \mathbb Q_\ell)\right)^{\otimes n} \otimes \rho \right)^{S_n}  \cong \Bigl( (\mathbb C^{m} )^{\otimes (n-d) } \otimes \rho \otimes \sgn \Bigr)^{S_{n-d} } = V^d {\rho}  \] so \[ \Bigl| \frac{1}{ \phi(g) }  \sum_{ \substack{ f \in \mathcal M_n \\ \gcd(f,g)=1 }}  F_\rho(f)  \Bigr| \leq \frac{1}{\phi(g)} \sum_{d=0}^{m-1} q^d \dim V^d_\rho= O (\rho,q) \] by definition.
 
 \[\Bigl| \sum_{ \substack{ f \in \mathcal M_n \\  f \equiv a \mod g }}  F_\rho(f)  -\frac{1}{ \phi(g) }  \sum_{ \substack{ f \in \mathcal M_n \\ \gcd(f,g)=1 }}  F_\rho(f) \Bigr|  \leq  \Bigl| \sum_{ \substack{ f \in \mathcal M_n\\  f \equiv a \mod g }}  F_\rho(f) \Bigr|  +  \sum_{ a' \in \mathbb F_q[t]^\times} \frac{1}{ \phi(g) } \Bigl|   \sum_{ \substack{ f \in  \mathcal M_n \\ f \equiv a' \mod g }}  F_\rho(f) \Bigr| \] \[\leq  ( C_1(\rho)  + C_2(\rho) \sqrt{q} )q^{\frac{n-m}{2} }  +  \sum_{ a' \in \mathbb F_q[t]^\times} \frac{1}{ \phi(g) } ( C_1(\rho)  + C_2(\rho) \sqrt{q} )q^{\frac{n-m}{2} }   \] \[\leq 2 ( C_1(\rho)  + C_2(\rho) \sqrt{q} )q^{\frac{n-m}{2} }   . \qedhere \] \end{proof}

We also have a variant result that applies to linear combinations of the functions $F_\rho$;

\begin{corollary}\label{virtual-representation-variant} Let  $\mathbb F_q$ be a finite field. Let $g$ be a squarefree polynomial of degree $m$ over $\mathbb F_q$ and $a$ an invertible residue class mod $g$.  

Let $n\geq m$ be a natural number. Let $\rho_1,\dots, \rho_d$ be representations of $S_n$ and $\alpha_1,\dots, \alpha_d$ complex numbers.
Then
\[ \Bigl| \sum_{ \substack{ f \in \mathcal M_n \\  f \equiv a \mod g }}  \sum_{i=1}^d \alpha_i F_{\rho_i}(f)  -\frac{1}{ \phi(g) }  \sum_{ \substack{ f \in\mathcal M_n \\ \gcd(f,g)=1 }} \sum_{i=1}^d \alpha_i F_{\rho_i}(f) \Bigr| \leq   2\sum_{i=1}^d\abs{\alpha_i}  ( C_1(\rho_i)  + C_2(\rho_i) \sqrt{q} ) q^{\frac{n-m}{2} }.  \]

If $\abs{\alpha_1},\dots, \abs{\alpha_d}=1$ then

\[ \Bigl| \sum_{ \substack{ f \in \mathcal M_n \\  f \equiv a \mod g }}  \sum_{i=1}^d \alpha_i F_{\rho_i}(f)  -\frac{1}{ \phi(g) }  \sum_{ \substack{ f \in\mathcal M_n \\ \gcd(f,g)=1 }} \sum_{i=1}^d \alpha_i F_{\rho_i}(f) \Bigr| \leq   2 \Bigl( C_1\Bigr( \bigoplus_{i=1}^d \rho_i\Bigr)  + C_2\Bigl(\bigoplus_{i=1}^d\rho_i\Bigr) \sqrt{q} \Bigr) q^{\frac{n-m}{2} }.  \]

\end{corollary}

\begin{proof} By the triangle inequality and Theorem \ref{squarefree-bound-precise} we have
\[ \Bigl| \sum_{ \substack{ f \in \mathcal M_n \\  f \equiv a \mod g }}\sum_{i=1}^d \alpha_i F_{\rho_i}(f) -\frac{1}{ \phi(g) }  \sum_{ \substack{ f \in\mathcal M_n\\ \gcd(f,g)=1  }} \sum_{i=1}^d \alpha_i F_{\rho_i}(f)\Bigr|\]
\[ \leq \sum_{i=1}^{d} \abs{\alpha_i} \Bigl| \sum_{ \substack{ f \in \mathcal M_n \\  f \equiv a \mod g }}  F_{\rho_i} (f)  -\frac{1}{ \phi(g) }  \sum_{ \substack{ f \in\mathcal M_n \\ \gcd(f,g)=1 }}  F_{\rho_i} (f) \Bigr|  \]
\[ \leq  \sum_{i=1}^{d}   \abs{\alpha_i}2 ( C_1(\rho_i )  + C_2(\rho_i) \sqrt{q} )  q^{\frac{n-m}{2} }     \]
and if $\abs{\alpha_i}=1$ for all $i$ then 
\[\sum_{i=1}^{d}   \abs{\alpha_i}2 ( C_1(\rho_i )  + C_2(\rho_i) \sqrt{q} )  q^{\frac{n-m}{2} }   = \sum_{i=1}^{d}   2 ( C_1(\rho_i )  + C_2(\rho_i) \sqrt{q} )  q^{\frac{n-m}{2} } =2  \Bigl( C_1\Bigl( \bigoplus_{i=1}^d \rho_i\Bigr)  + C_2\Bigl(\bigoplus_{i=1}^d \rho_i\Bigr) \sqrt{q} \Bigr)q^{\frac{n-m}{2} }   . \]   
\end{proof}

In applications, the constant $2$ should always be irrelevant, but in case it ever matters, we sketch how it can be removed by introducing an additional, usually much smaller, error term. Let
\[O(\rho,q)= \sum_{d=0}^{m-1}  \frac{q^d}{\phi(g) } \dim V^d_{\rho}  .\]

\begin{proposition} Let  $\mathbb F_q$ be a finite field. Let $g$ be a squarefree polynomial of degree $m$ over $\mathbb F_q$ and $a$ an invertible residue class mod $g$.  

Let $n\geq m$ be a natural number and let $\rho$ be a representation of $S_n$.  
Then
\[ \Bigl| \sum_{ \substack{ f \in \mathcal M_n \\  f \equiv a \mod g }}  F_\rho(f)  -\frac{1}{ \phi(g) }  \sum_{ \substack{ f \in\mathcal M_n \\ \gcd(f,g)=1 }}  F_\rho(f) \Bigr| \leq   ( C_1(\rho)  + C_2(\rho) \sqrt{q} )q^{\frac{n-m}{2} }   + O(\rho,q) \]

\end{proposition}

\begin{proof} The proof is the same as \cref{squarefree-bound-precise}, except that in the case where $\rho$ corresponds to a partition with all parts of size $<m$,  we use  \eqref{eq-trace-complex} to estimate $\sum_{ \substack{ f \in \mathcal M_n \\ \gcd(f,g)=1}}  F_\rho(f)  $.
\[ \Bigl|  \sum_{ \substack{ f \in \mathcal M_n \\ \gcd(f,g)=1}}  F_\rho(f) \Bigr| =\Bigl|   \sum_{j =0}^{2n}  (-1)^j \tr \Bigl(\Frob_q, H^i\Bigl( \Bigl(H^*(C_{\overline{\kappa}}, \mathbb Q_\ell)\Bigr)^{\otimes n} \otimes \rho \Bigr)^{S_n}\Bigr)\Bigr| \]
\[ \leq  \sum_{j =0}^{2n} \Bigl|    \tr \Bigl(\Frob_q, H^i\Bigl( \Bigl(H^*(C_{\overline{\kappa}}, \mathbb Q_\ell)\Bigr)^{\otimes n} \otimes \rho \Bigr)^{S_n}\Bigr)\Bigr| .\]
Because all eigenvalues of $\Frob_q$ on $H^1(C_{\overline{\kappa}}, \mathbb Q_\ell)$ have absolute value $1$, all eigenvalues of $\Frob_q$ on $H^j\Bigl( \Bigl(H^*(C_{\overline{\kappa}}, \mathbb Q_\ell)\Bigr)^{\otimes n} \otimes \rho \Bigr)^{S_n} $ have absolute value $q^{ j-n}$, so (using \cref{rho-cohomology} and \cref{other-vanishing-lemma}) 
\[ \Bigl|  \sum_{ \substack{ f \in \mathcal M_n \\ \gcd(f,g)=1}}  F_\rho(f) \Bigr| \leq  \sum_{j =0}^{2n}  q^{j-n} \dim  \Bigl( \Bigl( H^*(C_{\overline{\kappa}}, \mathbb Q_\ell)\Bigr)^{\otimes n} \otimes \rho \Bigr)^{S_n} \] \[= \sum_{j=0}^{2n} q^{j-n} \dim V^{j-n}_\rho = \sum_{d=0}^{m-1} q^d \dim V^d_\rho  = \phi(g) O (\rho,q)  \] by definition.

Also, in the case where $\rho$ corresponds to a partition with one part of size $\geq m$, we observe that $O(\rho,q)\geq 0$ since it is a sum of dimensions, so the inequality holds also in that case. \end{proof}

\section{Applications}\label{s-applications}

We now apply \cref{squarefree-bound-precise} to various functions of arithmetic interest. We handle the M\"obius function, divisor function, and von Mangoldt function in their own subsections, with additional subsections for the application to moments of $L$-functions (building off the divisor function sums) and one-level densities (building off von Mangoldt sums).

A commonality between these arguments will be the following bound for binomial coefficients.
%The following binomial coefficient bound will be helpful in getting simple upper bounds that are close to optimal for what our method can produce. We will apply it when $a,b$ are proportional to our main parameters $m,n$, while $x,y$ are small, and $\alpha$ is related to the level of distribution.

\begin{lemma}\label{binomial-bound} Let $\alpha>0$ be a real number. Let $a,b$ be natural numbers with $\alpha b \leq a $. Let $x,y$ be integers. Then
\[ \binom{a+b+x+y }{ a+x} \leq  \frac{ (\alpha+1)^{x+y}}{\alpha^x}  \left( \frac{  (\alpha+1)e }{ \alpha} \right)^a  \]
\end{lemma}

 We will apply \cref{binomial-bound} when $a,b$ are proportional to our main parameters $m,n$, while $x,y$ are small, and $\alpha$ is related to the level of distribution. Because it bounds a binomial coefficient in terms of an exponential in $a$, it will give a bound exponential in $n$ and $m$ and thus polynomial in $q^n$ and $q^m$. This will make it easy to compare to our savings over the trivial bound, which is $q^{ \frac{n-m}{2}}$, or to bounds typical in analytic number theory, which are polynomials in the key parameters.  Among all bounds for a binomial coefficient in terms of an exponential, this is not the optimal one, but it is close to optimal and gives very simple formulas.

\begin{proof} We have \[(\alpha+1)^{a+b+x+y} = \sum_{i=0}^{a+b+x+y} \binom{a+b+x+y}{i} \alpha^i  \geq \binom{a+b+x+y}{a+x} \alpha^{a+x} \] so 
\begin{equation}\label{binom-bound-1} \binom{a+b+x+y }{ a+x} \geq \frac{ (\alpha+1)^{a+b+x+y} }{ \alpha^{a+x} }.\end{equation}
We have $1+ \alpha \leq e^{\alpha}$ and $ \alpha b \leq a $ so \[ (1+\alpha)^b \leq e^a.\] Combining this with \eqref{binom-bound-1} gives
\[ \binom{a+b+x+y }{ a+b} \leq\frac{ (\alpha+1)^{a+x+y} e^{ a}  }{ \alpha^{a+x} } =  \frac{ (\alpha+1)^{x+y}}{\alpha^x } \left( \frac{  (\alpha+1)e}{ \alpha} \right)^a . \qedhere \]  \end{proof}

It is not hard to check that, when $\alpha b \leq a \leq (\alpha + o(1)) b$, the only inequality used which is not optimal to within a subexponential factor is $1+\alpha\leq  e^\alpha$. Thus, this bound gets closer to optimal as $\alpha$ goes to $0$, which will occur when the original analytic number theory problem gets more difficult -- when the level of distribution goes to $1$, or, in one-level density problems, when the Fourier support grows.

We will use throughout these arguments the notation $F(u) [ u^n]$ to refer to the coefficient of $u^n$ in the power series $F(u)$, and the straightforward identity
\[ \frac{u^a}{(1-u)^b} [u^n] = \binom{n+ b-1-a }{ b-1}.\]

\subsection{The M\"{o}bius function}

\begin{lemma}\label{mobius-1}  For $g$ squarefree of degree $m \geq 2$, a natural number $n \geq m$, and a residue class $a \in (\mathbb F_q[t]/g)^\times$, we have
\[\left | \sum_{ \substack{ f \in \mathcal M_n\\  f \equiv a \mod g }}  \mu(f) \right| \leq \binom{n+m-2}{ 2m-2} q^{ \frac{n-m}{2}} +\binom {n+m-2 }{ 2m-3} q^{\frac{n+1-m}{2}} . \] \end{lemma}

\begin{proof} We apply \cref{alt-main} to $\rho=\sgn$, noting that $F_\rho(f) = (-1)^n \mu(f)$ \cite[Lemma 3.5]{SawinShort}.

We have \[ V_{\rho} =\Bigl ( (\mathbb C^{m-1})^{\otimes n}   \Bigr)^{S_n} = \Sym^n (\mathbb C^{m-1} ) \] so  
\begin{align*}
C_2(\rho) =&  \frac{\partial^{m-1}}{ \partial \lambda_1\dots  \partial \lambda_{m-1} } \tr( \Diag(\lambda_1,\dots, \lambda_{m-1}, V_\rho) )  \Big|_{\lambda_1,\dots,\lambda_{m-1}=1} \\
 = & \frac{\partial^{m-1}}{ \partial \lambda_1\dots  \partial \lambda_{m-1} }\prod_{i=1}^{m-1} 1/(1- u \lambda_i)   [u^n]\Big|_{\lambda_1,\dots,\lambda_{m-1}=1} \\
 =& \prod_{i=1}^{m-1} \frac{ u} { (1-u \lambda_i)^2} [u^n]  \Big|_{\lambda_1,\dots,\lambda_{m-1}=1} \\
 =& \frac{u^{m-1}}{ (1-u)^{2(m-1)} }[u^n ] = \binom{n + m-2}{ 2m-3}.\end{align*} 

Furthermore we have \[V^d_\rho =\Bigl ( (\mathbb C^{m})^{\otimes (n-d)} \otimes \sgn_{S_d}   \Bigr)^{S_{n-d} \times S_d}   \] which is $  \Sym^n ( \mathbb C^m) $ if $d=0$, $\Sym^{n-1} (\mathbb C^m)$ if $d=1$, and $0$ otherwise. 

This gives \[C_1(\rho)=  \frac{\partial^{m}}{ \partial \lambda_1\dots  \partial \lambda_{m} } \left( \tr( \Diag(\lambda_1,\dots, \lambda_m),  \Sym^n ( \mathbb C^m) ) ) - \tr( \Diag(\lambda_1,\dots, \lambda_m),  \Sym^{n-1} ( \mathbb C^m))  \right)  \Big|_{\lambda_1,\dots,\lambda_{m}=1} \] \[=\frac{\partial^{m}}{ \partial \lambda_1\dots  \partial \lambda_{m} } (1-u)  \prod_{i=1}^{m} \frac{1}{1- u \lambda_i} [u^n] \Big|_{\lambda_1,\dots,\lambda_{m}=1}=(1-u) \prod_{i=1}^{m} \frac{ u} { (1-u \lambda_i)^2}[u^n] \Big|_{\lambda_1,\dots,\lambda_{m}=1} \] \[ = \frac{u^m}{ (1-u)^{2m-1} }[u^n ] = \binom{ n+m-2 }{2m-2 } .\] \end{proof} 

\begin{proposition}\label{mobius-2} Let $0<\theta<1 $ be a real number and let \[ q >  \frac{e^2 (1+\theta)^2 }{(1-\theta)^2}  \] be a prime power. Then for any natural numbers $n,m$ with $m\leq \theta n$,  squarefree $g \in \mathcal M_m$, and $a \in (\mathbb F_q[t]/g)^\times $, we have
\[ \sum_{\substack{  f \in \mathcal M_n\\ f \equiv a \mod g}} \mu(f) =  O \left( (q^{n-m})^{1-\delta} \right) \]
where \[ \delta =  \frac{1}{2} - \log_q \left( \frac{ (1+\theta) e}{1-\theta}\right) > 0\] and the implicit constant depends only on $q,\theta$.
\end{proposition}

\begin{proof} By \cref{mobius-1} we have \[ \left| \sum_{\substack{  f \in \mathcal M_n\\ f \equiv a \mod g}} \mu(f) \right| \leq  \binom{n+m-2}{2m-2} q^{\frac{n-m}{2} }+ \binom{n+m-3}{2m-3} q^{ \frac{n+1-m}{2} }.\] 
We apply \cref{binomial-bound} to each binomial coefficient, taking $a=n-m, b=2m$, $\alpha= \frac{1-\theta}{2\theta} $, $x=0$ and $1$ respectively, $y=-2$ and $-3$ respectively. We obtain
\[\left| \sum_{\substack{  f \in \mathcal M_n\\ f \equiv a \mod g}} \mu(f) \right| \leq \left( \frac{(2\theta)^2}{(1+\theta)^2} + \frac{(2\theta)^3 q^{1/2} }{(1+\theta)^2 (1-\theta) } \right)  \cdot \left( \frac{ e (1+ \theta) q^{1/2} }{ 1-\theta} \right)^{n-m} \ =  O(1) \cdot (q^{n-m})^{1-\delta} . \]
 
%
%\[  \leq \frac{ (1+ \theta)^{n+m-2} q^{ \frac{n-m}{2} } }{ (1-\theta)^{n-m} (2 \theta)^{2m-2} } + \frac{ (1+ \theta)^{n+m-3} q^{ \frac{n+-m}{2} } }{ (1-\theta)^{n-m} (2 \theta)^{2m-3} } = \left( \frac{(2\theta)^2}{ (1+\theta)^2} + \frac{ q^{\frac{1}{2}} (1+\theta)^3}{ (2\theta)^3} \right) \left( \frac{ 2 \theta q^{ \frac{1}{2} }}{ (1-\theta) }  \right)^{n-m} \left( \frac{ 1+\theta}{2\theta } \right)^{n+m} , \]
%Because \[  \frac{2\theta}{1+\theta} = 1 - \frac{1-\theta}{1+\theta} \leq e^{ \frac{1-\theta}{1+\theta}} \] and \[ \frac{n+m}{1+\theta} \leq n \leq \frac{n-m}{1-\theta}, \]
%we have
%\[  \left( \frac{ 1+\theta}{2\theta } \right)^{n+m}  \leq  e^{n-m} \]
%and so
%\[ \left| \sum_{\substack{  f \in \mathcal M_n\\ f \equiv a \mod g}} \mu(f) \right|  =O \left(  \left( \frac{ 2 e \theta q^{ \frac{1}{2} }}{ (1-\theta) }  \right)^{n-m} \right)\]
Because $q> \frac{e^2 (1+ \theta)^2 }{(1-\theta)^2}$,  we have \[\delta = \frac{1}{2} \log_q \left( \frac{ q (1-\theta)^2}{ e^2 (1+\theta)^2} \right)> 0. \qedhere \] % \[ \frac{ (1+\theta) e q^{ \frac{1}{2} }}{ 1-\theta }< q,\] and we can take \[ \delta = \log_q \left( \frac{ q^{\frac{1}{2}} (1-\theta)}{ (1+\theta) e} \right) > 0.\]
 \end{proof}

\subsection{The divisor function}

Fix a natural number $k$. For this subsection, let $\rho=  (\mathbb C^k)^{\otimes n} $.

\begin{lemma}\label{divisor-1-f} We have $F_\rho(f) =d_k(f)$. \end{lemma}

\begin{proof} A basis $v_1,\dots, v_k$ of $\mathbb C^k$ defines a basis for $\rho$ consisting of tensors $v_{j(1)} \otimes \dots \otimes v_{j(n)}$ where $j$ varies over functions $\{1,\dots, n\} \to \{1,\dots, k\}$. The group $S_n$ acts on this basis by permuting the function $j(n)$.  The $S_n$-orbits are parameterized by tuples $(n_1,\dots, n_k)$, with $n_i = \abs{ j^{-1}(i)}$, where the stabilizer of an element in the $(n_1,\dots, n_k)$-orbit is $S_{n_1} \times \dots \times S_{n_k}$. So \[ \rho= \bigoplus_{ \substack{n_1,\dots, n_k \in \mathbb N \\ n_1 + \dots + n_k =n }  } \operatorname{Ind}_{S_{n_1} \times \dots \times S_{n_k}}^{S_n} \mathbb C.\]  The result then follows from \cite[Lemma 3.8]{SawinShort}.
%Alternately, this induces a basis for $V_f \otimes \rho$ parameterized by pairs of a tuple $(a_1,\dots, a_n) \in \overline{\mathbb F}_q$ with $\prod_{i=1}^n (t-a_i)=f$ and a function $j$ from $\{1,\dots,n\}$ to $\{1,\dots, k\}$. Because $S_n$ acts by permutation, $(V_f \otimes \rho)^{S_n}$ has a basis parametrized by $S_n$-orbits on this basis. Given $((a_1,\dots, a_n), j)$, we can define $f_l = \prod_{ i \in j^{-1} (l)} (t-a_i)$. Then two tuples lie int he same $S_n$-orbit if and only if they have the same $f_1,\dots, f_k$, and any tuple $f_1,\dots, f_k \in \overline{\mathbb F}_q[t]$ of monic polynomials with product $f$ comes from an $S_n$-orbit.
%
%Thus $(V_f \otimes \rho)^{S_n}$ has a basis consisting of tuples $f_1,\dots, f_k$, on which $\Frob_q$ now acts by permutation. The trace of $\Frob_q$ is the number of basis vectors fixed by $\Frob_q$. Since a tuple $f_1,\dots,f_k$ is fixed by $\Frob_q$ if and only if $f_1,f_k \in \mathbb F_q[t]$, this is $d_k(f)$.
\end{proof}

\begin{lemma}\label{divisor-1-C2} We have \[  C_2(\rho) = k^{m-1} \binom{ mk-k-m+1 }{n+1-m} .\] \end{lemma}

\begin{proof} We have \[V_{\rho} = \left( (\mathbb C^{m-1})^{\otimes n} \otimes (\mathbb C^k ) ^{\otimes n}  \otimes \sgn \right)^{S_n} = \left( (\mathbb C^{m-1} \otimes \mathbb C^k )^{\otimes n} \otimes \sgn \right)^{S_n}= \wedge^n (\mathbb C^{m-1} \otimes \mathbb C^k )\] 
so \[\tr(\Diag(\lambda_1,\dots, \lambda_{m-1}), V_{\rho} ) = \prod_{i=1}^{m-1} ( 1+ \lambda_i u)^k [ u^n ] .\]

Thus
\[  C_2(\rho) = \frac{\partial^{m-1}} {\partial \lambda_1 \dots \partial \lambda_{m-1}} \prod_{i=1}^{m-1} ( 1+ \lambda_i u)^k [ u^n ]  \Big|_{\lambda_1,\dots,\lambda_{m-1}=1} =\prod_{i=1}^{m-1}k u ( 1+ \lambda_i u)^{k-1} [ u^n ]  \Big|_{\lambda_1,\dots,\lambda_{m-1}=1}  \] 
\[ = k^{m-1} u^{m-1} (1+u)^{ (m-1)(k-1) } [ u^n] =k^{m-1} \binom{ mk-k-m+1}{n+1-m} . \qedhere \] \end{proof}

\begin{lemma}\label{divisor-1-C1}We have \begin{equation}\label{eq-divisor-C1} C_1(\rho) = \left(k^m - \binom{m+k-1}{m} \right) \binom{mk-m-k}{n-m}.\end{equation}  \end{lemma}

\begin{proof} First observe that \[ V^d_{\rho} = \Bigl( (\mathbb C^m)^{\otimes (n-d)} \otimes (\mathbb C^k)^{\otimes n} \otimes \sgn_{S_{n-d}}\Bigr)^{S_{n-d} \times S_d} =\Bigl( (\mathbb C^{m} \otimes \mathbb C^{k} )^{\otimes (n-d)} \otimes \sgn_{S_{n-d}}\Bigr)^{S_{n-d}} \otimes ( (\mathbb C^k)^{\otimes d} )^{S_d} \] \[= \wedge^{n-d} ( \mathbb C^{m} \otimes \mathbb C^{k}  ) \otimes \Sym^{d} ( \mathbb C^k). \] Thus \[ \sum_{d=0}^n (-1)^ d \tr(\Diag(\lambda_1,\dots, \lambda_{m}), V_{\rho}^d )  = \sum_{d=0}^n (-1)^d  \binom{ d+k-1}{k-1}\prod_{i=1}^m (1+\lambda_i u)^k  [u^{n-d} ] \] \[= \frac{1}{ (1+u)^k} \prod_{i=1}^m (1+\lambda_i u)^k [u^n] \]
and \begin{equation} \begin{split} & \label{d1-C1-1}  \frac{\partial^m}{ \partial \lambda_1 \dots \partial \lambda_m}  \frac{1}{ (1+u)^k} \prod_{i=1}^m (1+\lambda_i u)^k [u^n] \Big|_{\lambda_1,\dots,\lambda_{m}=1}=  \frac{1}{ (1+u)^k} \prod_{i=1}^m\left(k u (1+\lambda_i u)^{k-1} \right) [u^n]\Big|_{\lambda_1,\dots,\lambda_{m}=1} \\
& = k^m u^m (1+u)^{mk-m-k} [u^n] = k^m\binom{mk-m-k}{n-m} . \end{split} \end{equation} 
Furthermore,
\[  \dim V^d_\rho = \binom{mk}{n-d}   \binom{ d+k-1}{k-1} =   \frac{1}{ (1-vu)^k} (1+u)^{mk} [v^d] [u^n]. \] 
Since $ \frac{1}{m!} \frac{\partial^m }{ \partial v^m} v^d = \binom{d}{m} v^{d-m} $, we have 
\begin{equation}\label{d1-C1-2} \begin{split} \sum_{d=0}^m (-1)^{d-m} \binom{d}{m } \dim V^d_\rho &= \frac{1}{m!} \frac{\partial^m }{ \partial v^m} \frac{1}{ (1-vu)^k} (1+u)^{mk} [u^n ]  \Big|_{v=-1}\\
&=  \binom{m+k-1}{m}  \frac{u^m}{ (1-vu)^{k+m} } (1+u)^{mk} [u^n] \Big|_{v=-1} \\
&= \binom{m+k-1}{m}  \frac{u^m}{ (1+u)^{k+m}} (1+u)^{mk} [u^n]= \binom{m+k-1}{m} \binom{mk-m-k}{n-m} . \end{split}\end{equation}
Combining \eqref{d1-C1-1} and \eqref{d1-C1-2}, we have \eqref{eq-divisor-C1}. \end{proof}

\begin{lemma}\label{divisor-1} For $g$ a squarefree polynomial of degree $m$, and $n\geq m$ a natural number, we have
\[ \Bigl| \sum_{ \substack{ f \in  \mathcal M_n\\  f \equiv a \mod g }}  d_k(f)  -\frac{1}{ \phi(g) }  \sum_{ \substack{ f \in \mathcal M_n \\ \gcd(f,g)=1}}  d_k(f) \Bigr|\] \[ \leq   2\left(k^m - \binom{m+k-1}{m} \right) \binom{mk-m-k}{n-m}  q^{\frac{n-m}{2} }  + 2  k^{m-1} \binom{mk-k-m+1}{n+1-m} q^{ \frac{n+1-m}{2} }.\]

\end{lemma}

\begin{proof} This follows from \cref{squarefree-bound-precise} applied to $\rho $, using \cref{divisor-1-C2} and \cref{divisor-1-C1} to compute $C_2(\rho)$ and $C_1(\rho)$ respectively. \end{proof}

% and compute the two terms. We have \[V_{\rho} = \left( (\mathbb C^{m-1})^{\otimes n} \otimes (\mathbb C^k ) ^{\otimes n}  \otimes \sgn \right)^{S_n} = \left( (\mathbb C^{m-1} \otimes \mathbb C^k )^{\otimes n} \otimes \sgn \right)^{S_n}= \wedge^n (\mathbb C^{m-1} \otimes \mathbb C^k )\] 
%so \[\tr(\Diag(\lambda_1,\dots, \lambda_{m-1}), V_{\rho} ) = \prod_{i=1}^{m-1} ( 1+ \lambda_i u)^k [ u^n ] .\]
%

%Finally, $O(m,q) =\frac{  \sum_{d=0}^{m-1}   \binom{mk}{n-d}   \binom{ d+k-1}{k-1} q^d}{ \phi(g)} $. 

\begin{proposition}\label{divisor-2}
Let $k>1$ be a natural number, $\\theta<1 $ be a real number and let \[ q>  \frac{e^2 (k-1)^2  \theta^2}{ (1-\theta)^2 }   k^{ \frac{2 \theta }{1-\theta}} \]
 be a prime power. For any natural numbers $n,m$ with $m \leq \theta n$, $g \in \mathcal M_m$, and $a \in (\mathbb F_q[t]/g)^\times $,
\[ \sum_{\substack{  f \in \mathcal M_n\\ f \equiv a \mod g}} d_k(f) = \frac{1}{\phi(g) }\sum_{\substack{  f \in \mathcal M_n\\  \gcd(f,g)}} d_k(f)    + O \left( (q^{n-m})^{1-\delta} \right) \]
where 
\[ \delta = \frac{1}{2} -  \log_q \left( \frac { (k-1) e \theta  k^{\frac{\theta}{1-\theta} }} {  1-\theta } \right) >0\] and the implicit constant depends only on $q,k,\theta$.
\end{proposition} 
\begin{proof} It suffices to prove each term in \cref{divisor-2} is $O \left( (q^{n-m})^{1-\delta} \right)$ separately. For $\theta< \frac{1}{k}$ we have $mk \leq n$ so both error terms vanish, thus we may assume $\theta> \frac{1}{k}$. Applying \cref{binomial-bound}, taking $a=n-m$,  $b= km-n$, $x=0,$ $y=-k$, and $\alpha=  \frac{ 1- \theta}{ k \theta -1}$, we have

\[2 q^{\frac{n-m}{2} } \left(k^m - \binom{m+k-1}{m} \right) \binom{mk-m-k}{n-m} \leq 2q^{ \frac{n-m}{2} } k^m \binom{mk-m-k}{n-m}\] \[ \leq    2 k^m  \frac{ (k\theta-1)^k}{ (k\theta -\theta)^k}    \left( \frac{  q^{1/2} (k-1) e \theta  }{ 1-\theta } \right)^{n-m}  \leq 2\frac{ (k\theta-1)^k}{ (k\theta-\theta)^k}    \left( \frac{  q^{1/2} (k-1) e \theta  k^{\frac{1}{1-\theta}}  }{ k(1-\theta) } \right)^{n-m}  = 2\frac{ (k\theta-1)^k}{ (k\theta-\theta)^k}    (q^{n-m})^{1-\delta} \] 
because $k^{ \frac{m}{n-m}} \leq k^{ \frac{\theta}{1-\theta}} $.

Because $\frac{ (k\theta-1)^k}{ (k\theta-\theta)^k}  = O(1)$, this term is $ O \left( (q^{n-m})^{1-\delta} \right) $.

Applying \cref{binomial-bound} similarly, except with $x=1$,
\[2q^{ \frac{n+1-m}{2} } k^{m-1} \binom{mk-k-m+1}{n+1-m}  \leq   2q^{1/2} \frac{k^m}{k}  \frac{  (k\theta-1)^k}{ (1-\theta) (k \theta -k)^{k-1}  }    \left( \frac{  q^{1/2} (k-1) e \theta  }{ 1-\theta } \right)^{n-m}  \] and because \[  \frac{ 2  (k\theta-1)^k q^{1/2} }{ k (1-\theta) (k\theta - k)^{k-1} }  = O(1) ,\] this term is also  $ O \left( (q^{n-m})^{1-\delta} \right) $.\end{proof}

%Since both terms are $O \left(  q^{ (n-m) (1-\delta)} \right)$, our original sum is as well. \end{proof} 

\subsection{Moments of $L$-functions}

For a primitive character $\chi\colon (\mathbb F_q[t]/g)^\times \to \mathbb C^\times$ , define 
\[ L(s, \chi) = \sum_{ \substack{ f \in \mathbb F_q [t]^+ \\ \gcd(f,g)=1}} \chi(f) \abs{f}^{-s} =  \sum_{ \substack{ f \in \mathbb F_q [t]^+ \\ \gcd(f,g)=1}} \chi(f) q^{ - (\deg f) s}. \]

Fix a positive integer $k$. We will prove an estimate for the $k$th moment of $L(s,\chi)$.

First, for $g$ squarefree and $a$ an invertible residue class mod $g$, let
\[ M(g,a) = \sum_{ \substack{ h \in \mathbb F_q[t]^+  \\ h  | g \\ h \neq 1 }} \sum_{ \substack{\chi\colon  (\mathbb F_q[t]/h)^\times \to \mathbb C^\times \\ \textrm{primitive} }} \overline{\chi(a)}  L(s,\chi)^k \prod_{ \substack{ \pi | (g/h) \\ \textrm{prime}}} (1 - \chi(\pi) q^{ -( \deg \pi) s} )^k.  \]
So $M(g,a)$ is a moment that incorporates the imprimitive characters, suitably weighted, but not the trivial character. We will first prove an estimate for $M(g,a)$, and then remove the imprimitive characters by inclusion-exclusion.

\begin{lemma}\label{moments-1} For $g$ a squarefree polynomial of degree $m$, $a$ a polynomial of degree $<m$ coprime to $g$, and $s \in \mathbb C$ with $\operatorname{Re} s \geq 1/2$, we have
\[ \Bigl|M(g,a)   - \phi(g) q^{- (\deg a) s } d_k(a)   \Bigr |\]
\begin{equation}\label{l-1-error} \leq \phi(g)   \left( \frac{1}{2^{k-1}}  +\frac{  q^{\frac{1}{2}} }{k 2^{k-2}  }\right)  q^{ -\frac{m}{2} } k^m 2^{m (k-1)}+  \binom{m+k-1}{k} q^{ \frac{m-1}{2} } \end{equation} 

\end{lemma}

\begin{proof} Note that, for $\chi$ primitive mod $h$,
\[ L(s,\chi) \prod_{ \substack{ \pi | (g/h) \\ \textrm{prime}}} (1 - \chi(\pi) q^{ -( \deg \pi) s} ) = \sum_{ \substack{ f \in \mathbb F_q [t]^+ \\ \gcd(f,g)=1}} \chi(f) q^{ - (\deg f) s} \] because the extra factors cancel the  Euler factors prime to $h$ but not prime to $g$, so
\[   L(s,\chi)^k \prod_{ \substack{ \pi | (g/h) \\ \textrm{prime}}} (1 - \chi(\pi) q^{ -( \deg \pi) s} ) ^k= \sum_{ \substack{ f \in \mathbb F_q [t]^+ \\ \gcd(f,g)=1}} \chi(f) d_k(f) q^{ - (\deg f) s} .\]
Summing, over $h$ dividing $g$ other than $1$, a sum over primitive characters mod $h$ has the effect of summing over all characters except the trivial character. Using the orthogonality of characters identity  \[ \sum_{\chi\colon (\mathbb F_q[t]/g)^\times \to \mathbb C^\times }\overline{\chi(a)}   \chi(f) = \begin{cases} \phi(g) & f\equiv a \mod g \\ 0 & f \not\equiv a\mod g\end{cases}
, \] and subtracting the contribution of the trivial character, we get
\[M(g,a) =  \sum_{ \substack{ h \in \mathbb F_q[t]^+ \\ h  | g \\ h \neq 1 }} \sum_{\substack{ \chi  \colon (\mathbb F_q[t]/h)^\times \to \mathbb C^\times \\ \textrm{primitive}} }  \overline{\chi(a)} L(s,\chi)^k \prod_{ \substack{ \pi | (g/h) \\ \textrm{prime}}} (1 - \chi(\pi) q^{ -( \deg \pi) s} )^k \] \[=   \sum_{n=0}^{\infty}  \Bigl( \sum_{ \substack{ f \in \mathcal M_n \\ f \equiv a\mod g }}  \phi(g) d_k(f)  -  \sum_{ \substack{ f \in  \mathcal M_n\\ \gcd(f,g)=1 }}  d_k(f) \Bigr) q^{-ns }  \]
Using $\operatorname{Re} s\geq 1/2$, and cancelling the $f=a$ term, it suffices to show that
\[\sum_{n=0}^{\infty} \Bigl|   \sum_{ \substack{ f \in \mathcal M_n \\ f \equiv a\mod g  \\ f \neq a }}  \phi(g) d_k(f)  -  \sum_{ \substack{ f \in \mathcal M_n\\ \gcd(f,g)=1 }}  d_k(f)    \Bigr| q^{-n/2}  \] is at most \eqref{l-1-error}.

We handle the $n\geq m$ and $n<m$ terms separately. For $n<m$, there are no monic polynomials of degree $n$ congruent to $a$ mod $g$ other than $a$, so the summand is simply \[\sum_{ \substack{ f \in  \mathcal M_n\\ \gcd(f,g)=1 }}  d_k(f) q^{-n/2}  \leq  \sum_{ f \in  \mathcal M_n }  d_k(f) q^{-n/2}  = \binom{ n + k-1 }{ k-1} q^{n/2} \]  where the sum of $d_k(f)$ is the coefficient of $q^{-ns}$ in $L(1,\chi)^k = \frac{1}{(1-q^{1-s})^k}$ and thus is $\binom{ n + k-1 }{ k-1}  q^n$. Thus the terms with $ n< m$ contribute at most
\[ \sum_{n=0}^{m-1} \binom{ n + k-1 }{ k-1} q^{\frac{n}{2} }  \leq  \sum_{n=0}^{m-1} \binom{ n + k-1 }{ k-1} q^{\frac{m-1}{2}} = \binom{m+k-1}{k} q^{ \frac{m-1}{2} } \] which is the second term of \eqref{l-1-error}.

For $n \geq m$, the summand is exactly $q^{-n/2} \phi(g)$ times the quantity bounded in \cref{divisor-1}, and thus by \cref{divisor-1} the $n\geq m$ terms are at most
\[ \sum_{n=m}^{\infty}  2 \phi(g) q^{\frac{-n}{2} } \left( q^{ \frac{n-m}{2}} k^m  \binom{mk-m-k}{n-m}  + q^{ \frac{n+1-m}{2} } k^{m-1} \binom{mk-k-m+1}{n+1-m}  \right) \]
We have
\[  \sum_{n=m}^{\infty}   q^{\frac{-n}{2} }q^{ \frac{n-m}{2}} k^m  \binom{mk-m-k}{n-m}  =q^{ - \frac{m}{2}}  k^m \sum_{n=m}^{\infty}  \binom{mk-m-k}{n-m}  =  q^{ - \frac{m}{2}}  k^m 2^{mk-m-k} \]
and
\[  \sum_{n=m}^{\infty}    q^{\frac{-n}{2} }q^{ \frac{n+1-m}{2}} k^{m-1}  \binom{mk-m-k+1}{n+1-m}  =  q^{  \frac{1- m}{2}}  k^{m-1}  \sum_{n=m}^{\infty}  \binom{mk-m-k+1}{n+1 -m} \] \[ =  q^{ \frac{1- m}{2}}  k^{m-1} (2^{mk-m-k+1}-1)\leq  q^{ \frac{1- m}{2}}  k^{m-1} 2^{mk-m-k+1} ,\]
while
\[  q^{ - \frac{m}{2}}  k^m 2^{mk-m-k} + q^{ \frac{1- m}{2}}  k^{m-1} 2^{mk-m-k+1} =  \left( \frac{1}{2^k}  +\frac{  q^{\frac{1}{2}} }{k 2^{k-1}}  \right)  q^{ -\frac{m}{2} } k^m 2^{m (k-1)} \] which, adding back the $2\phi(g)$ factor, matches the first term of \eqref{l-1-error}. \end{proof}

Since we will not normally need an estimate as precise as \eqref{moments-1}, we prove a simpler version: 

\begin{lemma}\label{moments-simplified} For $g$ a squarefree polynomial of degree $m$, $a$ a polynomial of degree $<m$ coprime to $g$, and $s \in \mathbb C$ with $\operatorname{Re} s \geq 1/2$,  and $k>1$, we have
\[ \Bigl|M(g,a)   - \phi(g) q^{- (\deg a) s } d_k(a)   \Bigr | = O \left (  \phi(g)  q^{ -\frac{m}{2} } k^m 2^{m (k-1)}\right)  .\]

\end{lemma}

\begin{proof}  By \cref{moments-1}, it suffices to prove that
\begin{equation*}2\phi(g)   \left( \frac{1}{2^k}  +\frac{  q^{\frac{1}{2}} }{k 2^{k-1}  }\right)  q^{ -\frac{m}{2} } k^m 2^{m (k-1)}+  \binom{m+k-1}{k} q^{\frac{m-1}{2}}  =    O \left ( \phi(g)q^{ -\frac{m}{2} } k^m 2^{m (k-1)}\right)  . \end{equation*} 
For the first term, this is obvious since $ \left( \frac{1}{2^k}  +\frac{  q^{\frac{1}{2}} }{k 2^{k-1}  }\right)$, so it remains to prove
\begin{equation*}  \binom{m+k-1}{k}  q^{\frac{m-1}{2}}  = \phi(g)   O \left ( q^{ -\frac{m}{2} } k^m 2^{m (k-1)}\right)  \end{equation*} 
Since  $q^m/\phi(g) =O ( q^{  \epsilon m })$ and $\binom{m+k-1}{k}  = O ( q^{\epsilon m})$ for any $\epsilon >0$, it suffices to take $\epsilon$ small enough that $q^{ \epsilon} \leq k 2^{k-1} $, which is possible as long as $k>1$. \end{proof}

For $f$ a monic polynomial, let \[ \mu_k(f) = \sum_{\substack{f_1,\dots, f_k \in \mathbb F_q[t]^+ \\ f=\prod_{i=1}^k f_i}} \prod_{i=1}^k \mu(f) .\]

For $v$ an invertible polynomial mod $g'$, let $\overline{v}$ denote the inverse of $v$ mod $g'$, or, more precisely, the unique polynomial of degree $< \deg g'$ congruent to the inverse of $v$ mod $g'$.

The inclusion-exclusion that relates the moments of $L(s,\chi)$ to $M(g,a)$ is as follows. 

\begin{lemma}\label{moments-sieve} For $g$ a squarefree polynomial of degree $m$, we have

\begin{equation}\label{moments-primitive-reduction} 
\sum_{\substack{ \chi \colon (\mathbb F_q[t]/g)^\times \to \mathbb C^\times \\ \textrm{primitive} }} L(s,\chi)^k =  \sum_{\substack{ g' \in \mathbb F_q[t] \\ g' | g}} \mu (g/g')  \sum_{\substack{v \in \mathbb F_q[t]^+ \\ v | (g/g')^k }} \mu_k(v)  q^{ - (\deg v) s}  M( g', \overline{v} ). \end{equation} \end{lemma}

\begin{proof}By M\"{o}bius inversion,

\begin{align*} 
& \sum_{\substack{ \chi \colon (\mathbb F_q[t]/g)^\times \to \mathbb C^\times \\ \textrm{primitive} }} L(s,\chi)^k   =  \sum_{\substack{ g' \in \mathbb F_q[t] \\ g' | g}} \mu (g/g') \sum_{ \substack{ h \in \mathbb F_q^+  \\ h  | g' \\ h \neq 1 }} \sum_{ \substack{ \chi \colon (\mathbb F_q[t]/h)^\times \to \mathbb C^\times \\ \textrm{primitive} }  } L(s,\chi)^k \prod_{ \substack{ \pi | (g/h) \\ \textrm{prime}}} (1 - \chi(\pi) q^{ -( \deg \pi) s} )^k \\
= &  \sum_{\substack{ g' \in \mathbb F_q[t] \\ g' | g}} \mu (g/g') \sum_{ \substack{ h \in \mathbb F_q^+  \\ h  | g' \\ h \neq 1 }} \sum_{\substack{ \chi \colon (\mathbb F_q[t]/h)^\times \to \mathbb C^\times \\ \textrm{primitive} }}  L(s,\chi)^k \prod_{ \substack{ \pi | (g'/h) \\ \textrm{prime}}} (1 - \chi(\pi) q^{ -( \deg \pi) s} )^k  \prod_{ \substack{ \pi | (g/g') \\ \textrm{prime}}} (1 - \chi(\pi) q^{ -( \deg \pi) s} )^k \\ 
=& \sum_{\substack{ g' \in \mathbb F_q[t] \\ g' | g}} \mu (g/g')  \sum_{ \substack{ h \in \mathbb F_q^+  \\ h  | g' \\ h \neq 1 } } \sum_{ \substack{\chi \colon (\mathbb F_q[t]/h)^\times \to \mathbb C^\times \\ \textrm{primitive} }}   L(s,\chi)^k  \prod_{ \substack{ \pi | (g'/h) \\ \textrm{prime}}} (1 - \chi(\pi) q^{ -( \deg \pi) s} )^k   \sum_{\substack{v \in \mathbb F_q[t]^+ \\ v | (g/g')^k }} \mu_k(v) \chi(v) q^{ - (\deg v) s}  \\ 
= & \sum_{\substack{ g' \in \mathbb F_q[t] \\ g' | g}} \mu (g/g')  \sum_{\substack{v \in \mathbb F_q[t]^+ \\ v | (g/g')^k }} \mu_k(v)  q^{ - (\deg v) s}  M( g', \overline{v} ) . \qedhere
\end{align*}  \end{proof}

\begin{proposition}\label{moments-2} For $g$ a squarefree polynomial of degree $m$,  $s \in \mathbb C$ with $\operatorname{Re} s \geq 1/2$, and $\epsilon>0$, we have
\begin{equation}\label{l-2-main} \Biggl| \frac{1}{\phi^*(g) } \sum_{\substack{ \chi \colon (\mathbb F_q[t]/g)^\times \to \mathbb C^\times \\ \textrm{primitive} }} L(s,\chi)^k   - 1   \Biggr |  \leq O  \left( |g|^{  \frac{\log \small( k 2^{k-1}\small) }{ \log q} - \frac{1}{2} + \epsilon }  \right)\end{equation}
where the implicit constant in the big $O$ depends only on $q,k, \epsilon$ and $\phi^* (g)$ is the number of primitive characters mod $g$. \end{proposition}

\begin{proof} 
We apply \cref{moments-simplified} to each $M(g', v^{-1} \mod g')$ appearing in \eqref{moments-primitive-reduction}, and group like terms together. From this, we obtain

\begin{equation}\label{moments-reduction-expansion} \Bigl| \sum_{\substack{ \chi \colon (\mathbb F_q[t]/g)^\times \to \mathbb C^\times \\ \textrm{primitive} }} L(s,\chi)^k  - A \Bigr| \ll B \end{equation}
where $A$ and $B$ are as follows:

\begin{equation}\label{moments-main-term} A =  \sum_{\substack{ g' \in \mathbb F_q[t]^+ \\ g' | g}} \mu (g/g')  \phi(g') \sum_{\substack{v \in \mathbb F_q[t]^+ \\ v | (g/g')^k }}\mu_k(v)  q^{ - (\deg v) s - (\deg \overline{v}) s } d_k (\overline{v}) \end{equation} 
\begin{equation}\label{moments-big-error-term}B  =  \sum_{\substack{ g' \in \mathbb F_q[t]^+ \\ g' | g}}\phi(g') \sum_{\substack{v \in \mathbb F_q[t]^+ \\ v | (g/g')^k }} \left| \mu_k(v) \right|  \left| q^{ - (\deg v) s } \right|    q^{ -\frac{ \deg g' }{2} } k^{\deg g'}  2^{ (\deg g') (k-1)} \end{equation}
%\begin{equation}\label{moments-small-error-term}  \sum_{\substack{ g' \in \mathbb F_q[t]^+ \\ g' | g}}\sum_{\substack{v \in \mathbb F_q[t]^+ \\ v | (g/g')^k }}  \left| \mu_k(v) \right|   \left|q^{ - (\deg v) s } \right| \binom{(\deg g') +k-1}{k} ( 1 + q^{-1/2})^{ (\deg g') k}q^{\frac{(\deg g') -1}{2}}  \end{equation}

We now evaluate these two terms.

We divide $A$ into the terms where $v=1$, which will give the main term, and $v \neq 1$, which will give an additional error term. For $v=1$, we observe that $\mu_k(v)=1$, $\overline{v}=1$, and $d_k(\overline{v})=1$, so 
\[\sum_{\substack{ g' \in \mathbb F_q[t]^+ \\ g' | g}} \mu (g/g')  \phi(g') \sum_{\substack{v \in \mathbb F_q[t]^+ \\ v | (g/g')^k \\ v=1 }}\mu_k(v)  q^{ - (\deg v) s - (\deg \overline{v}) s } d_k (\overline{v}) = \sum_{\substack{ g' \in \mathbb F_q[t]^+ \\ g' | g}} \mu (g/g')  \phi(g')  =\prod_{ \substack{\pi | g\\ \textrm{prime}} } (q^{\deg g} - 2) \]
and the $v=1$ contribution will cancel the main term $1$ in \eqref{l-2-main}.  For the $v\neq 1$ terms, we observe that $v \overline{v}$ is congruent to $1$ mod $g'$ but not equal to $1$, so $\deg v + \deg{\overline{v}} = \deg (v\overline{v}) \geq \deg g'$. Thus 
\[\Bigl| \sum_{\substack{ g' \in \mathbb F_q[t]^+ \\ g' | g}} \mu (g/g')  \phi(g') \sum_{\substack{v \in \mathbb F_q[t]^+ \\ v | (g/g')^k \\ v\neq 1 }}\mu_k(v)  q^{ - (\deg v) s - (\deg \overline{v}) s } d_k (\overline{v})\Bigr | \leq \sum_{\substack{ g' \in \mathbb F_q[t]^+ \\ g' | g}} \phi(g') \sum_{\substack{v \in \mathbb F_q[t]^+ \\ v | (g/g')^k \\ v\neq 1 }} \left| \mu_k(v) \right| q^{ - \deg g' /2} d_k ( \overline{v} ) \]
Now \[ d_k ( \overline{v}) \leq k^ { \omega(\overline{v}) } \leq k^{ \deg \overline{v}  } \leq k^{\deg g' } \] and \[\sum_{\substack{v \in \mathbb F_q[t]^+ \\ v | (g/g')^k \\ v\neq 1 }} \left| \mu_k(v) \right|  = 2^{k \omega(g/g') }\] so
\begin{equation}\label{moments-main-term-evaluation} \begin{split} & \left | A - \prod_{ \substack{\pi | g\\ \textrm{prime}} } (q^{\deg g} - 2) \right| \leq   
 \sum_{\substack{ g' \in \mathbb F_q[t]^+ \\ g' | g}} \phi(g') \sum_{\substack{v \in \mathbb F_q[t]^+ \\ v | (g/g')^k \\ v\neq 1 }} \left| \mu_k(v) \right| q^{ - \deg g' /2} k^{ \deg g' }   \\ 
 =  & \sum_{\substack{ g' \in \mathbb F_q[t]^+ \\ g' | g}} \phi(g') 2^{ k\omega (g/g')}  q^{ - \deg g' /2} k^{ \deg g'}  = \prod_{\substack{\pi|g \\ \textrm{prime}}}  \left(  k^{\deg \pi} \frac{q^{\deg \pi}-1}{q^{ \deg \pi/2} } + 2^k \right)  \leq   \prod_{\substack{\pi|g \\ \textrm{prime}} } \left(  \left( kq^{\frac{1}{2}}\right)^{\deg \pi}  + 2^k \right) \end{split} \end{equation} 
 
For any $\epsilon>0$,  since $ \left( kq^{\frac{1}{2}}\right)^{\deg \pi}  + 2^k \leq  \left( kq^{\frac{1}{2}+\epsilon }\right)^{\deg \pi} $ for all but finitely many primes $\pi$, we have \[ \prod_{\substack{\pi|g \\ \textrm{prime}} } \left(  \left( kq^{\frac{1}{2}}\right)^{\deg \pi}  + 2^k \right) \ll  \left( k q^{\frac{1}{2} + \epsilon }\right)^m .\] 

We now evaluate $B$.
\[ B =  \sum_{\substack{ g' \in \mathbb F_q[t]^+ \\ g' | g}} \phi(g')  \Bigl( \prod_{\pi | (g/g') } \left(1 + \left |q^{ - (\deg \pi) s} \right| \right)^k \Bigr) \Bigl( \prod_{\pi | g'}   ( q^{-1/2} k 2^{(k-1)})^{\deg \pi} \Bigr)  \]
\[=   \sum_{\substack{ g' \in \mathbb F_q[t]^+ \\ g' | g}}   \Bigl( \prod_{\pi | (g/g') } \left(1 + \left |q^{ - (\deg \pi) s}\right |\right )^k \Bigr) \Bigl( \prod_{\pi | g'}   ( q^{-1/2} k 2^{(k-1)})^{\deg \pi} (q^{\deg \pi }- 1) \Bigr)  \] 
\[= \prod_{\pi | g} \left( \left(1 + \left |q^{ - (\deg \pi) s} \right|\right )^k  +  ( q^{-1/2} k 2^{(k-1)})^{\deg \pi} (q^{\deg \pi }- 1) \right) \]
\[ \leq   \prod_{\pi | g} \left( (1 +q^{-\deg \pi /2} )^k  +  ( q^{1/2} k 2^{k-1})^{\deg \pi} \right)\]
\[ \ll  ( q^{\frac{1}{2}+ \epsilon } k 2^{k-1})^m  \] since, for any $\epsilon>0$, we have  $(1 +q^{-\deg \pi /2} )^k  +  ( q^{1/2} k 2^{k-1})^{\deg \pi}  <   ( q^{\frac{1}{2} + \epsilon } k 2^{k-1})^{\deg \pi} $ for all but finitely many primes $\pi$.

%Finally, we evaluate \eqref{moments-small-error-term}. Using $\binom{(\deg g') +k-1}{k} \leq \binom{m+k-1}{m} $ , we get
%\[ \eqref{moments-small-error-term}\leq \binom{m+k-1}{m} q^{- \frac{1}{2}}  \sum_{\substack{ g' \in \mathbb F_q[t]^+ \\ g' | g}} \Bigl( \prod_{\pi | (g/g') } \left(1 + \left |q^{ - (\deg \pi) s}\right |\right )^k \Bigr)   ( 1 + q^{-1/2})^{ (\deg g') k}q^{\frac{\deg g' }{2}} \]
%\[ =  \binom{m+k-1}{m} q^{- \frac{1}{2}} \prod_{ \pi | g}  \left( ( 1 + q^{-1/2})^{ (\deg \pi ) k}q^{\frac{ \deg \pi  }{2}}  + \left(1 + \left |q^{ - (\deg \pi) s}\right |\right )^k \right) \]
%\[ \leq  \binom{m+k-1}{m} q^{- \frac{1}{2}} \prod_{ \pi | g}  \left(  \left( ( 1 + q^{-1/2})^{ k}q^{\frac{ 1 }{2}} \right)^{ \deg \pi} + \left(1 + q^{ - \frac{1}{2} }\right )^k \right) \]
%\[  \leq  \binom{m+k-1}{m} q^{- \frac{1}{2}} \prod_{ \pi | g}  \left(  \left( 2 ( 1 + q^{-1/2})^{ k}q^{\frac{ 1 }{2}} \right)^{ \deg \pi}  \right) \ll  \left(  \left(  ( 1 + q^{-1/2})^{ k}q^{\frac{ 1 }{2}+\epsilon} \right)^{ m}  \right).\]
%since $2^{ \omega(g) } \ll |g|^\epsilon = q^{\epsilon m}$ and $\binom{m+k-1}{m} \ll q^{ \epsilon m}$. 
%
Plugging these into \eqref{moments-reduction-expansion}, we get
\[ \Bigl| \sum_{\substack{ \chi \colon (\mathbb F_q[t]/g)^\times \to \mathbb C^\times \\ \textrm{primitive} }} L(s,\chi)^k  - \prod_{ \substack { \pi \mid g \\ \textrm{prime}} } (q^{\deg \pi} -2 )  \Bigr|  \]
\[\ll  \max \left(  k q^{\frac{1}{2} + \epsilon }, q^{\frac{1}{2}+ \epsilon } k 2^{k-1}\right)^m = \left( q^{\frac{1}{2}+ \epsilon } k 2^{k-1} \right)^m \]
Now dividing by $\phi^*(g) = \prod_{ \substack { \pi \mid g \\ \textrm{prime}}  } (q^{\deg \pi} -2 ) $ and using $\prod_{ \substack { \pi \mid g \\ \textrm{prime}}  } (q^{\deg \pi} -2 )  \gg q^{(1- \epsilon)m }$, we obtain 

\[ \Bigl| \frac{1}{  \phi^*(g)  } \sum_{\substack{ \chi \colon (\mathbb F_q[t]/g)^\times \to \mathbb C^\times \\ \textrm{primitive} }} L(s,\chi)^k  -   1 \Bigr| \ll  \left( q^{-\frac{1}{2}+ \epsilon } k 2^{k-1} \right)^m = O \left( |g|^{ \frac{\log \small( k 2^{k-1}\small)  }{\log q} - \frac{1}{2} + \epsilon } \right) . \qedhere    \]
\end{proof}

\subsection{The von Mangoldt function}

%Before studying the von Mangoldt function, we prove some preparatory lemmas.
Let $\operatorname{std}$ be the $n-1$-dimensional standard representation of $S_n$ acting on vectors in $\mathbb C^n$ whose entries sum to $0$.

\begin{lemma}\label{mangoldt-alternating}  For a monic polynomial $f$ of degree $n$, we have  \begin{equation}\label{eq-mangoldt-alternating} \Lambda(f) =   \sum_{i=0}^{n-1}  (-1 )^i F_{ \wedge^i \operatorname{std}} (f).\end{equation}
\end{lemma}
\begin{proof} This is \cite[Lemma 3.6]{SawinShort}. \end{proof}

For this subsection, let $\rho= \bigoplus_{i=0}^{n-1}  \wedge^i \operatorname{std}$. This representation is relevant to us because we will bound trivially the sum over the $i$ variable in \eqref{eq-mangoldt-alternating} and thus can ignore the signs $(-1)^i$. We begin by calculating the characters of $V_\rho$ and $V^d_\rho$ (which will also be used in Section \ref{s-anatomy}).

\begin{lemma}\label{mangoldt-V-trace}
We have 
\begin{equation}\label{vm-V-gf} \tr ( \Diag(\lambda_1,\dots, \lambda_{m-1}, V_ \rho)  )= \frac{1}{2}  \prod_{i=1}^{ m-1} \frac{ 1 + u \lambda_i}{1-u \lambda_i} [u^n] \end{equation}
and
 \begin{equation}\label{vm-Vd-gf}  \tr (\Diag (\lambda_1,\dots, \lambda_m) , V_{\rho}^d ) = \frac{1}{2}  \frac{1 + uv}{ 1-uv}  \prod_{i=1}^m \frac{ 1+u \lambda _i  }{ 1- u\lambda_i } [ u^{n}] [ v^d] . \end{equation}  \end{lemma}
 \begin{proof} 
 
 Let \begin{equation}\label{rho-prime-def} \rho' =  \bigoplus_{k=0}^n \Ind^{S_n}_{S_k \times S_{n-k}} \sgn_{S_k} .\end{equation} Next we can check that \begin{equation}\label{rho-prime-rho} \rho' = \bigoplus_{k=0}^n \wedge^k ( \operatorname{std} \oplus \mathbb C)  = \rho \oplus \rho .\end{equation} First view $\operatorname{std} \oplus \mathbb C$ as a representation with basis vectors $v_1,\dots, v_n$, so that $\wedge^k ( \operatorname{std} \oplus \mathbb C)$ is a representation with basis consisting of $v_{i_1} \wedge \dots \wedge v_{i_k}$ for tuples $i_1 < \dots <i_k$. Observe that the one-dimensional subspace generated by $v_1\wedge \dots \wedge v_k$ is a $S_k \times S_{n-k}$-subrepresentation isomorphic to $\sgn_{S_k}$, whose translates under $S_n/(S_k\times S_{n-k})$ are the one-dimensional subspaces generated by the basis vectors, so that $\wedge^k ( \operatorname{std} \oplus \mathbb C)$ is the sum of these translates. By definition, this means $\wedge^k ( \operatorname{std} \oplus \mathbb C)$ is the induced representation $\Ind^{S_n}_{S_k \times S_{n-k}} \sgn_{S_k} $. For the second equality of \eqref{rho-prime-rho}, observe $\wedge^k ( \operatorname{std} \oplus \mathbb C)  = \wedge^k(\operatorname{std} )\oplus \wedge^{k-1}(\operatorname{std})$ and sum both terms over $k$ from $0$ to $n$. This establishes \eqref{rho-prime-rho} and thus \begin{equation}\label{V-half} \tr ( \Diag(\lambda_1,\dots, \lambda_{m-1}, V_ \rho)  )= \frac{1}{2} \tr ( \Diag(\lambda_1,\dots, \lambda_{m-1}, V_ {\rho'} )  ) \end{equation} and \begin{equation}\label{Vd-half} \tr ( \Diag(\lambda_1,\dots, \lambda_{m}, V^d_ \rho)  )= \frac{1}{2} \tr ( \Diag(\lambda_1,\dots, \lambda_{m}, V^d_ {\rho'} )  ).\end{equation}
 
 By the definition \eqref{rho-prime-def} of $\rho'$ and Frobenius reciprocity, we have \[ V_{\rho'} =  \bigoplus_{k=0}^n\Bigl ( ( \mathbb C^{m-1})^{\otimes n} \otimes \Ind^{S_n}_{S_k \times S_{n-k}} \sgn_{S_k} \otimes \sgn\Bigr)^{S_n} =  \bigoplus_{k=0}^n\Bigl ( ( \mathbb C^{m-1})^{\otimes n} \otimes \Ind^{S_n}_{S_k \times S_{n-k}} \sgn_{S_{n-k}} \Bigr)^{S_n} \] \[= \bigoplus_{k=0}^n  \Bigl ( ( \mathbb C^{m-1})^{\otimes n}\otimes \sgn_{S_{n-k}} )^{S_k \times S_{n-k}} = \bigoplus_{k=0}^n  \wedge^k (\mathbb C^{m-1}) \otimes \Sym^{n-k} (\mathbb C^{m-1})\] so \begin{equation}\label{V-rho'}  \tr ( \Diag(\lambda_1,\dots, \lambda_{m-1}, V_ \rho' )  )= \prod_{i=1}^{ m-1} \frac{ 1 + u \lambda_i}{1-u \lambda_i} [u^n] . \end{equation}
 
 Combining \eqref{V-rho'} and \eqref{V-half}, we obtain \eqref{vm-V-gf}. We now apply \cref{induction-restriction}, taking $n_1=k, n_2=n-k, \rho_1=\sign, \rho_2 = \mathbb C$, to obtain that  the restriction of $\Ind_{S_k \times S_{n-k}} \sgn_{S_k} $ to $S_d \times S_{n-e}$ is  
 \[  \bigoplus_{e =\max(0, k+d-n ) }^{ \min(k ,d) }   \operatorname{Ind}_{ S_e \times S_{d-e} \times S_{k-e} \times  S_{n-k-d+e}  }^{S_d\times S_{n-d}}  (\sgn_{S_e} \otimes \sgn_{S_{k-e}} )\]
  \[ = \bigoplus_{e =\max(0,k+d-n) }^{ \min(k,d) }  \Bigl(  \operatorname{Ind}_{ S_e \times S_{d-e}}^{S_d} \sgn_{S_e}  \Bigr) \otimes \Bigl( \operatorname{Ind}_{S_{k-e} \times S_{n-d-k+e}}^{S_{n-d}} \sgn_{S_{k-e}} \Bigr) . \] 

This gives \[ \Bigl( ( \mathbb C^{m} )^{\otimes n-d} \otimes \Ind_{S_k \times S_{n-k}} \sgn_{S_k}  \otimes \sgn_{S_{n-d}} \Bigr)^{S_d\times S_{n-d}}\] \[ =\bigoplus_{e =\max(0,k+d-n) }^{ \min(k,d) } \Bigl(  \operatorname{Ind}_{ S_e \times S_{d-e}}^{S_d} \sgn_{S_e}  \Bigr)^{S_d} \otimes   \Bigl( ( \mathbb C^{m} )^{\otimes n-d} \otimes \bigl(    \operatorname{Ind}_{S_{k-e} \times S_{n-d-k+e}}^{S_{n-d}} \sgn_{S_{k-e}}  \bigr) \otimes \sgn_{S_{n-d}} \Bigr) ^{S_{n-d}}\]
\[ = \bigoplus_{e =\max(0,k+d-n) }^{ \min(k,d) } \Bigl(  \operatorname{Ind}_{ S_e \times S_{d-e}}^{S_d} \sgn_{S_e}  \Bigr)^{S_d}  \otimes \Bigl(     \operatorname{Ind}_{S_{k-e} \times S_{n-d-k+e}}^{S_{n-d}} \bigl(  (\mathbb C^m)^{ k-e} \otimes( \mathbb C^m)^{n-d-k+e} \otimes \sgn_{S_{n-d-k+e}} \bigr)  \Bigr)^{S_{n-d}}\]
\[ = \bigoplus_{e =\max(0,k+d-n) }^{ \min(k,d) }   \mathbb C \otimes \sgn_{S_e}^{S_e} \otimes  \Sym^{k-e} \mathbb C^m \otimes \wedge^{ n-d-k+e} \mathbb C^m \]
\[ = \begin{cases} \Sym^k \mathbb C^m \otimes \wedge^{n-d-k} \mathbb C^m  \oplus \Sym^{k-1} \mathbb C^m \otimes \wedge^{n-d-k+1} \mathbb C^m  & \textrm{if } d>0 \\ 
\Sym^k \mathbb C^m \otimes \wedge^{n-d-k} \mathbb C^m   & \textrm{ if }d=0 \end{cases}
\]
since $ \sgn_{S_e}^{S_e}=0$ unless $e=0$ or $e=1$, with symmetric powers or wedge powers interpreted to vanish if the power is negative. 

Thus  \[ \Bigl( ( \mathbb C^{m} )^{\otimes n-d} \otimes  \rho'   \otimes \sgn_{S_{n-d}} \Bigr)^{S_d\times S_{n-d}}  \] 
\[ = \begin{cases} \Bigl (  \bigoplus_{k=0}^{n-d} \Sym^k \mathbb C^m \otimes \wedge^{n-d-k} \mathbb C^m \Bigr)  \oplus  \Bigl (  \bigoplus_{k=1}^{n-d+1} \Sym^{k-1} \mathbb C^m \otimes \wedge^{n-d-k+1} \mathbb C^m \Bigr) & \textrm{if }d>0\\
 \bigoplus_{k=0}^{n-d} \Sym^k \mathbb C^m \otimes \wedge^{n-d-k} \mathbb C^m  & \textrm{if }d=0\end{cases} \]

Note that both summands in the $d>0$ case are equal. Hence

\begin{equation}\label{Vd-rho'} \tr (\Diag (\lambda_1,\dots, \lambda_m) , V_{\rho'}^d ) = (1 + 1_{d>0})  \prod_{i=1}^m \frac{ 1+u \lambda _i  }{ 1- u\lambda_i } [ u^{n-d} ]  = \frac{1 + uv}{ 1-uv}  \prod_{i=1}^m \frac{ 1+u \lambda _i  }{ 1- u\lambda_i } [ u^{n}] [ v^d] . \end{equation}

  Combining \eqref{Vd-rho'} and \eqref{Vd-half}, we obtain \eqref{vm-Vd-gf}.\end{proof}

\begin{lemma}\label{primes-1} For $g$ squarefree of degree $m \geq 2$, a natural number $n \geq m$, and a residue class $a \in (\mathbb F_q[t]/g)^\times$, we have \[\left | \sum_{ \substack{ f \in \mathcal M_n \\  f \equiv a \mod g }}  \Lambda(f)  -  \frac{1}{ \phi(g) }  \sum_{ \substack{ f \in\mathcal M_n\\ \gcd(f,g)=1  }}  \Lambda(f)  \right| \leq  \sum_{r=0}^{m-2} 2^{m-1-r}  \binom{ n+m-r }{ 2m-r}  q^{ \frac{n-m}{2}} + 2^{m-1} {n + m \choose 2m-1}  q^{\frac{n+1-m}{2}}  . \] \end{lemma}

\begin{proof}
Using \cref{mangoldt-alternating} and \cref{virtual-representation-variant} taking $d=n, \rho_i = \wedge^{i-1} \operatorname{std}, \alpha_i = (-1)^{i-1}$, we have
\[ \Bigl| \sum_{ \substack{ f \in \mathcal M_n \\  f \equiv a \mod g }}\Lambda(f)  -\frac{1}{ \phi(g) }  \sum_{ \substack{ f \in\mathcal M_n\\ \gcd(f,g)=1  }}  \Lambda(f) \Bigr|
 \leq   2 \Bigl( C_1\Bigr( \bigoplus_{i=0}^{n-1} \wedge^i \operatorname{std}\Bigr)  + C_2\Bigl(\bigoplus_{i=0}^{n-1} \wedge^i\operatorname{std} \Bigr) \sqrt{q} \Bigr) q^{\frac{n-m}{2} }\]
\[ =2  ( C_1(\rho )  + C_2(\rho) \sqrt{q} )q^{\frac{n-m}{2} }   . \]   

%Let \[ \rho' =  \sum_{k=0}^n \Ind_{S_k \times S_{n-k}} \sgn_{S_k} .\] Then we can check that \[ \rho' = \bigoplus_{i=0}^n \wedge^i ( \operatorname{std} \oplus \mathbb C)  = \rho \oplus \rho \] so \[ C_1 ( \rho') =2 C_1(\rho)  \hspace{.2in} \textrm{ and }\hspace{.2in}  C_2(\rho') =2 C_2(\rho). \] 

From \eqref{vm-V-gf}, we obtain 
\[2 C_2(\rho) = \frac{\partial^{m-1}}{ \partial \lambda_1\dots  \partial \lambda_{m-1} }\prod_{i=1}^{m-1} \frac{ 1 + u \lambda_i}{1-u \lambda_i} [u^n] \Big|_{\lambda_1,\dots,\lambda_{m-1}=1} \] \[= \prod_{i=1}^{m-1} \frac{ u (1-u \lambda_i) + u (1 + u \lambda_i) } { (1-u \lambda_i)^2} [u^n] \Big|_{\lambda_1,\dots,\lambda_{m-1}=1} =\prod_{i=1}^{m-1} \frac{ 2u}{ (1-u\lambda_i)^2 } [u^n] \Big|_{\lambda_1,\dots,\lambda_{m-1}=1} \]
\[= \frac{ 2^{m-1} u^{m-1} }{ (1-u)^{2m-2}} [u^n]  = 2^{m-1} \binom{n+m}{ 2m-1}. \] 

From \eqref{vm-Vd-gf} we obtain \[\sum_{d=0}^n (-1)^d\tr (\Diag (\lambda_1,\dots, \lambda_m) , V_{\rho}^d ) =\frac{1}{2}\frac{1 + uv}{ 1-uv}  \prod_{i=1}^m \frac{ 1+u \lambda _i  }{ 1- u\lambda_i } [ u^{n}] \Big|_{v=-1} =   \frac{1-u}{1+u} \prod_{i=1}^m \frac{ 1+u \lambda _i  }{ 1- u\lambda_i } [ u^n ] \] and \[  \frac{\partial^{m}}{ \partial \lambda_1\dots  \partial \lambda_{m} } \sum_{d=0}^n (-1)^d\tr (\Diag (\lambda_1,\dots, \lambda_m) , V_{\rho}^d )  \Big|_{\lambda_1,\dots,\lambda_{m}=1}=\frac{\partial^{m}}{ \partial \lambda_1\dots  \partial \lambda_{m} }  \frac{1}{2} \frac{1-u}{1+u} \prod_{i=1}^m \frac{ 1+u \lambda _i  }{ 1- u\lambda_i } [ u^n ] \Big|_{\lambda_1,\dots,\lambda_{m}=1} \]
\[ =\frac{1}{2} \frac{1-u}{1+u} \prod_{i=1}^m  \frac{ 2 u }{ (1- u \lambda_i)^2 }  [u^n]\Big|_{\lambda_1,\dots,\lambda_{m}=1} = \frac{ 2^{m-1} u^m }{ (1+u) (1-u)^{2m-1} } [ u^n] .\]

From \eqref{vm-Vd-gf} we also obtain

\[\sum_{d=m}^n (-1)^{d-m} \binom{d}{m} \dim V_{\rho}^d    =  \frac{1}{m!} \frac{\partial^m}{\partial v^m}  \frac{1}{2}  \frac{1 + uv}{ 1-uv}  \prod_{i=1}^m \frac{ 1+u \lambda _i  }{ 1- u\lambda_i } [ u^{n}] \Big|_{\lambda_1,\dots,\lambda_m=1, v=-1} \]
\[ =     \frac{u^{m} }{ (1-uv)^{m+1}}  \prod_{i=1}^m \frac{ 1+u \lambda _i  }{ 1- u\lambda_i } [ u^{n}] \Big|_{\lambda_1,\dots,\lambda_m=1, v=-1} \]
\[ = \frac{ u^m }{ (1+u)^{m+1}} \frac{ (1+u)^m }{ (1-u)^m} [u^n] = \frac{ u^m}{ (1+u) (1-u)^m } [u^n] .\]
Thus \[ 2C_1(\rho) = \left(  \frac{ 2^m u^m }{ (1+u) (1-u)^{2m-1} } - \frac{2 u^m}{ (1+u) (1-u)^m } \right) [u^n] .\]

%whereas  for $d\geq m>0$, \[ \dim V_{\rho'}^d = 2 \prod_{i=1}^m \frac{ 1+u   }{ 1- u } [ u^{n-d} ]=2 \frac{ (1+u)^m}{(1-u)^m} [u^{n-d} ] \] so
%
%\[ \sum_{d=m}^n (-1)^{d-m} \binom{d}{m} \dim V_{\rho}^d  =2 \sum_{d=m}^n (-1)^{d-m} \binom{d}{m} \frac{ (1+u)^m}{(1-u)^m} [u^{n-d} ] = 2 \frac{u^m }{ (1+u)^{m+1}}  \frac{ (1+u)^m}{(1-u)^m} [ u^{n}] \] \[= 2 \frac{ u^m}{ (1+u) (1-u)^m } [u^n] \] 
%so \[ C_1(\rho') = \left(  \frac{ 2^m u^m }{ (1+u) (1-u)^{2m-1} } - \frac{2 u^m}{ (1+u) (1-u)^m } \right) [u^n] .\]

We have \[   \frac{ 2^m u^m }{ (1+u) (1-u)^{2m-1} } - \frac{ 2u^m}{ (1+u) (1-u)^m }= \frac{2u^m}{ (1-u)^{2m-1} }   \frac{  2^{m-1} - (1-u)^{m-1}  }{ 1+u}  \] \[ =  \frac{2u^m}{ (1-u)^{2m-1} }  \sum_{r=0}^{m-2} (1-u)^r 2^{m-2-r}  = \sum_{r=0}^{m-2} \frac{ 2^{m-1-r}  u^m  }{ (1-u)^{2m-1-r} } \] so \[2 C_1(\rho) =\sum_{r=0}^{m-2} 2^{m-1-r}  \binom{ n+m-r }{ 2m-r}. \qedhere \]  \end{proof}
 
 \begin{proposition}\label{primes-2} Let $0<\theta<1 $ be a real number and let \[ q >  4^{ \frac{1}{1-\theta}  }  \left( \frac{ 1}{ 1- \theta} - \frac{1}{2} \right)^2e^2 \] be a prime power. For any natural numbers $n,m$ with $m \leq \theta n $,  squarefree $g \in \mathcal M_m$, and $a \in (\mathbb F_q[t]/g)^\times $,
\[ \Bigl | \sum_{ \substack{ f \in \mathcal M_n \\  f \equiv a \mod g }}  \Lambda(f)  -  \frac{q^n}{ \phi(g) }  \Bigr|  =  O \left( (q^{n-m})^{1-\delta} \right) \]
where  \[\delta =  \frac{1}{2} - \log_q  \left(  \frac{ 2^{ \frac{\theta}{1-\theta } } e (1+\theta)  }{ 1-\theta} \right)> 0 \]
 and the implicit constant depends only on $q, \theta$.

\end{proposition}

\begin{proof} 

By the triangle inequality and \cref{primes-1}, the left side is at most  \[\frac{ 1}{\phi(g)} \Bigl( q^n - \sum_{ \substack{ f \in\mathcal M_n  \\ \gcd(f,g)=1}}  \Lambda(f) \Bigr) + \sum_{r=0}^{m-2} 2^{m-1-r}  \binom{ n+m-r }{ 2m-r}  q^{ \frac{n-m}{2}} + 2^{m-1} {n + m \choose 2m-1}  q^{\frac{n+1-m}{2}}   .  \]

Since $\sum_{f \in \mathcal M_n} \Lambda(f) = q^n$ by a zeta function identity, the first term is $\frac{ 1}{\phi(g)} \sum_{\substack{ f \in \mathcal M_n \\ \gcd(f,g)\neq 1}} \Lambda(f)$. Because each prime factor $\pi$ of $g$ has at most one prime power that contributes to the sum, and contributes at most $\deg \pi$, the sum is at most $\deg g$, and so the first term is $\frac{\deg g}{\phi(g)} \ll 1 \ll (q^{n-m})^{1-\delta}$.

We apply \cref{binomial-bound} to the remaining two. For the simpler term, $ 2^{m-1} {n + m \choose 2m-1}  q^{\frac{n+1-m}{2}} $, we take $a = n-m$, $b=2m$, $\alpha = \frac{1-\theta}{2\theta} $, $x=1$, $y=-1$ and obtain 
\[ 2^{m-1} {n + m \choose 2m-1}  q^{\frac{n+1-m}{2}}  \leq  2^{m-1}  \frac{ 2\theta \sqrt{q} }{1-\theta} \left(  \frac{ e (1+\theta)  \sqrt{q}  }{ (1-\theta)} \right)^{n-m} \] \[=  \frac{ \theta \sqrt{q} }{ 1-\theta} \left(  \frac{  2^{ \frac{m}{ n-m}} e (1+\theta)  \sqrt{q}  }{ (1-\theta)} \right)^{n-m}\leq  \frac{ \theta \sqrt{q} }{ 1-\theta} \left(  \frac{  2^{ \frac{\theta}{ 1-\theta }} e (1+\theta)  \sqrt{q}  }{ (1-\theta)} \right)^{n-m} \]

To bound $2^{m-1-r}  \binom{ n+m-r }{ 2m-r}  q^{ \frac{n-m}{2}} $ we take $a=n-m$, $b=2m$, $\alpha = \frac{1- \theta}{2\theta}$, $x=0$, $y = -r$ and obtain
\[ \sum_{r=0}^{m-2} 2^{m-1-r}  \binom{ n+m-r }{ 2m-r}  q^{ \frac{n-m}{2}}  \leq \sum_{r=0}^{m-2} 2^{m-1-r}  \left( \frac{2\theta}{1+ \theta} \right)^r  \left(  \frac{ e (1+\theta)  \sqrt{q}  }{ (1-\theta)} \right)^{n-m}  \]\[= \sum_{r=0}^ {m-2}  \frac{1}{2} \left( \frac{\theta}{1+ \theta} \right)^r   \left(  \frac{ 2^{ \frac{m}{n-m} } e (1+\theta)  \sqrt{q}  }{ (1-\theta)} \right)^{n-m} \leq 
  \frac{1}{2}  \frac{1}{ 1 - \frac{\theta}{ 1+ \theta}}   \left(  \frac{ 2^{ \frac{\theta}{1-\theta } } e (1+\theta)  \sqrt{q}  }{ 1-\theta} \right)^{n-m}  \] 
  
Thus
\[  \Bigl | \sum_{ \substack{ f \in \mathcal M_n \\  f \equiv a \mod g }}  \Lambda(f)  -  \frac{q^n}{ \phi(g) }    \Bigr|  \ll \left(  \frac{ 2^{ \frac{\theta}{1-\theta } } e (1+\theta)  \sqrt{q}  }{ 1-\theta} \right)^{n-m} = q^{ (n-m) (1-\delta) } .  \]

Lastly, observe that \[ \delta  = \frac{1}{2} \log_q\left( \frac{q  (1-\theta)^2 } { 4^{ \frac{\theta}{1-\theta } } (1+\theta)^2 e^2  } \right) = \frac{1}{2} \log_q \left( \frac{q}{  4^{ \frac{1}{1-\theta}  }  \left( \frac{ 1}{ 1- \theta} - \frac{1}{2} \right)^2e^2 }  \right) >0 \] by assumption.   \end{proof} 
  
  \subsection{One-level density}
  
 Estimates for the one-level density are closely related to sums of the von Mangoldt function. We begin with an identity that expresses the relevant von Mangoldt sum in a more convenient form. 
 
 For $\alpha \in \mathbb F_q^\times$, let  \[w_\alpha = \begin{cases} 1- \frac{1}{q-1} & \alpha=1 \\ - \frac{1}{q-1} & \alpha\neq 1 \end{cases}.\]
 
 \begin{lemma}\label{ol-identity}
 For $n, m$ natural numbers and $g$ a squarefree polynomial of degree $m$, we have  \begin{equation}\label{ol-1-main}  \begin{split} & \sum_{\substack{  \chi \colon (\mathbb F_q[t] / g)^\times \to \mathbb C^\times  \\ \textrm{primitive} \\ \textrm{odd}}}  \sum_{ \substack{   f \in \mathcal M_n \\ \gcd(f,g)=1} }  \Lambda(f) \chi(f) \\ = & \sum_{ \substack{ h \mid g\\ h \neq 1} } \phi(h) \mu(g/h)       \sum_{\alpha \in \mathbb F_q^\times} w_\alpha \sum_{ \substack{ f \in \mathcal M_n \\ f \equiv \alpha  \mod h}}  \Lambda(f)     - \sum_{ \substack{ h \mid g\\ h \neq 1} } \phi(h) \mu(g/h)  \sum_{\alpha \in \mathbb F_q^\times} w_\alpha \sum_{ \substack{ f \in \mathcal M_n \\ f \equiv \alpha  \mod h \\ \gcd(f,g) \neq 1}}  \Lambda(f)   \end{split} \end{equation} \end{lemma}

  \begin{proof} For $f \in \mathbb F_q[t]$ prime to $g$, we have
  \begin{align*}
  \sum_{\substack{  \chi \colon (\mathbb F_q[t] / g)^\times \to \mathbb C^\times  \\ \textrm{primitive} \\ \textrm{odd}}}  \chi(f)&  =  \sum_{h \mid g} \mu(g/h) \sum_{\substack{  \chi \colon (\mathbb F_q[t] / h)^\times \to \mathbb C^\times  \\ \textrm{odd} }}  \chi(f)\\ & =  \sum_{\substack{ h \mid g \\ h \neq 1} } \mu(g/h) \sum_{\substack{  \chi \colon (\mathbb F_q[t] / h)^\times \to \mathbb C^\times   \\ \textrm{odd}}}  \chi(f) . \\
  %&= \sum_{ \substack{ h \mid g\\ h \neq 1} } \mu(g/h) \Bigl (  \sum_{\substack{  \chi \colon (\mathbb F_q[t] / h)^\times \to \mathbb C^\times  }}  \chi(f) - \sum_{\substack{  \chi \colon (\mathbb F_q[t] / h)^\times \to \mathbb C^\times  \\ \textrm{even}  }}  \chi(f) \Bigr).
  \end{align*}
  
  Because $f$ is prime to $g$, \[   \sum_{\substack{  \chi \colon (\mathbb F_q[t] / h)^\times \to \mathbb C^\times  }}  \chi(f) = \begin{cases} \phi(h) & f \equiv 1 \mod h \\ 0 & \textrm{otherwise} \end{cases} \] and, as long as $h \neq 1$,  
  \[   \sum_{\substack{  \chi \colon (\mathbb F_q[t] / h)^\times \to \mathbb C^\times  \\ \textrm{even}  }}  \chi(f) = \begin{cases} \frac{\phi(h)}{q-1}  & f \equiv \alpha \mod h \textrm{ for some } \alpha \in \mathbb F_q^\times  \\ 0 & \textrm{otherwise} \end{cases} \]
  so 
    \[   \sum_{\substack{  \chi \colon (\mathbb F_q[t] / h)^\times \to \mathbb C^\times  \\ \textrm{odd}  }}  \chi(f) = \begin{cases} w_\alpha \phi(h) & f \equiv \alpha \mod h \textrm{ for some } \alpha \in \mathbb F_q^\times  \\ 0 & \textrm{otherwise} \end{cases}. \]
  Thus
 \begin{align*}  
 &  \sum_{\substack{  \chi \colon (\mathbb F_q[t] / g)^\times \to \mathbb C^\times  \\ \textrm{primitive} \\ \textrm{odd}}}  \sum_{ \substack{   f \in \mathcal M_n \\ \gcd(f,g)=1} }  \Lambda(f) \chi(f) =    \sum_{ \substack{  f \in \mathcal M_n \\ \gcd(f,g)=1 }} \Lambda(f)  \sum_{\substack{  \chi \colon (\mathbb F_q[t] / g)^\times \to \mathbb C^\times  \\ \textrm{primitive} \\ \textrm{odd}}}  \chi(f)\\ 
    & =    \sum_{ \substack{ f \in \mathcal M_n \\ \gcd(f,g)=1} } \Lambda(f)  \sum_{ \substack{ h \mid g\\ h \neq 1} } \mu(g/h)   \sum_{\substack{  \chi \colon (\mathbb F_q[t] / h)^\times \to \mathbb C^\times \\ \textrm{odd}   }}  \chi(h) \\
    & = \sum_{ \substack{ h \mid g\\ h \neq 1} } \phi(h) \mu(g/h)    \sum_{\alpha \in \mathbb F_q^\times}  w_\alpha \sum_{ \substack{ f \in \mathcal M_n \\ \gcd(f,g) =1 \\ f \equiv \alpha  \mod h}}  \Lambda(f)   \\
   &= \sum_{ \substack{ h \mid g\\ h \neq 1} } \phi(h) \mu(g/h)     \sum_{\alpha \in \mathbb F_q^\times} w_\alpha \sum_{ \substack{ f \in \mathcal M_n \\ f \equiv \alpha  \mod h}}  \Lambda(f)    - \sum_{ \substack{ h \mid g\\ h \neq 1} } \phi(h) \mu(g/h)    \sum_{\alpha \in \mathbb F_q^\times} w_\alpha \sum_{ \substack{ f \in \mathcal M_n\\ f \equiv \alpha  \mod h \\ \gcd(f,g) \neq 1 }}  \Lambda(f)     .  \qedhere \end{align*}    
  \end{proof}
  
  We now bound the two terms appearing in \eqref{ol-1-main}.
  
  \begin{lemma}\label{ol-easy} Let $n,m$ be natural numbers and $g$ a squarefree polynomial of degree $m$. Then 
\begin{equation}\label{ol-1-easy} \sum_{ \substack{ h \mid g\\ h \neq 1} } \phi(h) \mu(g/h)     \sum_{\alpha \in \mathbb F_q^\times} w_\alpha \sum_{ \substack{ f \in \mathcal M_n \\ f \equiv \alpha  \mod h \\ \gcd(f,g) \neq 1}}  \Lambda(f)   = O ( m q^{ \min(m,n)}) .\end{equation}  \end{lemma}
  
  \begin{proof} We have \begin{align*}   \Bigl| \sum_{ \substack{ h \mid g\\ h \neq 1} } \phi(h) \mu(g/h)     \sum_{\alpha \in \mathbb F_q^\times} w_\alpha \sum_{ \substack{ f \in \mathcal M_n \\ f \equiv \alpha  \mod h \\ \gcd(f,g) \neq 1}}  \Lambda(f)    \Bigr|   &= \Bigl| \sum_{ \substack{ f \in \mathcal M_n \\ \gcd(f,g) \neq 1 }} \Lambda(f) \sum_{\alpha \in \mathbb F_q^\times} w_\alpha \sum_{ \substack{ h \mid  \gcd(g, f-\alpha) \\ h \neq 1  }} \phi(h) \mu(g/h)\Bigr|   \\
  \leq \sum_{ \substack{ f \in \mathcal M_n \\ \gcd(f,g) \neq 1 }} \Lambda(f) \sum_{\alpha \in \mathbb F_q^\times} \abs{w_\alpha} \sum_{ \substack{ h \mid  \gcd(g, f-\alpha)   }} \phi(h)  &=  \sum_{ \substack{ f \in \mathcal M_n \\ \gcd(f,g) \neq 1 }} \Lambda(f) \sum_{\alpha \in \mathbb F_q^\times} \abs{w_\alpha} q^{ \deg \gcd(g, f-\alpha) }  \\
\leq \sum_{ \substack{ f \in \mathcal M_n \\ \gcd(f,g) \neq 1 }} \Lambda(f) \sum_{\alpha \in \mathbb F_q^\times} \abs{w_\alpha} q^{ \min (m,n) }  & \ll \sum_{ \substack{ f \in \mathcal M_n \\ \gcd(f,g) \neq 1 }}  \Lambda(f)  q^{ \min (m,n) } .\end{align*}

If $\Lambda(f) \neq 0$ then $f$ is a prime power, and if, in addition, $\gcd(f,g)\neq 1$, then $f$ must be a power of a prime dividing $g$. For each prime $\pi$ dividing $g$, at most one power of $\pi$ has degree $n$, so $ \sum_{ \substack{ f \in \mathcal M_n \\ \gcd(f,g) \neq 1 }} \Lambda(f)\leq m$, which implies \eqref{ol-1-easy}. \end{proof} 

\begin{lemma} \label{ol-hard} Let $\lambda>1,\epsilon>0$ be real numbers and let $n,m$ be natural numbers with $n\leq \lambda m$. If $q> \frac{4}{ (\lambda^2-1)^2}$ then for a squarefree polynomial $g$ of degree $m$, we have
\begin{equation}\label{ol-1-hard}  \Bigl | \sum_{ \substack{ h \mid g\\ h \neq 1} } \phi(h) \mu(g/h)    \sum_{\alpha \in \mathbb F_q^\times} w_\alpha \sum_{ \substack{ f \in \mathcal M_n \\ f \equiv \alpha  \mod h}}  \Lambda(f)   \Bigr |  \ll   \left( \frac{  (\lambda+1)^2 e^2 }{2 \sqrt{q} } \right)^{m} q^{ \frac{n}{2} }  q^{ (1 + \epsilon ) m}   . \end{equation} 

\end{lemma}

\begin{proof} We observe that  $ \sum_{\alpha\in \mathbb F_q^\times} w_\alpha=0$ so  \[ \sum_{\alpha \in \mathbb F_q^\times}  w_\alpha \Bigl(\sum_{ \substack{ f \in \mathcal M_n \\   f \equiv \alpha  \mod h}}  \Lambda(f) - \frac{1}{\phi(h)}  \sum_{ \substack{ f \in \mathcal M_n \\   \gcd(f,h) =1 }}  \Lambda(f) \Bigr) = \sum_{\alpha \in \mathbb F_q^\times} w_\alpha \sum_{ \substack{ f \in \mathcal M_n \\   f \equiv \alpha  \mod h}}  \Lambda(f) .\]
Thus if $n \geq h$ then by \cref{primes-1}  (letting $c= m-\deg h$),  \begin{equation}\label{density-1} \begin{split}
 & \Bigl| \sum_{\alpha \in \mathbb F_q^\times} w_\alpha \sum_{ \substack{ f \in \mathcal M_n \\   f \equiv \alpha  \mod h}}  \Lambda(f) \Bigr|  \\
  =& \Bigl|  \sum_{\alpha \in \mathbb F_q^\times} w_\alpha \Bigl(\sum_{ \substack{ f \in \mathcal M_n \\   f \equiv \alpha  \mod h}}  \Lambda(f) - \frac{1}{\phi(h)}  \sum_{ \substack{ f \in \mathcal M_n \\   \gcd(f,h) =1 }}  \Lambda(f) \Bigr) \Bigr| \\
 \leq & \sum_{\alpha \in \mathbb F_q^\times} \abs{w_\alpha} \Bigl| \sum_{ \substack{ f \in \mathcal M_n \\   f \equiv \alpha  \mod h}}  \Lambda(f) - \frac{1}{\phi(h)}  \sum_{ \substack{ f \in \mathcal M_n \\   \gcd(f,h) =1 }}  \Lambda(f) \Bigr| \\
 \leq & \sum_{\alpha \in \mathbb F_q^\times} \abs{w_\alpha}  \Bigl( \sum_{r=0}^{m-c-2} 2^{m-1-c-r} \binom{ n + m-c-r }{ 2 m-2c -r } q^{ \frac{n-m +c }{2} } + 2^{m-c -1 } \binom{n + m-c  }{ 2 m-2c -1 } q^{ \frac{n + 1 + m-c }{2}} \Bigr) \\
 \ll &  \sum_{r=0}^{m-c-2} 2^{m-1-c-r} \binom{ n + m-c-r }{ 2 m-2c -r } q^{ \frac{n-m +c }{2} } + 2^{m-c -1 } \binom{n + m-c  }{ 2m-2c -1 } q^{ \frac{n + 1 + m-c }{2}} 
 \end{split}. \end{equation} 
 
 The same bound also holds if $n< \deg h$, where we cannot apply \cref{primes-1} but instead observe that $f \equiv \alpha \mod h$ is impossible if $0< \deg f=n  < \deg h$.
 
 We now apply \cref{binomial-bound} with  $\alpha = \frac{2}{ \lambda-1}$,  $a =2m$, $b =n-m$,  $x= -2c-r$ and $-2c-1$ respectively, $y= c$ and $c+1$ respectively, obtaining 
 
 \begin{equation}\label{density-2} \begin{split} & \sum_{r=0}^{m-c-2} 2^{m-1-c-r} \binom{ n + m-c-r }{ 2 m-2c -r } q^{ \frac{n-m +c }{2} } + 2^{m-c -1 } \binom{n + m-c  }{ 2m-2c -1 } q^{ \frac{n + 1 + m-c }{2}} \\
 \leq & \sum_{r=0}^{m-c-2} 2^{m-1-c-r}   \frac{2^{2c+r} }{ (\lambda +1)^{ c+r}  (\lambda-1)^{c} } \left( \frac{(\lambda+1) e}{2} \right)^{2m} q^{ \frac{n-m +c }{2} } \\
  & + 2^{m-c -1 } \frac{2^{2c+1} }{ (\lambda +1)^{ c+1 }  (\lambda-1)^{c} } \left( \frac{(\lambda+1) e}{2} \right)^{2m} q^{ \frac{n + 1 -  m+c }{2}} \\
 =&  \left(\sum_{r=0}^{m-c-2}  \frac{1}{( \lambda+1)^r}  +  \frac{2 q^{ \frac{1}{2}}  }{ \lambda +1 }   \right)  2^{m-1-c}  \frac{2^{2c} }{ (\lambda +1)^{ c}  (\lambda-1)^{c} }  \left( \frac{(\lambda+1) e}{2} \right)^{2m} q^{ \frac{n-m +c }{2} } \\
 \leq &  \left(  \frac{\lambda +1}{\lambda}   +  \frac{2 q^{ \frac{1}{2}}  }{ \lambda +1 }   \right)  2^{m-1-c}  \frac{2^{2c} }{ (\lambda +1)^{ c}  (\lambda-1)^{c} }  \left( \frac{(\lambda+1) e}{2} \right)^{2m} q^{ \frac{n-m +c }{2} }  \\
  \ll & 2^{m-c}  \frac{2^{2c} }{ (\lambda +1)^{ c}  (\lambda-1)^{c} }  \left( \frac{(\lambda+1) e}{2} \right)^{2m} q^{ \frac{n-m +c }{2} }   =  \left( \frac{2 \sqrt{q} }{\lambda^2-1} \right)^c  \left( \frac{ (\lambda+1)^2 e^2 }{2 \sqrt{q} } \right)^{m} q^{ \frac{n}{2} } 
  \end{split} \end{equation}
  
  Combining \eqref{density-1} and \eqref{density-2}, we obtain 
  
  \begin{equation}\label{density-3} \begin{split} 
  & \Bigl | \sum_{ \substack{ h \mid g\\ h \neq 1} } \phi(h) \mu(g/h)      \sum_{\alpha \in \mathbb F_q^\times} w_\alpha \sum_{ \substack{ f \in \mathcal M_n \\ f \equiv \alpha  \mod h}}  \Lambda(f) \Bigr|\\  \ll&
\sum_{ \substack{ h \mid g} } \phi(h)   \left( \frac{2 \sqrt{q} }{\lambda^2-1} \right)^{ m- \deg h }  \left( \frac{ (\lambda+1)^2 e^2 }{2 \sqrt{q} } \right)^{m} q^{ \frac{n}{2} }\\
=&  \left( \frac{ (\lambda+1)^2 e^2 }{2 \sqrt{q} } \right)^{m} q^{ \frac{n}{2} }   \prod_{\substack{ \pi \mid g \\ \pi \textrm{ prime}}} \left(q^{\deg \pi} - 1 + \left( \frac{2 \sqrt{q} }{\lambda^2-1} \right)^{\deg \pi } \right).
  \end{split} \end{equation}
 Fix $\epsilon >0$. Since $q >  \frac{4}{ (\lambda^2-1)^2}$, we have  $q^{\deg \pi} - 1 + \left( \frac{\sqrt{q} }{\lambda^2-1} \right)^{\deg \pi } < q^{ (1+\epsilon) \deg \pi}$ for all but finitely many $\pi$ and thus 
 
 \begin{equation}\label{density-4}  \left( \frac{  (\lambda+1)^2 e^2 }{2 \sqrt{q} } \right)^{m} q^{ \frac{n}{2} }   \prod_{\substack{ \pi \mid g \\ \pi \textrm{ prime}}} \left(q^{\deg \pi} - 1 + \left( \frac{2 \sqrt{q} }{\lambda^2-1} \right)^{\deg \pi } \right) \ll \left( \frac{  (\lambda+1)^2 e^2 }{2 \sqrt{q} } \right)^{m} q^{ \frac{n}{2} }  q^{ (1 + \epsilon ) m}  \end{equation} 
 
 Combining \eqref{density-3} and \eqref{density-4}, we obtain \eqref{ol-1-hard}.  \end{proof}

  We can now bound our original von Mangoldt sum.
  
  \begin{lemma}\label{ol-1} Let $\lambda>1.07$ be a real number and let $q$ be a prime power with \[ q > \frac{ (\lambda+1)^4 e^4}{4} .\]

  Let $n$ and $m$ be natural numbers with $n \leq \lambda m$.  For a squarefree polynomial $g$ of degree $m$, we have
  \[\Bigl|  \sum_{\substack{  \chi \colon (\mathbb F_q[t] / g)^\times \to \mathbb C^\times  \\ \textrm{primitive} \\ \textrm{odd}}}  \sum_{ \substack{  f \in \mathcal M_n \\ \gcd(f,g)=1}}  \Lambda(f) \chi(f) \Bigr| = O \left(   q^{ \frac{n}{2}} (q^{ m})^{1-\delta} \right) ,\] with the constant depending only on $q,\lambda, \delta$, for any  \[ \delta < \frac{1}{2} - \log_q \left( \frac{(\lambda+1)^2 e^2}{2} \right)  , \] and there exists such $\delta>0$.  \end{lemma}
  
  \begin{proof}  By our assumptions on $q$ and $\lambda$, we have  $q> \frac{ (\lambda+1)^4 e^4}{4}> \frac{ 4}{ (\lambda^2-1)^2}. $
  
  We apply \cref{ol-identity,ol-easy,ol-hard} to bound both terms on the right side of \eqref{ol-1-main}, obtaining
 
\[\Bigl | \sum_{\substack{  \chi \colon (\mathbb F_q[t] / g)^\times \to \mathbb C^\times  \\ \textrm{primitive} \\ \textrm{odd}}}  \sum_{ \substack{   f \in \mathcal M_n \\ \gcd(f,g)=1} }  \Lambda(f) \chi(f)   \Bigr | \ll \left( \frac{  (\lambda+1)^2 e^2 }{2 \sqrt{q} } \right)^{m} q^{ \frac{n}{2} }  q^{ (1 + \epsilon ) m}  + m q^{ \min(m,n)} \]
 Since $m \ll q^{\epsilon m}$ for every $\epsilon>0$ and $q^{ \min (m,n)} \leq q^{ \frac{m}{2}} q^{ \frac{n}{2}}$, the second term is dominated by the first. Thus, it suffices to observe that, for any 
 \[ \delta < \frac{1}{2} - \log_q \left( \frac{(\lambda+1)^2 e^2}{2} \right)  , \] 
 if we choose
 \[ \epsilon = \frac{1}{2} - \log_q \left( \frac{(\lambda+1)^2 e^2}{2} \right)   - \delta\] then \[ \left( \frac{  (\lambda+1)^2 e^2 }{2 \sqrt{q} } \right)^{m} q^{ \frac{n}{2} } q^{(1+\epsilon)m}=  q^{ (1-\delta) m} q^{ \frac{n}{2} .}  \]

Finally, because $q > \frac{ (\lambda+1)^4 e^4}{4}$, we have $ \frac{1}{2} - \log_q \left( \frac{(\lambda+1)^2 e^2}{2} \right)>0$ so there exists such $\delta>0$.
%\[  \sum_{ \substack{ f \in \mathcal M_n \\ \gcd(f,g) =\alpha \\ f \equiv \alpha  \mod h}}  \Lambda(f)  = \Bigl(\sum_{ \substack{ f \in \mathcal M_n \\   f \equiv \alpha  \mod h}}  \Lambda(f) - \frac{1}{\phi(h)}  \sum_{ \substack{ f \in \mathcal M_n \\   \gcd(f,h) =1 }}  \Lambda(f) \Bigr) + \frac{1}{\phi(h)}  \sum_{ \substack{ f \in \mathcal M_n \\   \gcd(f,h) =1 }}  \Lambda(f)   - \sum_{\substack{  f \in \mathcal M_n \\ \gcd(f,g) \neq 1 \\ f \equiv \alpha \mod h}} \Lambda(f) .\]
  \end{proof} 
  
  Next we prove a lemma relating von Mangoldt sums to roots.
  
  Fix a primitive character $\chi$. Since $L(s,\chi)$ is a polynomial in $q^{-s}$, the roots of $L(s,\chi)$ are periodic with period $\frac{2 \pi i}{\log q}$.  Since this polynomial has degree $m-1$, there are $m-1$ roots of the form $\frac{1}{2} + i \gamma$ with $\gamma \in [0, \frac{2\pi}{\log q})$. Let $\gamma_1,\dots, \gamma_{m-1}$ be the $\gamma$ values of these roots, taken with multiplicity.

  \begin{lemma} \label{ol-transfer} Fix a primitive character $\chi$ and $n  \in \mathbb Z$, $n\neq 0$. Let  $\chi^{\sgn(n)}$ be $\chi$ if $n>0$ and $\chi^{-1}$ if $n<0$. We have
  
  \[\sum_{j=1}^{m-1} e^{ in  \log(q) \gamma_j } = q^{-\frac{\abs{n}}{2} } \sum_{ \substack{ f \in \mathcal M_{\abs{n}} \\\ \gcd(f,g)=1}} \Lambda(f) \chi^{\sgn(n)} (f).\]

\end{lemma}

  \begin{proof} We have \[ L(s,\chi) = \prod_{j=1}^{m-1} ( 1 - q^{ \frac{1}{2} + i \gamma_j - s} )  \] so \[ \frac{1}{\log q}  \frac{ \operatorname{dlog}  L(s,\chi) }{ds} = \sum_{j=1}^{m-1} \frac{ q^{ \frac{1}{2} + i \gamma_j - s} }{ 1 - q^{ \frac{1}{2} + i \gamma_j - s}}  = \sum_{j=1}^{m-1} \sum_{n=1}^{\infty} q^{ \frac{n}{2} + i n \gamma_j}  q^{-ns} = \sum_{n=1}^{\infty} q^{ \frac{n}{2} } \sum_{j=1}^{m-1} e^{ i n \log(q) \gamma_j} q^{-ns}   . \]
The Euler product of $L(s,\chi)$ gives
 \[ \frac{1}{\log q}  \frac{ \operatorname{dlog}  L(s,\chi) }{ds}  = \sum_{ \substack{ \pi \textrm{ prime} \\ \pi \nmid g} }  \frac{  \chi(\pi) q^{-s}}{ 1- \chi(\pi) q^{-s}}  = \sum_{n=1}^{\infty}  \sum_{ \substack{ f \in \mathcal M_n \\\ \gcd(f,g)=1}} \Lambda(f)  \chi(f) q^{-ns} .\]
Equating the coefficients of $q^{-ns}$, we have for $n>0$
\[\sum_{j=1}^{m-1} e^{ in  \log(q) \gamma_j }  = q^{-\frac{n}{2} } \sum_{ \substack{ f \in \mathcal M_n \\\ \gcd(f,g)=1}} \Lambda(f)  \chi(f) = q^{-\frac{\abs{n}}{2} } \sum_{ \substack{ f \in \mathcal M_{\abs{n}} \\\ \gcd(f,g)=1}} \Lambda(f)  \chi^{\sgn(n)}(f)  .\]
The result for $n<0$ follows by taking complex conjugates of both sides. 
  \end{proof}

  \begin{proposition}\label{ol-2} Let $\lambda>1.07$ be a real number and let $q$ be a prime power with \[ q > \frac{ (\lambda+1)^4 e^4}{4} .\] 

 Let $\psi$ be a Schwartz function on $\mathbb R$ whose Fourier transform $\hat{\psi}$ is supported on $[-\lambda,\lambda]$. Then for a squarefree polynomial $g$ of degree $m$, we have
 \[ \sum_{\substack{  \chi \colon (\mathbb F_q[t] / g)^\times \to \mathbb C^\times  \\ \textrm{primitive} \\ \textrm{odd}}}  \sum_{\gamma_\chi}  \psi \Bigl( \frac{ m \log(q) \gamma_\chi}{2\pi} \Bigr ) = \left( \hat{\psi}(0) + o(1) \right)  \phi^{*,odd}(g)\]
 where $\frac{1}{2} + i \gamma_\chi$ are the (infinitely many) roots of $L(s,\chi)$, $\phi^{*,odd}(g) $ is the number of primitive odd Dirichlet characters mod $g$, and $o(1)$ goes to $0$ as $m$ goes to $+\infty$ for fixed $q,\lambda$ and bounded $ || \hat{\psi}||_\infty$.
 \end{proposition}
 
 \begin{proof} We can express the multiset of all  $\gamma_\chi$ as 
 \[ \left\{ \gamma_j + k \frac{ 2\pi }{\log q} \mid  1 \leq j\leq m-1, k\in \mathbb Z\right\} .\]
 Thus, by the Poisson summation formula,
 \[ \sum_{\gamma_\chi}  \psi \Bigl( \frac{ m \log(q) \gamma_\chi}{2\pi} \Bigr )  = \sum_{j=1}^{m-1} \sum_{k=-\infty}^{\infty}  \psi \Bigl ( \frac{ m \log(q) \gamma_j}{2\pi }+ m k \Bigr)\]
 \[ =  \sum_{j=1}^{m-1}  \sum_{n=-\infty}^{\infty}  e^{ \frac{ m \log(q) \gamma_j}{2\pi } \cdot \frac{n}{m} \cdot 2\pi i} \frac{1}{m}   \hat{\psi} \Bigl( \frac{n}{m} \Bigr)  = \sum_{n= -\infty}^\infty  \frac{1}{m}  \sum_{j=1}^{m-1} e^{ i n \log(q) \gamma_j } \hat{\psi} \Bigl( \frac{n}{m} \Bigr). \]
\[ = \sum_{ \substack{ n \in \mathbb Z \\ n \neq 0 \\ \abs{n}\leq \lambda m} } \frac{q^{-\frac{\abs{n}}{2} }}{m} \hat{\psi} \Bigl( \frac{n}{m} \Bigr) \sum_{ \substack{ f \in \mathcal M_{\abs{n}} \\\ \gcd(f,g)=1}} \Lambda(f)  \chi^{\sgn(n)}(f) + \frac{m-1}{m} \hat{\psi} (0) \]  where the last equality uses \cref{ol-transfer}, our assumption implying that $\hat{\psi} \left( \frac{n}{m} \right) =0 $ if $ \abs{n} > \lambda m$, and the fact that 
  $ \sum_{j=1}^{m-1} e^{ in  \log(q) \gamma_j } = m-1$ if $n=0$.

 Now we sum over $\chi$. We obtain
 \[ \sum_{\substack{  \chi \colon (\mathbb F_q[t] / g)^\times \to \mathbb C^\times  \\ \textrm{primitive} \\ \textrm{odd}}}  \sum_{\gamma_\chi}  \psi \Bigl( \frac{ m \log(q) \gamma_\chi}{2\pi} \Bigr )\]
\[ = \sum_{ \substack{ n \in \mathbb Z \\ n \neq 0 \\ \abs{n}\leq \lambda m} } \frac{q^{-\frac{\abs{n}}{2} }}{m} \hat{\psi} \Bigl( \frac{n}{m} \Bigr)   \sum_{\substack{  \chi \colon (\mathbb F_q[t] / g)^\times \to \mathbb C^\times  \\ \textrm{primitive} \\ \textrm{odd}}}\sum_{ \substack{ f \in \mathcal M_{\abs{n}} \\\ \gcd(f,g)=1}} \Lambda(f)  \chi^{\sgn(n)}(f) +    \sum_{\substack{  \chi \colon (\mathbb F_q[t] / g)^\times \to \mathbb C^\times  \\ \textrm{primitive} \\ \textrm{odd}}} \frac{m-1}{m} \hat{\psi} (0). \] 
 Thus
  \[ \Bigl | \sum_{\substack{  \chi \colon (\mathbb F_q[t] / g)^\times \to \mathbb C^\times  \\ \textrm{primitive} \\ \textrm{odd}}}\Bigl(   \sum_{\gamma_\chi}  \psi \Bigl( \frac{ m \log(q) \gamma_\chi}{2\pi} \Bigr ) -\hat{\psi}(0)  \Bigr)  \Bigr|   \]
\[ \leq \sum_{ \substack{ n \in \mathbb Z \\ n \neq 0 \\ \abs{n}\leq \lambda m} } \frac{q^{-\frac{\abs{n}}{2} }}{m} \abs{\hat{\psi} \Bigl( \frac{n}{m} \Bigr) }\Biggl|   \sum_{\substack{  \chi \colon (\mathbb F_q[t] / g)^\times \to \mathbb C^\times  \\ \textrm{primitive} \\ \textrm{odd}}}\sum_{ \substack{ f \in \mathcal M_{\abs{n}} \\\ \gcd(f,g)=1}} \Lambda(f)  \chi(f) \Biggr| +    \sum_{\substack{  \chi \colon (\mathbb F_q[t] / g)^\times \to \mathbb C^\times  \\ \textrm{primitive} \\ \textrm{odd}}} \frac{1}{m}\abs{ \hat{\psi} (0)}. \] 
 \[ \ll  \sum_{ \substack{ n \in \mathbb Z \\ n \neq 0 \\ \abs{n}\leq \lambda m} } \frac{q^{-\frac{\abs{n}}{2} }}{m} \abs{\hat{\psi} \Bigl( \frac{n}{m} \Bigr) }   q^{ \frac{\abs{n}}{2}} (q^{ m})^{1-\delta}  +    \sum_{\substack{  \chi \colon (\mathbb F_q[t] / g)^\times \to \mathbb C^\times  \\ \textrm{primitive} \\ \textrm{odd}}} \frac{1}{m}\abs{ \hat{\psi} (0)}. \] 
\[ \leq 2 \lambda m  ||\hat{\psi}||_{\infty} (q^m)^{1-\delta}  + \sum_{\substack{  \chi \colon (\mathbb F_q[t] / g)^\times \to \mathbb C^\times  \\ \textrm{primitive} \\ \textrm{odd}}} \frac{1}{m}\abs{ \hat{\psi} (0)}.\]

The second term is certainly $o(\phi^{*,odd}(g) )  $ since $ \frac{1}{m}\abs{ \hat{\psi} (0)} = o(1)$.

For the first term, it suffices to check that $m (q^m)^{1-\delta} = o (\phi^{*,odd}(g))$ which holds because
\[ \phi^{*,odd}(g) =  \frac{q-2}{q-1}  \prod_{ \pi \mid g} \left( q^{ \deg \pi - 2 } \right)  - \frac{q-2}{q-1} (-1)^{\omega(g)} \gg  q^{ (1-\epsilon)m}  \]
as soon as $\deg g>0$.
  \end{proof}

\section{Anatomy of polynomials in progressions}\label{s-anatomy}

\subsection{The GEM distribution and the Poisson-Dirichlet process} 

Let $\Delta_\infty$ be the simplex $\Delta_\infty$ of functions $p\colon  \mathbb N^+ \to[0,1]$ satisfying $\sum_{m=1}^{\infty} p(m) =1$ (i.e. the space of probability distributions on $\mathbb N^+$), endowed with the (compact) product topology.

Let $x_1,x_2,\dots$ be a sequence of independent random variables uniformly distributed on $[0,1]$, and let 
\[ p(m) = (1-x_m)  \prod_{i=1}^{m-1} x_i .\]
Then we define the GEM distribution as the distribution of the $\Delta_\infty$-valued random variable $p$. Our goal in this subsection will be to find a tail bound on the entropy \[ - \sum_{m=1}^{\infty} p(m) \log p(m) \] of $p$, which we will use in the remaining subsections to bound the entropy of random polynomials and thus to control the total variation distance between two distributions on factorization functions.

The closely-related Poisson-Dirichlet process is the distribution of the random variable attained by sorting $p$ in decreasing order $p(1) \geq p(2) \geq p(3) \dots$.  Since the entropy is invariant under permuting the $p(m)$, including sorting, the same results will hold for the Poisson-Dirichlet process.

For background on the GEM distribution, the Poisson-Dirichlet process, and their relation to factorizations of random integers, see \cite{DG}.

To control the entropy, we will place the GEM distribution in a Markov chain, and then compare with a simpler Markov chain. Let $y_n$ be a sequence indexed by $n \in \mathbb Z$ of independent random variables uniformly distributed on $[0,1]$.  Define for $n \in \mathbb Z$
\[ p_n(m) =( 1 - y_{n-m} ) \prod_{i=1}^{m-1} y_{ n-i} .\]

Then for all $n\in \mathbb Z$, the function $p_n(m)$ is GEM distributed, because $x_i = y_{n-i}$ are independent and uniformly distributed in $[0,1]$.

We have \[ p_{n+1} (m) =  \begin{cases} 1-y_n & m=1 \\ y_n  p_n(m-1)  & m>1 \end{cases} .\]
Note that $p_n $ depends only on $y_{n-1},y_{n-2},\dots$ and thus $y_n$ is independent of $p_n$. Hence $(p_n)$ is a Markov chain on $\Delta_\infty$, with the GEM process as a stationary distribution.

Let \[ H_n = - \sum_{m=1}^{\infty} p_n(m) \log p_n(m) \] be the entropy of the probability distribution $p_n$. We have
\begin{equation}\begin{split}\label{next-entropy-formula} H_{n+1} &= - (1-y_n) \log (1-y_n) - \sum_{m=2}^{\infty}  y_n p_n(m-1) \log ( y_n p_n(m-1) )\\
=&  - (1-y_n) \log (1-y_n) -  \sum_{m=2}^{\infty}  y_n p_n(m-1) \log ( y_n ) -  \sum_{m=2}^{\infty}  y_np_n(m-1) \log ( p_n(m-1))\\ =& - (1-y_n) \log(1-y_n) - y_n \log(y_n)  + y_n H_n .\end{split}\end{equation}

Because $H_n$ depends only on $y_{n-1}, y_{n-2},\dots$, $H_n$ is independent of $y_n$, and so \eqref{next-entropy-formula} implies that $(H_n)$ is a Markov chain. Our first lemma controls how much $H_n$ can increase over one step of this Markov chain:

%characterized by the recurrence relation that if we sample $p$ from the GEM process and $y$ independently uniformly from $[0,1]$, and let \[ \tilde{p} (n ) =  \begin{cases} 1-y & n=1 \\ y p(n-1)  & n>2 \end{cases} \] then $\tilde{p}$ is distributed according to the GEM process. For instance, we can construct the GEM process by starting from a probability distribution which assigns measure one to $(1,0,0,\dots)$, iteratively applying the operation $p \to \tilde{p}$, and taking a limit.
%
%For $p$ a probability distribution on $\mathbb N$, let $H(p)$ be the entropy \[ H(p)= - \sum_{n=1}^{\infty} p(n) \log p(n) .\] 
%
%Then we have \[ H( \tilde{p}) = - (1-y) \log (1-y) - \sum_{n=1}^{\infty}  y p(n) \log ( y p(n))\]
%\[ = - (1-y) \log (1-y) -  \sum_{n=1}^{\infty}  y p(n) \log ( y ) -  \sum_{n=1}^{\infty}  y p(n) \log ( p(n)  ) = - (1-y) \log(1-y) - y \log(y)  + y H(p) .\] 

\begin{lemma}\label{entropy-increase-bound}   Let $H\geq0$ and $0\leq y \leq 1$ be real numbers, and let
\[ \tilde{H} = - (1-y) \log(1-y) - y \log(y)  + y H .\]

We have \[ e^{ \tilde{H} } \leq e^{ H} +1 ,\] we have \[ e^{ \tilde{H} } \leq e^{H} \]  as long as $y\leq 1 - e^{ 1-  H}   $, and we have  \[ e^{ \tilde{H} } \leq e^{H}-1 \] as long as $y\leq 1 - 4e^{-  H }  $. \end{lemma}

\begin{proof}  Let $w = e^{ H}$. Then because log is concave, we have
\[ (1-y) \log \left( \frac{1}{(1-y) (w+1) } \right) + y \log \left( \frac{ w}{ y (w+1) } \right)  \leq \log\left( \frac{ (1-y)}{ (1-y) (w+1)}  + \frac{ y w}{ y (w+1) } \right)\] \[ = \log \left( \frac{1}{w+1} + \frac{w}{w+1}\right) = \log 1 =0.\]

Expanding out the terms, we get
\begin{align*}
  - (1-y) \log (1-y) -  \log(w+1) - y \log (y) + y \log(w)&  \leq 0\\
  - (1-y) \log (1-y) - y \log (y) + y \log(w)  &\leq \log (w+1) \\
\tilde{H}  =   - (1-y) \log (1-y) - y \log (y) + y H &  \leq \log ( e^{ H} +1) .\end{align*}

This proves the first claim.

%Let $z = e^{H}$, then $e^{ \tilde{H}  } = e^{ - (1-y) \log(1-y) - y \log(y)  } z^y$.  This is a convex function of $z$. If $z=y/(1-y) $ then \[ e^{ - (1-y) \log(1-y) - y \log(y)  }  z^y = e^{ - (1-y) \log(1-y) - y \log(y)  }  \left(\frac{y}{1-y} \right)^y  = e^{ - (1-y) \log(1-y) - y \log(y)   + y \log (y ) - y \log(1-y) } = \frac{1}{1-y} = \frac{y}{1-y} + 1 =z+1 \] and \[ \frac{d}{dz}  e^{ - (1-y) \log(1-y) - y \log(y)  }  z^y  =  e^{ - (1-y) \log(1-y) - y \log(y)  }  y z^{y-1} =  e^{ - (1-y) \log(1-y) - y \log(y)  }  y \left(\frac{y}{1-y} \right)^{y-1} \] \[ = e^{-(1-y) \log (1-y) - y\log(y) + \log(y) - (1-y) \log (y) + (1-y) \log(1-y) }  = 1\] and since $e^{ - (1-y) \log(1-y) - y \log(y)  } z^y$ is a convex function tangent to the line $z+1$ it is bounded by the line $z+1$ everywhere.

We have \[ - \log (y) = \log \left( \frac{1}{y} \right) = \log \left( 1 + \frac{1-y}{y} \right) \leq \frac{1-y}{y}\]  so   \begin{equation}\label{useful-tilde} \begin{split}  \tilde{H} \leq & - (1-y) \log (1-y) + (1-y) + y H \\ =&  (1-y) (- \log(1-y) + 1- H ) + H . \end{split} \end{equation} 

Thus, if $ y \leq 1 - e^{ 1-  H } $ then
\[ \tilde{H}  \leq    (1-y)  (- \log( e^{ 1 - H } ) + 1- H ) + H =  H .\]

Similarly, if $y \leq 1 - 4 e^{-H} $ then
\[ \tilde{H}  \leq   H + (1-y) (- \log(4  e^{ - H } ) + 1- H ) =  H -  (1-y) (\log 4-1) .\]

Thus

\[ e^{\tilde{H}  }  \leq e^{ H} e^{ - (1-y) (\log 4-1) } \leq  e^{ H} e^{  -  4 e^{ -H }(\log 4-1) }.\]
We have $ e^{ - (\log 4-1) 4 x} < 1-x$ for $x < .61$ by examining the graphs, so applying this to
\[ x= e^{-H}  \leq (1-y)/4 \leq 1/4 \leq .61,\]
we get
\[ e^{\tilde{H}  }  \leq  e^{ H} (1 - e^{ - H}) = e^{H}-1 ,\]
as desired.

%If $y \leq 1 -e^{ 2 -  H } $ then by the same reasoning, $ H_{n+1}  - H  \leq - (1-y)$ so \[ e ^{ H(\tilde{p}) } \leq  e^{ - (1-y) } e^{H } \leq e^{ - e^{ 2 + H}} e^{H(p)}  \] and using $e^{ - e^2 x} \leq 1-x$ for $x< .999$, which is satisfied since $e^{ H( p)} \geq e^2 / (1-y) \geq e^2 > 1 / .999$. 
\end{proof}
 We will now bound $H_n$ by another Markov chain, whose limiting distribution can be found explicitly. 

Let $N_0=\max (4, \lceil e^{H_0} \rceil) $ and, for $n\geq 0$, let \[N_{n+1} = \begin{cases} N_n+1 &  y_n> 1 - \frac{e }{ N_n} \\  N_n &   1- \frac{4}{N_n}  < y_n  \leq 1 - \frac{e}{N_n} \\  N_n-1 &  y_n \leq  1 - \frac{4}{N_n}\end{cases}  .\]

\begin{lemma}\label{coupling-comparison} For all $n\geq 0$, we have $N_n \geq e^{H_n}$. \end{lemma}

\begin{proof} We prove this by induction on $n$. The case $n=0$ follows from the definition of $N_0$. For the induction step, observe that 
\[ N_{n+1} \geq e^{  - (1-y_n) \log(1-y_n) - y_n \log(y_n)   + y_n \log N_n} \geq e^{  - (1-y_n) \log(1-y_n) - y_n \log(y_n)   + y_n H_n} = e^{ H_{n+1}}\]
by \cref{entropy-increase-bound} (applied to $\log N_n$ and $y_n$), the induction hypothesis, and \eqref{next-entropy-formula}. \end{proof}

For $N \geq 4$ an integer, let  \[ \pi(N)  =  \frac{1}{ (4+e) e^{ (4+e) } }    \frac{N e^N }{ (N- 4)! }   . \]

\begin{lemma}\label{limiting-distribution} We have 
\[ \lim_{N \to \infty} \mathbb P ( N_n = N) = \pi(N).\] \end{lemma}

\begin{proof} The sequence $N_n$ is a Markov chain with state space the integers at least $4$, and the following transition probabilities: The probability of transitioning from $N$ to $N+1$ is  $ \frac{e}{N} $, the probability of transitioning from $N$ to $N-1$ is $ 1- \frac{4}{N} $, and the probability of transitioning from $N$ to $N$ is the remaining $ \frac{4-e }{N}  $.

We claim that $\pi$ is a time-reversible stationary distribution of this Markov chain. 

We can check that $\pi$ is a probability measure using 
\[ \sum_{N=4}^{\infty} N  e^N / (N-4)! =  \sum_{N=0}^{\infty} (N+4 ) e^ {4 +N }/ N! = \sum_{N=0}^{\infty} 4 e^4  e^N/ N!  + \sum_{N=1}^{\infty} e^{ 5 + (N-1)} / (N-1)! \] \[ = 4 e^4 e^e + e^5 e^4 = (4 + e) e^{ 4+e} . \] 

To check that $\pi$ is stationary and time-reversible, it suffices to check that \[  \frac{e}{N}  \pi (N)  = \left( 1- \frac{4}{ N+1} \right)  \pi (N+1)  \] for all $N$, which is true by definition.

Because the probabilities of transitioning from $N$ to $N+1$ or from $N$ to $N-1$ are positive (in the second case, as long as $N-1\geq 4$, the Markov chain $N_n$ is irreducible. Because it is irreducible and has a stationary distribution $\pi$, it is positive recurrent \cite[Theorem 21.13]{lpw}. Because each state has positive probability of transitioning to itself, the Markov chain is also aperiodic, and thus it converges to the unique stationary distribution $\pi$ \cite[Theorem 21.16]{lpw}.\end{proof}

%For integer $C$, we have the probability that $ N \geq C$ is  \[ \sum_{N=C}^{\infty}  \frac{1}{ (4+e) e^{ (4+e) } }    \frac{N e^N }{ (N- 4)! }  =  \frac{1}{ (4+e) e^{ (4+e) } }   \left( \sum_{N=C}^{\infty}   \frac{ e^N }{ (N- 5)! }   + \sum_{N=C}^{\infty} \frac{4 e^N}{ (N-4)!} \right)\] \[ \leq   \frac{1}{ (4+e) e^{ (4+e) } }   \left( \sum_{N=C}^{\infty}   \frac{ e^N }{ (C- 5)! (C-4)^{N-C}  }   + \sum_{N=C}^{\infty} \frac{4 e^N}{ (C-4)! (C-3)^{N-C} } \right)  \] \[= \frac{1}{ (4+e) e^{ (4+e) } }   \left(  \frac{e^C}{ (C-5)!} \frac{1}{ 1- \frac{e}{C-4}} + \frac{4 e^C}{ (C-4)!}  \frac{1}{1- \frac{e}{ (C-3) }  }  \right)  \] \[ =\frac{ e^C}{ (4+e) e^{ (4+e) }  (C-5)! }   \left(  \frac{C-4}{ C- 4 -e} + \frac{  C-3 }{ (C-3- e) (C-4) }  \right)  \] 
%

\begin{lemma}\label{summation-bound} For $C\geq 8$, we have
\[ \sum_{ N = C}^\infty \pi (N) \leq   \frac{1}{ (4+e) e^{ (4+e) } }  \frac{  C^5 e^C}{ C!} \] 
\end{lemma}
\begin{proof} By a telescoping sum, it suffices to check that for $N \geq 8 $, \[  \frac{1}{ (4+e) e^{ (4+e) } }    \frac{N e^N }{ (N- 4)! }  = \pi(N) \leq  \frac{1}{ (4+e) e^{ (4+e) } }\frac{ N^5  e^N}{ N !} -  \frac{1}{ (4+e) e^{ (4+e) } }\frac{  (N+1)^5 e^{N+1}}{ (N+1)!}. \]   Factoring out  $\frac{e^N}{ (4+e) e^{ (4+e)  } N!} $, it suffices to prove
\[ N^2 (N-1) (N-2) (N-3) \leq  N^5   - e (N+1)^4 \] 
which one can check holds for $N \geq  8$.  \end{proof} 

\begin{lemma}\label{dp-bound} For $L > 7$, the probability that the entropy of a GEM random variable is at least $\log L $ is at most
\[ \leq \frac{1}{ \sqrt{2\pi} (4+e) e^{ (4+e) } } \frac {  L^{9/2}   e^{2L} }{  L^L} .\] \end{lemma}

\begin{proof} Because for each $n$, $ H_n$ is the entropy of a GEM random variable, the probability that the entropy of a GEM random variable is at least $\log L$ is equal to $\mathbb P(H_n\geq \log L)$ for all $n$ and thus is equal to
\[ \lim_{n \to \infty} \mathbb P ( H_n\geq \log L ) =  \lim_{n \to \infty} \mathbb P ( e^{H_n} \geq  L )  \leq \lim_{n \to \infty} \mathbb P ( N_n \geq  L ) .\] We use \cref{coupling-comparison} to relate $H_n$ and $N_n$.

So it suffices to check for $L> 7 $ that 
\[  \lim_{n \to \infty} \mathbb P ( N_n \geq L ) \leq \frac{1}{ \sqrt{2\pi} (4+e) e^{ (4+e) } } \frac {  L^{9/2}   e^{2L} }{  L^L} .\]

Let $C = \lceil L \rceil$. Then we have 
\[  \lim_{n \to \infty} \mathbb P ( N_n \geq  L )  =\lim_{n \to \infty} \mathbb P ( N_n \geq  C )= \sum_{ N = C}^\infty \pi (N)  \leq   \frac{1}{ (4+e) e^{ (4+e) } }  \frac{  C^5 e^C}{ C!}   \] \[\leq \frac{1}{ \sqrt{2\pi} (4+e) e^{ (4+e) } } \frac {  C^{9/2}   e^{2C} }{  C^C}   \leq \frac{1}{ \sqrt{2\pi} (4+e) e^{ (4+e) } } \frac {  L^{9/2}   e^{2L} }{  L^L} \]
using, respectively, the fact that $N_n$ is always an integer, \cref{limiting-distribution}, \cref{summation-bound}, the explicit Stirling's formula $C! \geq \sqrt{2\pi } e^{-C} C^{ (C+1/2) }$, and the fact that $\frac {  L^{9/2}   e^{2L} }{  L^L} $ is a decreasing function of $L$ for $L>7$ (which may be checked straightforwardly by calculating the logarithmic derivative).

\end{proof}

\begin{remark}  It is not so hard to calculate that the mean of $H_n$ is $1$ and the variance is $1 - \frac{\pi^2}{12} = .1775 \ldots$, so the true distribution of $e^{H_n}$ may be quite concentrated around $e$.  However, using this mean and variance and applying Chebyshev's inequality would give a considerably worse tail bound for large $L$. \end{remark}

The following result will also be convenient:

\begin{lemma}\label{no-atoms} For any $ x\in \mathbb R$, the probability that the entropy of a GEM random variable is exactly $x$ is zero. \end{lemma}

The same argument can be used to show the measure on $\mathbb R$ given by the distribution of the entropy of a GEM random variable is absolutely continuous with respect to Lebesgue measure, but we will not need this stronger statement.

\begin{proof} By \eqref{next-entropy-formula}, there are random variables $H, \tilde{H}, y$ with $H, \tilde{H}$ distributed as the entropy of a GEM random variable, $y$ independent of $H$, and $ \tilde{H} = - (1-y) \log(1-y) - y \log(y)  + y H$. It suffices to show the probability that $\tilde{H}=x$ is zero. By Fubini's theorem applied to calculate the integral of the characteristic function of the set where $\tilde{H}=x$ against the product of the probability measure of $H$ and the probability measure of $y$, it suffices to show for each $H$ that the probability that $- (1-y) \log(1-y) - y \log(y)  + y H=x$ is zero for $y$ uniformly distributed in $[0,1]$. Since $- (1-y) \log(1-y) - y \log(y)  + y H$ is a nonconstant analytic function on $y\in [0,1]$, it takes the value $x$ for only finitely many $y$, so it indeed takes the value $x$ with probability zero.

\end{proof}

\subsection{Sums of Dirichlet convolutions of von Mangoldt functions}

Fix an integer $\omega \geq 1$. Let $n_1,\dots, n_\omega$ be a tuple of natural numbers with $\sum_{i=1}^{\omega} n_i = n$. Define
\[F_{n_1,\dots,n_{\omega}} (f) = \sum_{ \substack { (f_1,\dots, f_{\omega}) \in (\mathbb F_q[t]^+) ^{\omega} \\ \deg f_i = n_i \\ \prod_{i=1}^\omega f_i = f}} \prod_{i=1}^{\omega} \Lambda (f_i) .\]

Because $\Lambda$ is supported primarily on primes, $F_{n_1,\dots, n_{\omega}}$ is an arithmetic function supported mainly on polynomials whose prime factors have degrees $n_1,\dots, n_\omega$.  More precisely, for squarefree $f$, $F_{n_1,\dots, n_{\omega}}(f)$ is simply the indicator function of polynomials whose prime factors have degrees $n_1,\dots, n_\omega$, times a constant factor $\prod_{i=1}^\omega n_i \cdot \prod_{k=1}^{\infty} \abs{ \{ j \mid n_j =k \} } !$. 

If we wish to count polynomials whose prime factors have specified degrees in an arithmetic progression, it suffices to sum the functions $F_{n_1,\dots, n_{\omega}}(f)$ and adjust for the non-squarefree terms using a sieve. We will do exactly this, first handling in \cref{vm-c-1} the sums of  $F_{n_1,\dots, n_{\omega}}$ in arithmetic progressions, and later sieving.

This first lemma will help us work with Dirichlet convolutions such as $F_{n_1,\dots, n_{\omega}}$, and should help more generally in studying other Dirichlet convolutions:

\begin{lemma}\label{convolution-lemma} Let $n_1$ and $n_2$ be natural numbers with $n_1+n_2=n$. Let $\rho_1$ be a representation of $S_{n_1}$ and $\rho_2$ a representation of $S_{n_2}$. Let $\rho= \operatorname{Ind}_{S_{n_1} \times S_{n_2}}^{S_n} \rho_1 \otimes \rho_2$.  Then

\begin{enumerate}
\item We have \begin{equation}\label{Dirichlet-convolution-formula}  F_{\rho} (f) = \sum_{\substack{  f_1 \in  \mathcal M_{n_1}, f_2 \in \mathcal M_{n_2} \\ f_1 f_2 = f}}  F_{\rho_1}(f_1) F_{\rho_2}(f_2) .\end{equation}

\item We have \[ V_\rho = V_{\rho_1} \otimes V_{\rho_2} \] as representations of $GL_{m-1}$.

\item We have \begin{equation}\label{VD-convolution-formula} V^d_\rho = \bigoplus_{e=0}^d V^{e}_{\rho_1} \otimes V^{ d - e}_{\rho_2} \end{equation} as representations of $GL_m$. \end{enumerate}

\end{lemma}

\begin{proof}

\begin{enumerate}

\item By definition, $V_f$ is the free vector space generated by tuples $(a_1,\dots, a_n ) \in \overline{\mathbb F}_q$ with $\prod_{i=1}^{n}  (t-a_i) = f$. For each pair $f_1, f_2  \in \overline{\mathbb F_q}[t]^+$ with $\deg f_1 =n_1, \deg f_2 = n_2, f_1f_2=f$, there is a linear map $\mu_{f_1,f_2} \colon V_{f_1} \otimes V_{f_2} \to V_f$ that maps $(a_1,\dots, a_{n_1}) \otimes (b_{n_1},\otimes , b_{n_2})$ to $(a_1,\dots, a_{n_1},b_1,\dots, b_{n_2})$. Summing the $\mu_{f_1,f_2}$ defines a linear map
\[\overline{\mu} \colon \bigoplus_{ \substack{ f_1, f_2 \in \overline{\mathbb F_q}[t]^+ \\ \deg f_i = n_i \\ f_1f_2 =f}} V_{f_1} \otimes V_{f_2} \to V_f.\]
We can check that $\overline{\mu}$ is an isomorphism since the basis vectors of $V_f$ are in its image, the individual maps $\mu_{f_1,f_2}$ are injective, and the image of each map $\mu_{f_1,f_2}$ is spanned by a disjoint set of basis vectors, so $\overline{\mu}$ is injective as well. Thus
\[ V_f = \bigoplus_{ \substack{ f_1, f_2 \in \overline{\mathbb F_q}[t]^+ \\ \deg f_i = n_i \\ f_1f_2 =f}} V_{f_1} \otimes V_{f_2} .\]
Hence by Frobenius reciprocity \[ (V_f \otimes \rho  )^{S_n} = \Bigl(V_f \otimes \operatorname{Ind}_{S_{n_1} \times S_{n_2}}^{S_n} ( \rho_1  \otimes \rho_2 )\Bigr)^{S_n} =  (V_f \otimes \rho_1  \otimes \rho_2 )^{S_{n_1} \times S_{n_2} } \] \[ =  \bigoplus_{ \substack{ f_1, f_2 \in \overline{\mathbb F_q}[t]^+ \\ \deg f_i = n_i \\ f_1f_2 =f}}  ( V_{f_1} \otimes V_{f_2} \otimes \rho_1 \otimes \rho_2)^{S_{n_1} \times S_{n_2} } =  \bigoplus_{ \substack{ f_1, f_2 \in \overline{\mathbb F_q}[t]^+ \\ \deg f_i = n_i \\ f_1f_2 =f}}  ( V_{f_1}  \otimes \rho_1 )^{S_{n_1}}  \otimes (V_{f_2}  \otimes \rho_2 )^{ S_{n_2} }. \] 

Now consider how Frobenius acts on this sum. Frobenius sends \[ ( V_{f_1}  \otimes \rho_1  )^{S_{n_1} } \otimes (V_{f_2}  \otimes \rho_2 )^{ S_{n_2} }\] to \[ ( V_{\operatorname{Frob}_q(f_1)}  \otimes \rho_1  )^{S_{n_1}} \otimes (V_{\operatorname{Frob}_q(f_2)}  \otimes \rho_2 )^{ S_{n_2} }.\]  Thus, Frobenius permutes the summands of $(V_f \otimes \rho  )^{S_n}$ and hence the trace of Frobenius on $(V_f \otimes \rho  )^{S_n}$ is the sum of the trace of Frobenius on each invariant summand. The summand associated to $f_1,f_2$ is invariant if and only if $f_1, f_2 \in \mathbb F_q[t]$, in which case the trace of Frobenius on that summand is  $ F_{\rho_1}(f_1) F_{\rho_2}(f_2) $. This gives \eqref{Dirichlet-convolution-formula}. 

\item  We have \[ \Bigl (  (\mathbb C^m)^{n} \otimes \rho \otimes \sgn \Bigr )^{S_n} =   \Bigl (  (\mathbb C^m)^{n} \otimes \operatorname{Ind}_{S_{n_1} \times S_{n_2}}^{S_n}( \rho_1 \otimes \rho_2)
 \otimes \sgn \Bigr )^{S_n} \] \[ =  \Bigl (   \operatorname{Ind}_{S_{n_1} \times S_{n_2}}^{S_n} \Bigl ( (\mathbb C^m)^{n_1} \otimes (\mathbb C^m)^{n_2} \otimes \rho_1 \otimes \rho_2 \otimes \sgn_{S_{n_1}} \otimes \sgn_{S_{n_2}} \Bigr) \Bigr)^{S_n} \] \[ =   \Bigl( (\mathbb C^m)^{n_1} \otimes  \rho_1 \otimes\sgn_{S_{n_1}} \Bigr)^{S_{n_1}} 
 \otimes  \Bigl( (\mathbb C^m)^{n_2} \otimes  \rho_2 \otimes  \sgn_{S_{n_2}}  \Bigr)^{S_{n_2} }= V_{\rho_1} \otimes V_{\rho_2} . \] 
\item By definition, \[ V_\rho^d = \Bigl (  (\mathbb C^m)^{\otimes n-d} \otimes \operatorname{Ind}_{S_{n_1} \times S_{n_2}}^{S_n} ( \rho_1  \otimes \rho_2 ) \otimes \sgn_{S_{n-d} }  \Bigr)^{S_{n-d} \times S_d} .\]

%As a representation of $S_d \times S_{n-d}$,  we can represent \[\Ind_{S_{n_1} \times S_{n_2}}^{S_n}   (\rho_1   \otimes \rho_2 )\] as \[ \bigoplus_{g \in (S_d \times S_{n-d} ) \backslash S_n / (S_{n_1} \times S_{n_2}) } \operatorname{Ind}_{ (S_{n_1} \times S_{n_2}) \cap  g^{-1} (S_d \times S_{n-d} )  g }^{ S_d \times S_{n-d} }  (\rho_1   \otimes \rho_2)  \] \[ = \bigoplus_{e =\max(0, d- n_2 ) }^{ \min(n_1 ,d) }   \operatorname{Ind}_{ S_e \times S_{d-e} \times S_{n_1-e} \times  S_{n_2-d+e}  }^{S_d\times S_{n-d}}  (\rho_1   \otimes \rho_2) \] where we view $\rho_1  \otimes \rho_2$ as a representation of $S_e \times S_{d-e} \times S_{n_1-e} \times  S_{n_2-d+e} $ by embedding $S_e \times S_{n_1-e}$ into $S_{n_1}$ and $S_{d-e} \times S_{n_2-d+e}$ into $S_{n_2}$, and we view $S_e \times S_{d-e} \times S_{n_1-e} \times  S_{n_2-d+e} $ as a subgroup of $S_d \times S_{n-d}$ by embedding  $S_e \times S_{d-e} $ into $S_d$ and $ S_{n_1-e} \times  S_{n_2-d+e} $  into $S_{n-d}$. This is because such double cosets are classified by the number $e$ of elements of $\{1,\dots, k\}$ whose image under $g\colon  \{1,\dots n\} \to \{1,\dots ,n\}$ is contained in $\{1,\dots,d\}$.

 By \cref{induction-restriction}, as a representation of $S_{n-d} \times S_d$, 
\[  (\mathbb C^m)^{\otimes n-d} \otimes \operatorname{Ind}_{S_{n_1} \times S_{n_2}}^{S_n} ( \rho_1  \otimes \sgn_{S_{n_1} } \otimes \rho_2\otimes \sgn_{S_{n_2}} ) \otimes \sgn_{S_{n-d} }  \]
\[ = \bigoplus_{e =\max(0, d- n_2 ) }^{ \min(n_1 ,d) }   \operatorname{Ind}_{ S_e \times S_{d-e} \times S_{n_1-e} \times  S_{n_2-d+e}  }^{S_d\times S_{n-d}}  \Bigl( (\mathbb C^m)^{ (n_1-e) + (n_2 -d+e) } \otimes \rho_1  \otimes \sgn_{S_{n_1-e} } \otimes \rho_2\otimes \sgn_{S_{n_2-d+e}} )  \] 
and thus
\[ \Bigl (  (\mathbb C^m)^{\otimes n-d} \otimes \operatorname{Ind}_{S_{n_1} \times S_{n_2}}^{S_n} ( \rho_1\otimes \rho_2\otimes \sgn_{S_{n_2}} ) \Bigr)^{ S_{n-d} \times S_d} \]
\[ = \bigoplus_{e =\max(0, d- n_2 ) }^{ \min(n_1 ,d) } \Bigl( (\mathbb C^m)^{ (n_1-e) + (n_2 -d+e) } \otimes \rho_1  \otimes \sgn_{S_{n_1-e} } \otimes \rho_2\otimes \sgn_{S_{n_2-d+e}} )^{ S_e \times S_{d-e} \times S_{n_1-e} \times  S_{n_2-d+e} } \]
\[ = \bigoplus_{e =\max(0, d- n_2 ) }^{ \min(n_1 ,d) }  \Bigl ( (\mathbb C^m)^{n_1-e} \otimes \rho_1 \otimes \sgn_{S_{n_1-e} }  \Bigr) ^{ S_{e} \times S_{n_1-e}} \otimes \Bigl ( (\mathbb C^m)^{n_2-d +e } \otimes \rho_2 \otimes \sgn_{S_{n_2-d+e} }  \Bigr) ^{ S_{d-e} \times S_{n_2-d +e }}  \]
\[ = \bigoplus_{e =\max(0, d- n_2 ) }^{ \min(n_1 ,d) } V^e_{\rho_1} \otimes V^{d-e}_{\rho_2 } = \bigoplus_{e =0 }^{ d} V^e_{\rho_1} \otimes V^{d-e}_{\rho_2}   \] 
since the terms for $e>n_1$ and $e< d-n_2$ vanish, giving \eqref{VD-convolution-formula}. 
\end{enumerate} 
\end{proof}

\begin{lemma}\label{vm-c-1} 
For $g$ squarefree of degree $m \geq 2$, a natural number $n \geq m$, a tuple $n_1,\dots, n_{\omega}$ of natural numbers summing to $n$, and a residue class $a \in (\mathbb F_q[t]/g)^\times$, we have \[\left | \sum_{ \substack{ f \in \mathcal M_n \\  f \equiv a \mod g }}  F_{n_1,\dots, n_\omega} (f)  -  \frac{1}{ \phi(g) }  \sum_{ \substack{ f \in\mathcal M_n  }}  F_{n_1,\dots, n_\omega} (f)  \right| \] \[ \leq  \binom{n}{n_1,\dots, n_\omega}  \Bigl ( 2^{m+1-\omega}  \binom{ n+2m +1 } { 3m+1} q^{ \frac{n-m}{2}}  + 2^{ m-\omega} \binom{ n+2m-3  }{ 3m-4}  q^{\frac{n+1-m}{2}}  \Bigr)   . \]
\end{lemma}

\begin{proof} By \cref{mangoldt-alternating}, $\Lambda$, restricted to $\mathcal M_{n_i}$, is $\sum_{i=0}^{n_i-1 }(-1)^i  F_{\wedge^i (\operatorname{std}_{n_i} )} $. Thus by \cref{convolution-lemma}(1) we have
\begin{equation}\label{F-alternating}  F_{n_1,\dots, n_\omega} = \sum_{j_1=0}^{n_1} \sum_{j_2=0}^{n_2} \dots \sum_{j_\omega=0}^{n_\omega} (-1)^{j_1 + \dots + j_{\omega}} F_{ \operatorname{Ind}_{S_{n_1} \times \dots \times S_{n_\omega}}^{ S_n} \wedge ^{j_1}( \operatorname{std}_{n_1} ) \otimes \dots \otimes \wedge ^{j_\omega}( \operatorname{std}_{n_\omega} ) }  .\end{equation} Letting \[ \rho_{n_1,\dots,n_\omega}= \bigoplus_{j_1=0}^{n_1} \bigoplus_ {j_2=0}^{n_2} \dots \bigoplus_{j_\omega=0}^{n_\omega}  \operatorname{Ind}_{S_{n_1} \times \dots \times S_{n_\omega}}^{ S_n} \wedge ^{j_1}( \operatorname{std}_{n_1} ) \otimes \dots \otimes \wedge ^{j_\omega}( \operatorname{std}_{n_\omega} )  ,\] we have by \eqref{F-alternating}, the triangle inequality, and \cref{squarefree-bound-precise},
\[ \left | \sum_{ \substack{ f \in \mathcal M_n \\  f \equiv a \mod g }}  F_{n_1,\dots, n_\omega} (f)  -  \frac{1}{ \phi(g) }  \sum_{ \substack{ f \in\mathcal M_n  }}  F_{n_1,\dots, n_\omega} (f)  \right| 
%\[  \sum_{j_1=0}^{n_1} \sum_{j_2=0}^{n_2} \dots \sum_{j_\omega=0}^{n_\omega} \left | \sum_{ \substack{ f \in \mathcal M_n \\  f \equiv a \mod g }}  F_{ \operatorname{Ind}_{S_{n_1} \times \dots \times S_{n_\omega}}^{ S_n} \wedge ^{j_1}( \operatorname{std}_{n_1} ) \otimes \dots \otimes \wedge ^{j_\omega}( \operatorname{std}_{n_\omega} ) } (f)  -  \frac{1}{ \phi(g) }  \sum_{ \substack{ f \in\mathcal M_n  }}  F_{ \operatorname{Ind}_{S_{n_1} \times \dots \times S_{n_\omega}}^{ S_n} \wedge ^{j_1}( \operatorname{std}_{n_1} ) \otimes \dots \otimes \wedge ^{j_\omega}( \operatorname{std}_{n_\omega} ) } (f)  \right| \]
%\[ = \sum_{j_1=0}^{n_1} \sum_{j_2=0}^{n_2} \dots \sum_{j_\omega=0}^{n_\omega}  2 ( C_1( \operatorname{Ind}_{S_{n_1} \times \dots \times S_{n_\omega}}^{ S_n} \wedge ^{j_1}( \operatorname{std}_{n_1} ) \otimes \dots \otimes \wedge ^{j_\omega}( \operatorname{std}_{n_\omega} )) + C_2( \operatorname{Ind}_{S_{n_1} \times \dots \times S_{n_\omega}}^{ S_n} \wedge ^{j_1}( \operatorname{std}_{n_1} ) \otimes \dots \otimes \wedge ^{j_\omega}( \operatorname{std}_{n_\omega} ) ) \sqrt{q} q^{\frac{n-m}{2} }\]
\leq 2( C_1(\rho_{n_1,\dots,n_\omega} )  + C_2(\rho_{n_1,\dots,n_\omega}) \sqrt{q} ) q^{\frac{n-m}{2} }.\]

By \cref{convolution-lemma}(2), \[ V_{\rho_{n_1,\dots,n_\omega}} = \bigotimes_{i=1}^\omega V_{\rho_{n_i}} \]  
and thus, by \eqref{vm-V-gf}, %\[  \tr (\Diag (\lambda_1,\dots, \lambda_m) , V_{\rho_{n_i}} ) =  \frac{1}{2}  \prod_{j=1}^{m-1} \frac{ 1+u \lambda _j }{ 1- u\lambda_j } [ u^{n_i} ] \]
 \[   \tr (\Diag (\lambda_1,\dots, \lambda_m) , V_{\rho_{n_1,\dots,n_\omega}} )  = \frac{1}{2^\omega}  \prod_{i=1}^\omega \prod_{j=1}^{m-1} \frac{ 1+u_i \lambda _j }{ 1- u_i \lambda_j } \Bigl[ \prod_{i=1}^{\omega} u_i^{n_i} \Bigr]  .\]
For power series $F, G$ in $u_1,\dots, u_\omega ,\lambda_1,\lambda_j$, we write $F \preccurlyeq G$ if the coefficient in $F$ of each monomial in the $u_1,\dots, u_\omega, \lambda_j$ is less than or equal to the coefficient in $G$ of the same monomial. Observe that 
\begin{equation}\label{pce-equation} \prod_{i=1}^\omega  \frac{ 1+u_i \lambda _j }{ 1- u_i \lambda_j }  \preccurlyeq \prod_{i=1}^{\omega} \frac{1}{ (1- u_i \lambda_j)^2} \preccurlyeq   \frac{1}{ (1 - (\sum_{i=1}^{\omega} u_i )\lambda_j)^2} \end{equation}  This gives
\[ C_2 (\rho_{n_1,\dots, n_\omega} ) = \frac{ \partial^{m-1}}{ \partial \lambda_1 \dots \partial \lambda_{m-1}}   \tr (\Diag (\lambda_1,\dots, \lambda_m) , V_{\rho_{n_1,\dots,n_\omega}} ) \Big|_{\lambda_1,\dots,\lambda_{m-1}=1}  \] \[ \leq   \frac{ \partial^{m-1}}{ \partial \lambda_1 \dots \partial \lambda_{m-1}}  \frac{1}{2^\omega} \prod_{j=1}^{m-1}   \frac{1}{ (1 - (\sum_{i=1}^{\omega} u_i )\lambda_j)^2} \Bigl[ \prod_{i=1}^{\omega} u_i^{n_i} \Bigr]  \Big|_{\lambda_1,\dots,\lambda_{m-1}=1}.\]
For any power series $S$ in a variable $u$, we have \begin{equation}\label{ui-u-eq} S\Bigl( \sum_{i=1}^{\omega} u_i\Bigr) \Bigl[ \prod_{i=1}^{\omega} u_i^{n_i} \Bigr] = \binom{n}{n_1,\dots, n_\omega}  S(u) [u^n].\end{equation} Applying this, we get
\[ C_2 (\rho_{n_1,\dots, n_\omega} ) \leq    \binom{n}{n_1,\dots, n_\omega}   \frac{ \partial^{m-1}}{ \partial \lambda_1 \dots \partial \lambda_{m-1}}  \frac{1}{2^\omega} \prod_{j=1}^{m-1}   \frac{1 }{ (1-u \lambda_j )^2} [ u^n ]   \Big|_{\lambda_1,\dots,\lambda_{m-1}=1}\] \[=  \binom{n}{n_1,\dots, n_\omega}   \frac{1}{2^\omega} \prod_{j=1}^{m-1}   \frac{2u }{ (1-\lambda_j u)^3}   [u^n]  \Big|_{\lambda_1,\dots,\lambda_{m-1}=1} =  \binom{n}{n_1,\dots, n_\omega}   2^{ m-1-\omega} \frac{ u^{m-1}}{ (1-u)^{3m-3} } [u^n] \] \[=  \binom{n}{n_1,\dots, n_\omega}   2^{ m-1-\omega} \binom{ n+2m-3  }{ 3m-4}.\]

Furthermore, by \cref{convolution-lemma}(3) and \eqref{vm-Vd-gf}, %we have \[  \tr (\Diag (\lambda_1,\dots, \lambda_m) , V_{\rho_{n_i}} ^d ) =  \frac{1}{2}  \frac{1+vu}{1-uv} \prod_{j=1}^m \frac{ 1+u \lambda _j }{ 1- u\lambda_j } [ u^{n_i} v^d ] \]
we have \[   \tr (\Diag (\lambda_1,\dots, \lambda_m) , V^d_{\rho_{n_1,\dots,n_\omega}} )  = \frac{1}{2^\omega}  \prod_{i=1}^\omega  \frac{1+vu_i}{1-u_i v} \prod_{j=1}^{m} \frac{ 1+u_i \lambda _j }{ 1- u_i \lambda_j } \Bigl[ v^d  \prod_{i=1}^{\omega} u_i^{n_i} \Bigr]  .\]  Thus by \cref{weak-C1-bound}, 
\[ C_1(\rho) \leq  \sum_{d=0}^n  \frac{\partial^m}{ \partial \lambda_1 \dots \partial \lambda_m } \frac{1}{2^\omega}  \prod_{i=1}^\omega  \frac{1+vu_i}{1-u_i v} \prod_{j=1}^{m} \frac{ 1+u_i \lambda _j }{ 1- u_i \lambda_j } \Bigl[v^d  \prod_{i=1}^{\omega} u_i^{n_i} \Bigr]   \Big|_{\lambda_1,\dots,\lambda_{m}=1}\]
(because summing the coefficient of $v^d$ over $d$ is equivalent to evaluating at $v=1$)
\[= \frac{1}{2^\omega}  \frac{\partial^m}{ \partial \lambda_1 \dots \partial \lambda_m }   \prod_{i=1}^\omega  \frac{1+u_i}{1-u_i } \prod_{j=1}^{m} \frac{ 1+u_i \lambda _j }{ 1- u_i \lambda_j } \Bigl[  \prod_{i=1}^{\omega} u_i^{n_i} \Bigr]  \Big|_{\lambda_1,\dots,\lambda_{m}=1}\]
(applying \eqref{pce-equation} and the similar inequality $\prod_{i=1}^{\omega} \frac{1+u_i}{1-u_i} \preccurlyeq \frac{1}{ (1- \sum_{i=1}^{\omega} u_i)^2} $ )
\[ \leq \frac{1}{2^\omega}   \frac{\partial^m}{ \partial \lambda_1 \dots \partial \lambda_m }   \frac{1}{ (1- \sum_{i=1}^{\omega} u_i)^2 } \prod_{j=1}^{m} \frac{ 1}{ (1-(\sum_{i=1}^{\omega}  u_i)  \lambda_j )^2 } \Bigl[  \prod_{i=1}^{\omega} u_i^{n_i} \Bigr] \Big|_{\lambda_1,\dots,\lambda_{m}=1}\]
(applying \eqref{ui-u-eq})
\[ \leq \binom{n}{ n_1,\dots, n_\omega}  \frac{1}{2^\omega}   \frac{\partial^m}{ \partial \lambda_1 \dots \partial \lambda_m } \frac{1}{ (1-u)^2} \prod_{j=1}^m \frac{1}{(1- u \lambda_j)^2}  [ u^n ]  \Big|_{\lambda_1,\dots,\lambda_{m}=1}\]
\[ = \binom{n}{ n_1,\dots, n_\omega}  \frac{1}{2^\omega}    \frac{1}{ (1-u)^2} \prod_{j=1}^m \frac{2u}{(1- u \lambda_j)^3}  [ u^n ]  \Big|_{\lambda_1,\dots,\lambda_{m}=1}\]
\[ = \binom{n}{ n_1,\dots, n_\omega}  2^{m-\omega} \frac{u^m }{ (1-u)^{3m+2} }  [ u^n ]  =\binom{n}{ n_1,\dots, n_\omega}  2^{m-\omega}  \binom{ n+2m +1 } { 3m+1}  . \qedhere\]
\end{proof} 

We now give a variant sum that replaces the binomial coefficients with simpler exponentials. We keep the multinomial coefficients, because their multiplicative properties will be helpful to us.

For our sieve argument, to study polynomials of degree $n$, we will sometimes have to study polynomials of degree less than $n$, and the most convenient way to study this will be to introduce the variable $c$ and work with a polynomial of degree $n-c$.

\begin{lemma}\label{vm-c-2} 
For $\theta$ a real number between $0$ and $1$, $g$ squarefree of degree $m \geq 2$, natural numbers $c,n$ with $c \leq n$ and $m \leq \theta n$, a tuple $n_1,\dots, n_{\omega}$ of natural numbers summing to $n-c $, and a residue class $a \in (\mathbb F_q[t]/g)^\times$, we have \begin{equation}\label{eq-vm-c-2} \begin{split} & \left | \sum_{ \substack{ f \in \mathcal M_{n-c} \\  f \equiv a \mod g }}  F_{n_1,\dots, n_\omega} (f)  -  \frac{1}{ \phi(g) }  \sum_{ \substack{ f \in\mathcal M_{n-c} }}  F_{n_1,\dots, n_\omega} (f)  \right| \\ \ll &    \frac{(n-c) !}{ \prod_{i=1}^\omega (n_i)!}  \left(\frac{1-\theta}{ (1+2 \theta)  \sqrt{q} }\right)^c \left( \frac{(1+2\theta) e \sqrt{q}  2^{ \frac{\theta}{1-\theta}} }{ 1-\theta} \right)^{n-m}  \end{split}\end{equation} where the implicit constant depends only on $q,\theta$.
\end{lemma}

\begin{proof} By \cref{vm-c-1} the left side of \eqref{eq-vm-c-2} is
\begin{equation}\label{vm-c-1-summary} \leq  \frac{(n-c)!}{ \prod_{i=1}^\omega (n_i)!}  \Bigl ( 2^{m+1-\omega}  \binom{ n-c +2m +1 } { 3m+1} q^{\frac{n-c-m}{2}}   + 2^{ m-\omega} \binom{ n-c+2m-3  }{ 3m-4}  q^{\frac{n+1-m}{2}}  \Bigr)   . \end{equation}

To bound $  \binom{ n-c +2m +1 } { 3m+1}$ we apply \cref{binomial-bound} with $a=n-m$, $b=3m$,  $x=-c$, $y=1$, $\alpha=\frac{1-\theta}{3\theta}$ to obtain
\[2^{m+1 -\omega}  \binom{ n-c +2m +1 } { 3m+1} q^{\frac{n-c-m}{2}}  \leq 2^{-\omega}  2^{m} \frac{ 1+2\theta}{3\theta }  \left(\frac{1-\theta}{ 1+2 \theta }\right)^c \left( \frac{(1+2\theta) e}{ 1-\theta} \right)^{n-m} q^{  \frac{n-m}{2} } q^{- \frac{c}{2} } \]
\[\leq  2^{-\omega}  \frac{ 2+4\theta}{3\theta }  \left(\frac{1-\theta}{ (1+2 \theta)  \sqrt{q} }\right)^c \left( \frac{(1+2\theta) e \sqrt{q}  2^{ \frac{\theta}{1-\theta}} }{ 1-\theta} \right)^{n-m} .\]

To bound $ \binom{ n-c+2m-3  }{ 3m-4}  $ we take $a=n-m$, $b=3m$, $x=1-c$,  $y=-4$, $\alpha=\frac{1-\theta}{3\theta}$  to obtain

\[ 2^{ m-\omega} \binom{ n-c+2m-3  }{ 3m-4}  q^{\frac{n+1-m}{2}}  \leq  2^{- \omega} 2^m   \frac{ (3\theta)^4}{ (1+2\theta)^3 (1-\theta) }   \left(\frac{1-\theta}{ (1+2 \theta) }\right)^c\left( \frac{(1+2\theta) e}{ 1-\theta} \right)^{n-m} q^{  \frac{n-m}{2} } q^{- \frac{c}{2} }  q^{- \frac{1}{2}} \]
\[ = 2^{-\omega}  \frac{ (3\theta)^4}{  (1+2\theta)^3 (1-\theta) \sqrt{q}  }  \left(\frac{1-\theta}{ (1+2 \theta)  \sqrt{q} }\right)^c \left( \frac{(1+2\theta) e \sqrt{q}  2^{ \frac{\theta}{1-\theta}} }{ 1-\theta} \right)^{n-m}.\]

Dropping the constant terms and $2^{-\omega}$, we obtain 
\[ \Bigl ( 2^{m+1-\omega}  \binom{ n-c +2m +1 } { 3m+1} q^{\frac{n-c-m}{2}}   + 2^{ m-\omega} \binom{ n-c+2m-3  }{ 3m-4}  q^{\frac{n+1-m}{2}}  \Bigr) \] \[  \ll \left(\frac{1-\theta}{ (1+2 \theta)  \sqrt{q} }\right)^c \left( \frac{(1+2\theta) e \sqrt{q}  2^{ \frac{\theta}{1-\theta}} }{ 1-\theta} \right)^{n-m} .\] 

Combined with \eqref{vm-c-1-summary}, we obtain \eqref{eq-vm-c-2}.
\end{proof}

We would like to use this estimate to count polynomials with a given prime factorization. There are two changes we must make to $F_{n_1,\dots, n_\omega}$ to count squarefree polynomials with a  factorization into primes of fixed degrees. First, we must replace the von Mangoldt functions with the indicator functions of the primes, replacing $F_{n_1,\dots, n_\omega}$ with a function supported on products of exactly $\omega$ primes, with the $i$th prime of degree $n_i$. Second, we must remove the terms where these $\omega$ primes are not all distinct. To count non-squarefree polynomials with a factorization into primes of fixed degrees and multiplicities, we will first sum over the possibilities for the maximum squareful part, which reduces the problem to counting squarefree polynomials prime to a given squareful polynomial. It will be convenient to add this additional coprimality condition in the first step.

Let $h$ be a polynomial of degree $c\leq n$ and let $n_1,\dots, n_\omega$ be a tuple of natural numbers with $\sum_{i=1}^{\omega} n_i = n-c$. Define
\[G_{n_1,\dots,n_{\omega}}^h (f) = \sum_{ \substack { (f_1,\dots, f_{\omega}) \in (\mathbb F_q[t]^+) ^{\omega} \\ \deg f_i = n_i \\ f_i \textrm{ prime} \\ f_i \nmid h  \\ \prod_{i=1}^\omega f_i = f}} \prod_{i=1}^{\omega} n_i  .\]

\begin{lemma}\label{vm-c-s-1} 
For $\theta$ a real number between $0$ and $1$, $g$ squarefree of degree $m \geq 2$, natural numbers $c,n$ with $c \leq n$ and $m \leq \theta n$, a tuple $n_1,\dots, n_{\omega}$ of natural numbers summing to $n-c $, a residue class $a \in (\mathbb F_q[t]/g)^\times$, and a polynomial $h$ of degree $c$, we have \begin{equation}\label{eq-vm-c-s-1} \begin{split} & \Bigl | \sum_{ \substack{ f \in \mathcal M_{n-c} \\  f \equiv a \mod g }}  G_{n_1,\dots,n_{\omega}}^h(f)  -  \frac{1}{ \phi(g) }  \sum_{ \substack{ f \in\mathcal M_{n-c} }}  G_{n_1,\dots,n_{\omega}}^h (f)  \Bigr| \\ & \ll     e^{ \sqrt{c} } \frac{(n-c) !}{ \prod_{i=1}^\omega (n_i)!}   \left(\frac{1-\theta}{ (1+2 \theta)  \sqrt{q} }\right)^c \left( \frac{(1+2\theta) e \sqrt{q}  2^{ \frac{\theta}{1-\theta}} }{ 1-\theta} \right)^{n-m} . \end{split} \end{equation} 
\end{lemma}

\begin{proof}  We have
\[G_{n_1,\dots,n_{\omega}}^h (f) = \sum_{ \substack { (f_1,\dots, f_{\omega}) \in (\mathbb F_q[t]^+) ^{\omega} \\ \deg f_i = n_i \\ \prod_{i=1}^\omega f_i = f}} \prod_{i=1}^{\omega} \Bigl( \Lambda(f_i) - \sum_{ \substack{ \pi \textrm{ prime}\\ r \in \mathbb N \\ r \deg \pi = n_i \\ r>1 \textrm{ or } \pi \mid h}}  (\deg \pi  )\cdot \delta_{f_i = \pi^r}  \Bigr) \]
\[= \sum_{ \substack { S \subseteq \{1,\dots, \omega\} \\ (\pi_i)_{i\in S} \textrm { prime} \\ (r_i)_{i \in S} \in \mathbb N^S \\  r_i \deg \pi_i= n_i \\ r_i>1 \textrm{ or } \pi_i \mid h}} (-1)^{\abs{S}} F_{ (n_i)_{i\notin S} } \Bigl (\frac{f}{\prod_{i \in S} \pi_i^{r_i} }\Bigr)   \prod_{i \in S} \deg \pi_i  .\]
Applying \cref{vm-c-2} to each term, ignoring those for which $\prod_{i\in S} \pi_i^{r_i}$ is not coprime to $g$, we obtain
\begin{equation}\label{first-prime-sieve} \begin{split} &  \Bigl| \sum_{ \substack{ f \in \mathcal M_{n-c} \\  f \equiv a \mod g }}  G_{n_1,\dots,n_{\omega}}^h(f)  -  \frac{1}{ \phi(g) }  \sum_{ \substack{ f \in\mathcal M_{n-c} }}  G_{n_1,\dots,n_{\omega}}^h (f)  \Bigr| \\ & \ll \sum_{ \substack { S \subseteq \{1,\dots, \omega\} \\ (\pi_i)_{i\in S} \textrm { prime} \\ (r_i)_{i \in S} \in \mathbb N^S \\  r_i \deg \pi_i= n_i \\ r_i>1 \textrm{ or } \pi_i \mid h}} \frac{ (n-c- \sum_{i \in S} n_i)!}{ \prod_{i \notin S} (n_i)!}  \left(\frac{1-\theta}{ (1+2 \theta)  \sqrt{q} }\right)^{c + \sum_{i \in S} n_i }  \left( \frac{(1+2\theta) e \sqrt{q}  2^{ \frac{\theta}{1-\theta}} }{ 1-\theta} \right)^{n-m} \prod_{i \in S} \deg \pi_i   \\
&=   \frac{ (n-c)!} { \prod_{i=1}^\omega (n_i)!}  \left(\frac{1-\theta}{ (1+2 \theta)  \sqrt{q} }\right)^{c }  \left( \frac{(1+2\theta) e \sqrt{q}  2^{ \frac{\theta}{1-\theta}} }{ 1-\theta} \right)^{n-m} \\
& \times  \sum_{  S \subseteq\{1,\dots, \omega\} }    \frac{ (n-c- \sum_{i \in S} n_i)! \prod_{i \in S} (n_i)! }{( n-c)! }  \prod_{i \in S} \Biggl(  \left(\frac{1-\theta}{ (1+2 \theta)  \sqrt{q} }\right)^{n_i}   \sum_{ \substack{ \pi \textrm{ prime}, r \in \mathbb N \\ r \deg \pi = n_i \\ r>1 \textrm{ or } \pi \mid h}}   \deg \pi  \Biggr).  
\end{split} \end{equation} 

We have the bound
\[ \sum_{ \substack{ \pi \textrm{ prime}\\  r \in \mathbb N \\ r \deg \pi = n_i \\ r>1 \textrm{ or } \pi \mid h}}   \deg \pi=   \sum_{ \substack{ \pi \textrm{ prime}\\  r \in \mathbb N \\ r \deg \pi = n_i \\ r>1 }}   \deg \pi +  \sum_{ \substack{ \pi \textrm{ prime} \\  \deg \pi = n_i \\  \pi \mid h}}   \deg \pi \] \[ \leq \frac{q^{ \frac{n_i}{2}}} { 1- \frac{1}{q}}  +  \max ( c, q^{n_i} ) \leq \frac{q^{ \frac{n_i}{2}}} { 1- \frac{1}{q}}  +  \sqrt {  c q^{n_i} } = q^{ \frac{n_i}{2}} \Bigl( \sqrt{c} + \frac{q}{q-1} \Bigr)\] 
so
\[ \left(\frac{1-\theta}{ (1+2 \theta)  \sqrt{q} }\right)^{n_i}   \sum_{ \substack{ \pi \textrm{ prime}\\ r \in \mathbb N \\ r \deg \pi = n_i \\ r>1 \textrm{ or } \pi \mid h}} \deg \pi   \leq q^{ - \frac{n_i}{2}} \sum_{ \substack{ \pi \textrm{ prime}\\ r \in \mathbb N \\ r \deg \pi = n_i \\ r>1 \textrm{ or } \pi \mid h}} \deg \pi    \leq \sqrt{c} + \frac{q}{q-1} .\]
Let $\overline{S}$ denote the complement of $S$. We will next need to use the inequality.
\begin{equation}\label{easy-combinatorics} \frac {  \omega!}  { \abs{ \overline{S}} !}  \leq \frac{( n-c)! }{ (n-c- \sum_{i \in S} n_i)! \prod_{i \in S} (n_i)! }.\end{equation}
To establish \eqref{easy-combinatorics}, interpret its right side as counting partitions of $\{1,\dots,n-c\}$ into subsets of size $n_i$ for $i \in S$ plus one remaining subset.  Interpret the left side as the cardinality of the set $S_{\omega}^{\overline{S}}$ of permutations $\sigma\in S_\omega$ such that $\sigma(i)<\sigma(j)$ if $i<j$ and $\sigma(i),\sigma(j)\in \overline{S}$. For $\sigma \in S(\omega)^{\overline{S}}$, let $\alpha(\sigma)$ be the partition where we put $ 1+ \sum_{j<i} n_{\sigma(j)}, 2+ \sum_{j<i} n_{\sigma(j)},\dots,  \sum_{j\leq i} n_{\sigma(j) }$ into the subset of size $n_{\sigma(i)}$ if $\sigma(i) \in S$ and into the remaining subset otherwise. If $\alpha(\sigma_1)=\alpha(\sigma_2) $ then we can check $\sigma_1(i) =\sigma_2(i)$ for all $i$ by induction on $i$, showing that $\alpha$ is injective, giving \eqref{easy-combinatorics}.

Thus we have the bound
\[\frac{ (n-c- \sum_{i \in S} n_i)! \prod_{i \in S} (n_i)! }{( n-c)! } \leq \frac { \abs{ \overline{S}} !}  {  \omega!}.\]
Therefore, using the fact that the number of $ S \subseteq \{1,\dots,\omega\}$ of size $s$ is $\frac{\omega!}{s!( \omega-s)!} $, we obtain
\[ \sum_{  S \subseteq\{1,\dots, \omega\} }    \frac{ (n-c- \sum_{i \in S} n_i)! \prod_{i \in S} (n_i)! }{( n-c)! }  \prod_{i \in S} \Biggl(  \left(\frac{1-\theta}{ (1+2 \theta)  \sqrt{q} }\right)^{n_i}   \sum_{ \substack{ \pi \textrm{ prime}, r \in \mathbb N \\ r \deg \pi = n_i \\ r>1 \textrm{ or } \pi \mid h}}   \deg \pi_i  \Biggr).\]
\[ \leq  \sum_{  S \subseteq\{1,\dots, \omega\} }  \frac { \abs{ \overline{S}} !}  {  \omega! } \Bigl( \sqrt{c} + \frac{q}{q-1} \Bigr)^{ \abs{S}} \]
\[ \leq \sum_{s=0}^{\omega}  \frac{1}{ s!}  \Bigl( \sqrt{c} + \frac{1}{q-1} \Bigr)^s \leq e^{ \sqrt{c} + \frac{q}{q-1}} \ll e^{ \sqrt{c}}.\]
Plugging this into \eqref{first-prime-sieve}, we obtain \eqref{eq-vm-c-s-1}.
\end{proof} 

Let $h$ be a polynomial of degree $c\leq n$ and let $n_1,\dots, n_\omega$ be a tuple of natural numbers with $\sum_{i=1}^{\omega} n_i = n$. Define
\[H_{n_1,\dots,n_{\omega}}^h (f) = \sum_{ \substack { (f_1,\dots, f_{\omega}) \in (\mathbb F_q[t]^+) ^{\omega} \\ \deg f_i = n_i \\ f_i \textrm{ prime} \\ f_i \nmid h \\ f_i \textrm{ distinct} \\ \prod_{i=1}^\omega f_i = f}} \prod_{i=1}^{\omega} n_i  .\]
Note the definition is the same as $H_{n_1,\dots,n_{\omega}}^h (f) $ except that we require the prime factors $f_i$ be distinct. 

The inclusion-exclusion sum to extract $H_{n_1,\dots,n_{\omega}}^h (f) $ from $G_{n_1,\dots,n_{\omega}}^h (f) $ is straightforward, but notationally complicated. We explain it first, and then bound sums of $H_{n_1,\dots,n_{\omega}}^h (f) $ in progressions, after proving a purely combinatorial inequality.

Let $\mathcal P$ be the set of partitions $P$ of $\{1,\dots, \omega\}$ into disjoint nonempty subsets $S$, subject to the condition that for each $S \in P$, for $i,j\in S$, $n_i=n_j$. For $P \in \mathcal P$ and $S \in P$, let $n_S$ be the unique integer such that $n_i=n_S$ for all $i\in S$. For $P \in \mathcal P$, let $P^{\geq 2}$ be the set of sets $S \in P$ with cardinality at least two.  For $P \in \mathcal P$, let \[s(P) = \bigcup_{ S \in P^{\geq 2}} S \] and for a function $v$ from $P^{\geq 2}$ to $\mathbb F_q[t]$ let 
  \[ H(v) =  \prod_{S \in P^{\geq 2}} v(S)^{\abs{S} }. \] Define \[ d(P) = \sum_{ S \in P^{\geq 2} } \abs{S} n_S \] so that \[ \deg H(v) = d(P)\] if \[ \deg v(S) = n_S.\]

  \begin{lemma}\label{partition-sieve} For natural numbers $c,n$ with $c \leq n$, a tuple $n_1,\dots, n_{\omega}$ of natural numbers summing to $n-c $, and $h$ a polynomial of degree $c$, we have
 \begin{equation}\label{eq-partition-sieve}H_{n_1,\dots,n_{\omega}}^h(f) = \sum_{ P \in \mathcal P} (-1)^{ \abs{ P^{\geq 2}}} \sum_{ \substack{ v\colon P^{\geq 2} \to \mathbb F_q[t]^+ \\ v(S) \textrm{ prime}  \\ v(S) \nmid h \\  v\textrm{ injective } \\ \deg v(S) = n_S\\ H(v) \mid f}}  G^{h \cdot H(v) }_{ (n_i)_{i\notin  s(P) }}  \left( \frac{f}{ H(v)} \right) \prod_{i \in s(P) } n_i  . \end{equation} \end{lemma}
 
 \begin{proof} Expanding $G^{, h \cdot H(v)}$, we can express the right side of \eqref{eq-partition-sieve} as a sum over triples of $P\in \mathcal P$, $v\colon P^{\geq 2} \to \mathbb F_q[t]^+ $, and $f_i'$ for each $i\notin s(P)$ such that $\prod_{i \notin s(P) }  f_i' = \frac{f}{H(v)}$, $f_i'$ is prime of degree $n_i$, $f_i' \nmid h$, $v(S)$ is prime of degree $n_S$, $v(S) \nmid h$, and $v$ is injective.

We will define a bijection between triples $P, v , (f_i')_{i\notin s(P)}$ satisfying these conditions and pairs $(f_i)_{i = 1}^{\omega} , U$ where $f_i$ is prime of degree $n_i$, $f_i \nmid h$, $\prod_{i=1}^{\omega} f_i =f$, and $U$ is a subset of \[\{ \pi \in \mathbb F_q[t]^+ \mid \pi \textrm{ prime} , \abs{ \{ i \mid f_i = \pi \} } > 1  \}.\]

Indeed, to obtain $(f_i)_{i=1}^{\omega}, U$ from $P, v , (f_i')_{i\notin s(P)}$ we set $f_i = f_i'$ for $i \notin s(P)$, $f_i =v (S)$ for $i\in S$, and $U $ to be the image of $v$. Conversely, to obtain $P, v , (f_i')_{i\notin s(P)}$ from $(f_i)_{i=1}^{\omega}, U$ we take $P$ to be the set consisting of $ \{ i \mid f_i = \pi\ \}$ for $\pi \in U$ and $\{ i \}$ for all $i$ with $f_i \notin U$, $v$ to send $ \{ i \mid f_i = \pi\ \}$ to $\pi$, and $f_i'= f_i$ if $f_i \notin U$. 

It is routine to check that these maps are well-defined (i.e. satisfy all the conditions) and are inverses. Using this bijection, the right side of \eqref{eq-partition-sieve} is equal to
\[ \sum_{ \substack { (f_1,\dots, f_{\omega}) \in (\mathbb F_q[t]^+) ^{\omega} \\ \deg f_i = n_i \\ f_i \textrm{ prime} \\ f_i \nmid h  \\ \prod_{i=1}^\omega f_i = f}} \hspace{.1in} \sum_{ U \subseteq \{ \pi \in \mathbb F_q[t]^+ \mid \pi \textrm{ prime} , \abs{\{ i \mid f_i=\pi\} } > 1 \}} (-1)^{ \abs{U}}  \prod_{i=1}^{\omega} n_i  \]
\[ =  \sum_{ \substack { (f_1,\dots, f_{\omega}) \in (\mathbb F_q[t]^+) ^{\omega} \\ \deg f_i = n_i \\ f_i \textrm{ prime} \\ f_i \nmid h \\ f_i \textrm{ distinct} \\ \prod_{i=1}^\omega f_i = f}} \prod_{i=1}^{\omega} n_i  = H_{n_1,\dots,n_{\omega}}^h (f) ,\] verifying \eqref{eq-partition-sieve}.\end{proof}

The following purely combinatorial inequality will help to control the sum over $P$ in \cref{eq-partition-sieve}. 

\begin{lemma} For any natural numbers $n,c,d$ with $d \leq n-c$, we have \begin{equation}\label{sieve-combinatorial-bound} \sum_{ \substack{ P \in \mathcal P \\ d(P) = d} }  \frac{(n-c - d ) ! \prod_{i \in s(P)} (n_i)! }{ (n-c)! }  \prod_{ S \in P^{\geq 2} }n_S^{ \abs{S}-1} \leq 1. \end{equation}  \end{lemma}

\begin{proof} First note for any partition $P$ that $d(P) = \sum_{i\in s(P) } n_i$. Next observe that, for all $d$, there are at most $\binom{n-c}{d}$ distinct $s \subseteq \{1,\dots,\omega\}$ that equal $s(P)$ for some $P \in \mathcal P$ and such that $\sum_{i \in s} n_i =d$. This is because, after fixing a partition of a set of size $n-c$ into subsets of size $n_1,\dots, n_\omega$, each allowable $s$ defines a subset of size $d$, being the union of the corresponding sets of size $n_i$, and no two distinct $s$ define the same subsets of size $d$.

So to prove \eqref{sieve-combinatorial-bound}, it suffices to prove for each $s$ with $\sum_{i \in s} n_i =d$  that 
\begin{equation}\label{sieve-combinatorial-bound-2} \sum_{ \substack{ P \in \mathcal P \\ s(P) = s } }  \frac{\prod_{i \in s } (n_i)! }{ d! }  \prod_{ S \in P^{\geq 2} }n_S^{ \abs{S}-1} \leq 1. \end{equation} 
Now \[ \prod_{ S \in P^{\geq 2}} n_S^{ \abs{S}-1} \leq \frac{  \prod_{ i \in s} n_i }{ \max_{i \in s} n_i} \] so it suffices to prove that the number of $P \in \mathcal P$ with $s(P) =s$ is at most  \[ \frac{d!}{ \prod_{i\in s}   (n_i)!} \frac{ \max_{i \in s} n_i}{ \prod_{i \in s } n_i }.\] The number of $P \in \mathcal P$ is at most $\abs{s}!$ since every partition of $s$ is the cycle decomposition of at least one permutation of $s$.  So it suffices to check that
\begin{equation}\label{last-combinatorial-step} \abs{s}!  \frac{ \prod_{i \in s} n_i} {\max_{i \in s} n_i} \leq \frac{d!}{ \prod_{i\in s}   (n_i)!} .\end{equation}  To do this, note that to each partition of $\{1,\dots, d\}$ into sets $A_i$ of size $n_i$ labeled by $i\in s$, we obtain a bijection $\beta$ between $1,\dots, \abs{s}$ and $s$, where $A_{\beta(1)}$ contains the first element, $A_{\beta(2)}$ contains the first element not in $A_{\beta(1)}$, $A_{\beta(3)}$ contains the first element not in $A_{\beta(1)} \cup A_{\beta(2)}$, and so on, and a tuple of natural numbers $(x_j)_{j=1}^{\abs{s}-1}$, where $x_j$ is the number of elements of $A_{\beta_j}$ before the first element of $A_{\beta_{j+1}}$. Every pair of a bijection $\beta$ and an $\abs{s}-1$ tuple of natural numbers satisfying $1 \leq x_j \leq n_{\beta(j)}$ arises this way, because we can take $x_1$ elements of $A_{\beta(1)}$, then $x_2$ elements of $A_{\beta(2)}$, and so on, and place the remaining elements last in arbitrary order. This gives a surjection from a set of size $ \frac{d!}{ \prod_{i\in s}   (n_i)!} $ to a set of size $\geq  \abs{s}!  \frac{ \prod_{i \in s} n_i} {\max_{i \in s} n_i}$, verifying \eqref{last-combinatorial-step}. \end{proof}

\begin{lemma}\label{vm-c-s-2} 
For $\theta$ a real number between $0$ and $1$, $g$ squarefree of degree $m \geq 2$, natural numbers $c,n$ with $c \leq n$ and $m \leq \theta n$, a tuple $n_1,\dots, n_{\omega}$ of natural numbers summing to $n-c $, a residue class $a \in (\mathbb F_q[t]/g)^\times$, and a polynomial $h$ of degree $c$, we have \begin{equation}\label{eq-vm-c-s-2} \begin{split} & \Bigl | \sum_{ \substack{ f \in \mathcal M_{n-c} \\  f \equiv a \mod g }}  H^{h}_{n_1,\dots, n_\omega}(f)  -  \frac{1}{ \phi(g) }  \sum_{ \substack{ f \in\mathcal M_{n-c} }} H_{n_1,\dots,n_{\omega}}^h(f)  \Bigr| \\ & \ll     e^{ \sqrt{c} } \frac{(n-c) !}{ \prod_{i=1}^\omega (n_i)!}   \left(\frac{1-\theta}{ (1+2 \theta)  \sqrt{q} }\right)^c \left( \frac{(1+2\theta) e \sqrt{q}  2^{ \frac{\theta}{1-\theta}} }{ 1-\theta} \right)^{n-m} . \end{split} \end{equation} 
\end{lemma}
  
  \begin{proof} 
  
% Then
%by a straightforward but notationally complicated inclusion-exclusion.  Given $P,v$ and $f_i'$ for $i\notin s(P)$ such that $\prod_{i \notin s(P) f_i' = \frac{f}{H(v)}}$, we obtain a tuple $(f_i)_{i=1}^{\omega}$ such that $\prod_{i=1}^{\omega} f_i = f$ by setting $f_i= f_i'$ for $i\notin s(P)$ and $f_i = v(S)$ if $i \in S$ .   Each tuple $(f_i)_{i=1}^{\omega}$ with $f_i$ prime, $f_i$ coprime to $h$, $\deg f_i=n_i$, and $\prod_{i=1}^{\omega} f_i =f$ arises $2^k$ times this way, where $k$ is the number of primes $\pi$ such that $f_i= \pi$ for more than one value of $i$. Indeed, for a subset $U$ of the set of such primes $\pi$, we take $P$ to be the set consisting of $ \{ i \mid f_i = \pi\ \}$ for $\pi \in U$ and $\{ i \}$ for all $i$ with $f_i \notin U$. For this $P$,  $\abs{P^{\geq 2}} =\abs{U}$ so the sum over all possible values of $P$ vanishes unless $k=0$, which happens if the $f_i$ are all distinct.  

We apply \cref{partition-sieve} and then sum \eqref{eq-vm-c-s-1} over all pairs $P,v$, obtaining
 \begin{align*}
 & \Bigl | \sum_{ \substack{ f \in \mathcal M_{n-c} \\  f \equiv a \mod g }}  H^{h}_{n_1,\dots, n_\omega}(f)  -  \frac{1}{ \phi(g) }  \sum_{ \substack{ f \in\mathcal M_{n-c} }} H_{n_1,\dots,n_{\omega}}^h(f)  \Bigr| \\
 & \ll       \sum_{ P \in \mathcal P}  \sum_{ \substack{ v\colon P^{\geq 2} \to \mathbb F_q[T]^+ \\ v(S) \textrm{ prime} \\ v(S) \nmid h  \\  v\textrm{ injective } \\ \deg v(S) = n_S}}  e^{ \sqrt{c + d(P)} } \frac{(n-c - d(P) ) !}{ \prod_{i\notin s(P)} (n_i)!}   \left(\frac{1-\theta}{ (1+2 \theta)  \sqrt{q} }\right)^{c + d(P) }  \left( \frac{(1+2\theta) e \sqrt{q}  2^{ \frac{\theta}{1-\theta}} }{ 1-\theta} \right)^{n-m}  \prod_{i \in s(P) } n_i   . 
 \end{align*} Using  $\sqrt{c + d(P)} \leq \sqrt{c} + \sqrt{d(P)}$, we have
 \begin{equation}\label{sieve-summing-step} \begin{split}& \Bigl | \sum_{ \substack{ f \in \mathcal M_{n-c} \\  f \equiv a \mod g }}  H^{h}_{n_1,\dots, n_\omega}(f)  -  \frac{1}{ \phi(g) }  \sum_{ \substack{ f \in\mathcal M_{n-c} }} H_{n_1,\dots,n_{\omega}}^h(f)  \Bigr| \\ & \ll      e^{ \sqrt{c}}  \frac{ (n-c)!}{ \prod_{ i =1}^\omega (n_i)!} \left(\frac{1-\theta}{ (1+2 \theta)  \sqrt{q} }\right)^{c}    \left( \frac{(1+2\theta) e \sqrt{q}  2^{ \frac{\theta}{1-\theta}} }{ 1-\theta} \right)^{n-m} \times \\
 & \sum_{ P \in \mathcal P}  \sum_{ \substack{ v\colon P^{\geq 2} \to \mathbb F_q[T]^+ \\ v(S) \textrm{ prime}\\ v(S) \nmid h  \\  v\textrm{ injective } \\ \deg v(S) = n_S}}  e^{ \sqrt{d(P)} } \frac{(n-c - d(P) ) ! \prod_{i \in s(P)} (n_i)! }{ (n-c)! }   \left(\frac{1-\theta}{ (1+2 \theta)  \sqrt{q} }\right)^{d(P)  }  \prod_{i \in s(P) } n_i   . \end{split} \end{equation} 

Now the term being summed is independent of $v$ so to estimate the sum over $v$, it suffices to bound the number of injective maps $v\colon  P^\geq 2 \to \mathbb F_q[T]^+$ with $\deg v(S) =n_S$ and $v(S)$ a prime not dividing $h$. An upper bound can be obtained by noting that there are at most $\frac{ q^{ n_S}}{n_S}$ primes of degree $n_S$, so the total number of possibilities is at most $\prod_{ S \in P^{\geq 2}} \frac{ q^{n_S}}{n_S}$.  We have \[ \prod_{i \in s(P)} n_i  =\prod_{ S \in P^{\geq 2}} n_S^{ \abs{S}}\] so we obtain 
\begin{equation}\label{v-counting-step}  \begin{split} &  \sum_{ P \in \mathcal P}  \sum_{ \substack{ v\colon P^{\geq 2} \to \mathbb F_q[T]^+ \\ v(S) \textrm{ prime} \\  v\textrm{ injective } \\ \deg v(S) = n_S}}  e^{ \sqrt{d(P)} } \frac{(n-c - d(P) ) ! \prod_{i \in s(P)} (n_i)! }{ (n-c)! }   \left(\frac{1-\theta}{ (1+2 \theta)  \sqrt{q} }\right)^{d(P)  }  \prod_{i \in s(P) } n_i   \\
& \leq \sum_{ P \in \mathcal P} e^{ \sqrt{d(P)} } \frac{(n-c - d(P) ) ! \prod_{i \in s(P)} (n_i)! }{ (n-c)! }   \left(\frac{1-\theta}{ (1+2 \theta)  \sqrt{q} }\right)^{d(P)  }   \prod_{ S \in P^{\geq 2} } ( q^{n_S} n_S^{ \abs{S}-1} ) \\
& \leq \sum_{ P \in \mathcal P} e^{ \sqrt{d(P)} } \frac{(n-c - d(P) ) ! \prod_{i \in s(P)} (n_i)! }{ (n-c)! }   \left(\frac{1-\theta}{ (1+2 \theta) }\right)^{d(P)  }   \prod_{ S \in P^{\geq 2} }n_S^{ \abs{S}-1} \end{split} \end{equation} 
since \[ d(P) = \sum_{ S \in P^{\geq 2} } \abs{S} n_S \geq \sum_{ S \in P^{\geq 2} } 2 n_S.\]

 Now \eqref{sieve-combinatorial-bound} implies
\begin{align*} &  \sum_{ P \in \mathcal P} e^{ \sqrt{d(P)} } \frac{(n-c - d(P) ) ! \prod_{i \in s(P)} (n_i)! }{ (n-c)! }   \left(\frac{1-\theta}{ (1+2 \theta) }\right)^{d(P)  }   \prod_{ S \in P^{\geq 2} }n_S^{ \abs{S}-1} \\  & \leq \sum_{d =0}^{ n-c} e^{ \sqrt{d} }  \left(\frac{1-\theta}{ (1+2 \theta) }\right)^{d }   \sum_{ \substack{ P \in \mathcal P\\ d(P)=d}} \frac{(n-c - d ) ! \prod_{i \in s(P)} (n_i)! }{ (n-c)! }  \prod_{ S \in P^{\geq 2} }n_S^{ \abs{S}-1} \\ &  \leq  \sum_{d =0}^{ n-c} e^{ \sqrt{d} }  \left(\frac{1-\theta}{ (1+2 \theta) }\right)^{d }  \leq  \sum_{d =0}^{ \infty} e^{ \sqrt{d} }  \left(\frac{1-\theta}{ (1+2 \theta) }\right)^{d }  \ll 1. \end{align*} 

Plugging this into \eqref{v-counting-step} and \eqref{sieve-summing-step}, we obtain \eqref{eq-vm-c-s-2}.
\end{proof}

\subsection{Distributions of factorization types} 

We keep notation from Subsection \ref{ss-ft}.

In addition, we define the \emph{entropy} $H(w)$ of a factorization type $w \in \EFT_n$ as  \[ \sum_{ (d,l) \in w} \frac{dl}{n} \log \left( \frac{n}{d} \right). \]  Thus $H(\omega_f) $ is the entropy of the probability distribution consisting of one event of probability $\frac{ \deg \pi} {\deg f}$ for each prime factor $\pi$ of $f$, counted with multiplicity.

\begin{lemma}\label{d-ft-1} For $\theta$ a real number between $0$ and $1$, $\lambda$ a positive real number, natural numbers $n$, $m$ with $2 \leq m \leq \theta n$, $g$ squarefree of degree $m$, and an invertible residue class $a$ mod $g$, we have 
\[ \hspace{-.7in}  \sum_{ \substack{ w\in \EFT_n \\ H(w) \leq \lambda}} \left | \frac{ \abs{ \{ f \in \mathcal M_n \mid f \equiv a \mod g , \omega_f=w \} } }{q^{n-m}} -  \frac{ \abs{ \{ f \in \mathcal M_n \mid \gcd(f,g)=1  , \omega_f=w \} } }{\phi(g) q^{n-m}} \right|   \ll  \left( \frac{(1+2\theta) e   2^{ \frac{\theta}{1-\theta}}  e^{ \frac{\lambda}{1-\theta}}  }{ (1-\theta) \sqrt{q} } \right)^{n-m}  .  \] \end{lemma}

\begin{proof} First fix one extended factorization type $w$. Let $w^{ \geq 2}$ be the multiset of pairs $(d,l) \in w$ with $l\geq 2$. This is the analogue for factorization types of the maximum squareful number dividing a given number. Let $c = \sum_{ (d,l)\in w^{\geq 2}} dl$ so that $c =\deg h$ if $\omega_h = w^{\geq 2}$. Let $r$ be the number of pairs $(d,1) \in w$ and let $n_1,\dots, n_r$ be the $d$ values appearing in these pairs.

Then $\omega_f=w$ if and only if $f$ is the product of a monic polynomial $h$ with $\omega_h= w^{\geq 2}$ and primes $f_1,\dots, f_r$ of degrees $n_1,\dots, n_r$ where the $f_i$ are distinct and coprime to $h$.  Furthermore the number of ways of writing $f$ as such a product is $\prod_{k=1}^{\infty} \abs{ \{ i \mid n_i = k \}}!$.  Thus
\[  \frac{1}{ \prod_{i=1}^r n_i}  \frac{ 1}{ \prod_{k=1}^{\infty} \abs{ \{ i \mid n_i = k \}}!} \sum_{ \substack{ h \in \mathbb F_q[T]^+ \\ h\mid f \\ \omega_h = w^{\geq 2}}} H^{h}_{n_1,\dots, n_r} \left( \frac{f}{h} \right) = \begin{cases} 1 & \omega_f  =w \\ 0 & \omega_f \neq w \end{cases} .\]

Hence by \cref{vm-c-s-2} 
\[ \left | \frac{ \abs{ \{ f \in \mathcal M_n \mid f \equiv a \mod g , \omega_f=w \} } }{q^{n-m}} -  \frac{ \abs{ \{ f \in \mathcal M_n \mid \gcd(f,g)=1  , \omega_f=w \} } }{\phi(g) q^{n-m}} \right|\]
\[ \leq  \frac{1}{ \prod_{i=1}^r n_i}  \frac{ 1}{ \prod_{k=1}^{\infty} \abs{ \{ i \mid n_i = k \}}!}  \frac{1}{q^{n-m} }  \sum_{ \substack{ h \in \mathbb F_q[T]^+ \\ \omega_h = w^{\geq 2}}} \Bigl | \sum_{ \substack{ f \in \mathcal M_{n-c} \\  f \equiv a/h  \mod g }}  H^{h}_{n_1,\dots, n_r}(f)  -  \frac{1}{ \phi(g) }  \sum_{ \substack{ f \in\mathcal M_{n-c} }}  H^{h}_{n_1,\dots, n_r} (f)  \Bigr|  \]
\[ \ll   \frac{1}{ \prod_{i=1}^r n_i}  \frac{ 1}{ \prod_{k=1}^{\infty} \abs{ \{ i \mid n_i = k \}}!}   \frac{1}{q^{n-m} }  \sum_{ \substack{ h \in \mathbb F_q[T]^+ \\ \omega_h = w^{\geq 2}}}   e^{ \sqrt{c} } \frac{(n-c) !}{ \prod_{i=1}^r (n_i)!}  \left(\frac{1-\theta}{ (1+2 \theta)  \sqrt{q} }\right)^c \left( \frac{(1+2\theta) e \sqrt{q}  2^{ \frac{\theta}{1-\theta}} }{ 1-\theta} \right)^{n-m}   \] 
Now, $c = \sum_{(d,l) \in w^{\geq 2} }dl$ so $c! \geq \prod_{(d,l) \in w^{\geq 2} }(d!)^l$,  and $\prod_{(d,l) \in w\setminus w^{\geq 2}} (d!)^l = \prod_{i=1}^r (n_i)!$, so  \[ \frac{(n-c) !}{ \prod_{i=1}^r (n_i)! } \leq   \frac{n!}{c! \prod_{i=1}^r (n_i)!  }  \leq \frac{ n!}{ \prod_{ (d,l) \in w} (d!)^l } .\]
Furthermore, by the multinomial theorem, we have \[  1 = \Bigl(  \sum_{(d,l) \in w} l \cdot \frac{d}{n} \Bigr)^n \geq \frac{ n!}{ \prod_{ (d,l) \in w} (d!)^l } \prod_{(d,l)\in w}   \Bigl( \frac{d}{n} \Bigr)^{dl}= \frac{ n!}{ \prod_{ (d,l) \in w} (d!)^l }  e^{- \sum_{(d,l)\in w} dl \log \left(\frac{n}{d} \right)} =\frac{ n!}{ \prod_{ (d,l) \in w} (d!)^l }   e^{ - n H(w) } \]
so 
 \[ \frac{(n-c) !}{ \prod_{i=1}^r (n_i)! } \leq   \frac{ n!}{ \prod_{ (d,l) \in w} (d!)^l }  \leq  e^{ n H(w) } \leq e^{n \lambda} \leq e^{ \frac{ (n-m) \lambda}{ 1-\theta} } \] thus we obtain

\[ \left | \frac{ \abs{ \{ f \in \mathcal M_n \mid f \equiv a \mod g , \omega_f=w \} } }{q^{n-m}} -  \frac{ \abs{ \{ f \in \mathcal M_n \mid \gcd(f,g)=1  , \omega_f=w \} } }{q^{n-m}} \right|\]
\[ \ll   \frac{1}{ \prod_{i=1}^r n_i}  \frac{ 1}{ \prod_{k=1}^{\infty} \abs{ \{ i \mid n_i = k \}}!}    \sum_{ \substack{ h \in \mathbb F_q[T]^+ \\ \omega_h = w^{\geq 2}}}   e^{ \sqrt{c} }   \left(\frac{1-\theta}{ (1+2 \theta)  \sqrt{q} }\right)^c \left( \frac{(1+2\theta) e   2^{ \frac{\theta}{1-\theta}}  e^{ \frac{\lambda}{1-\theta}}  }{ (1-\theta) \sqrt{q} } \right)^{n-m}   \] 

For each polynomial $h$ of degree $c$, the $w \in \EFT_n$ with $w^{\geq 2} = \omega_h$ are parameterized by unordered tuples $n_1,\dots, n_r$ of natural numbers summing to $n-c$.  The sum over all such unordered tuples of $ \frac{1}{ \prod_{i=1}^r n_i}  \frac{ 1}{ \prod_{k=1}^{\infty} \abs{ \{ i \mid n_i = k \}}!}   $ is $1$ since $\frac{(n-c)!}{ \prod_{i=1}^r n_i\prod_{k=1}^{\infty}}$ is the number of elements of the symmetric group $S_{n-c}$ with cycle type $(n_1,\dots, n_r)$.

Adding the condition that $H(\omega) \leq \lambda$ only shrinks the value of the sum. Thus \[   \sum_{ \substack{ w\in \EFT_n \\ H(w) \leq \lambda}} \left | \frac{ \abs{ \{ f \in \mathcal M_n \mid f \equiv a \mod g , \omega_f=w \} } }{q^{n-m}} -  \frac{ \abs{ \{ f \in \mathcal M_n \mid \gcd(f,g)=1  , \omega_f=w \} } }{q^{n-m}} \right|  \] 
\[ \ll \sum_{c=0}^n \sum_{ \substack{ h \in \mathcal M_c \\ \textrm{squarefull} } } e^{ \sqrt{c} }   \left(\frac{1-\theta}{ (1+2 \theta)  \sqrt{q} }\right)^c  \left( \frac{(1+2\theta) e   2^{ \frac{\theta}{1-\theta}}  e^{ \frac{\lambda}{1-\theta}}  }{ (1-\theta) \sqrt{q} } \right)^{n-m}    .\] 

Thus, it remains to establish \[ \sum_{c=0}^n \sum_{ \substack{ h \in \mathcal M_c \\ \textrm{squarefull}} }  e^{ \sqrt{c} }   \left(\frac{1-\theta}{ (1+2 \theta)  \sqrt{q} }\right)^c  \ll 1\]  To do this, choose any $s$ with  \[  \frac{1-\theta}{ (1+2 \theta)  \sqrt{q} } <  q^{-s} < \frac{1}{ \sqrt{q} }\] and note that \[e ^{ \sqrt{c} }   \left(\frac{1-\theta}{ (1+2 \theta)  \sqrt{q} }\right)^c \ll q^{-sc} \]  so that

\[  \sum_{c=0}^n \sum_{ \substack{ h \in \mathcal M_c \\ \textrm{squarefull} } } e^{ \sqrt{c} }   \left(\frac{1-\theta}{ (1+2 \theta)  \sqrt{q} }\right)^c \ll \sum_{c=0}^n \sum_{ \substack{ h \in \mathcal M_c \\ \textrm{squarefull} } }  q^{-sc} \ll \sum_{\substack{h \in \mathbb F_q[t]^+ \\ \textrm{squarefull}} } q^{-s \deg h}  = \prod_{ \substack { \pi \in \mathbb F_q[t]^+ \\ \textrm{prime}}} \left( 1 + \frac{\abs{ \pi}^{-2s}}{ 1 - \abs{\pi}^{-s}} \right) \ll 1 ,\] because $s>\frac{1}{2}$ and so the Euler product converges. 
\end{proof}

\begin{lemma}\label{d-ft-2} For $\theta$ a real number between $0$ and $1$, $\lambda$ a positive real number, natural numbers $n$, $m$ with $2\leq m \leq \theta n$, $g$ squarefree of degree $m$, and an invertible residue class $a$ mod $g$, the total variation distance between the probability measures 
\[ \frac{1}{q^{n-m}}  \sum_{ \substack{ f \in \mathcal M_n \\ f \equiv a \mod g}} \delta_{\omega_f} \]
and
\[ \frac{1}{\phi(g) q^{n-m}}  \sum_{ \substack{ f \in \mathcal M_n \\ \gcd(f,g)=1 }} \delta_{\omega_f} \]
is at most
\[ \frac{\abs{ \{ f \in \mathcal M_n \mid \gcd(f,g)=1, H(\omega_f) > \lambda\} }}{ q^{n-m} \phi(g) } + O \left(  \left( \frac{(1+2\theta) e   2^{ \frac{\theta}{1-\theta}}  e^{ \frac{\lambda}{1-\theta}}  }{ (1-\theta) \sqrt{q} } \right)^{n-m}  \right) .\]
\end{lemma}

\begin{proof}Let $\mu_1 = \frac{1}{q^{n-m}}  \sum_{ \substack{ f \in \mathcal M_n \\ f \equiv a \mod g}} \delta_{\omega_f} $ and $\mu_2 = \frac{1}{\phi(g) q^{n-m}}  \sum_{ \substack{ f \in \mathcal M_n \\ \gcd(f,g)=1 }} \delta_{\omega_f}$. 
It is easy to verify that these measures are probability measures, because there are $q^{n-m}$ monic $f$ of degree $n$ congruent to $a$ mod $g$ and $\phi(g)q^{n-m}$ monic $f$ of degree $m$ coprime to $g$.

 For probability measures $\mu_1,\mu_2$ on a finite set $S$, one equivalent definition of the total variation distance between $\mu_1$ and $\mu_2$ is \cite[Remark 4.3]{lpw}
\[ \sum_{ s \in S}  \max ( \mu_2(s) - \mu_1(s),0 ) .\]

We apply \cref{d-ft-1} to bound the difference $\abs{ \mu_2(w) -\mu_1(w)}$ whenever $H(w) \leq \lambda$ and use \[ \max ( \mu_2(s) - \mu_1(s),0 ) \leq \mu_2(s)\] to bound the measure whenever $H(w) > \lambda$. These give the two stated terms.  \end{proof}

The next two lemmas describes the limiting distribution of the entropy by comparing the prime factorization of a random element of $\mathcal M_n$ coprime to $g$ to a GEM random variable.

 For each natural number $n$ and any polynomial $g$ coprime to at least one polynomial of degree $n$, define a random variables $p^{n, g}\in \Delta_\infty$.  First choose a polynomial $f$ uniformly at random among polynomials of degree $n$ prime to $g$. Then choose a random prime factor $\pi_1$ of $f$, choosing each prime factor $\pi$ with probability $\frac{\deg \pi}{\deg f}$, and set $p^{n,g}(1) = \frac{ \deg \pi_1}{n}$. Then choose a random prime factor $\pi_1$ of $f/\pi_1$, choosing each prime factor $\pi$ with probability $\frac{\deg \pi}{ \deg (f/\pi_1)}$, and set $p^{n,g}(2)= \frac{ \deg \pi_2}{n f}$. Continue, choosing $\pi_j $ a prime factor of $f/\prod_{i=1}^{j-i} \pi_i$, where each $\pi$ occurs with probability $\frac{\deg \pi} { \deg (f/\prod_{i=1}^{j-i} \pi_i)}$, and set $p^{n,g}(j) =\frac{ \deg \pi_j}{n}$. If a prime factor occurs multiple times, we treat these multiple occurrences as distinct - each has their own probability proportional to $\deg \pi$ as occurring. 
 
 The following generalizes \cite[Theorem 1]{DG} by first passing from integers to polynomials over a finite field, and then restricting to polynomials coprime to $g$. An analogue involving a uniform distribution over integers coprime to a given integer can likely be proven by similar means.

\begin{lemma}\label{weak-convergence} As $n$ goes to $\infty$, the random variable $p^{n,g}$ converges weakly in distribution to the GEM distribution, uniformly in $g$ as long as $\deg g = e^{ o (n/\log n)}$. \end{lemma}

\begin{proof}

 By the argument on \cite[p. 401]{DG}, which is purely probabilistic and transfers without modification, it suffices to prove for each positive integer $k$ and each tuple $a_1,\dots,a_k, b_1,\dots, b_k$ with $0<a_i<b_i<1$ that
\[ \lim \inf_{n\to \infty} \mathbb P \left(  a_i < \frac{  p^{n,g}(i)}{1 - \sum_{j=1}^{i-1} p^{n,g}(j) } < b_i \textrm{ for all } i\right) \geq \prod_{i=1}^k (b_i-a_i) .\]

Note that $\frac{  p^{n,g}(i)}{1 - \sum_{j=1}^{i-1} p^{n,g}(j) } = \frac{ \deg \pi_i} { n - \sum_{j=1}^{i-1} \deg \pi_j}$. Next observe that for any fixed $f$, the probability that $\pi_1,\dots, \pi_k$ attain a given tuple $\overline{\pi}_1,\dots, \overline{\pi}_k$ of prime polynomials is at least $\prod_{i=1}^k \frac{ \deg \overline{\pi}_i}{ n- \sum_{j=1}^{i-1} \deg \overline{\pi}_j } $ as long as $\overline{\pi}_1,\dots, \overline{\pi}_k$ divide $f$. Thus it suffices to prove that
\begin{equation}\label{dg-lower-bound} \sum_{ \substack{ \overline{\pi}_1,\dots, \overline{\pi}_k \in \mathbb F_q[t] \\ \textrm{monic, prime} \\ a_i <\frac{ \deg \pi_i} { n - \sum_{j=1}^{i-1} \deg \pi_j} < b_i  \textrm{ for all } i}} \prod_{i=1}^k \frac{ \deg \overline{\pi}_i}{ n- \sum_{j=1}^{i-1} \deg \overline{\pi}_j }  \mathbb P (\overline{ \pi}_1\overline{\pi}_2 \dots \overline{\pi}_k \mid f)  \geq (1-o(1)) \prod_{i=1}^k (b_i-a_i).\end{equation}

We will do this in two steps. We will first lower bound the probability in the $g=1$ case, i.e. with no coprimality condition on $f$, and then show the difference between the probability in the $g=1$ case and the general case converges to $0$ as $g$ goes to $\infty$. (Everything in this proof but this bounding-the-difference step follows closely the argument of \cite[Proof of Theorem 1 on p. 401- 403]{DG}.)

In both cases, we observe that for $d_i=\deg \overline{\pi}_i$, the conditions $a_i <\frac{ d_i} { n - \sum_{j=1}^{i-1} d_j} < b_i$ imply
\begin{equation}\label{i-upper-bound}  n - \sum_{j=1}^{i-1} d_j > n \prod_{i=1}^{j-1} (1 -b_j)\end{equation}
 by induction on $i$ so \begin{equation}\label{all-k-lower-bound} n - \sum_{j=1}^k d_j > n \prod_{i=1}^k (1-b_j) \end{equation}
 and \begin{equation}\label{i-lower-bound}  d_i > a_i n \prod_{i=1}^{j-1} (1 -b_j).\end{equation}
  
  Because of \eqref{all-k-lower-bound}, $\deg ( \overline{\pi}_1\overline{\pi}_2 \dots \overline{\pi}_k ) \leq n$ so, in the $g=1$ case, we have
\[\mathbb P ( \overline{\pi}_1\overline{\pi}_2 \dots \overline{\pi}_k \mid f)  = \frac{1}{ q^{ \sum_{j=1}^k \deg \overline{\pi}_j}} \] thus
\[ \sum_{ \substack{ \overline{\pi}_1,\dots, \overline{\pi}_k \in \mathbb F_q[t] \\ \textrm{monic, prime} \\ a_i <\frac{ \deg \overline{\pi}_i} { n - \sum_{j=1}^{i-1} \deg \pi_j} < b_i  \textrm{ for all } i}} \prod_{i=1}^k \frac{ \deg \overline{\pi}_i}{ n- \sum_{j=1}^{i-1} \deg \overline{\pi}_j }  \mathbb P ( \pi_1\pi_2 \dots \pi_k \mid f)\] \[=  \sum_{ \substack{ \overline{\pi}_1,\dots, \overline{\pi}_k \in \mathbb F_q[t] \\ \textrm{monic, prime} \\ a_i <\frac{ \deg \pi_i} { n - \sum_{j=1}^{i-1} \deg \pi_j} < b_i  \textrm{ for all } i}} \prod_{i=1}^k \frac{ \deg \overline{\pi}_i}{ n- \sum_{j=1}^{i-1} \deg \overline{\pi}_j }   \frac{1}{ q^{ \sum_{j=1}^k \deg \pi_j}}\]
\[ = \sum_{ \substack{ d_1,\dots, d_k \in \mathbb N^+ \\ a_i <\frac{ d_i} { n - \sum_{j=1}^{i-1} d_j} < b_i  \textrm{ for all } i}} \prod_{i=1}^k \frac{ d_i}{ n- \sum_{j=1}^{i-1} d_j }   \frac{1}{ q^{ \sum_{j=1}^k d_j}} \prod_{i=1}^k \abs{ \{ \overline{\pi}_i \in \mathcal M_{d_i} \mid \overline{\pi}_i \textrm{ prime}} \]
(by \eqref{i-lower-bound} and the prime number theorem for polynomials)
\[ \geq (1-o(1) ) \sum_{ \substack{ d_1,\dots, d_k \in \mathbb N^+ \\ a_i <\frac{ d_i} { n - \sum_{j=1}^{i-1} d_j} < b_i  \textrm{ for all } i}} \prod_{i=1}^k \frac{ d_i}{ n- \sum_{j=1}^{i-1} d_j }   \frac{1}{ q^{ \sum_{j=1}^k d_j}} \prod_{i=1}^k \frac{ q^{d_i}}{d_i} \]
\[=  (1-o(1) ) \sum_{ \substack{ d_1,\dots, d_k \in \mathbb N^+ \\ a_i <\frac{ d_i} { n - \sum_{j=1}^{i-1} d_j} < b_i  \textrm{ for all } i}} \prod_{i=1}^k \frac{ 1}{ n- \sum_{j=1}^{i-1} d_j }  .\]

Fixing $d_1,\dots, d_{j-1}$, the final variable  $d_k$ may take any value in $(a_k (n - \sum_{j=1}^{k-1} d_j) , b_k (n - \sum_{j=1}^{k-1} d_j)$, and the summand $ \prod_{i=1}^k \frac{ 1}{ n- \sum_{j=1}^{i-1} d_j }$ is independent of $d_k$. We may thus replace the sum over $d_k$ by the number of possible choices of $d_k$, which is 
\[ b_k (n - \sum_{j=1}^{k-1} d_j) - a_k (n - \sum_{j=1}^{k-1} d_j) +O(1) = (1-o(1)) (b_k-a_k)  (n - \sum_{j=1}^{k-1} d_j) \] by \eqref{i-upper-bound} for $i=k$. This cancels the $k$th term in the product $\prod_{i=1}^k \frac{ 1}{ n- \sum_{j=1}^{i-1} d_j }$, giving the sum
\[=  (1-o(1) )  (b_k-a_k) \sum_{ \substack{ d_1,\dots, d_{k-1} \in \mathbb N^+ \\ a_i <\frac{ d_i} { n - \sum_{j=1}^{i-1} d_j} < b_i  \textrm{ for all } i<k}} \prod_{i=1}^{k-1} \frac{ 1}{ n- \sum_{j=1}^{i-1} d_j }  \]
 and we may repeat this process $k$ times, obtaining
 \[ (1-o(1)) \prod_{i=1}^k (b_k-a_k),\] as desired.
 
 To handle the general $g$ case, we observe that the number of monic polynomials of degree $n$ divisible by a fixed polynomial $h$ is $q^{n -\deg h}$ if $\deg h \leq n$ and $0$ otherwise, so the number of monic polynomials of degree $n$ prime to $g$ is $\sum_{\substack{ h\mid g\\ \deg h\leq n}}\mu(h) q^{n-\deg h}$, and
 \begin{equation}\label{rp-probability} \mathbb P ( \overline{\pi}_1\overline{\pi}_2 \dots \overline{\pi}_k \mid f ) =  \frac{\sum\limits_{\substack{ h\mid g\\ \deg h + \sum_{i=1}^k \deg \overline{\pi}_i \leq n}}\mu(h) q^{n-\deg h- \sum_{i=1}^k \deg \overline{\pi}_i  }  }{  \sum\limits_{\substack{ h\mid g\\ \deg h\leq n}}\mu(h) q^{n-\deg h} } \end{equation} 
as long as $\overline{\pi}_1,\dots, \overline{\pi}_k \nmid g$.

To estimate both numerator and denominator, we will show for every fixed $c>0$ that \begin{equation}\label{rp-removal-lb}  \sum_{ \substack{h \mid g \\ \deg h> cn}} \abs{\mu(h)} q^{n- \deg h} = o \left( q^n \frac{\phi(g)}{\abs{g}} \right).\end{equation} 
To do this, observe for every $\alpha>0$ that
\[  \sum_{ \substack{h \mid g \\ \deg h> cn}} \abs{\mu(h)} q^{n- \deg h} \leq \frac{1}{q^{\alpha cn} }  \sum_{ \substack{h \mid g }} \abs{\mu(h)} q^{n- (1-\alpha) \deg h} =\frac{q^n}{q^{\alpha cn }} \prod_{\substack{ \pi \mid g \\ \textrm{prime}}} (1 + q^{ - (1-\alpha) \deg \pi }).\]
%\frac{1}{cn}  \sum_{ \substack{h \mid g }} \sum_{ \substack{ \pi \mid h \\ \textrm{prime}}}  \deg \pi \abs{\mu(h)} q^{n- \deg h}  \] \[= \frac{1}{cn} \sum_{ \substack{\pi \mid g\\ \textrm{prime}}} \deg \pi \sum_{\substack{h\mid g \\ \pi \mid h}}\abs{\mu(h)} q^{n-\deg h} = \frac{1}{cn} \sum_{ \substack{\pi \mid g\\ \textrm{prime}}} \deg \pi   q^{n-\deg \pi}  \prod_{\substack{\pi' \mid g/\pi \\ \textrm{prime}}} (1 + q^{-\deg \pi'})\] \[ \leq \frac{q^n}{cn} \sum_{\substack{\pi \mid g \\ \textrm{prime}}} \frac{ \deg \pi}{q^{\deg \pi}} \prod_{\substack{\pi' \mid g \\ \textrm{prime}}} (1+ q^{-\deg \pi} ) \]
and $\frac{\phi(g)}{\abs{g}} = \prod_{ \substack{\pi \mid g \\ \textrm{prime}}} (1- q^{-\deg \pi}) $ so it suffices to prove that
\[ \frac{1}{q^{\alpha cn }} \prod_{\substack{ \pi \mid g \\ \textrm{prime}}} (1 + q^{ - (1-\alpha) \deg \pi })= o \Bigl(  \prod_{ \substack{\pi \mid g \\ \textrm{prime}} }(1- q^{-\deg \pi}) \Bigr) .\]

For fixed $\deg g$, the  products on the left side are minimized, and the product on the right side is maximized, when each $\deg \pi$ is as small as possible, so it suffices to handle the ``primorial" case where $g$ consists of the products of all primes of degree $<l$ with some of the primes of degree $l$.  The assumption $\deg g=e^{ o (n/\log n) }$ gives $l=o (n/\log n) $. Taking $\alpha= 1/l$, we have
%\[ \sum_{\substack{\pi \mid g \\ \textrm{prime}}} \frac{ \deg \pi}{q^{\deg \pi}}  \leq \sum_{\substack{\deg \pi \leq l \\ \textrm{prime}}} \frac{ \deg \pi}{q^{\deg \pi}} \leq \sum_{d=1}^l \frac{q^d}{d} \frac{d}{q^d} \leq l = o (n^{1/3})\]
\[   \prod_{ \substack{\pi \mid g \\ \textrm{prime}} }(1+ q^{-(1-\alpha) \deg \pi}) \leq \prod_{ \substack{\pi \mid g \\ \textrm{prime}} } e^{ q^{-(1-\alpha) \deg \pi} } = e^{ \sum_{\substack{\pi \mid g} }q^{-(1-\alpha) \deg \pi}} \leq e^{ \sum_{ \deg \pi \leq l } q^{-(1-\alpha) \deg \pi}} \] \[\leq  e^{ q \sum_{ \deg \pi \leq l } q^{-\deg \pi}} \leq  e^{ q \sum_{d=1}^{l} \frac{q^d}{d} q^{-d} } = e^{ q  \log l + O(1) } =O( l^q) = O( n^q )\]
and
\[\prod_{ \substack{\pi \mid g \\ \textrm{prime}} }(1- q^{-\deg \pi})= \frac{ \prod_{ \substack{\pi \mid g \\ \textrm{prime}} }(1 -q^{-2\deg \pi}) }{\prod_{ \substack{\pi \mid g \\ \textrm{prime}} }(1+ q^{-\deg \pi})} \gg \frac{1}{ \prod_{ \substack{\pi \mid g \\ \textrm{prime}} }(1+ q ^{-\deg \pi})}\gg \frac{1}{ O( n^q) } \]
while $\alpha cn = c n / l  = cn / o(n/\log n ) = c \log n / o(1) = \log n / o(1)$ is eventually greater than any constant times $\log n$, so $q^{ \alpha c n}$ is eventually greater than any power of $n$, thus 

\[ \frac{1}{q^{\alpha cn }} \prod_{\substack{ \pi \mid g \\ \textrm{prime}}} (1 + q^{ - (1-\alpha) \deg \pi }) = \frac{ O(n^q)}{ q^{\alpha c n }} = o(n^q)  = o \Bigl(  \prod_{ \substack{\pi \mid g \\ \textrm{prime}} }(1- q^{-\deg \pi}) \Bigr) ,\]
establishing \eqref{rp-removal-lb}.
Taking \eqref{rp-probability} and then combining \eqref{all-k-lower-bound} with \eqref{rp-removal-lb},  we get
\[\mathbb P ( \overline{\pi}_1\overline{\pi}_2 \dots \overline{\pi}_k \mid f ) =  \frac{ \sum\limits_{ h\mid g }\mu(h) q^{n-\deg h- \sum_{i=1}^k \deg \overline{\pi}_i }  - \sum\limits_{\substack{ h\mid g\\ \deg h + \sum_{i=1}^k \deg \overline{\pi}_i > n}}\mu(h) q^{n-\deg h- \sum_{i=1}^k \deg \overline{\pi}_i }   }{  \sum\limits_{\substack{ h\mid g}}\mu(h) q^{n-\deg h} -  \sum\limits_{\substack{ h\mid g\\ \deg h> n}}\mu(h) q^{n-\deg h} } \] \[= \frac{ q^{n- \sum_{i=1}^k \deg \overline{\pi_i} } \phi(g)/\abs{g} -  o ( q^{n- \sum_{i=1}^k \deg \overline{\pi_i} } \phi(g)/\abs{g})}{ q^n \phi(g)/\abs{g} - o ( q^n \phi(g)/\abs{g} )} = q^{ -\sum_{i=1}^k \deg \overline{\pi}_i } (1+o(1) ) \]
so to establish \eqref{dg-lower-bound}, it suffices to prove
  \begin{equation}\label{dg-lower-bound-rp} \sum_{ \substack{ \overline{\pi}_1,\dots, \overline{\pi}_k \in \mathbb F_q[t] \\ \textrm{monic, prime} \\ a_i <\frac{ \deg \pi_i} { n - \sum_{j=1}^{i-1} \deg \pi_j} < b_i  \textrm{ for all } i\\ \gcd( \overline{\pi}_i, g) =1 \textrm{ for all } i }} \prod_{i=1}^k \frac{ \deg \overline{\pi}_i}{ n- \sum_{j=1}^{i-1} \deg \overline{\pi}_j } q^{ - \sum_{i=1}^k \deg \overline{\pi}_i} \geq (1-o(1)) \prod_{i=1}^k (b_i-a_i).\end{equation}

Another inclusion-exclusion gives
\[\sum_{ \substack{ \overline{\pi}_1,\dots, \overline{\pi}_k \in \mathbb F_q[t] \\ \textrm{monic, prime} \\ a_i <\frac{ \deg \pi_i} { n - \sum_{j=1}^{i-1} \deg \pi_j} < b_i  \textrm{ for all } i\\ \gcd( \overline{\pi}_i, g) =1 \textrm{ for all } i }} \prod_{i=1}^k \frac{ \deg \overline{\pi}_i}{ n- \sum_{j=1}^{i-1} \deg \overline{\pi}_j } q^{ - \sum_{i=1}^k \deg \overline{\pi}_i}\]
\[ = \sum_{h\mid g} \mu(h) \sum_{ \substack{ \overline{\pi}_1,\dots, \overline{\pi}_k \in \mathbb F_q[t] \\ \textrm{monic, prime} \\ a_i <\frac{ \deg \pi_i} { n - \sum_{j=1}^{i-1} \deg \pi_j} < b_i  \textrm{ for all } i\\ h \mid \prod_{i=1}^k \overline{\pi}_i }} \prod_{i=1}^k \frac{ \deg \overline{\pi}_i}{ n- \sum_{j=1}^{i-1} \deg \overline{\pi}_j } q^{ - \sum_{i=1}^k \deg \overline{\pi}_i}\]
For any $h\neq 1$, the inner sum vanishes unless some prime factor of $h$ is one of the $\overline{\pi}_i$, so by \eqref{i-lower-bound}, the inner sum vanishes unless $\deg h> a_i n \prod_{i=1}^{j-1} (1 -b_j)$. In any case, the inner sum is what was already seen to be a lower bound for the probability that, choosing $f$ uniformly at random from $\mathcal M_n$ and then choosing $\pi_1,\dots, \pi_k$ as usual, $h\mid \prod_{i=1}^k   \pi_i$, and thus is clearly at most the probability that, choosing $f$ uniformly at random, $h\mid f$, and hence is at most $q^{-\deg h} $, so the $h \neq 1$ terms of this sum are at most
\[ \sum_{\substack{ h \mid g\\ \deg h > a_i n \prod_{i=1}^{j-1} (1 -b_j) }} \abs{\mu(h)} q^{-\deg h} \] which was already seen in \eqref{rp-removal-lb} to be $o ( \phi(g)/\abs{g} ) = o(1)$, while the $h=1$ sum was already estimated in the $g=1$ case above, establishing \eqref{dg-lower-bound-rp} and completing the argument.
 \end{proof}

\begin{lemma}\label{stronger-convergence}  For $\lambda$ a positive real number and $\theta \in (0,1)$ the quantity \[ \frac{\abs{ \{ f \in \mathcal M_n \mid \gcd(f,g)=1, H(\omega_f) > \lambda\} }}{ q^{n-m} \phi(g) },\]   converges as $n$ goes to $\infty$, uniformly in $g,m$ as long as $m=\deg g \leq \theta n$, to the probability that the entropy of a GEM variable is $>\lambda$.

\end{lemma} 

\begin{proof}We continue analyzing the random variable $p^{n,g}$. By construction, we have
\[ H (\omega_f) =  -\sum_{j=1}^{\infty}  p^{n,g}(j) \log p^{n,g}(j) .\]

Let $p$ be a GEM distributed random variable. Because the probability that $H(p)= \lambda$ is zero by \cref{no-atoms}, it suffices to show that $ -\sum_{j=1}^{\infty}  p^{n,g}(j) \log p^{n,g}(j)$ converges in distribution to $ -\sum_{j=1}^{\infty} p(j) \log p(j)$. Using $\mathcal L$ to represent the law of a random variable, it suffices to show that.
\begin{equation}\label{entropy-limits-dirichlet} \lim_{n \to \infty} \mathcal L  \Bigl( -\sum_{j=1}^{\infty}  p^{n,g}(j) \log p^{n,g}(j) \Bigr)  =  \mathcal L \Bigl( -\sum_{j=1}^{\infty} p(j) \log p(j) \Bigr). \end{equation} 

It follows from Lemma \ref{weak-convergence} since $m \leq \theta n = e^{ o(n^{1/3} ) } $, that $p^{n,g}(1),\dots p^{n,g}(k)$ converge in distribution as $n \to \infty$ with $\deg g \leq  \theta n$ to $p(1) ,\dots, p(k)$.   Thus
\[  \lim_{n \to \infty} \mathcal L \Bigl(  -\sum_{j=1}^{k}  p^{n,g}(j) \log p^{n,g}(j) \Bigr)  = \mathcal L \Bigl( -\sum_{j=1}^{k} p(j) \log p(j)  \Bigr) \]
and hence
\[ \lim_{k \to \infty}  \lim_{n \to \infty} \mathcal L \Bigl( -\sum_{j=1}^{k}  p^{n,g}(j) \log p^{n,g}(j) \Bigr)  = \lim_{k\to\infty} \mathcal L \Bigl(-\sum_{j=1}^{k} p(j) \log p(j) \Bigr) .\]

Now 
\[\lim_{k\to\infty} \mathcal L \Bigl( -\sum_{j=1}^{k} p(j) \log p(j)\Bigr)  =\mathcal L \Bigl(-\sum_{j=1}^{\infty} p(j) \log p(j) \Bigr) \]
because for all $\epsilon>0$,
\[ \lim _{k \to\infty} \mathbb P \Bigl(\bigl| - \sum_{j=k+1}^{\infty} p(j) \log p(j) \bigr| > \epsilon\Bigr) =0 ,\] that is, the sums of the first $k$ terms converge in probability, and thus distribution, to the full sum.

So to obtain \eqref{entropy-limits-dirichlet}, it suffices to prove 
\[  \lim_{k \to \infty}  \lim_{n \to \infty}\mathcal L \Bigl( -\sum_{j=1}^{k}  p^{n,g}(j) \log p^{n,g}(j) \Bigr)  = \lim_{n \to \infty} \mathcal L \Bigl( -\sum_{j=1}^{\infty}  p^{n,g}(j) \log p^{n,g}(j)  \Bigr) , \]
for which it similarly suffices to check, for all $\epsilon>0$, that
\[  \lim_{k \to \infty}  \lim_{n \to \infty} \mathbb P \bigl(  \bigl| -\sum_{j=k+1}^{\infty}  p^{n,g}(j) \log p^{n,g}(j)  \bigr| > \epsilon\bigr).  \] 

%Thus the limit as $k$ goes to $\infty$ of the limit as $n$ goes to $\infty$ of the distribution of $-\sum_{j=1}^kp^{n,g}(j) \log p^{n,g}(j) $ equals the limit as $k$ goes to $\infty$ of the distribution of $\sum_{j=1}^{k} p(j) \log p(j)$.
%
%Since the limit of $k$ goes to $\infty$ of the distribution of $-\sum_{j=k+1}^{\infty} p(j) \log p(j)$ is supported at $0$, it suffices to show that the limit as $k$ goes to $\infty$ of the limit as $n$ goes to $\infty$ of  the distribution of $-\sum_{j=k+1}^\infty p^{n,g}(j) \log p^{n,g}(j) $ is the point mass at $0$.

 Since $-\sum_{j=k+1}^\infty p^{n,g}(j) \log p^{n,g}(j) \geq 0$, it suffices to show
\[ \lim_{k\to\infty} \lim_{n \to\infty} \mathbb E \left[ -\sum_{j=k+1}^\infty p^{n,g}(j) \log p^{n,g}(j) \right ]  =0 .\]

We have \[ \mathbb E \left[ -\sum_{j=k+1}^\infty p^{n,g}(j) \log p^{n,g}(j) \right ] = \sum_{ \pi \mid f} v_\pi(f)  \frac{ \deg \pi}{n}  \log \left( \frac{ n}{\deg \pi} \right)  \mathbb P ( \pi \neq \pi_1,\dots \pi_k)  
\]

\[\leq  \sum_{ \pi \mid f} v_\pi(f)  \frac{ \deg \pi}{n}  \log \left( \frac{ n}{\deg \pi} \right)  \left( 1 - \frac{\deg \pi} {n} \right)^k  = \sum_{ \pi \textrm{ prime}} \sum_{r=1}^{\infty} \delta_{\pi^r \mid f}  \frac{ \deg \pi}{n}  \log \left( \frac{ n}{\deg \pi} \right)  \left( 1 - \frac{\deg \pi} {n} \right)^k.\]

Thus \[ \mathbb E \left[ -\sum_{j=k+1}^\infty p^{n,g}(j) \log p^{n,g}(j) \right ] \] \[ \leq   \sum_{ \pi \textrm{ prime}} \sum_{r=1}^{\infty}  \frac{ \abs{\{ f \in \mathbb F_q[t]^+ \mid \deg f=n, \gcd(f,g)=1, \pi^r \mid f\}}}{ \phi(g) q^{n-m}}   \frac{ \deg \pi}{n}  \log \left( \frac{ n}{\deg \pi} \right)  \left( 1 - \frac{\deg \pi} {n} \right)^k \]
\[ =  \sum_{ \substack{ \pi \textrm{ prime} \\ \pi \nmid g}} \sum_{r=1}^{\infty}  \frac{ \abs{\{ f \in \mathbb F_q[t]^+ \mid \deg f=n- r \deg \pi , \gcd(f,g)=1\}}}{ \phi(g) q^{n-m}}   \frac{ \deg \pi}{n}  \log \left( \frac{ n}{\deg \pi} \right)  \left( 1 - \frac{\deg \pi} {n} \right)^k \]
\[ = \sum_{d=1}^n \sum_{ \substack{ \pi \textrm{ prime} \\ \pi \nmid g \\ \deg \pi \mid d }}  \frac{ \abs{\{ f \in \mathbb F_q[t]^+ \mid \deg f=n- d , \gcd(f,g)=1\}}}{ \phi(g) q^{n-m}}   \frac{ \deg \pi}{n}  \log \left( \frac{ n}{\deg \pi} \right)  \left( 1 - \frac{\deg \pi} {n} \right)^k.\]

We now split the sum into terms where $d \leq n-m$ and $d> n-m$ and bound them separately. For $d \leq n-m$, we have $\abs{\{ f \in \mathbb F_q[t]^+ \mid \deg f=n- d , \gcd(f,g)=1\}}  = q^{n-m-d} \phi(g)$ so 

\[\sum_{d=1}^{n-m}  \sum_{ \substack{ \pi \textrm{ prime} \\ \pi \nmid g \\ \deg \pi \mid d }}  \frac{ \abs{\{ f \in \mathbb F_q[t]^+ \mid \deg f=n- d , \gcd(f,g)=1\}}}{ \phi(g) q^{n-m}}   \frac{ \deg \pi}{n}  \log \left( \frac{ n}{\deg \pi} \right)  \left( 1 - \frac{\deg \pi} {n} \right)^k\]
\[ =\sum_{d=1}^{n-m}   \sum_{ \substack{ \pi \textrm{ prime} \\ \pi \nmid g \\ \deg \pi \mid d }}  q^{-d}    \frac{ \deg \pi}{n}  \log \left( \frac{ n}{\deg \pi} \right)  \left( 1 - \frac{\deg \pi} {n} \right)^k \]
\[ \leq \sum_{ \substack{ \pi \textrm{ prime} \\ \pi \nmid g \\ \deg \pi \leq n-m }}  \sum_{r =1}^{\infty} q^{-r \deg \pi}     \frac{ \deg \pi}{n}  \log \left( \frac{ n}{\deg \pi} \right)  \left( 1 - \frac{\deg \pi} {n} \right)^k\]
\[ \ll \sum_{ \substack{ \pi \textrm{ prime} \\ \pi \nmid g \\ \deg \pi \leq n-m }}  q^{ - \deg \pi}     \frac{ \deg \pi}{n}  \log \left( \frac{ n}{\deg \pi} \right)  \left( 1 - \frac{\deg \pi} {n} \right)^k\]
\[= \sum_{e=1}^{n-m}  \abs{ \{\pi \textrm{ prime}, \deg \pi =e\}}  q^{-e } \frac{ e}{n}  \log \left( \frac{ e}{n} \right)  \left( 1 - \frac{e} {n} \right)^k \]
\[ \leq \sum_{e=1}^{n-m} \frac{1}{e}  \frac{ e}{n}  \log \left( \frac{ e}{n} \right)  \left( 1 - \frac{e} {n} \right)^k\]
which is a Riemann sum for $\int_{0}^{1-\theta}  \log(x) (1-x)^k dx$, and thus converges as $n$ goes to $\infty$ to $\int_{0}^{1-\theta}  \log(x) (1-x)^k dx$, which converges to $0$ as $k$ goes to $\infty$, as desired.

For $d >n-m$, we observe that 

\[ \sum_{d=1}^{n}  \frac{ \abs{\{ f \in \mathbb F_q[t]^+ \mid \deg f=n- d , \gcd(f,g)=1\}}}{ \phi(g) q^{n-m}}   \sum_{ \substack{ \pi \textrm{ prime} \\ \pi \nmid g \\ \deg \pi \mid d }}   \deg \pi  \] is the average of the total degree of the prime factors of a polynomial of degree $n$ prime to $g$ and thus must be $n$. So 

\[ \sum_{d=n+1-m}^{n}  \sum_{ \substack{ \pi \textrm{ prime} \\ \pi \nmid g \\ \deg \pi \mid d }}  \frac{ \abs{\{ f \in \mathbb F_q[t]^+ \mid \deg f=n- d , \gcd(f,g)=1\}}}{ \phi(g) q^{n-m}}   \frac{ \deg \pi}{n}  \log \left( \frac{ n}{\deg \pi} \right)  \left( 1 - \frac{\deg \pi} {n} \right)^k \] 

\[ \leq \max_{ n+1-m \leq d \leq n}  \frac{ n} {\sum_{ \substack{ \pi \textrm{ prime} \\ \pi \nmid g \\ \deg \pi \mid d }}   \deg \pi }   \sum_{ \substack{ \pi \textrm{ prime} \\ \pi \nmid g \\ \deg \pi \mid d }}    \frac{ \deg \pi}{n}  \log \left( \frac{ n}{\deg \pi} \right)  \left( 1 - \frac{\deg \pi} {n} \right)^k.\]

The denominator is \[ \sum_{ \substack{ \pi \textrm{ prime} \\ \pi \nmid g \\ \deg \pi \mid d }}   \deg \pi = q^d - \sum_{ \substack{ \pi \textrm{ prime} \\ \pi \mid g \\ \deg \pi \mid d }} \geq q^d - m \gg q^d \] so 

\[ \sum_{d=n+1-m}^{n}  \sum_{ \substack{ \pi \textrm{ prime} \\ \pi \nmid g \\ \deg \pi \mid d }}  \frac{ \abs{\{ f \in \mathbb F_q[t]^+ \mid \deg f=n- d , \gcd(f,g)=1\}}}{ \phi(g) q^{n-m}}   \frac{ \deg \pi}{n}  \log \left( \frac{ n}{\deg \pi} \right)  \left( 1 - \frac{\deg \pi} {n} \right)^k \] \begin{equation}\label{convergence-tail-max}  \ll \max_{ n+1-m \leq d \leq n}    \frac{1}{ q^d}  \sum_{ \substack{ \pi \textrm{ prime} \\ \pi \nmid g \\ \deg \pi \mid d }}   \deg \pi    \log \left( \frac{ n}{\deg \pi} \right)  \left( 1 - \frac{\deg \pi} {n} \right)^k . \end{equation}

Those $\pi$ with $\deg \pi < d$  contribute $\ll q^{d/2}$ to $ \sum_{ \substack{ \pi \textrm{ prime} \\ \pi \nmid g \\ \deg \pi \mid d }}   \deg \pi $, so they contribute $\ll q^{d/2}\cdot  \log n\cdot 1 $ to  $\sum_{ \substack{ \pi \textrm{ prime} \\ \pi \nmid g \\ \deg \pi \mid d }}   \deg \pi    \log \left( \frac{ n}{\deg \pi} \right)  \left( 1 - \frac{\deg \pi} {n} \right)^k$ and thus $\ll q^{-d/2} \log n$ to \eqref{convergence-tail-max}. Since $d > n-m \geq (1-\theta) n$, $\lim_{n \to \infty}q^{-d/2} \log n =0$ and these terms can be ignored. The remaining terms of  \eqref{convergence-tail-max} are
\begin{equation*}\max_{ n+1-m \leq d \leq n}    \frac{1}{ q^d}  \sum_{ \substack{ \pi \textrm{ prime} \\ \pi \nmid g \\ \deg \pi =  d }}   d    \log \left( \frac{ n}{d} \right)  \left( 1 - \frac{d} {n} \right)^k . \end{equation*}
\[ \leq  \max_{ n+1-m \leq d \leq n}   \log \left( \frac{ n}{d} \right)  \left( 1 - \frac{d} {n} \right)^k \leq   \log \left( \frac{ 1}{1-\theta} \right)  \left( 1 -  (1-\theta)  \right)^k \]
which goes to $0$ as $k$ goes to $\infty$, as desired.
\end{proof}

\begin{proposition}\label{eft-final}  For $\theta$ a real number between $0$ and $1$, $q$ a prime power with  \[ q>\left ( \frac{(1+2\theta) e   2^{ \frac{\theta}{1-\theta}}  7^{ \frac{1}{ (1-\theta)} }    }{ (1-\theta) }  \right)^2,\]     a natural number $n$, $g$ squarefree of degree $m$ with $2 \leq m \leq \theta n$, and an invertible residue class $a$ mod $g$, the total variation distance between the probability measures  
\[ \frac{1}{q^{n-m}}  \sum_{ \substack{ f \in \mathcal M_n \\ f \equiv a \mod g}} \delta_{\omega_f} \]
and
\[ \frac{1}{\phi(g) q^{n-m}}  \sum_{ \substack{ f \in \mathcal M_n \\ \gcd(f,g)=1 }} \delta_{\omega_f} \]
is at most
\[  \frac{1}{ \sqrt{2\pi} (4+e) e^{ (4+e) } } \frac {  L^{9/2}   e^{2L} }{  L^L} +o(1)\]
where \[ L = \frac{1}{2}  \left( \frac{(2- 2 \theta) \sqrt{q} } { (1+2 \theta)  e} \right)^{ (1-\theta) }  \] and the constant in $o(1)$ may depend on $q,\theta$.
\end{proposition}

\begin{proof} We apply \cref{d-ft-2} with $\lambda = \log (L) - \frac{1}{ \sqrt{n-m}}$. We have 
\[ O \left(  \left( \frac{(1+2\theta) e   2^{ \frac{\theta}{1-\theta}}  e^{ \frac{\lambda}{1-\theta}}  }{ (1-\theta) \sqrt{q} } \right)^{n-m}  \right) = O \left(  \left( \frac{(1+2\theta) e   2^{ \frac{\theta}{1-\theta}}  e^{ \frac{\log L}{1-\theta} - \frac{1}{ (1-\theta) \sqrt{n-m}} }  }{ (1-\theta) \sqrt{q} } \right)^{n-m}  \right)\]
\[ =O \left( e^{ - \frac{ \sqrt{n-m}}{1-\theta}}  \left( \frac{(1+2\theta) e   2^{ \frac{\theta}{1-\theta}}  L^{ \frac{1}{1-\theta}}  }{(1-\theta) \sqrt{q} } \right)^{n-m} \right) .\]

Now 
\[ \frac{(1+2\theta) e   2^{ \frac{\theta}{1-\theta}}  L^{ \frac{1}{1-\theta}}  }{(1-\theta) \sqrt{q} }  = \frac{(1+2\theta) e   2^{ \frac{\theta}{1-\theta}} }{(1-\theta)  2^{ \frac{1}{1-\theta}} \sqrt{q} } \frac{(2- 2 \theta) \sqrt{q} } { (1+2 \theta)  e}   = \frac{   2^{  1 +\frac{\theta}{1-\theta}} }{ 2^{ \frac{1}{1-\theta}}  } =1\]
so
\[ O \left( e^{ - \frac{ \sqrt{n-m}}{1-\theta}}  \left(  \frac{(1+2\theta) e   2^{ \frac{\theta}{1-\theta}}  L^{ \frac{1}{1-\theta}}  }{(1-\theta) \sqrt{q} } \right)^{n-m} \right) = O \left( e^{ - \frac{ \sqrt{n-m}}{1-\theta}} \right) =o (1).\]

Thus, it suffices to prove that 
\begin{equation}\label{final-entropy} \frac{\abs{ \{ f \in \mathcal M_n \mid \gcd(f,g)=1, H(\omega_f) > \log (L) - \frac{1}{ \sqrt{n-m} }\} }}{ q^{n-m} \phi(g) } \leq  o(1) + \frac{1}{ \sqrt{2\pi} (4+e) e^{ (4+e) } } \frac {  L^{9/2}   e^{2L} }{  L^L}  .\end{equation}

Defining $N_{g,n} ( x) = \abs{ \{ f \in \mathcal M_n \mid \gcd(f,g)=1, H(\omega_f) > x\} }$, we have 

We have
\begin{equation}\begin{split}\label{delta-n-limit} &  \lim_{ \delta \to 0}  \lim_{n \to \infty} \max_{ \substack{ m \in \mathbb N, g \in \mathbb F_q[t]^+ \\ m \leq \theta n \\ \deg g  =m \\ g\textrm{ squarefree} }}  \frac{N_{g,n} (\log L -\delta)}{ q^{n-m} \phi(g) }= \lim_{ \delta\to 0}  \mathbb P (  H> \log (L) -\delta )  \\
 = &\mathbb P ( H \geq  \log(L))= \mathbb P ( H >  \log(L)) \leq  \frac{1}{ \sqrt{2\pi} (4+e) e^{ (4+e) } } \frac {  L^{9/2}   e^{2L} }{  L^L} ,\end{split} \end{equation} where $H$ is the entropy of a GEM random variable, by \cref{stronger-convergence}, a basic property of probability, \cref{no-atoms}, and \cref{dp-bound}. 

For any fixed $\delta>0, \theta<1 $ we have $\delta \geq \frac{1}{ \sqrt{n-\theta n }} \geq \frac{1}{ \sqrt{n-m}}$ for all $n$ sufficiently large and all $m \leq \theta n$, which since $N_{g,n}$ is nondecreasing in $x$  implies that 
\[\max_{ \substack{ m \in \mathbb N, g \in \mathbb F_q[t]^+ \\ m \leq \theta n \\ \deg g  =m \\ g\textrm{ squarefree} }}  \frac{N_{g,n} (\log L -\delta)}{ q^{n-m} \phi(g) }  \geq  \max_{ \substack{ m \in \mathbb N, g \in \mathbb F_q[t]^+ \\ m \leq \theta n \\ \deg g  =m \\ g\textrm{ squarefree} }}\frac{N_{g,n} \left(\log L -\frac{1}{\sqrt{n-m}} \right)}{ q^{n-m} \phi(g) }\]
for all $n$ sufficiently large, so that
\[ \lim_{n \to \infty} \max_{ \substack{ m \in \mathbb N, g \in \mathbb F_q[t]^+ \\ m \leq \theta n \\ \deg g  =m \\ g\textrm{ squarefree} }}  \frac{N_{g,n} (\log L -\delta)}{ q^{n-m} \phi(g) }  \geq \lim\sup_{n\to\infty}  \max_{ \substack{ m \in \mathbb N, g \in \mathbb F_q[t]^+ \\ m \leq \theta n \\ \deg g  =m \\ g\textrm{ squarefree} }} \frac{N_{g,n} \left(\log L -\frac{1}{\sqrt{n-m} }\right)}{ q^{n-m} \phi(g) }\]
for all $\delta>0$, and thus
\[  \lim_{ \delta \to 0}  \lim_{n \to \infty} \max_{ \substack{ m \in \mathbb N, g \in \mathbb F_q[t]^+ \\ m \leq \theta n \\ \deg g  =m \\ g\textrm{ squarefree} }}  \frac{N_{g,n} (\log L -\delta)}{ q^{n-m} \phi(g) }  \geq \lim\sup_{n\to\infty}  \max_{ \substack{ m \in \mathbb N, g \in \mathbb F_q[t]^+ \\ m \leq \theta n \\ \deg g  =m \\ g\textrm{ squarefree} }}\frac{N_{g,n} \left(\log L -\frac{1}{\sqrt{n-m} }\right)}{ q^{n-m} \phi(g) } , \]
which together with \eqref{delta-n-limit} verifies \eqref{final-entropy}.

%It suffices to show that for all $\epsilon$, for $n,m$ sufficiently large and for $\delta$ sufficiently small,
%\[ \frac{\abs{ \{ f \in \mathcal M_n \mid \gcd(f,g)=1, H(\omega_f) > \log (L) -\delta \} }}{ q^{n-m} \phi(g) } \leq  \epsilon  + \frac{1}{ \sqrt{2\pi} (4+e) e^{ (4+e) } } \frac {  L^{9/2}   e^{2L} }{  L^L}  \]
%(because the left side is an increasing function of $\delta$). 
%
%Applying \cref{stronger-convergence}, it suffices to show that 
%\[ \lim_{ \delta \to 0 }  \mathbb P (  H> \log (L) -\delta )  \leq   \frac{1}{ \sqrt{2\pi} (4+e) e^{ (4+e) } } \frac {  L^{9/2}   e^{2L} }{  L^L}  \]
%for $H$ the entropy of a GEM random variable.  By the assumption on $q$, we have $L>7$, and so this follows from \cref{dp-bound} and the fact that $ \lim_{ \delta \to 0 }  \mathbb P (  H> \log (L) -\delta )  = \mathbb P (H \geq \log(L))$ for any random variable $H$.
\end{proof}

\bibliographystyle{plainnat}
 \bibliography{references}

\end{document}